\documentclass{amsart}

\usepackage{amsmath}
\usepackage{amssymb}
\usepackage{amsthm}
\usepackage{arydshln}
\usepackage{bbm}
\usepackage{bm}
\usepackage{colonequals}
\usepackage{datetime2}
\usepackage{mathrsfs}
\usepackage{mathtools}
\usepackage{stmaryrd}
\usepackage[all]{xy}
\usepackage{yhmath}

\usepackage{imakeidx}
\makeindex[name=symbols,title=List of Symbols,columns=2]
\newcommand{\symdef}[2]{\index[symbols]{#1@#2}}

\usepackage[dvipsnames]{xcolor}
\usepackage[
  bookmarks=true,
  bookmarksnumbered=true,
  colorlinks=true,
  linkcolor=RoyalBlue,  
  citecolor=ForestGreen,
  urlcolor=Magenta,     
  linktoc=all
]{hyperref}

\theoremstyle{plain}
\newtheorem{thm}{Theorem}[section]
\newtheorem*{thm*}{Theorem}
\newtheorem{prop}[thm]{Proposition}
\newtheorem{lem}[thm]{Lemma}
\newtheorem{cor}[thm]{Corollary}

\newtheorem{expect}[thm]{Expectation}

\theoremstyle{definition}
\newtheorem{defn}[thm]{Definition}

\theoremstyle{remark}
\newtheorem{rem}[thm]{Remark}

\newtheorem*{claim*}{Claim}

\newcommand{\ad}{\mathrm{ad}}
\newcommand{\asym}{\mathrm{asym}}

\newcommand{\desc}{\mathrm{desc}}
\newcommand{\new}{\mathrm{new}}
\newcommand{\pr}{\mathrm{pr}}
\newcommand{\ram}{\mathrm{ram}}
\newcommand{\red}{\mathrm{red}}
\newcommand{\res}{\mathrm{res}}
\newcommand{\rs}{\mathrm{rs}}
\renewcommand{\sc}{\mathrm{sc}}
\newcommand{\spec}{\mathrm{spec}}
\newcommand{\spl}{\mathbf{spl}}
\newcommand{\srs}{\mathrm{srs}}
\renewcommand{\ss}{\mathrm{ss}}
\newcommand{\st}{\mathrm{st}}
\newcommand{\supp}{\mathrm{supp}}
\newcommand{\sym}{\mathrm{sym}}
\newcommand{\ur}{\mathrm{ur}}

\newcommand{\Art}{\mathrm{Art}}
\newcommand{\KS}{\mathrm{KS}}

\newcommand{\TN}{\mathrm{TN}}
\newcommand{\Wal}{\mathrm{Wal}}

\newcommand{\I}{\mathrm{I}}
\newcommand{\II}{\mathrm{II}}
\newcommand{\III}{\mathrm{III}}
\newcommand{\IV}{\mathrm{IV}}

\newcommand{\C}{\mathbb{C}}
\newcommand{\F}{\mathbb{F}}
\newcommand{\Q}{\mathbb{Q}}
\newcommand{\R}{\mathbb{R}}
\newcommand{\Z}{\mathbb{Z}}
\newcommand{\Gm}{\mathbb{G}_{\mathrm{m}}}

\newcommand{\bbD}{\mathbb{D}}
\newcommand{\bbH}{\mathbb{H}}

\newcommand{\x}{\mathbf{x}}

\newcommand{\bfA}{\mathbf{A}}
\newcommand{\bfB}{\mathbf{B}}
\newcommand{\B}{\mathbf{B}}
\newcommand{\bfD}{\mathbf{D}}
\newcommand{\bfG}{\mathbf{G}}
\newcommand{\G}{\mathbf{G}}
\newcommand{\bfH}{\mathbf{H}}
\renewcommand{\H}{\mathbf{H}}

\newcommand{\J}{\mathbf{J}}
\newcommand{\bfN}{\mathbf{N}}
\newcommand{\bfS}{\mathbf{S}}
\newcommand{\bfT}{\mathbf{T}}
\newcommand{\T}{\mathbf{T}}
\newcommand{\bfU}{\mathbf{U}}
\newcommand{\bfV}{\mathbf{V}}
\newcommand{\bfX}{\mathbf{X}}
\newcommand{\bfZ}{\mathbf{Z}}

\newcommand{\mfg}{\mathfrak{g}}
\newcommand{\mfh}{\mathfrak{h}}
\newcommand{\mfj}{\mathfrak{j}}
\newcommand{\mfp}{\mathfrak{p}}
\newcommand{\mfs}{\mathfrak{s}}
\newcommand{\mft}{\mathfrak{t}}
\newcommand{\mfw}{\mathfrak{w}}
\newcommand{\mfz}{\mathfrak{z}}
\newcommand{\mfG}{\mathfrak{G}}
\newcommand{\mfH}{\mathfrak{H}}
\newcommand{\mfV}{\mathfrak{V}}

\newcommand{\mcA}{\mathcal{A}}
\newcommand{\mcB}{\mathcal{B}}
\newcommand{\mcC}{\mathcal{C}}
\newcommand{\mcD}{\mathcal{D}}

\newcommand{\mcH}{\mathcal{H}}
\newcommand{\mcJ}{\mathcal{J}}

\newcommand{\mcM}{\mathcal{M}}
\newcommand{\mcO}{\mathcal{O}}
\newcommand{\mcS}{\mathcal{S}}
\newcommand{\mcT}{\mathcal{T}}

\newcommand{\bmfg}{\boldsymbol{\mathfrak{g}}}
\newcommand{\bmfh}{\boldsymbol{\mathfrak{h}}}
\newcommand{\bmfj}{\boldsymbol{\mathfrak{j}}}
\newcommand{\bmfs}{\boldsymbol{\mathfrak{s}}}
\newcommand{\bmft}{\boldsymbol{\mathfrak{t}}}
\newcommand{\bmfz}{\boldsymbol{\mathfrak{z}}}

\newcommand{\ol}{\overline}
\newcommand{\ul}{\underline}
\renewcommand{\t}{\tilde}
\newcommand{\h}{\hat}
\renewcommand{\b}{\bar}
\renewcommand{\j}{\jmath}
\renewcommand{\L}{{}^{L}}
\newcommand{\Lj}{{}^{L}\! j}
\renewcommand{\-}{\text{-}}

\newcommand{\dia}{\diamondsuit}
\newcommand{\nat}{\natural}
\newcommand{\mr}{\mathring}

\DeclareMathOperator{\cInd}{c-Ind}

\DeclareMathOperator{\id}{id}
\DeclareMathOperator{\inv}{inv}
\DeclareMathOperator{\ord}{ord}
\DeclareMathOperator{\sgn}{sgn}

\DeclareMathOperator{\tr}{tr}
\DeclareMathOperator{\tran}{\mathfrak{tran}}
\DeclareMathOperator{\val}{val}

\DeclareMathOperator{\Aut}{Aut}

\DeclareMathOperator{\Gal}{Gal}
\DeclareMathOperator{\Hom}{Hom}

\DeclareMathOperator{\Ker}{Ker}
\DeclareMathOperator{\Lie}{Lie}
\DeclareMathOperator{\Nr}{Nr}
\DeclareMathOperator{\Res}{Res}
\DeclareMathOperator{\Tr}{Tr}

\DeclareMathOperator{\GL}{GL}
\DeclareMathOperator{\SL}{SL}
\DeclareMathOperator{\Sp}{Sp}

\title{Twisted endoscopic character relation for toral supercuspidal $L$-packets of classical groups}
\author{Masao Oi}
\address{Department of Mathematics, National Taiwan University, Astronomy Mathematics Building 5F, No.\ 1, Sec.\ 4, Roosevelt Rd., Taipei 10617, Taiwan.}
\email{masaooi@ntu.edu.tw}

\begin{document}

\begin{abstract}
We prove that Kaletha's toral supercuspidal $L$-packets satisfy the twisted endoscopic character relation in some cases, including the case of general linear groups equipped with an involution.
Consequently, we verify that Kaletha's construction of the local Langlands correspondence for toral supercuspidal representations of quasi-split symplectic or special orthogonal groups coincides with Arthur's.
The strategy is to emulate Kaletha's proof of the standard endoscopic character relation in the twisted setting by appealing to Waldspurger's framework ``l'endoscopie tordue n'est pas si tordue''.
\end{abstract}


\maketitle

\begingroup
\renewcommand{\thefootnote}{}
\footnotetext{\DTMnow}
\footnotetext{Note: The material of this paper was originally incorporated as part of \cite{Oi23-TECR} and has been extracted. The remaining part of \cite{Oi23-TECR} is now contained in \cite{Oi25}.}
\endgroup

\setcounter{tocdepth}{2}
\tableofcontents


\section{Introduction}\label{sec:intro}

One fundamental objective in representation theory of reductive groups over local fields is to establish the \textit{local Langlands correspondence}.
For a connected reductive group $\G$ defined over a local field $F$, the local Langlands correspondence is a natural map from the set $\Pi(\G)$ of isomorphism classes of irreducible admissible representations of $G\colonequals \G(F)$ to the set $\Phi(\G)$ of equivalence classes of $L$-parameters of $\G$.
Here, each fiber of the map is expected to be finite; we let $\Pi_{\phi}^{\G}$ denote the fiber at $\phi\in\Phi(\G)$ and call it an \textit{$L$-packet}.
Thus, we may think of the local Langlands correspondence as a natural partition of $\Pi(\G)$ into finite sets labeled by $L$-parameters:
\[
\Pi(\G)=\bigsqcup_{\phi\in\Phi(\G)}\Pi_{\phi}^{\G}.
\]

While the local Langlands correspondence was constructed by Langlands (\cite{Lan89}) when $F$ is archimedean, its existence is still conjectural in general when $F$ is non-archimedean.
However, numerous results have been obtained up to the present.
Let us review some of them in the following by focusing only on the case where $F$ is a $p$-adic field, i.e., a non-archimedean local field of characteristic zero.

Firstly, the local Langlands correspondence has been completely established for several specific groups.
The particularly important examples include the results of Harris--Taylor and Henniart for $\GL_{n}$ (\cite{HT01,Hen00}), Arthur for quasi-split special orthogonal or symplectic groups (\cite{Art13}), and Mok for quasi-split unitary groups (\cite{Mok15}).

On the other hand, one might also attempt to construct the local Langlands correspondence by restricting the class of representations instead of the class of groups.
An important example of this direction is the work of DeBacker--Reeder (\cite{DR09}), which established the local Langlands correspondence for depth zero regular supercuspidal representations of arbitrary unramified groups.
After the work of DeBacker--Reeder, Kaletha investigated some particular cases of positive depth supercuspidal representations (\cite{Kal13-wild,Kal15}).
Currently, all these constructions have been uniformly generalized by Kaletha himself to a considerably broad class of supercuspidal representations called \textit{regular} (more generally, \textit{non-singular/semisimple}) supercuspidal representations of tamely ramified connected reductive groups (\cite{Kal19,Kal19-sc}).

Taking into account these two approaches toward the local Langlands correspondence (i.e., the ``vertical'' direction which restricts the class of groups and the ``horizontal'' direction which restricts the class of representations), it is natural to ask whether two different constructions indeed give rise to the identical correspondence on the ``intersection'' of the domains of the constructions.
This problem is not only interesting in itself but also technically important.
For example, the above-mentioned constructions for specific groups have a favorable compatibility with the global classification theory of automorphic representations.
On the other hand, Kaletha's construction is highly explicit since it is ultimately based on the local Langlands correspondence for tori (note that this is parallel to Langlands' construction \cite{Lan89} in the archimedean case).
If we get the coincidence of these different constructions, we can combine their advantages.

Based on this motivation, we first investigated the case of $\GL_{n}$ in a joint work with Tokimoto.
In \cite{OT21}, we proved that Kaletha's construction exactly realizes the local Langlands correspondences established by Harris--Taylor and Henniart for any regular supercuspidal representation of $\GL_{n}(F)$ whenever $p\neq2$ 
(note that now this result has been generalized by Tokimoto to all inner forms of $\GL_{n}$ in \cite{Tok23}).

The aim of this paper is to establish a methodology for comparing Kaletha's construction with others for more general groups.
Especially, we prove the following:
\begin{thm}[Theorem {\ref{thm:Arthur=Kaletha}}]\label{thm:Arthur=Kaletha-intro}
Let $\H$ be a quasi-split special orthogonal or symplectic group over $F$.
Suppose that $p$ is sufficiently large.
The Local Langlands correspondences of Arthur and Kaletha coincide for any ``toral'' supercuspidal representation of $H\colonequals \H(F)$.
\end{thm}

We briefly explain what ``toral'' supercuspidal representations are (see Section \ref{subsec:rsc} for more details).
Let $\G$ be a tamely ramified connected reductive group over $F$ in the following.
In \cite{Yu01}, Yu established an explicit method for producing a broad class of supercuspidal representations, which are called \textit{tame supercuspidal representations}.
Yu's construction associates a tame supercuspidal representation to each tuple $(\vec{\G},\vec\vartheta,\vec{r},\x,\rho_{0})$ called a cuspidal $\G$-datum.
Here, we only recall that $\vec{\G}=(\G^{0}\subsetneq\cdots\subsetneq\G^{d})$ is a sequence of tame Levi subgroups of $\G$ and $\rho_{0}$ is a depth zero cuspidal representation of an open compact-modulo-center subgroup of $\G^{0}$ (hence regarded as a representation of a finite reductive group).
In \cite{Kal19}, by invoking the Deligne--Lusztig theory \cite{DL76}, Kaletha introduced the notion of \textit{regularity} for tame supercuspidal representations and discovered that regular supercuspidal representations can be re-parametrized by much simpler data $(\bfS,\vartheta)$ called \textit{tame elliptic regular pairs}, which consist only of a tame elliptic maximal torus $\bfS$ of $\G$ contained in $\G^{0}$ and a character $\vartheta$ of $S\colonequals \bfS(F)$ satisfying a certain regularity condition.
Based on this re-parametrization, he assigned an $L$-parameter to each regular supercuspidal representation and analyzed the internal structures of the resulting $L$-packets.
\textit{Toral supercuspidal representations} constitute a special class among regular supercuspidal representations; they are tame supercuspidal representations obtained from cuspidal $\G$-data whose $\vec{\G}$ are of the form $(\bfS\subsetneq\G)$.

To obtain Theorem \ref{thm:Arthur=Kaletha-intro}, we verify that Kaletha's toral supercuspidal $L$-packets satisfy the \textit{twisted endoscopic character relation}, which is the characterization of Arthur's correspondence.
Thus we next review the general framework of twisted endoscopy.
Suppose that $\H$ is an endoscopic group for $(\G,\theta)$ in the sense of Kottwitz--Shelstad, where $\theta$ is an $F$-rational pinning-preserving automorphism of $\G$ (see Section \ref{sec:tw-endo}).
In particular, $\H$ is equipped with an $L$-embedding $\h\xi\colon\L\H\hookrightarrow\L\G$, which enables us to regard any $L$-parameter $\phi_{\H}$ of $\H$ as an $L$-parameter of $\G$ (write $\phi$) by composing $\phi_{\H}$ with $\h\xi$.
Suppose that the local Langlands correspondence both for $\G$ and $\H$ are available, so that we may associate $L$-packets $\Pi_{\phi_{\H}}^{\H}$ and $\Pi_{\phi}^{\G}$ to both $\phi_{\H}$ and $\phi$.
Assume that $\phi_{\H}$ and $\phi$ are tempered.
\[
\xymatrix@R=10pt{
\Pi(\G) \supset \Pi_{\phi}^{\G} \quad \ar@{<~>}[rr]^-{\text{LLC for $\G$}} &&& \L\G\\
\Pi(\H) \supset \Pi_{\phi_{\H}}^{\H} \quad \ar@{<~>}[rr]^-{\text{LLC for $\H$}} &&\quad W_F\times\SL_2(\C) \ar[r]_-{\phi_{\H}} \ar[ru]^-{\phi} & \L\H \ar[u]_-{\h\xi} \\
}
\]

In this situation, it is expected that the following holds:

\begin{expect}[Twisted endoscopic character relation]\label{expectation:TECR-intro}
For each $\pi\in\Pi_{\phi}^{\G}$ there exists a constant $\Delta^{\spec}_{\phi,\pi}\in\C$ such that the following identity holds for any strongly regular semisimple $\delta\in\t{\G}(F)$:
\begin{align}\label{eq:TECR-intro}
\sum_{\pi\in\Pi_{\phi}^{\G}}
\Delta^{\spec}_{\phi,\pi}\cdot\Phi_{\t{\pi}}(\delta)
=
\sum_{\gamma\in H/{\st}} \mr{\Delta}(\gamma,\delta)
\sum_{\pi_{\H}\in\Pi_{\phi_{\H}}^{\H}}\Phi_{\pi_{\H}}(\gamma).
\end{align}
\end{expect}
Here,
\begin{itemize}
\item
$\t{\G}$ denotes the twisted space determined by $\G$ and $\theta$, i.e., the connected component of the semi-direct product group $\G\rtimes\langle\theta\rangle$ containing $1\rtimes\theta$ (see Section \ref{subsec:twisted-spaces});
\item
$\Phi_{\pi_{\H}}$ is the normalized (Harish-Chandra) character of $\pi_{\H}\in\Pi_{\phi_{\H}}^{\H}$ and $\Phi_{\t\pi}$ is the normalized twisted character of $\pi\in\Pi_{\phi}^{\G}$ (see Section \ref{subsec:TCF}), which is defined only when $\pi$ is $\theta$-stable (thus the coefficient $\Delta^{\spec}_{\phi,\pi}$ is zero unless $\pi$ is $\theta$-stable);
\item
the sum on the right-hand side is over the stable conjugacy classes of norms of $\delta$ in $H$ in the sense of twisted endoscopy (see Section \ref{subsec:norm});
\item
$\mr{\Delta}$ on the right-hand side is the transfer factor of Kottwitz--Shelstad (see Section \ref{subsec:tran}).
\end{itemize}
We note that, by linear independence of twisted characters, a family $\{\Delta^{\spec}_{\phi,\pi}\}_{\pi\in\Pi_{\phi}^{\G}}$ of constants as above is unique if exists.

In the untwisted case (i.e., $\theta$ is trivial), Kaletha proved that Expectation \ref{expectation:TECR-intro} is indeed true for toral supercuspidal $L$-packets under some assumptions on $p$ (\cite[Theorem 6.3.4]{Kal19}; see \cite[Section 4.4]{FKS23} for a more general result in the untwisted situation).

The point is that when $\H$ is a quasi-split special orthogonal or symplectic group, we can find a general linear group $\G=\GL_{n}$ with an involution $\theta$ such that $\H$ is an endoscopic group for $(\G,\theta)$.
The expected identity \eqref{eq:TECR-intro} then tells us the sum of characters of representations in each $L$-packet $\Pi_{\phi_{\H}}^{\H}$ of $\H$ in terms of the twisted characters of representations of $\GL_{n}(F)$; this information is enough to characterize $\Pi_{\phi_{\H}}^{\H}$ as a finite set of representations by linear independence of characters.
What Arthur did is to prove that there indeed exists a finite set $\Pi_{\phi_{\H}}^{\H}$ for each $\phi_{\H}$ satisfying the identity \eqref{eq:TECR-intro} with $\Pi_{\phi}^{\G}$ which is determined by the local Langlands correspondence for $\GL_{n}$.
Therefore, as we already know the coincidence of Kaletha's construction with the one by Harris--Taylor and Henniart, it is enough to verify the twisted endoscopic character relation for Kaletha's $L$-packets in order to obtain the coincidence of Kaletha's and Arthur's constructions.

The main result of this paper is as follows:

\begin{thm}[{Theorems \ref{thm:spec-well-def-GL}, \ref{thm:spec-well-def-unram}}]\label{thm:main-intro}
Suppose that $p$ is sufficiently large.
Kaletha's toral supercuspidal $L$-packets satisfy Expectation \ref{expectation:TECR-intro} in the following cases:
\begin{enumerate}
\item
$\G=\GL_{n}$ or
\item
$\G$ is general, $\theta$ is involutive, and toral supercuspidal $L$-packets arise from a torus splitting over a finite extension $E/F$ with odd ramification index.
\end{enumerate}
\end{thm}

We explain the outline of our proof of Theorem \ref{thm:main-intro} in the following.
Our strategy is quite simple in some sense; we reproduce Kaletha's proof of the standard (untwisted) endoscopic character relation while taking into account the effect of the twist $\theta$.
Thus let us first review Kaletha's proof in the untwisted setting briefly.

The starting point of Kaletha's proof is an explicit formula of the characters of tame supercuspidal representations due to Adler--DeBacker--Spice (\cite{AS09,DS18}).
In the toral setting, it is as follows.
Let $\pi_{(\bfS,\vartheta)}$ be the toral supercuspidal representation of $G$ arising from a tame elliptic regular pair $(\bfS,\vartheta)$.
Let $r\in\R_{>0}$ be the depth of the character $\vartheta$ and $X^{\ast}\in(\Lie\bfS)^{\ast}(F)$ be an element representing the restriction of $\vartheta$ to the depth $r$ subgroup $S_{r}$ of $S$ (see Section \ref{subsec:TCF} for details).
Let $\delta\in G$ be any elliptic regular semisimple element.
Then, by the theory of Adler--Spice \cite{AS08}, we can take a \textit{normal $r$-approximation} $\delta=\delta_{<r}\cdot\delta_{\geq r}$.
Roughly speaking, this is a nice product decomposition of $\delta$ into a part ``$p$-adically shallower than $r$'' and a part $``p$-adically deeper than or equal to $r$''.
One of its important properties is that two parts $\delta_{<r}$ and $\delta_{\geq r}$ commute; even more strongly, the deeper part $\delta_{\geq r}$ belongs to the connected centralizer $\G_{\delta_{<r}}\colonequals \bfZ_{\G}(\delta_{<r})^{\circ}$ of the shallower part.
The Adler--DeBacker--Spice formula is expressed by using a normal $r$-approximation as follows (the symbol ${}^{g}(-)$ denotes the conjugate $g(-)g^{-1}$):
\begin{align}\label{eq:ADS}
\Phi_{\pi_{(\bfS,\vartheta)}}(\delta)
=
\sum_{\begin{subarray}{c}g\in S\backslash G/G_{\delta_{<r}} \\ {}^{g}\delta_{<r}\in S\end{subarray}} \varepsilon({}^{g}\delta_{<r})\cdot\vartheta({}^{g}\delta_{<r})\cdot\h\iota^{\G_{\delta_{<r}}}_{g^{-1}X^{\ast}g}(\log(\delta_{\geq r})),
\end{align}
where $\varepsilon({}^{g}\delta_{<r})$ is a certain root of unity determined by ${}^{g}\delta_{<r}$, and $\h\iota^{\G_{\delta_{<r}}}_{g^{-1}X^{\ast}g}(-)$ denotes the normalized Fourier transform of the orbital integral with respect to $g^{-1}X^{\ast}g$ taken in the (Lie algebra of) $\G_{\delta_{<r}}$.
The important observation here is that $\h\iota^{\G_{\delta_{<r}}}_{g^{-1}X^{\ast}g}(-)$, which is nothing but the Lie algebra analogue of the Harish-Chandra characters of representations, is used to express the contribution of the deeper part $\delta_{\geq r}$.
In fact, by the theorem of Waldspurger and Ng\^o (\cite{Wal06,Ngo10}; see Section \ref{subsec:Lie-tran}), we can compare the Fourier transforms of Lie algebra orbital integrals between any group and its standard endoscopic group (this can be thought of as a Lie algebra analogue of the standard endoscopic character relation).
The principal idea of Kaletha's strategy is to reduce the standard endoscopic character relation to the Lie algebra transfer theorem of Waldspurger--Ng\^o through the Adler--DeBacker--Spice character formula.

Now let $\H$ be a standard (untwisted) endoscopic group of $\G$ and suppose that both $\Pi_{\phi_{\H}}^{\H}$ and $\Pi_{\phi}^{\G}$ consist of toral supercuspidal representations (of depth $r\in\R_{>0}$).
When $\gamma\in H$ is a norm of $\delta\in G$, we may transfer a normal $r$-approximation $\delta=\delta_{<r}\cdot\delta_{\geq r}$ to $\gamma=\gamma_{<r}\cdot\gamma_{\geq r}$.
Therefore, by applying the Adler--DeBacker--Spice formula to all the characters of representations in $\Pi_{\phi}^{\G}$ and $\Pi_{\phi_{H}}^{\H}$ with respect to these $r$-approximations, the $\G$-side and the $\H$-side of \eqref{eq:TECR-intro} are rewritten as follows:
\begin{align}\label{eq:G-side-intro}
\sum_{\pi_{(\bfS,\vartheta)}\in\Pi_{\phi}^{\G}}
\Delta_{\phi,\pi_{(\bfS,\vartheta)}}^{\spec}
\sum_{\begin{subarray}{c}g\in S\backslash G/G_{\delta_{<r}} \\ {}^{g}\delta_{<r}\in S\end{subarray}} \varepsilon({}^{g}\delta_{<r})\cdot\vartheta({}^{g}\delta_{<r})\cdot\h\iota^{\G_{\delta_{<r}}}_{g^{-1}X^{\ast}g}(\log(\delta_{\geq r})),
\end{align}
\begin{multline}\label{eq:H-side-intro}
\sum_{\gamma\in H/{\st}} \mr{\Delta}(\gamma,\delta)
\!\!\!\!\!\!
\sum_{\pi_{(\bfS_{\H},\vartheta_{\H})}\in\Pi_{\phi_{\H}}^{\H}}
\sum_{\begin{subarray}{c}h\in S_{\H}\backslash H/H_{\gamma_{<r}} \\ {}^{h}\gamma_{<r}\in S_{\H}\end{subarray}}
\!\!\!\!\!\!
 \varepsilon({}^{h}\gamma_{<r})\cdot\vartheta_{\H}({}^{h}\gamma_{<r})\cdot\h\iota^{\H_{\gamma_{<r}}}_{h^{-1}X_{\H}^{\ast}h}(\log(\gamma_{\geq r})).
\end{multline}

Note that the Lie algebra orbital integrals are taken not in $\G$ and $\H$ but in the ``descended'' groups $\G_{\delta_{<r}}$ and $\H_{\gamma_{<r}}$.
The crucially important ingredient here is the theory of \textit{descent for standard endoscopy} due to Langlands--Shelstad \cite{LS90}, which guarantees that the group $\H_{\gamma_{<r}}$ again has a structure of a standard endoscopic group of $\G_{\delta_{<r}}$.
Moreover, the transfer factor for the pair $(\G,\H)$ is related to that of the descended pair $(\G_{\delta_{<r}},\H_{\gamma_{<r}})$.
Then the basic setup for utilizing the Lie algebra transfer is done.
\[
\xymatrix@R=10pt{
\G\ar@{-}_{\text{standard endoscopy}}[d]\ar@{~>}^{\text{descent}}[rr]&&\G_{\delta_{<r}}\ar@{-}^-{\text{standard endoscopy}}[d]\\
\H\ar@{~>}^{\text{descent}}[rr]&&\H_{\gamma_{<r}}
}
\]

However, there are still several subtle points remaining.
Firstly, before thinking about comparing the summands of \eqref{eq:G-side-intro} and \eqref{eq:H-side-intro} via the Lie algebra transfer, we must investigate how the index sets of those sums can be compared.
Another related task is to rewrite both sides \eqref{eq:G-side-intro} and \eqref{eq:H-side-intro} in a way such that the Fourier transforms of orbital integrals are summed up over rational conjugacy classes within a stable conjugacy class (in $\G_{\delta_{<r}}$ or $\H_{\gamma_{<r}}$), so that the Lie algebra transfer can be applied.
Kaletha resolved these issues by an ingenious trick of rearranging the sums.
By construction, the members of $\Pi_{\phi}^{\G}$ are labeled by the rational conjugacy classes within the stable conjugacy class of \textit{admissible embeddings} (see Definition \ref{defn:adm-emb}) of $\bfS$ into $\G$.
The second index set of \eqref{eq:G-side-intro} can be thought of as a set measuring the difference between the rational conjugacy in $\G$ and that in $\G_{\delta_{<r}}$.
Hence the double sums in \eqref{eq:G-side-intro} are combined into a single sum over $G_{\delta_{<r}}$-conjugacy classes within the stable $\G$-conjugacy class of admissible embeddings.
If we again partition this sum based on the stable $\G_{\delta_{<r}}$-conjugacy, we can obtain a sum over the desired index set, i.e., $G_{\delta_{<r}}$-conjugacy classes within the stable $\G_{\delta_{<r}}$-conjugacy class.
The same argument can be also applied to the $\H$-side \eqref{eq:H-side-intro}.
Then, by utilizing the ``\textit{descent lemma}'', which was established in \cite[Section 5.4]{Kal15}, we can relate the sum over stable $\G_{\delta_{<r}}$-conjugacy classes within the stable $\G$-conjugacy class to that over $\H_{\gamma_{<r}}$-conjugacy classes within the stable $\H$-conjugacy class.
\[
\xymatrix@R=0pt{
&\G\ar@{-}_-{\text{$\Pi_{\phi}^{\G}$}}[ld]\ar@{-}^-{\qquad\text{descent lem.}}[rd]&&&&\H\ar@{-}_-{\text{$\Pi_{\phi_{\H}}^{\H}$}}[ld]\ar@{-}^-{\qquad\text{descent lem.}}[rd]&\\
G\ar@{-}_-{\text{ADS formula}\quad}[rd]&&\G_{\delta_{<r}}\ar@{-}^-{\qquad\text{Lie alg.\ trans.}}[ld]&&H\ar@{-}_-{\text{ADS formula}\quad}[rd]&&\H_{\gamma_{<r}}\ar@{-}^-{\qquad\text{Lie alg.\ trans.}}[ld]\\
&G_{\delta_{<r}}&&&&H_{\gamma_{<r}}&
}
\]

Secondly, we also have to relate the roots of unity ``$\varepsilon$'' in the summands in \eqref{eq:G-side-intro} and \eqref{eq:H-side-intro}.
These factors are explicitly computed in \cite{AS09,DS18}; they reflect the symmetry of the root system $\Phi(\G,\bfS)$, which is a finite set equipped with a (typically, highly nontrivial) Galois action.
Kaletha first showed that this part can be re-interpreted in terms of several invariants having a more ``endoscopic'' nature such as the second transfer factor $\Delta_{\II}$ (\cite[Corollary 4.8.2]{Kal19}).
Then, by computing the transfer factor $\mr{\Delta}(\gamma,\delta)$ explicitly and also by utilizing various nontrivial results on the arithmetic invariants such as local root numbers and Weil constants, he eventually proved that all of these subtle quantities perfectly fit together. 

Now, let us move on to the twisted situation.
In the following, we let $\theta$ be a nontrivial $F$-rational involution of $\G$ and $\H$ is an endoscopic group of $(\G,\theta)$.
Our first task is to establish a twisted version of the Adler--DeBacker--Spice formula.
For this, we have to start with investigating a twisted version of the notion of a normal $r$-approximation because the Adler--DeBacker--Spice formula is based on it.
These are studied in our preceding work \cite{Oi25}.
More precisely, we first established a definition/construction of a ``normal $r$-approximation'' in the setting of twisted spaces, and then established an explicit formula for twisted characters of toral supercuspidal representations, which is completely parallel to the Adler--DeBacker--Spice formula as follows (see \cite[Theorem 7.3]{Oi25} and also Theorem \ref{thm:TCF-final}):
\begin{align}\label{eq:ADS-twisted}
\Phi_{\t\pi_{(\bfS,\vartheta)}}(\delta)
=
\sum_{\begin{subarray}{c}g\in S\backslash G/G_{\delta_{<r}} \\ {}^{g}\delta_{<r}\in \t{S}\end{subarray}} \t\varepsilon({}^{g}\delta_{<r})\cdot\t\vartheta({}^{g}\delta_{<r})\cdot\h\iota^{\G_{\delta_{<r}}}_{g^{-1}X^{\ast}g}(\log(\delta_{\geq r})).
\end{align}
Here, we put ``$\sim$'' on the symbols to indicate that these are quantities determined in this twisted context.
(Although we wrote the above formula \eqref{eq:ADS-twisted} in a way parallel to \eqref{eq:ADS}, the actual expression of $\t\varepsilon$ is much more complicated than the untwisted case; see Proposition \ref{prop:TCF-normal} for the details).

By looking at the formula \eqref{eq:ADS-twisted}, we notice that the contribution of the deeper part $\delta_{\geq r}$ is expressed via the Fourier transform of a Lie algebra orbital integral with respect to the descended group $\G_{\delta_{<r}}$ as well as in the untwisted case.
Therefore, one might expect that the same strategy again works in this twisted setting.
Unfortunately, in the twisted setting, it is not always the case that $\H_{\gamma_{<r}}$ has a structure of an endoscopic group of $\G_{\delta_{<r}}$.
Nevertheless, it is still possible to relate $\H_{\gamma_{<r}}$ to $\G_{\delta_{<r}}$ by introducing another variant of the notion of standard endoscopy called \textit{non-standard endoscopy}.
More precisely, there exists a group $\bar{\H}$ such that $\bar{\H}$ is a standard endoscopic group of the simply-connected cover of the derived subgroup of $\G_{\delta_{<r}}$ and also that the simply-connected covers of the derived subgroups of $\bar{\H}$ and $\H_{\gamma_{<r}}$ form a non-standard endoscopic pair.
Furthermore, the Lie algebra transfer for Fourier transforms of orbital integrals is also available for the non-standard endoscopic pair.
This is the framework ``\textit{l'endoscopie tordue n'est pas si tordue}'' established by Waldspurger (\cite{Wal08}).
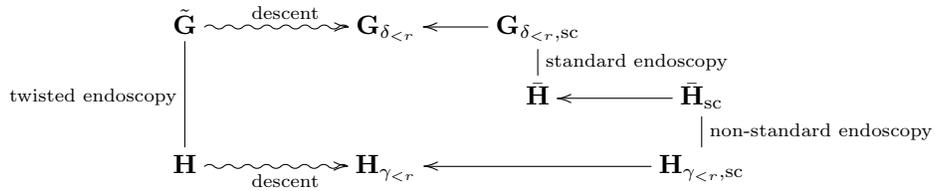
\[
\xymatrix@R=10pt{
\t\G\ar@{-}_{\text{twisted endoscopy}}[dd]\ar@{~>}^{\text{descent}}[rr]&&\G_{\delta_{<r}}&\G_{\delta_{<r},\sc}\ar[l]\ar@{-}^-{\text{standard endoscopy}}[d]&\\
&&&\bar{\H}&\bar{\H}_{\sc}\ar[l]\ar@{-}^-{\text{non-standard endoscopy}}[d]\\
\H\ar@{~>}_{\text{descent}}[rr]&&\H_{\gamma_{<r}}&&\H_{\gamma_{<r},\sc}\ar[ll]
}
\]

Now we gained the right to attempt to mimic Kaletha's proof.
Let us discuss the rearranging argument on the index sets of the sums in the twisted endoscopic character relation.
The first difficulty is that only $\theta$-stable members of $\Pi_{\phi}^{\G}$ contribute to the twisted endoscopic character relation.
Hence we must clarify the $\theta$-stability condition in terms of admissible embeddings, which parametrize the members of $\Pi_{\phi}^{\G}$.
We deal with this issue by examining the notion of a \textit{twisted maximal torus}.
The second difficulty is that Kaletha's descent lemma, which is necessary for the index sets comparison, needs a major modification.
The idea of the descent lemma in the untwisted case is to utilize  \textit{admissible isomorphisms}, which are $F$-rational isomorphisms between maximal tori of $\G$ and those of $\H$.
For a given $F$-rational admissible embedding of a maximal torus $\bfS_{\H}$ into $\H$, by composing it with an admissible isomorphism between $\bfS_{\H}$ and a maximal torus (say $\bfS$) in $\G$, we may produce an $F$-rational admissible embedding of $\bfS$ into $\G$.
However, this construction no longer works in the twisted setting because an admissible isomorphism in twisted endoscopy is an $F$-rational isomorphism between an $F$-rational maximal torus of $\H$ and the coinvariant (with respect to the ``twist'') of a maximal torus of $\G$.
To resolve this issue, we utilize the notion of a \textit{diagram} introduced by Waldspurger (see Definition \ref{defn:diag}).
A diagram induces an admissible isomorphism, but also encapsulates more information.
Hence it can be thought of as a ``rigidification'' of an admissible isomorphism.
Using diagrams instead of admissible isomorphisms, we can reproduce Kaletha's descent lemma in the twisted setting.

We next discuss comparing the roots of unity appearing in the $\G$-side with those in the $\H$-side.
Our basic strategy is to unravel various arithmetic or root-theoretic invariants via similar techniques to the untwisted case.
However, somehow our computation left us with a very complicated quantity as a ratio of a summand in the $\G$-side to that in the $\H$-side (see \eqref{spectral} and also \eqref{eq:spectral-tf-final}).
What we do is, rather than trying to express this ratio more explicitly, just defining the coefficient ``$\Delta_{\phi,\pi}^{\spec}$'' in the endoscopic character relation to be exactly this ratio.
Then the endoscopic character relation holds almost tautologically.
Instead, the well-definedness of $\Delta_{\phi,\pi}^{\spec}$ becomes quite nontrivial; a priori $\Delta_{\phi,\pi}^{\spec}$ heavily depends on the elliptic regular semisimple element $\delta\in\t{G}$ taken at the beginning.
Thus what we do next is to show that $\Delta_{\phi,\pi}^{\spec}$ is in fact independent of the choice of $\delta$.
By examining each factor involved in $\Delta_{\phi,\pi}^{\spec}$, we see that it is enough to check that the quantity \eqref{ramified-spec}, which is much simpler than \eqref{spectral}, is independent of $\delta$.
In fact, every factor appearing in \eqref{spectral} is related to ramified symmetric roots contained in the restricted root system (in the sense of Kottwitz--Shelstad; see Section \ref{subsec:twisted-tori}) of a maximal torus $\bfS$ of $\G$ associated to a $\theta$-stable toral supercuspidal representation $\pi$.
Therefore, if the restricted root system does not contain any ramified symmetric root, then \eqref{ramified-spec} is trivial.
This is how we obtained Theorem \ref{thm:main-intro} (2).
What we eventually verified is that \eqref{ramified-spec} is trivial also when $\G=\GL_{n}$; this is achieved by explicitly classifying the possible Galois actions on (restricted) root systems of $\GL_{n}$.

Let us finish this introduction by giving several concluding remarks.
We believe that it is possible to generalize our results in various directions.
The artificial assumption on $\G$ or $\bfS$ of Theorem \ref{thm:main-intro} stems only from the last part of the proof, which is about the well-definedness of $\Delta_{\phi,\pi}^{\spec}$.
Probably this part can be dealt with in general by a case-by-case computation based on a classification of twisted endoscopy; cf.\ \cite[Chapitre 14--18]{Wal08}.
We expect that it is also possible to establish a depth-zero version of our result by replacing the Adler--DeBacker--Spice character formula with the one of DeBacker--Reeder \cite{DR09}.
It is a natural problem to extend our result to the case of general regular (or even non-singular) supercuspidal representations, but it should require twisting the recent work of Spice \cite{Spi18,Spi21}, which are quite deep.

We finally would like to emphasize that our arguments are also inspired by Mezo's proof of the twisted endoscopic character relation for discrete series $L$-packets of real reductive groups (cf.\ \cite{Mez13}).
We think that our constant $\Delta_{\phi,\pi}^{\spec}$ is nothing but the $p$-adic version of what is called the \textit{spectral transfer factor} in the archimedean setting; see Chapter 1 and also (115)--(117) of \cite{Mez13}.

\medbreak
\noindent{\bfseries Organization of this paper.}\quad
In Section \ref{sec:notation}, we list our basic notation.
In Section \ref{sec:LLC}, we review Kaletha's construction of the local Langlands correspondence for regular supercuspidal representations.
In Section \ref{sec:TCF}, we review our preceding result \cite{Oi25} on the twisted character formula for toral supercuspidal representations.
In Section \ref{sec:tw-endo}, we review the framework of twisted endoscopy.
In Section \ref{sec:theta-stable}, we investigate the structure of a $\theta$-stable regular supercuspidal $L$-packet.
In Section \ref{sec:descent}, we examine the notion of a diagram and establish a twisted version of Kaletha's descent lemma.
In Section \ref{sec:Waldspurger}, we briefly review Waldspurger's framework.
In Sections \ref{sec:res-toral-inv} and \ref{sec:transfer-factor}, we prove some technical lemmas needed in the computation of the spectral transfer factor.
In Section \ref{sec:TECR}, we compare the $\G$-side and the $\H$-side of the endoscopic character relation.
Especially, we introduce the spectral transfer factor.
In Section \ref{sec:GL}, we prove that the spectral transfer factor is well-defined in the $\GL_{n}$ case, which implies that the twisted endoscopic character relation holds for toral supercuspidal $L$-packets of twisted $\GL_{n}$.

\medbreak
\noindent{\bfseries Acknowledgment.}\quad
This project was started when I was visiting Tasho Kaletha at the University of Michigan in 2018 winter as a part of the Program for Leading Graduate Schools, MEXT.
Since then, he has kindly responded to a lot of my questions and also has provided a number of valuable advice and encouragement.
Atsushi Ichino and Yoichi Mieda have dedicated a considerable amount of time to have discussions with me on this project.
I would like to express my deep gratitude, particularly to these three individuals.
I also would like to thank Jeffrey Adler, Alexander Bertoloni Meli, Charlotte Chan, Guy Henniart, Tamotsu Ikeda, Wen-Wei Li, Sug Woo Shin, Loren Spice, Kazuki Tokimoto, Sandeep Varma, and Alex Youcis for very helpful discussions on the content of this paper.

This work was carried out with the support from the Program for Leading Graduate Schools MEXT, JSPS Research Fellowship for Young Scientists, and KAKENHI Grant Number 17J05451, 19J00846, 20K14287.
This work was also supported by the Yushan Young Fellow Program, Ministry of Education, Taiwan.
I am grateful to the Hakubi Center of Kyoto University for providing a wonderful research environment between 2019--2024.


\section{Notation and assumptions on $p$}\label{sec:notation}

Here we summarize our basic notation and assumptions on $p$.
(See also the list of symbols at the end of this paper.)

\subsection{Notation}

\subsubsection{$p$-adic fields}
Let $p$ be a prime number.
We fix a $p$-adic field $F$, i.e., $F$ is a finite extension of $\Q_{p}$.
We fix an algebraic closure $\overline{F}$ of $F$ and always work within $\overline{F}$ when we consider algebraic extensions of $F$.
For any finite extension $E$ of $F$, we write $\mcO_{E}$, $\mfp_{E}$, $k_{E}$, and $\Gamma_{E}$ for the ring of integers of $E$, the maximal ideal of $\mcO_{E}$, the residue field, and the absolute Galois group $\Gal(\overline{F}/E)$ of $E$, respectively.
When $E=F$, we omit the subscript $F$ from these symbols. 
For any finite extension $E$ of $F$, we write $W_{E}$, $I_{E}$, and $P_{E}$ for the Weil group of $E$, its inertia subgroup, and its wild inertia subgroup, respectively.
For any $r\in\R_{>0}$, let $I_{F}^{r}$ denote the $r$-th upper ramification filtration of $I_{F}$.
We fix a valuation $\val_{F}$ of $F$ such that $\val_{F}(F^\times)=\Z$.
We extend it to $\overline{F}$ and again write $\val_{F}$ for it.
We define an absolute value $|\cdot|_{\overline{F}}$ of $\overline{F}$ by $|\cdot|_{\overline{F}}\colonequals q^{-\val_{F}(\cdot)}$, where $q$ is the cardinality of the residue field $k_{F}$.

We fix an additive character $\psi_{F}\colon F\rightarrow\C^{\times}$ of level $1$, i.e., $\psi_{F}|_{\mfp}\equiv\mathbbm{1}$ and $\psi_{F}|_{\mcO}\not\equiv\mathbbm{1}$.

\subsubsection{Algebraic varieties and algebraic groups}
In this paper, we use a bold letter for an algebraic variety and use an italic letter for the set of its $F$-valued points when it is defined over $F$.
For example, if $\bfX$ is an algebraic variety defined over $F$, then $X\colonequals \bfX(F)$.

For any algebraic group $\G$, we write $X^{\ast}(\G)$ and $X_{\ast}(\G)$ for the group of characters and cocharacters of $\G$, respectively.
We let $\bfZ_{\G}$ denote the center of $\G$.
When $\G$ is defined over $F$, so is $\bfZ_{\G}$ and the set of its $F$-valued points is denoted by $Z_{\G}$.

For any torus $\bfS$ equipped with an automorphism $\theta_{\bfS}$, we let $\bfS^{\theta_{\bfS}}$ and $\bfS_{\theta_{\bfS}}$ denote the invariant and coinvariant of $\bfS$ with respect to $\theta_{\bfS}$, respectively.

\subsubsection{Centralizers and normalizers}
Suppose that $\G$ is an algebraic group and $\bfX$ is an algebraic variety equipped with left and right actions of $\G$, for which we write $\G\times\bfX\times\G\rightarrow\bfX\colon(g_{1}, x, g_{2})\mapsto g_{1}\cdot x\cdot g_{2}$.
Then we define the conjugate action of $\G$ on $\bfX$ by $\G\times\bfX\rightarrow\bfX\colon (g,x)\mapsto g\cdot x\cdot g^{-1}$.
We introduce the following notation:
\begin{itemize}
\item
For $g\in\G$, let $[g]$ denote the conjugation automorphism $\bfX\rightarrow\bfX\colon x\mapsto g\cdot x\cdot g^{-1}$.
We also often write ${}^{g}x\colonequals [g](x)=g\cdot x\cdot g^{-1}$.
\item
For $x\in\bfX$, let $\G^{x}$ denote the full stabilizer of $x$ in $\G$ with respect to the conjugate action, i.e., $\G^{x}\colonequals \{g\in\G\mid [g](x)=x\}$.
\item
For $x\in\bfX$, let $\G_{x}$ denote the connected stabilizer of $x$ in $\G$ with respect to the conjugate action, i.e., $\G_{x}\colonequals \G^{x,\circ}$.
\end{itemize}
Note that, when $\G$ and $\bfX$ are $F$-rational, $[g]$ is also $F$-rational if $g\in \G(F)$.
Similarly, $\G^{x}$ and $\G_{x}$ are $F$-rational if $x\in\bfX(F)$.

For any subsets $\bfH\subset \bfG$ and $Y\subset \bfX$, we put 
\begin{itemize}
\item
$\bfZ_{\bfH}(Y)\colonequals \{g\in\bfH\mid \text{$[g](y)=y$ for any $y\in Y$}\}$ and 
\item
$\bfN_{\bfH}(Y)\colonequals \{g\in\bfH\mid [g](Y)\subset Y\}$.
\end{itemize}
When $Y$ is a singleton $\{y\}$, we simply write $\bfZ_{\bfH}(y)\colonequals \bfZ_{\bfH}(Y)=\bfN_{\bfH}(Y)$.
If $\G$ and $\bfX$ are defined over $F$ and $Y$ is a subset of $\bfX(F)$, then $\bfZ_{\G}(Y)$ and $\bfN_{\G}(Y)$ are defined over $F$ and the sets of their $F$-valued points are denoted by $Z_{\G}(Y)$ and $N_{\G}(Y)$, respectively.

\subsubsection{Reductive groups}
For any connected reductive group $\G$ and its maximal torus $\bfS$, we let $\Phi(\G,\bfS)$ and $\Phi^{\vee}(\G,\bfS)$ denote the set of roots and coroots of $\bfS$ in $\G$, respectively.
Note that, when both $\G$ and $\bfS$ are defined over $F$, the sets $\Phi(\G,\bfS)$ and $\Phi^{\vee}(\G,\bfS)$ are equipped with an action of $\Gamma$.
We let $\Omega_{\G}(\bfS)$ be the Weyl group of $\bfS$ in $\G$, i.e., $\Omega_{\G}(\bfS)\colonequals \bfN_{\G}(\bfS)/\bfS$.
We sometimes loosely write $\Omega_{\G}$ for $\Omega_{\G}(\bfS)$ when the choice of a maximal torus $\bfS$ is clear from the context (e.g., when $\bfS$ is a maximal torus belonging to a splitting of $\G$).

We write $\bmfg$ for the Lie algebra of $\G$.
When $\G$ is defined over $F$, $\bmfg$ is an algebraic variety over $F$, hence we write $\mfg\colonequals \bmfg(F)$ following the convention explained above.

\subsubsection{Bruhat--Tits theory}

Suppose that $\G$ is a connected reductive group over $F$.
We follow the notation on Bruhat--Tits theory used by \cite{AS08, AS09, DS18}.
(See, for example, \cite[Section 3.1]{AS08} for details.)
Especially, $\mcB(\G,F)$ (resp.\ $\mcB^{\red}(\G,F)$) denotes the enlarged (resp.\ reduced) Bruhat--Tits building of $\G$ over $F$.
For $\x\in\mcB(\G,F)$ (or $\x\in\mcB^{\red}(\G,F)$), we let $G_{\x}$ be the stabilizer of $\x$ in $G$.

We define $\tilde{\R}$ to be the set $\R\sqcup\{r+\mid r\in\R\}\sqcup\{\infty\}$ with a natural order.
Then, for any $r\in\tilde{\R}_{\geq0}$ and $\x\in\mcB^{\red}(\G,F)$, we can consider the $r$-th Moy--Prasad filtration $G_{\x,r}$ of $G$ with respect to the point $\x$.
For any $r,s\in\tilde{\R}_{\geq0}$ satisfying $r<s$, we write $G_{\x,r:s}$ for the quotient $G_{\x,r}/G_{\x,s}$.
We put $G_{r}\colonequals \bigcup_{\x\in\mcB^{\red}(\G,F)}G_{\x,r}$ for $r\in\tilde{\R}_{\geq0}$.
Similarly, we have the Moy--Prasad filtration $\{\mfg_{\x,r}\}_{r}$ on the Lie algebra $\mfg=\bmfg(F)$, their quotients $\mfg_{\x,r:s}$, and the unions $\mfg_{r}$.
We also have the Moy--Prasad filtration on the dual Lie algebra $\mfg^{\ast}\colonequals \Hom_{F}(\mfg,F)$ defined by 
\[
\mfg^{\ast}_{\x,r}\colonequals \{Y^{\ast}\in\mfg^{\ast} \mid \langle \mfg_{\x,(-r)+},Y^{\ast}\rangle\subset \mfp\}
\]
for any $r\in\R_{\geq0}$ and $\x\in\mcB^{\red}(\G,F)$ ($\mfg^{\ast}_{\x,r+}$ is defined to be $\bigcup_{s>r}\mfg^{\ast}_{\x,s}$).

Suppose that $\bfS$ is a tame (i.e., $F$-rational and split over a tamely ramified extension of $F$) maximal torus of $\G$.
By fixing an $S$-equivariant embedding of $\mcB(\bfS,F)$ into $\mcB(\G,F)$, we may regard $\mcB(\bfS,F)$ as a subset of $\mcB(\G,F)$.
Then, for any point $\x\in\mcB(\G,F)$, the property that ``$\x$ belongs to the image of $\mcB(\bfS,F)$'' does not depend on the choice of such an embedding (see the second paragraph of \cite[Section 3]{FKS23} for details).
For any point $\x\in\mcB(\G,F)$ which belongs to $\mcB(\bfS,F)$, we have $S_{\mathrm{b}}\subset G_{\x}$, where $S_{\mathrm{b}}$ denotes the maximal bounded subgroup of $S$.
When $\bfS$ is elliptic in $\G$, the image of $\mcB(\bfS,F)$ in $\mcB^{\red}(\G,F)$ consists of only one point.
If $\x\in\mcB(\G,F)$ belongs to the image of $\mcB(\bfS,F)$, we say that ``$\x$ is associated to $\bfS$''.

We also fix a family of mock-exponential maps $\mfg_{\x,r}\rightarrow G_{\x,r}$ for $x\in\mcB(\G,F)$ and $r\in\t\R_{>0}$ and simply write ``$\exp$'' for it (see \cite[Appendix A]{AS09}; cf.\ \cite[Section 3.4]{Hak18}).
We write ``$\log$'' for the inverse of $\exp$.
It is guaranteed that a mock exponential map in the sense of \cite[Appendix A]{AS09} always exists when $p$ does not divide the order of the absolute Weyl group of $\G$.

\subsubsection{Finite sets with Galois actions}\label{subsubsec:fin-Gal-sets}
We put $\Sigma\colonequals \Gamma\times\{\pm1\}$.
Suppose that $\Phi$ is a finite set with an action of $\Sigma$, e.g., the set of roots of an $F$-rational maximal torus in a connected reductive group ($-1$ acts on $\Phi$ via $\alpha\mapsto -\alpha$ in this case).
Following \cite{AS09}, we put $\dot{\Phi}\colonequals \Phi/\Gamma$\symdef{phidot}{$\dot{\Phi}$} and $\ddot{\Phi}\colonequals \Phi/\Sigma$\symdef{phidotdot}{$\ddot{\Phi}$}.
Also, whenever there is no risk of confusion, we simply write $\dot{\alpha}:=\Gamma\alpha$\symdef{alphadot}{$\dot{\alpha}$} and $\ddot{\alpha}:=\Sigma\alpha$\symdef{alphadotdot}{$\ddot{\alpha}$}.

For each $\alpha\in\Phi$, we put $\Gamma_{\alpha}$ (resp.\ $\Gamma_{\pm\alpha}$) to be the stabilizer of $\alpha$ (resp.\ $\{\pm\alpha\}$) in $\Gamma$.
Let $F_{\alpha}$ (resp.\ $F_{\pm\alpha}$) be the subfield of $\overline{F}$ fixed by $\Gamma_{\alpha}$ (resp.\ $\Gamma_{\pm\alpha}$).
Hence we have $\Gamma_{\alpha}=\Gamma_{F_{\alpha}}$ and $\Gamma_{\pm\alpha}=\Gamma_{F_{\pm\alpha}}$:
\[
F\subset F_{\pm\alpha}\subset F_{\alpha}
\quad
\longleftrightarrow
\quad
\Gamma\supset \Gamma_{\pm\alpha}\supset \Gamma_{\alpha}.
\]
We abbreviate the residue field $k_{F_{\alpha}}$ of $F_{\alpha}$ (resp.\ $k_{F_{\pm\alpha}}$ of $F_{\pm\alpha}$) as $k_{\alpha}$ (resp.\ $k_{\pm\alpha})$.

We say that $\alpha\in\Phi$ is \textit{asymmetric} if $F_{\alpha}=F_{\pm\alpha}$ and that $\alpha$ is \textit{symmetric} if $F_{\alpha}\supsetneq F_{\pm\alpha}$.
We remark that $\alpha$ is symmetric if and only if the $\Gamma$-orbit of $\alpha$ contains $-\alpha$.
By noting that the extension $F_{\alpha}/F_{\pm\alpha}$ is necessarily quadratic if $\alpha$ is symmetric, we say that $\alpha$ is \textit{(symmetric) unramified} (resp.\ \textit{ramified}) if $F_{\alpha}/F_{\pm\alpha}$ is unramified (resp.\ ramified).
We write $\Phi_{\asym}$, $\Phi_{\ur}$, $\Phi_{\ram}$, and $\Phi_{\sym}$ for the set of asymmetric elements, symmetric unramified elements, symmetric ramified elements, and symmetric elements of $\Phi$, respectively.

For $\alpha\in\Phi_{\sym}$, we let $\kappa_{\alpha}\colon F_{\pm\alpha}^{\times}\rightarrow\C^{\times}$ denote the quadratic character of $F_{\pm\alpha}^{\times}$ corresponding to the quadratic extension $F_{\alpha}/F_{\pm\alpha}$ under the local class field theory.

Note that, if $\alpha$ is symmetric, $\dot{\alpha}=\ddot{\alpha}$.
This implies that the sets $\dot{\Phi}_{\sym}$ and $\ddot{\Phi}_{\sym}$ can be naturally identified (and, of course, the same is true for $\Phi_{\ur}$ or $\Phi_{\ram}$).

\subsubsection{Several arithmetic invariants}
For any finite extension $E_{\pm}$ of $F$ and its quadratic extension $E$, we let $\lambda_{E/E_{\pm}}\colonequals \lambda_{E/E_{\pm}}(\psi_{F}\circ\Tr_{E_{\pm}/F})$ denote the Langlands constant with respect to the nontrivial additive character $\psi_{F}\circ\Tr_{E_{\pm}/F}$ of $E_{\pm}$ (see, e.g., \cite[30.4]{BH06}).
When the quadratic extension $E/E_{\pm}$ is given by $F_{\alpha}/F_{\pm\alpha}$ as in Section \ref{subsubsec:fin-Gal-sets}, we simply write $\lambda_{\alpha}$ for $\lambda_{F_{\alpha}/F_{\pm\alpha}}$.

In this paper, we often consider the root number $\varepsilon(\frac{1}{2},X^{\ast}(\bfS)_{\C},\psi_{F})$ of the $\varepsilon$-factor of the Galois representation $X^{\ast}(\bfS)_{\C}$ associated to an $F$-rational torus $\bfS$ (see \cite[Section 30]{BH06} or \cite[Section 3.6]{Tat79} for the definition of the $\varepsilon$-factor).
We shortly write $\varepsilon(\bfS)\colonequals \varepsilon(\frac{1}{2},X^{\ast}(\bfS)_{\C},\psi_{F})$.

\subsubsection{Finite fields}
Suppose that $\ul{k}$ is a finite field of odd characteristic $p$.
Then the multiplicative group $\ul{k}^{\times}$ is cyclic of even order, hence there exists a unique nontrivial quadratic character $\ul{k}^{\times}\rightarrow\{\pm1\}$.
We write $\sgn_{\ul{k}^{\times}}(-)$ for this character.

Next, we furthermore suppose that $[\ul{k}:\F_{p}]$ is even.
Then, there uniquely exists a subextension $\ul{k}_{\pm}$ satisfying $[\ul{k}:\ul{k}_{\pm}]=2$.
We let $\ul{k}^{1}$ denote the kernel of the norm map $\Nr_{\ul{k}/\ul{k}_{\pm}}\colon \ul{k}^{\times}\rightarrow\ul{k}_{\pm}^{\times}$.
By noting that $\ul{k}^{1}$ is also cyclic of even order, we write $\sgn_{\ul{k}^{1}}(-)$ for the unique nontrivial quadratic character of $\ul{k}^{1}$.

\subsection{Assumptions on $p$}
Throughout this paper, we assume that $\G$ is a tamely ramified connected reductive group over $F$.
Furthermore, we assume that $p$ is odd and does not divide the order of the absolute Weyl group of $\G$.
From Section \ref{sec:Waldspurger} to the end of this paper, we furthermore assume that $p$ is greater than or equal to $(2+e_{F})n$, where $n$ is the minimum of the dimension of a faithful representation of $\G$ and $e_{F}$ is the ramification degree of $F/\Q_{p}$.

\section{Kaletha's regular supercuspidal local Langlands correspondence}\label{sec:LLC}

\subsection{Regular supercuspidal representations}\label{subsec:rsc}

In \cite{Yu01}, Yu established a general methodology to produce supercuspidal representations of tamely ramified $p$-adic reductive groups in an explicit manner.
An input datum of Yu's construction is called a \textit{cuspidal $\G$-datum}; hence Yu's construction gives a map from the set of cuspidal $\G$-data to the set of isomorphism classes of supercuspidal representations of $G$.
The supercuspidal representations in the image of Yu's construction are called \textit{tame supercuspidal representations}.
It is known that Yu's construction gives a lot of supercuspidal representations; in fact, any supercuspidal representation is tame under our assumption on $p$ (Fintzen's exhaustion result \cite{Fin21-Ann}, which improves Kim's \cite{Kim07}).
We do not review the definition of a cuspidal $\G$-datum here since in this paper we only deal with ``regular supercuspidal representations'', a special class of tame supercuspidal representations, which can be re-parametrized by different input data as follows.

The ``fibers'' of Yu's construction were investigated by Hakim--Murnaghan \cite{HM08}.
They introduced an equivalence relation called \textit{$\G$-equivalence} and proved that two cuspidal $\G$-data give rise to the isomorphic supercuspidal representations if and only if two data are $\G$-equivalent.
Thus Yu's construction gives a bijection
\[
\xymatrix{
\{\text{cusp.\ $\G$-data}\}/\text{$\G$-eq.} \ar@{->}[rrr]^-{1:1}_-{\text{Yu's construction}} &&& \{\text{tame s.c.\ rep'ns of $G$}\}/{\sim}.
}
\]

In \cite{Kal19}, Kaletha introduced the notion of \textit{(extra) regularity} for cuspidal $\G$-data (see \cite[Section 3]{Kal19}).
Tame supercuspidal representations arising from (extra) regular cuspidal $\G$-data are called \textit{(extra) regular supercuspidal representations}.
By Kaletha's re-parametrizing result \cite[Proposition 3.7.8]{Kal19}, $\G$-equivalence classes of (extra) regular cuspidal $\G$-data bijectively correspond to $G$-conjugacy classes of much simpler data called \textit{tame elliptic (extra) regular pairs} (see \cite[Definition 3.7.5]{Kal19}).
We write $\pi_{(\bfS,\vartheta)}$ for the (extra) regular supercuspidal representation which corresponds to a tame elliptic (extra) regular pair $(\bfS,\vartheta)$.\symdef{pi-S-vartheta}{$\pi_{(\bfS,\vartheta)}$}
\[
\xymatrix{
\{\text{cusp.\ $\G$-data}\}/\text{$\G$-eq.} \ar@{->}[r]^-{1:1} & \{\text{tame s.c.\ rep'ns of $G$}\}/{\sim}\\
\{\text{(ex.) reg.\ cusp.\ $\G$-data}\}/\text{$\G$-eq.}\ar@{}[u]|{\bigcup} \ar@{->}[r]^-{1:1} & \{\text{(ex.) reg.\ s.c.\ rep'ns of $G$}\}/{\sim}\ar@{}[u]|{\bigcup}\\
\{\text{tame ell.\ (ex.) reg.\ pairs}\}/\text{$G$-conj.}\ar@{<->}[u]^-{1:1} \ar@{->}_-{\quad\qquad(\bfS,\vartheta)\mapsto\pi_{(\bfS,\vartheta)}}[ur]&
}
\]

\begin{defn}
    We say that a tame elliptic regular pair $(\bfS,\vartheta)$ is a \textit{toral pair} if $\vartheta$ is $\G$-generic of depth $r>0$ in the sense of \cite[Section 8]{Yu01}.
    We say that a regular supercuspidal representation $\pi$ of $G$ is \textit{toral} if $\pi$ arises from a toral pair.
\end{defn}

\begin{rem}
  \begin{enumerate}
    \item If a tame elliptic regular pair is toral, then it is extra regular.
    \item The meaning of the terminology ``toral'' depends on the literature. For example, our ``toral'' is called ``$0$-toral'' in \cite{FS21}. We use ``toral'' rather than ``$0$-toral'' in this paper following \cite{DS18} and \cite{Kal19}.
  \end{enumerate}
\end{rem}

Let us briefly review Yu's construction in the case of toral supercuspidal representations.
We fix a tame elliptic toral pair $(\bfS,\vartheta)$ of $\G$.
Let $\x\in\mcB^{\red}(\G,F)$ be the point associated to $\bfS$ and $r\in\R_{>0}$ be the depth of $\vartheta$.
We put $s\colonequals r/2$ and define the subgroups $K$\symdef{K}{$K$}, $J$\symdef{J}{$J$}, and $J_{+}$\symdef{J-}{$J_{+}$} of $G$ by
\[
K\colonequals  SG_{\x,s},\quad
J\colonequals  (S,G)_{\x,(r,s)},\quad
J_{+}\colonequals (S,G)_{\x,(r,s+)}, 
\]
where $(S,G)_{\x,(r,s)}$ and $(S,G)_{\x,(r,s+)}$ are the groups defined according to the manner of Yu (see \cite[Sections 1 and 2]{Yu01}).
Note that we have $K=SJ$.

Since the depth of $\vartheta$ is $r$, we can extend $\vartheta$ to a character $\hat{\vartheta}$\symdef{vartheta-hat}{$\hat{\vartheta}$} of $J_{+}$ satisfying $\hat{\vartheta}|_{(S,G)_{\x,(r+,s+)}}\equiv\mathbbm{1}$.
Then, by the definition of the $\G$-genericity, there exists an element $X^{\ast}$\symdef{X-ast}{$X^{\ast}$} of $\mfs^{\ast}_{-r}$ which is $\G$-generic of depth $r$ in the sense of \cite[Section 8]{Yu01} and satisfies
\[
\hat{\vartheta}(\exp(Y))
=
\psi_{F}(\langle Y,X^{\ast}\rangle)
\]
for any $Y\in\mfg_{\x,s+:r+}$ (or, equivalently, for any $Y\in\mfs_{s+:r+}$).
Here, as explained in \cite[Section 8]{Yu01}, we may regard $\mfs^{\ast}$ as a subspace of $\mfg^{\ast}$ by considering the coadjoint action of $\bfS$ on $\mfg^{\ast}$.
We recall that the definition of $\G$-genericity consists of two conditions \textbf{GE1} and \textbf{GE2}.
The condition \textbf{GE1} requires that $\val_{F}(\langle H_{\alpha},X^{\ast}\rangle)=-r$ for any $\alpha\in\Phi(\G,\bfS)$, where $H_{\alpha}\colonequals d\alpha^{\vee}(1)$.\symdef{H-alpha}{$H_{\alpha}$}
We do not review the condition \textbf{GE2} because \textbf{GE1} implies \textbf{GE2} by \cite[Lemma 8.1]{Yu01} when $p$ is not a torsion prime for the dual based root datum of $\G$, which follows from our assumption that $p\nmid|\Omega_{\G}|$ by \cite[Lemma 3.2]{Fin21-IMRN}.

The point of the construction is that, by putting $N\colonequals \Ker\hat{\vartheta}\subset J_{+}$, the quotient $J/N$ has the structure of a finite Heisenberg group:
\begin{itemize}
\item
The center of $J/N$ is given by $J_{+}/N$, which is isomorphic to $\mu_{p}\cong\F_{p}$ via $\hat{\vartheta}$ (here we fix an isomorphism $\mu_{p}\cong\F_{p}$).
\item
The quotient $J/J_{+}$ is a symplectic space with respect to the pairing
\[
(J/J_{+})\times(J/J_{+}) \rightarrow \mu_{p}\cong\F_{p}\colon (g,g')\mapsto\h\vartheta([g,g'])
\]
(see \cite[Section 11]{Yu01}; we will review the structure of the symplectic space $J/J_{+}$ in more detail in Section \ref{subsec:Heisen-structure}).
\end{itemize}
By the Stone--von Neumann theorem, there exists a unique irreducible representation of $J/N$ whose central character on $J_{+}/N$ is given by $\hat{\vartheta}$.
Furthermore, as the conjugate action of $S$ on $J$ preserves $J_{+}$ and $N$ and induces a symplectic action on $J/J_{+}$, we can extend the (inflation of) the representation of $J$ to the semi-direct group $S\ltimes J$, for which we write $\omega_{(\bfS,\vartheta)}$ (so-called the Heisenberg--Weil representation).
Then the tensor representation $\omega_{(\bfS,\vartheta)}\otimes(\vartheta\ltimes\mathbbm{1})$ of $S\ltimes J$ descends to $SJ=K$ (factors through the canonical map $S\ltimes J\twoheadrightarrow K$).
We let $\rho_{(\bfS,\vartheta)}$ be the descended representation of $K$.
The toral supercuspidal representation $\pi_{(\bfS,\vartheta)}$ is given by 
\[
\pi_{(\bfS,\vartheta)}\colonequals \cInd_{K}^{G}\rho_{(\bfS,\vartheta)}.
\]

\subsection{Several invariants for the construction of LLC}\label{subsec:invariants-for-LLC}

We next summarize several invariants which play important roles in Kaletha's construction of the local Langlands correspondence for regular supercuspidal representations \cite{Kal19}, mainly focusing on the case of toral supercuspidal representations.

In the following, by assuming that $\G$ is quasi-split over $F$, we fix an $F$-splitting $\spl_{\G}=(\bfB,\bfT,\{X_{\alpha}\}_{\alpha})$ of $\G$.
Let $\hat{\G}$ be the Langlands dual group of $\G$.
More precisely, $\h\G$ is a connected reductive group over $\C$ with
\begin{itemize}
\item
a $\Gamma$-action on $\hat{\G}$, 
\item
a $\Gamma$-stable splitting $\spl_{\hat{\G}}=(\hat{\bfB},\hat{\T},\{\mathcal{X}_{\alpha^{\vee}}\}_{\alpha^{\vee}})$ of $\hat{\G}$, and
\item
a $\Gamma$-equivariant isomorphism between the based root data $\Psi(\hat{\G})$ of $\hat{\G}$ and the dual $\Psi(\G)^{\vee}$ of that of $\G$.
\end{itemize}
We put ${}^{L}\G\colonequals \hat{\G}\rtimes W_{F}$.

Let $(\bfS,\vartheta)$ be a tame elliptic toral pair of $\G$.
Note that then the set $\Phi(\bfG,\bfS)$ of roots of $\bfS$ in $\G$ is a finite set equipped with a $\Gamma$-action.

\begin{defn}\label{defn:a-data}
A family $\{a_{\alpha}\}_{\alpha\in \Phi(\bfG,\bfS)}$ of elements $a_{\alpha}\in F_{\alpha}^{\times}$ is called a \textit{set of $a$-data (with respect to $\bfS$)} if the following conditions are satisfied:
\begin{itemize}
\item
$a_{-\alpha}=a_{\alpha}^{-1}$ for any $\alpha\in \Phi(\bfG,\bfS)$, and
\item
$a_{\sigma(\alpha)}=\sigma(a_{\alpha})$ for any $\alpha\in \Phi(\bfG,\bfS)$ and $\sigma\in\Gamma$.
\end{itemize}
\end{defn}

Following \cite[Section 4.7]{Kal19}, we associate a set $a_{\vartheta}=\{a_{\vartheta,\alpha}\}_{\alpha\in\Phi(\bfG,\bfS)}$\symdef{a-vartheta}{$a_{\vartheta}$} of $a$-data to $(\bfS,\vartheta)$ by the following (note that $a_{\vartheta}$ is simply denoted by $a$ in \cite[Section 4.7]{Kal19}):
\[
a_{\vartheta,\alpha}=\langle H_{\alpha},X^{\ast}\rangle,
\]
where $H_{\alpha}\colonequals d\alpha^{\vee}(1)\in \bmfs(F_{\alpha})$ and $X^{\ast}\in\mfs^{\ast}_{-r}$ is an element associated to $\vartheta$ (see Section \ref{subsec:rsc}).
We refer to $a_{\vartheta}$ as \textit{Kaletha's $a$-data} associated to $(\bfS,\vartheta)$.

\begin{defn}[{\cite[Definition 4.6.1]{Kal19}}]\label{defn:min-ram}
A family $\{\chi_{\alpha}\}_{\alpha\in \Phi(\bfG,\bfS)}$ of characters $\chi_{\alpha}\colon F_{\alpha}^{\times}\rightarrow\C^{\times}$ is called a \textit{set of $\chi$-data (with respect to $\bfS$)} if the following conditions are satisfied:
\begin{itemize}
\item
$\chi_{-\alpha}=\chi_{\alpha}^{-1}$ for any $\alpha\in \Phi(\bfG,\bfS)$,
\item
$\chi_{\sigma(\alpha)}=\chi_{\alpha}\circ\sigma^{-1}$ for any $\alpha\in \Phi(\bfG,\bfS)$ and $\sigma\in\Gamma$, and
\item
$\chi_{\alpha}|_{F_{\pm\alpha}^{\times}}$ equals the quadratic character $\kappa_{\alpha}$ for any $\alpha\in \Phi(\bfG,\bfS)_{\sym}$.
\end{itemize}
We say that $\{\chi_{\alpha}\}_{\alpha\in \Phi(\bfG,\bfS)}$ is \textit{minimally ramified} if the following conditions are furthermore satisfied:
\begin{itemize}
\item
$\chi_{\alpha}=\mathbbm{1}$ for any $\alpha\in \Phi(\bfG,\bfS)_{\asym}$,
\item
$\chi_{\alpha}$ is unramified for any $\alpha\in \Phi(\bfG,\bfS)_{\ur}$, and
\item
$\chi_{\alpha}$ is tamely ramified for any $\alpha\in \Phi(\bfG,\bfS)_{\ram}$.
\end{itemize}
\end{defn}

Following \cite[Section 4.7]{Kal19}, we associate a set $\chi_{\vartheta}=\{\chi_{\vartheta,\alpha}\}_{\alpha\in\Phi(\bfG,\bfS)}$\symdef{chi-vartheta}{$\chi_{\vartheta}$} of minimally ramified $\chi$-data to $(\bfS,\vartheta)$ as follows ($\chi_{\vartheta}$ is denoted by $\chi'$ in \cite[Section 4.7]{Kal19}):
\begin{itemize}
\item
For $\alpha\in\Phi(\bfG,\bfS)_{\asym}$, let $\chi_{\vartheta,\alpha}$ be the trivial character of $F_{\alpha}^{\times}$.
\item
For $\alpha\in\Phi(\bfG,\bfS)_{\ur}$, let $\chi_{\vartheta,\alpha}$ be the unique unramified nontrivial quadratic character of $F_{\alpha}^{\times}$.
\item
For $\alpha\in\Phi(\bfG,\bfS)_{\ram}$, let $\chi_{\vartheta,\alpha}$ be the unique tamely ramified character of $F_{\alpha}^{\times}$ characterized by the following properties:
\[
\chi_{\vartheta,\alpha}|_{F_{\pm\alpha}^{\times}}=\kappa_{\alpha}
\quad\text{and}\quad
\chi_{\vartheta,\alpha}(2a_{\vartheta,\alpha})=\lambda_{\alpha}.
\]
\end{itemize}

We refer to $\chi_{\vartheta}$ as \textit{Kaletha's $\chi$-data} associated to $(\bfS,\vartheta)$.

\begin{rem}\label{rem:non-toral-a-chi}
For a general (i.e., possibly non-toral) tame elliptic regular pair $(\bfS,\vartheta)$, Kaletha's sets of $a$-data $a_{\vartheta}$ and $\chi$-data $\chi_{\vartheta}$ are defined by noting the inductive structure of $\Phi(\G,\bfS)$ given by the tame twisted Levi subgroups of $\G$ determined by $(\bfS,\vartheta)$.
See \cite[Section 4.7]{Kal19} (and also \cite[Section 6]{OT21}) for the details.
\end{rem}

\begin{defn}
Let $\bfS$ be an $F$-rational maximal torus of $\G$.
Let $a=\{a_{\alpha}\}_{\alpha}$ be a set of $a$-data and $\chi=\{\chi_{\alpha}\}_{\alpha}$ a set of $\chi$-data with respect to $\bfS$.
We define a function $\Delta_{\G,\II}[a,\chi]\colon S\rightarrow\C^{\times}$ by
\[
\Delta_{\G,\II}[a,\chi](s)
\colonequals \prod_{\begin{subarray}{c}\dot{\alpha}\in\dot{\Phi}(\bfG,\bfS)\\ \alpha(s)\neq1\end{subarray}} \chi_{\alpha}\biggl(\frac{\alpha(s)-1}{a_{\alpha}}\biggr).
\]
\end{defn}

We next recall sign characters of $S$ which were introduced by DeBacker--Spice \cite[Section 4.3]{DS18} in their study of character formulas of tame supercuspidal representations.
We define a character $\epsilon_{\alpha}\colon S\rightarrow\C^{\times}$\symdef{epsilon-alpha}{$\epsilon_{\alpha}$} for $\alpha\in\Phi(\G,\bfS)\smallsetminus\Phi(\G,\bfS)_{\ram}$ as follows:
\[
\epsilon_{\alpha}(s)\colonequals 
\begin{cases}
\sgn_{k_{\alpha}^{\times}}(\overline{\alpha(s)}) & \text{if $\alpha\in\Phi(\G,\bfS)_{\asym}$,}\\
\sgn_{k_{\alpha}^{1}}(\overline{\alpha(s)}) & \text{if $\alpha\in\Phi(\G,\bfS)_{\ur}$}.
\end{cases}
\]
Here, $\overline{\alpha(s)}$ denotes the image of $\alpha(s)\in\mcO_{F_\alpha}$ in the residue field $k_{\alpha}$.

Recall that, in the construction of the regular supercuspidal representation $\pi_{(\bfS,\vartheta)}$, we note that the subquotient $J/N$ of $G$ has the structure of a finite Heisenberg group (see Section \ref{subsec:rsc}).
As we will review later, the symplectic quotient $V\colonequals J/J_+$ of $J/N$ has an orthogonal decomposition $V=\bigoplus_{\ddot{\alpha}\in\ddot{\Phi}} V_{\ddot{\alpha}}$\symdef{V}{$V$} into symplectic subspaces induced by the root space decomposition of $\bmfg$ with respect to $\bfS$ (see Section \ref{subsec:Heisen-structure}).
Each piece $V_{\ddot{\alpha}}$ could be zero; we define a subset $\ddot{\Xi}\subset\ddot{\Phi}$ to be the set of all $\ddot{\alpha}\in\ddot{\Phi}$ such that $V_{\ddot{\alpha}}\neq0$.
We define characters $\epsilon_{\vartheta,\asym}$\symdef{epsilon-vartheta-asym}{$\epsilon_{\vartheta,\asym}$} and $\epsilon_{\vartheta,\ur}$\symdef{epsilon-vartheta-ur}{$\epsilon_{\vartheta,\ur}$} of $S$ by taking the products of $\epsilon_{\alpha}$ over $\ddot{\Xi}$:
\[
\epsilon_{\vartheta,\asym}(s)
\colonequals \prod_{\ddot{\alpha}\in\ddot{\Xi}_{\asym}} \epsilon_{\alpha}(s)
\quad
\text{and}
\quad
\epsilon_{\vartheta,\ur}(s)
\colonequals \prod_{\ddot{\alpha}\in\ddot{\Xi}_{\ur}} \epsilon_{\alpha}(s).
\]

\begin{rem}
  In \cite[Section 4.3]{DS18}, the products are taken over $\ddot{\alpha}\in\ddot{\Phi}$ satisfying the condition ``$\frac{r}{2}\in\ord_{\x}(\alpha)$'', where $r$ is the depth of $\vartheta$ and $\ord_{\x}(\alpha)$ is the set defined in \cite[Definition 3.6]{DS18}.
  This condition is equivalent to that $\alpha\in\Xi$ (see the proof of \cite[Proposition 5.12]{OT21}).  
\end{rem}

We finally recall the character $\epsilon_{\bfS,\ram}$\symdef{epsilon-S-ram}{$\epsilon_{\bfS,\ram}$} of $S$ defined in \cite[Definition 4.7.3]{Kal19}.
As explained in \cite[Lemma 4.7.4]{Kal19}, this can be expressed as the product of \textit{toral invariants} $f_{(\G,\bfS)}(\alpha)$ for symmetric ramified roots $\alpha$, which are introduced in \cite[Section 4.1]{Kal15}.
We recall that the toral invariant $f_{(\G,\bfS)}(\alpha)$ for $\alpha\in\Phi(\G,\bfS)_{\sym}$ (not necessarily ramified) is defined as follows.
We fix an element $\tau_{\alpha}\in\Gamma_{\pm\alpha}\smallsetminus\Gamma_{\alpha}$ (i.e., $\tau_{\alpha}\in\Gamma_{\pm\alpha}$ is an element satisfying $\tau_{\alpha}(\alpha)=-\alpha$).
If we take an $F_{\alpha}$-rational root vector $X_{\alpha}\in\bmfg_{\alpha}(F_{\alpha})$, then $\tau_{\alpha}(X_{\alpha})$ belongs to $\bmfg_{-\alpha}(F_{\alpha})$ and the ratio of $[X_{\alpha},\tau_{\alpha}(X_{\alpha})]$ to $H_{\alpha}\colonequals d\alpha^{\vee}(1)\in\bmfs(F_{\alpha})$ lies in $F_{\pm\alpha}^{\times}$.
By noting that $\frac{[X_{\alpha},\tau_{\alpha}(X_{\alpha})]}{H_{\alpha}}$ is well-defined up to $\Nr_{F_{\alpha}/F_{\pm\alpha}}(F_{\alpha}^{\times})$-multiplication, we put
\[
f_{(\G,\bfS)}(\alpha)
\colonequals 
\kappa_{\alpha}\biggl(\frac{[X_{\alpha},\tau_{\alpha}(X_{\alpha})]}{H_{\alpha}}\biggr)\in\{\pm1\}.
\]
Then we have
\[
\epsilon_{\bfS,\ram}(s)
=
\prod_{\begin{subarray}{c}\dot{\alpha}\in\dot{\Phi}(\G,\bfS)_{\ram}\\ \alpha(s)\neq1 \\ \val_{F}(\alpha(s)-1)=0\end{subarray}}
f_{(\G,\bfS)}(\alpha).
\]

\begin{rem}\label{rem:non-toral-epsilon}
Similarly to the definition of $a_{\vartheta}$ and $\chi_{\vartheta}$, for a general tame elliptic regular pair $(\bfS,\vartheta)$, the characters $\epsilon_{\vartheta,\asym}$ and $\epsilon_{\vartheta,\ur}$ are defined by noting the inductive structure of $\Phi(\G,\bfS)$ given by the tame twisted Levi subgroups determined by $(\bfS,\vartheta)$.
See \cite[Section 4.3]{Kal19} (and also \cite[Section 6]{OT21}) for the details.
We also note that, in \cite{Kal19}, the product $\epsilon_{\vartheta,\asym}\cdot\epsilon_{\vartheta,\ur}$ (resp.\ $\epsilon_{\bfS,\ram}$) is denoted by $\epsilon^{\ram}$ (resp.\ $\epsilon_{\ram}$).
\end{rem}

\subsection{Construction of regular supercuspidal $L$-parameters}\label{subsec:construction}
Now we explain Kaletha's construction of regular supercuspidal $L$-parameters (\cite[Proof of Proposition 5.2.7]{Kal19}).

Suppose that a pair $(\bfS,\hat{\jmath})$ is given, where
\begin{itemize}
\item
$\bfS$ is an $F$-rational tame torus having the same rank as $\G$ and
\item
$\hat{\jmath}\colon\hat{\bfS}\hookrightarrow\hat{\G}$ is an embedding  whose $\hat{\G}$-conjugacy class is $\Gamma$-stable.
\end{itemize}
Here, by noting that $\hat{\jmath}(\hat{\bfS})$ is a maximal torus of $\hat{\G}$, we assume that $\hat{\jmath}(\hat{\bfS})$ itself equals $\hat{\T}$ by replacing $\hat{\jmath}$ with its conjugate.

\begin{defn}\label{defn:adm-emb}
Let $j\colon\bfS\hookrightarrow\G$ be an embedding of $\bfS$ into $\G$.
Since its image $\bfS_{j}\colonequals j(\bfS)$ is a maximal torus of $\G$ by the rank condition, there exists an element $g\in\G$ such that $[g](\bfS_{j})=\T$.
We say that $j$ is \textit{$\hat{\jmath}$-admissible} if the inverse of the dual to the isomorphism $[g]\circ j\colon\bfS\rightarrow\T$ is $\hat{\G}$-conjugate to $\hat{\jmath}\colon\hat{\bfS}\xrightarrow{\sim}\hat{\T}$.
\end{defn}
We write $\mcJ^{\G}_{\ol{F}}$\symdef{J-G-F-bar}{$\mcJ^{\G}_{\ol{F}}$} for the $\G$-conjugacy class of $\hat{\jmath}$-admissible embeddings of $\bfS$ into $\G$.
Since the $\hat{\G}$-conjugacy class of $\hat{\jmath}$ is $\Gamma$-stable, so is $\mcJ^{\G}_{\ol{F}}$ (see \cite[Section 5.1]{Kal19}).
Thus, by Kottwitz's result on the rational conjugacy (\cite[Corollary 2.2]{Kot82}), the quasi-splitness of $\G$ implies that $\mcJ^{\G}_{\ol{F}}$ has an $F$-rational point.
In other words, there exists an $F$-rational $\hat{\jmath}$-admissible embedding of $\bfS$ into $\G$.
We let $\mcJ^{\G}$\symdef{J-G}{$\mcJ^{\G}$} denote the set of $F$-rational points of $\mcJ^{\G}_{\ol{F}}$.
For each $j\in\mcJ^{\G}$ (i.e., $j$ is the $\G$-conjugacy class of an $F$-rational $\hat{\jmath}$-admissible embedding), by fixing a representative of $j$ and write again $j$ for it, $\bfS_j \colonequals j(\bfS)$\symdef{S-j}{$\bfS_j$} is an $F$-rational maximal torus of $\G$.
Consequently, we get
\begin{itemize}
\item
an $F$-rational embedding $\bfZ_{\G}\hookrightarrow\bfS_{j}$,
\item
a $\Gamma$-stable subset $\Phi(\G,\bfS_{j})$ of $X^{\ast}(\bfS_{j})$, and
\item
a $\Gamma$-stable subgroup $\Omega_{\G}(\bfS_{j})$ of $\Aut(\bfS_{j})$.
\end{itemize}
Since $j\colon\bfS\xrightarrow{\sim}\bfS_{j}$ is an $F$-rational isomorphism, by pulling back these via $j$, we get
\begin{itemize}
\item
an $F$-rational embedding $\bfZ_{\G}\hookrightarrow\bfS$,
\item
a $\Gamma$-stable subset $j^{\ast}\Phi(\G,\bfS_{j})$ of $X^{\ast}(\bfS)$, and
\item
a $\Gamma$-stable subgroup $j^{\ast}\Omega_{\G}(\bfS_{j})$ of $\Aut(\bfS)$.
\end{itemize}
Noting that all of these are independent of the choice of $j\in\mcJ^{\G}$ (and also of its representative), we write $\Phi(\G,\bfS_{\hat{\jmath}})\colonequals j^{\ast}\Phi(\G,\bfS_{j})$ and $\Omega_{\G}(\bfS_{\hat{\jmath}})\colonequals j^{\ast}\Omega_{\G}(\bfS_{j})$.

\begin{defn}[{\cite[Definition 5.2.4]{Kal19}}]\label{defn:rsc-L-datum}
A \textit{regular supercuspidal $L$-packet datum of $\G$} is a tuple $(\bfS,\hat{\jmath},\chi,\vartheta)$ consisting of 
\begin{enumerate}
\item
an $F$-rational tame torus $\bfS$ having the same rank as $\G$,
\item
an embedding $\hat{\jmath}\colon\hat{\bfS}\hookrightarrow\hat{\G}$ whose $\hat{\G}$-conjugacy class is $\Gamma$-stable,
\item
a set $\chi$ of minimally ramified $\chi$-data for $\Phi(\G,\bfS_{\hat{\jmath}})$, and
\item
a character $\vartheta\colon S\rightarrow\C^{\times}$
\end{enumerate}
satisfying the following conditions:
\begin{enumerate}
\item[(i)]
$\bfS$ is elliptic in $\G$ (i.e., $\bfS/\bfZ_{\G}$ is anisotropic),
\item[(ii)]
$\chi$ is $\Omega_{\G^{0}}(\bfS_{\hat{\jmath}})(F)$-invariant, and
\item[(iii)]
$(\bfS,\vartheta)$ is a tame elliptic extra regular pair of $\G$.
\end{enumerate}
By replacing ``extra regular'' with ``toral'' in (iii), we also define a \textit{toral supercuspidal $L$-packet datum of $\G$}.
\end{defn}

We give a few more comments about the precise meanings of the conditions (ii) and (iii) in Definition \ref{defn:rsc-L-datum} (see \cite[Sections 5.1 and 5.2]{Kal19} for the details).
The condition (i) implies that $\bfS_{j}$ is a tame elliptic maximal torus of $\G$ for any $j\in\mcJ^{\bfG}$.
The condition (iii) means that $(\bfS_{j},\vartheta_{j})$ is a tame elliptic extra regular pair of $\G$ for any $j\in\mcJ^{\bfG}$, where $\vartheta_{j}\colonequals \vartheta\circ j^{-1}$\symdef{vartheta-j}{$\vartheta_{j}$} (this is equivalent to that $(\bfS_{j},\vartheta_{j})$ is a tame elliptic extra regular pair of $\G$ for ``some'' $j\in\mcJ^{\bfG}$).
We define a subset $\Phi(\G^{0},\bfS_{\hat{\jmath}})$ of $\Phi(\G,\bfS_{\hat{\jmath}})$ by
\[
\Phi(\G^{0},\bfS_{\hat{\jmath}})
\colonequals 
\{\alpha\in\Phi(\G,\bfS_{\hat{\jmath}})\mid \vartheta\circ\Nr_{E/F}\circ\alpha^{\vee}(1+\mfp_E)=1 \},
\]
where $E$ is a tame finite extension of $F$ splitting $\bfS$.
Then $\Phi(\G^{0},\bfS_{\hat{\jmath}})$ is a Levi subsystem of $\Phi(\G,\bfS_{\hat{\jmath}})$ and associates a $\Gamma$-stable subgroup $\Omega_{\G^{0}}(\bfS_{\hat{\jmath}})$ of $\Omega_{\G}(\bfS_{\hat{\jmath}})$ canonically.

We take a regular supercuspidal $L$-packet datum $(\bfS,\hat{\jmath},\chi,\vartheta)$ of $\G$.
By applying the local Langlands correspondence for $\bfS$ to $\vartheta$, we get an $L$-parameter $\phi_{\vartheta}$ of $\bfS$, which is a homomorphism from $W_{F}$ to ${}^{L}\bfS$.
On the other hand, by the Langlands--Shelstad construction (\cite[Section 2.6]{LS87}), we can extend $\hat{\jmath}$ to an $L$-embedding ${}^{L}\!j_{\chi}$ from ${}^{L}\bfS$ to ${}^{L}\G$ by using the set $\chi$ of $\chi$-data.\symdef{L-j-chi}{${}^{L}\!j_{\chi}$}
By composing these two homomorphisms, we get an $L$-parameter $\phi$ of $\G$:
\[
\phi
\colon
W_{F}
\xrightarrow{\phi_{\vartheta}}
{}^{L}\bfS
\xrightarrow{{}^{L}\!j_{\chi}}
{}^{L}\G.
\]

The $L$-parameters constructed in this way can be characterized in a purely-Galois-theoretic manner:

\begin{prop}[{\cite[Propositions 5.2.7, 6.1.2]{Kal19}}]\label{prop:rsc-LLC}
Kaletha's construction gives a bijection between the isomorphism classes of (toral) regular supercuspidal $L$-packet data of $\G$ and the $\hat{\G}$-conjugacy classes of (toral) regular supercuspidal $L$-parameters of $\G$.
\end{prop}

Here, the definition of an isomorphism between regular supercuspidal $L$-packet data will be reviewed in the next subsection.
The definition of a (toral) regular supercuspidal $L$-parameter is as follows.

\begin{defn}[{\cite[Definition 5.2.3]{Kal19}}]\label{defn:rsc-L-par}
We say that an $L$-parameter $\phi\colon W_{F}\rtimes\SL_{2}(\C)\rightarrow\L\G$ is \textit{regular supercuspidal} if it satisfies the following:
\begin{enumerate}
\item[(0)]
$\phi|_{\SL_{2}(\C)}$ is trivial and $\phi$ is discrete, i.e., $S_{\phi}^{\circ}\colonequals  \bfZ_{\hat{\G}}(\phi(W_{F}))^{\circ}\subset\bfZ_{\hat{\G}}$,
\item[(1)]
$\phi(P_{F})$ is contained in a torus of $\hat{\G}$ (note that then $\mcM\colonequals \bfZ_{\hat{\G}}(\phi(P_{F}))^{\circ}$ is a Levi subgroup of $\hat{\G}$).
\item[(2)]
$\mcC\colonequals \bfZ_{\hat{\G}}(\phi(I_{F}))^{\circ}$ is a torus (note that then $\mcT\colonequals \bfZ_{\mcM}(\mcC)$ is a maximal torus of $\mcM$.
We put $\hat{\bfS}$ to be the $\Gamma$-module $\mcT$ with the $\Gamma$-action given by $[\phi(-)]$).
\item[(3)]
If $n\in \bfN_{\mcM}(\mcT)$ maps to a nontrivial element of $\Omega_{\mcM}(\hat{\bfS})^{\Gamma}$, then $n\notin \bfZ_{\hat{\G}}(\phi(I_{F}))$.
\end{enumerate}
\end{defn}

\begin{defn}[{\cite[Definition 6.1.1]{Kal19}}]\label{defn:toral-L-par}
We say that a regular supercuspidal $L$-parameter $\phi\colon W_{F}\rightarrow\L\G$ is \textit{toral of depth $r>0$} if it satisfies the following:
\begin{enumerate}
\item[(1)]
$\phi$ is trivial on $I_{F}^{r+}$, that is, $\phi(\sigma)=1\rtimes\sigma$ for any $\sigma\in I_{F}^{r+}$, and
\item[(2)]
$\bfZ_{\hat{\G}}(\phi(I_{F}^{r}))$ is a maximal torus of $\hat{\G}$ containing $\phi(P_{F})$.
\end{enumerate}
\end{defn}

\subsection{Construction of regular supercuspidal $L$-packets}\label{subsec:const-packet}
We finally explain Kaletha's construction of regular supercuspidal $L$-packets (\cite[Section 5.3]{Kal19}).

\begin{defn}[{\cite[Definition 5.3.2]{Kal19}}]\label{defn:rsc-datum}
Let $(\bfS,\hat{\jmath},\chi,\vartheta)$ be a regular supercuspidal $L$-packet datum of $\G$.
Let $\mcJ^{\G}$ be the $\G$-conjugacy classes of $F$-rational $\hat{\jmath}$-admissible embeddings of $\bfS$ into $\G$ (see Section \ref{subsec:construction}).
A \textit{regular supercuspidal datum} (over the regular supercuspidal $L$-packet datum $(\bfS,\hat{\jmath},\chi,\vartheta)$) is a tuple $(\bfS,\hat{\jmath},\chi,\vartheta, j)$ where $j$ is an element of $\mcJ^{\G}$.
\end{defn}

\begin{rem}\label{rem:rigid}
In the original definition given in \cite[Definition 5.3.2]{Kal19}, a regular supercuspidal datum is a tuple $(\bfS,\hat{\jmath},\chi,\vartheta,(\G',\psi,z),j)$ which furthermore contains a rigid inner twist $(\G',\psi,z)$ of $\G$ (in the sense of \cite{Kal16}).
In fact, Kaletha's $L$-packet constructed in \cite{Kal19} consist not only of representations of $\G$ but also those of all rigid inner forms of $\G$.
In this paper, since we focus only on the quasi-split case, we always take a rigid inner twist $(\G',\psi,z)$ in a regular supercuspidal datum to be the trivial twist $(\G,\id,1)$, and omit it from the notation.
\end{rem}

\begin{defn}[{\cite[Definitions 4.6.4, 4.6.5]{Kal19}}]
Let $\bfS$ be an $F$-rational maximal torus of $\G$.
A family $\{\zeta_{\alpha}\}_{\alpha\in \Phi(\bfG,\bfS)}$ of characters $\zeta_{\alpha}\colon F_{\alpha}^{\times}\rightarrow\C^{\times}$ is called a \textit{set of $\zeta$-data for $\Phi(\bfG,\bfS)$} if the following conditions are satisfied:
\begin{itemize}
\item
$\zeta_{-\alpha}=\zeta_{\alpha}^{-1}$ for any $\alpha\in \Phi(\bfG,\bfS)$,
\item
$\zeta_{\sigma(\alpha)}=\zeta_{\alpha}\circ\sigma^{-1}$ for any $\alpha\in \Phi(\bfG,\bfS)$ and $\sigma\in\Gamma$, and
\item
$\zeta_{\alpha}|_{F_{\pm\alpha}^{\times}}=\mathbbm{1}$ for any $\alpha\in \Phi(\bfG,\bfS)_{\sym}$.
\end{itemize}
\end{defn}

For a set $\zeta=\{\zeta_{\alpha}\}$ of $\zeta$-data for $\Phi(\bfG,\bfS)$, we define a character $\zeta_{S}\colon S\rightarrow\C^{\times}$ by 
\[
\zeta_{S}
\colonequals 
\prod_{\ddot{\alpha}\in \ddot{\Phi}(\G,\bfS)}\zeta_{\ddot{\alpha}},
\]\symdef{zeta-S}{$\zeta_{S}$}
where 
\begin{itemize}
\item
we put $\zeta_{\ddot{\alpha}}\colonequals \zeta_{\alpha}\circ\alpha$ if $\ddot{\alpha}\in\ddot{\Phi}(\G,\bfS)_{\asym}$ and
\item
we put $\zeta_{\ddot{\alpha}}$ to be the composition $S\xrightarrow{\alpha} F_{\alpha}^{1}\cong F_{\alpha}^{\times}/F_{\pm\alpha}^{\times}\xrightarrow{\zeta_{\alpha}}\C^{\times}$ if $\ddot{\alpha}\in\ddot{\Phi}(\G,\bfS)_{\ur}$ (here, $F_{\alpha}^{1}$ denotes the kernel of the norm map $\Nr_{F_{\alpha}/F_{\pm\alpha}}$ and the middle isomorphism is Hilbert's 90th theorem).
\end{itemize}

\begin{rem}\label{rem:zeta-data}
When we have two sets of $\chi$-data $\chi=\{\chi_{\alpha}\}_{\alpha\in\Phi(\G,\bfS)}$ and $\chi'=\{\chi'_{\alpha}\}_{\alpha\in\Phi(\G,\bfS)}$, we can produce a set of $\zeta$-data by taking the ratio of $\{\chi'_{\alpha}\}_{\alpha\in\Phi(\G,\bfS)}$ to $\{\chi_{\alpha}\}_{\alpha\in\Phi(\G,\bfS)}$.
We let $\zeta_{\chi'/\chi}$ denote the $\zeta$-data defined in this way:
\[
\zeta_{\chi'/\chi}
=
\{\zeta_{\chi'/\chi,\alpha}\}_{\alpha\in\Phi(\G,\bfS)},
\quad
\zeta_{\chi'/\chi,\alpha}\colonequals \chi'_{\alpha}\cdot\chi_{\alpha}^{-1}.
\]\symdef{zeta-chi'-chi}{$\zeta_{\chi'/\chi}$}
\end{rem}

\begin{defn}[{\cite[Definition 5.2.5]{Kal19}}]\label{defn:isom-rscp-datum}
An \textit{isomorphism} between two regular supercuspidal $L$-packet data $(\bfS,\hat{\jmath},\chi,\vartheta)$ and $(\bfS',\hat{\jmath}',\chi',\vartheta')$ is a tuple
\[
(\iota,g,\zeta)
\colon
(\bfS,\hat{\jmath},\chi,\vartheta)
\rightarrow
(\bfS',\hat{\jmath}',\chi',\vartheta')
\]
consisting of 
\begin{enumerate}
\item
an $F$-rational isomorphism $\iota\colon\bfS\rightarrow\bfS'$ of tori, 
\item
an element $g$ of $\hat{\G}$ satisfying $\hat{\jmath}\circ\hat{\iota}=[g]\circ\hat{\jmath}'$, and
\item
a set $\zeta=\{\zeta_{\alpha'}\}_{\alpha'\in\Phi(\G,\bfS'_{\h\j'})}$ of $\zeta$-data for $\Phi(\G,\bfS'_{\h\j'})$ given by $\chi_{\alpha'\circ\iota}=\chi'_{\alpha'}\cdot\zeta_{\alpha'}$
satisfying the equality 
\[
(\zeta_{S'}^{-1}\cdot\vartheta')\circ\iota
=
\vartheta.
\]
\end{enumerate}
\end{defn}

\begin{rem}\label{rem:isom-rscp-datum}
By the condition (2) of Definition \ref{defn:isom-rscp-datum}, for any $F$-rational $\h\j'$-admissible embedding $j'\colon\bfS'\hookrightarrow\G$, we can check that $j'\circ\iota\colon\bfS\hookrightarrow\G$ is an $F$-rational $\h\j$-admissible embedding.
This implies that we have an identification $\Phi(\G,\bfS'_{\h\j'})\cong\Phi(\G,\bfS'_{j'})\cong\Phi(\G,\bfS_{j'\circ\iota})\cong\Phi(\G,\bfS_{\h\j})$ given by $\alpha'\mapsto\alpha'\circ\iota$.
If we transport the set of $\chi$-data $\chi$ from $\Phi(\G,\bfS_{\h\j})$ to $\Phi(\G,\bfS'_{\h\j'})$ via this identification and write $\iota_{\ast}(\chi)$ for it, then the set of $\zeta$-data $\zeta$ in the condition (3) is nothing but $\zeta_{\iota_{\ast}(\chi)/\chi'}$ with the notation as in Remark \ref{rem:zeta-data}.
\end{rem}

\begin{defn}[{\cite[Definition 5.3.3]{Kal19}}]\label{defn:isom-rsc-datum}
An \textit{isomorphism} between two regular supercuspidal data $(\bfS,\hat{\jmath},\chi,\vartheta, j)$ and $(\bfS',\hat{\jmath}',\chi',\vartheta', j')$ is a tuple
\[
(\iota,g,\zeta,f)
\colon
(\bfS,\hat{\jmath},\chi,\vartheta, j)
\rightarrow
(\bfS',\hat{\jmath}',\chi',\vartheta', j')
\]
consisting of 
\begin{enumerate}
\item
an isomorphism of regular supercuspidal $L$-packet data 
\[
(\iota,g,\zeta)
\colon
(\bfS,\hat{\jmath},\chi,\vartheta)
\rightarrow
(\bfS',\hat{\jmath}',\chi',\vartheta'),\quad\text{and}
\]
\item
an automorphism $f$ of $\G$ given by a $G$-conjugation satisfying $j'\circ\iota=f\circ j$.
\end{enumerate}
\end{defn}

\begin{rem}
In \cite[Definitions 5.3.3]{Kal19}, the fourth parameter $f$ of a tuple $(\iota,g,\zeta,f)$ is an isomorphism of rigid inner twists.
As explained in Remark \ref{rem:rigid}, in this paper we always take a rigid inner twist to be trivial.
Since any automorphism of a rigid inner twist is necessarily a rational conjugation (\cite[Fact 5.1]{Kal16}), we may suppose that $f$ is as in Definition \ref{defn:isom-rsc-datum}.
\end{rem}

Let us investigate the isomorphism classes of regular supercuspidal data over a fixed regular supercuspidal $L$-packet datum.
Let $(\bfS,\hat{\jmath},\chi,\vartheta)$ be a regular supercuspidal $L$-packet datum.
If $(\iota,g,\zeta,f)\colon(\bfS,\hat{\jmath},\chi,\vartheta,j)\rightarrow(\bfS,\hat{\jmath},\chi,\vartheta,j')$ is an isomorphism of regular supercuspidal data ($j,j'\in\mcJ^{\G}$), then $\iota$ is necessarily the identity map by \cite[Lemma 5.2.6]{Kal19}.
In particular, the equality $j'\circ\iota=f\circ j$ in Definition \ref{defn:isom-rsc-datum} implies that $j$ and $j'$ are $G$-conjugate.
Conversely, whenever $j$ and $j'$ are $G$-conjugate, two regular supercuspidal data $(\bfS,\hat{\jmath},\chi,\vartheta,j)$ and $(\bfS,\hat{\jmath},\chi,\vartheta,j')$ are isomorphic.
Hence, the isomorphism classes of regular supercuspidal data with a fixed regular supercuspidal $L$-packet datum $(\bfS,\hat{\jmath},\chi,\vartheta)$ are parametrized by the set
\[
\mcJ^{\G}_{G}
\colonequals 
\mcJ^{\G}/{\sim_{G}}
=
\{\text{$\hat{\jmath}$-admissible $F$-rational embeddings $\bfS\hookrightarrow\G$}\}/{\sim_{G}},
\]\symdef{J-G-G}{$\mcJ^{\G}_{G}$}
where $\sim_{G}$ denotes the equivalence relation given by the $G$-conjugacy.
In the following, we often regard $\mcJ^{\G}_{G}$ as a subset of $\mcJ^{\G}$ by fixing a set of representatives as long as there is no risk of confusion.

Now we explain Kaletha's construction of regular supercuspidal $L$-packets (\cite[1153-1154 pages]{Kal19}).
Let $(\bfS,\hat{\jmath},\chi,\vartheta,j)$ be a regular supercuspidal datum with $j\in\mcJ^{\G}_{G}$.
Then $(\bfS_{j},\vartheta_{j})\colonequals (j(\bfS),\vartheta\circ j^{-1})$ is a tame elliptic extra regular pair of $\bfG$ by the definition of a regular supercuspidal datum.
We define a character $\epsilon_{\vartheta_{j}}$ of $S_{j}$ by
\[
\epsilon_{\vartheta_{j}}
\colonequals 
\epsilon_{\vartheta_{j},\asym}\cdot\epsilon_{\vartheta_{j},\ur}\cdot\epsilon_{\bfS_{j},\ram}.
\]\symdef{epsilon-vartheta-j}{$\epsilon_{\vartheta_{j}}$}
As in the manner of Section \ref{subsec:invariants-for-LLC}, we get a set $\chi_{\vartheta_{j}}$ of $\chi$-data for $\Phi(\bfG,\bfS_{j})$.
Via the identification $\Phi(\G,\bfS_{j})\cong\Phi(\G,\bfS_{\hat{\jmath}})$, this induces a set $\chi_{\vartheta_{\hat{\jmath}}}$ of $\chi$-data for $\Phi(\G,\bfS_{\hat{\jmath}})$, which is independent of the choice of $j$.
By comparing $\chi_{\vartheta_{\hat{\jmath}}}$ with the set of $\chi$-data $\chi$ contained in $(\bfS,\hat{\jmath},\chi,\vartheta)$, we get a set of $\zeta$-data $\zeta_{\chi_{\vartheta_{\hat{\jmath}}}/\chi}$ (Remark \ref{rem:zeta-data}).
We define a tame elliptic regular pair $(\bfS_{j},\vartheta'_{j})$ of $\G$ by 
\[
\vartheta'_{j}
\colonequals 
\epsilon_{\vartheta_{j}}\cdot(\vartheta\cdot\zeta_{\chi_{\vartheta_{\hat{\jmath}}}/\chi,S}^{-1})\circ j^{-1} \colon S_j\rightarrow\C^\times.
\]\symdef{vartheta-prime-j}{$\vartheta'_{j}$}

Then we get the regular supercuspidal representation $\pi_{(\bfS_{j},\vartheta'_{j})}$ of $G$ associated with the tame elliptic regular pair $(\bfS_{j},\vartheta'_{j})$ (see Section \ref{subsec:rsc}).
Note that the $G$-conjugacy class of $(\bfS_{j},\vartheta'_{j})$ is independent of the choice of a representative of $j$, hence so is the isomorphism class of $\pi_{(\bfS_{j},\vartheta'_{j})}$.
We put 
\[
\Pi_{\phi}^{\G}
\colonequals 
\{\pi_{(\bfS_{j},\vartheta'_{j})}\mid j\in\mcJ^{\G}_{G}\}.
\]\symdef{Pi-G-phi}{$\Pi_{\phi}^{\G}$}

\section{Twisted character formula for toral supercuspidal representations}\label{sec:TCF}

In this section, we review our previous work \cite{Oi25} on an explicit formula for twisted Harish-Chandra characters of toral supercuspidal representations, which is a generalization of the work of Adler--Spice \cite{AS09} and DeBacker--Spice \cite{DS18}.

\subsection{Twisted spaces and twisted characters}\label{subsec:twisted-spaces}

Let $\G$ be a quasi-split connected reductive group over $F$ with an $F$-splitting $\spl_{\G}=(\B,\T,\{X_{\alpha}\}_{\alpha})$.
Let $\theta$ be an $F$-rational involution of $\G$ (i.e., an automorphism of $\G$ of order $2$) preserving $\spl_{\G}$.
We assume that $(\bfG,\bfT,\theta)$ has no restricted root of type $2$ or $3$ in the sense of \cite{KS99}.

Following Labesse (\cite{Lab04}) and Waldspurger (\cite{Wal08}), we work with the formalism of \textit{twisted spaces} as follows.
We put 
\[
\t\G\colonequals\G\rtimes\theta= \G\theta,
\]
i.e., $\t\G$ is the non-identity connected component of the disconnected reductive group $\G\rtimes\langle\theta\rangle$.
This is a twisted space in the sense of Labesse, that is, an algebraic variety over $F$ which is a bi-$\G$-torsor.
As an algebraic variety, $\tilde{\G}$ is isomorphic to $\G$ by the map $g\mapsto g\theta$.
The right and left actions of $\G$ on $\t\G$ is given by
\[
g_{1}\cdot (g\theta)\cdot g_{2}
=
(g_{1}g\theta(g_{2}))\theta.
\]
Thus the conjugate action of $\G$ on $\t{\G}$ is given by
\[
[g_{1}](g\theta)
\colonequals 
g_{1}\cdot (g\theta)\cdot g_{1}^{-1}
=
(g_{1}g\theta(g_{1})^{-1})\theta.
\]
Note that the $\theta$-twisted conjugacy in $\G$ (as in \cite{KS99}) is amount to the $\G$-conjugacy in $\t\G$.
The conjugate action of $\t\G$ on $\G$ is also defined by, for $\delta=g\theta\in\t\G$, 
\[
[\delta]
\colonequals 
[g]\circ\theta
\colon\G\rightarrow\G.
\]

\begin{rem}\label{rem:Prasad}
We believe that our assumption on restricted roots is harmless for our purpose, i.e., study of twisted characters of toral supercuspidal representations. 
It is known that restricted roots of type $2$ or $3$ appear only when $\G$ contains a factor of type $A_{2n}$ on which $\theta$ acts nontrivially; \cite[(1.3.3)]{KS99}.
However, for example, it is known that $\GL_{2n+1}$ does not have a $\theta$-stable irreducible supercuspidal representation for such a $\theta$ whenever $p\neq2$ (\cite[Proposition 4]{Pra99}).
\end{rem}

Let us review some basics of twisted (Harish-Chandra) characters of irreducible admissible representations (see \cite[Section 2.6]{LH17} for more details).

Let $\eta\in\t{G}$.
Then $[\eta]$ is an $F$-rational automorphism of $\G$.
Recall that, for an irreducible admissible representation $\pi$ of $G$ realized on a $\C$-vector space $V$, its \textit{$\eta$-twist} $\pi^{\eta}$ is defined by the action
\[
\pi^{\eta}(g)\colonequals  \pi\circ[\eta](g)=\pi(\eta g\eta^{-1})
\]
on the same representation space $V$.
We say that $\pi$ is \textit{$\eta$-stable} if $\pi^{\eta}$ is isomorphic to $\pi$ as a representation of $G$.

Suppose that $\pi$ is an $\eta$-stable irreducible admissible representation of $G$.
We fix an intertwiner 
\[
I_{\pi}^{\eta}\colon\pi\xrightarrow{\sim}\pi^{\eta}
\]\symdef{I-eta-pi}{$I_{\pi}^{\eta}$}
(note that such an $I_{\pi}^{\eta}$ is unique up to $\C^{\times}$-multiple as $\pi$ is irreducible)
and put
\[
\t{\pi}(g\eta)\colonequals \pi(g)\circ I_{\pi}^{\eta}
\]
for any $g\eta\in\t{G}$ with $g\in G$.
Then we get a representation $\t{\pi}$ of $\t{G}$ on the representation space $V$ of $\pi$, i.e., a map $\t{\pi}\colon\t{G}\rightarrow\Aut_{\C}(V)$ satisfying the following relation for any $g_{1},g_{2}\in G$ and $\delta\in\t{G}$:
\[
\t{\pi}(g_{1}\cdot\delta\cdot g_{2})=\pi(g_{1})\circ\t{\pi}(\delta)\circ\pi(g_{2}).
\]

For any $f\in C_{c}^{\infty}(\t{G})$, an operator $\t{\pi}(f)$ on $V$ is defined by
\[
\t{\pi}(f)\colonequals  \int_{\t{G}} f(\delta) \t{\pi}(\delta) \, d\delta,
\]
where $d\delta$ is a measure on $\t{G}$ obtained by transferring a Haar measure $dg$ on $G$ by the bijection $G\rightarrow\t{G}\colon g\mapsto g\eta$.
Then, as in the untwisted case, the operator $\t{\pi}(f)$ is of finite rank and hence we can define its trace.
Moreover, the resulting distribution $f\mapsto \tr \t{\pi}(f)$ on $C_{c}^{\infty}(\t{G})$ is known to be locally $L^1$.
In this setting, the \textit{($\eta$-)twisted (Harish-Chandra) character} $\Theta_{\t{\pi}}$\symdef{Theta-pi-tilde}{$\Theta_{\t{\pi}}$} of $\pi$ is defined to be the unique locally constant function on $\t{G}_{\rs}$ such that 
\[
\tr \t{\pi}(f)
=
\int_{\t{G}_{\rs}} \Theta_{\t{\pi}}(\delta) f(\delta)\,d\delta
\]
for every $f \in C_{c}^{\infty}(\t{G})$ satisfying $\supp(f) \subseteq \t{G}_{\rs}$, where $\t{G}_{\rs}$ denotes the set of regular semisimple elements of $\t{G}$.

We emphasize that the twisted representation $\t{\pi}$ and the twisted character $\Theta_{\t{\pi}}$ depend on the choice of an intertwiner $I_{\pi}^{\eta}$ between $\pi$ and $\pi^{\eta}$ although this dependence is not reflected to the symbol $\t{\pi}$.


\subsection{Twisted maximal tori}\label{subsec:twisted-tori}
We will investigate the twisted character of a toral supercuspidal representation of $G$ arising from a ``twisted maximal torus'' of $\t\G$.
An \textit{$F$-rational twisted maximal torus} of $\t\G$ is a pair $(\t\bfS,\bfS)$ of 
\begin{itemize}
\item an $F$-rational maximal torus $\bfS$ of $\G$ and
\item an $F$-rational $\bfS$-twisted subspace $\t\bfS$ of $\t\G$ (i.e., subvariety of $\t{\G}$ which is a bi-$\bfS$-torsor under the bi-$\bfS$-action on $\t\bfS\subset\t\G$).
\end{itemize}
satisfying the following conditions:
\begin{enumerate}
\item
there exists a Borel subgroup $\bfB_{\bfS}$ of $\G$ (not necessarily $F$-rational) containing $\bfS$ and satisfying $\t\bfS=\bfN_{\t\G}(\bfS)\cap \bfN_{\t\G}(\bfB_{\bfS})$,
\item
the set $\t{S}=\t\bfS(F)$ of $F$-valued points of $\t\bfS$ is not empty.
\end{enumerate}
When $(\bfS,\t\bfS)$ is an $F$-rational twisted maximal torus of $\t\G$, we often simply say that ``$\t\bfS$ is an $F$-rational twisted maximal torus of $\t\G$''.
By the condition (1), every $\eta\in\t{\bfS}$ acts on $\bfS$ by the conjugation $[\eta]$.
Since $\bfS$ is commutative, this action is independent of the choice of $\eta$.
We let $\theta_{\bfS}$ denote this automorphism of $\bfS$.
Note that we can take $\eta$ to be $F$-rational by the condition (2), hence $\theta_{\bfS}$ is $F$-rational.
Moreover, since $\theta$ is involutive, so is $\theta_{\bfS}$ (see \cite[Lemma 3.3]{Oi25}).
We put $\bfS^{\nat}\colonequals \bfS^{\theta_{\bfS},\circ}$\symdef{S-natural}{$\bfS^{\nat}$}.
Note that, for any $\eta\in\t\bfS$, we have $\bfS^{\nat}=\bfS_{\eta}\subset\bfG_{\eta}$.

Recall that an element $\delta\in\t\G$ is said to be
\begin{itemize}
\item
\textit{semisimple} if $[\delta]$ is quasi-semisimple,
\item
\textit{regular semisimple} if $\delta$ is semisimple and $\G_{\delta}$ is a torus, and 
\item
\textit{strongly regular semisimple} if $\delta$ is semisimple and $\G^{\delta}$ is abelian
\end{itemize}
(see \cite[Sections 3.2 and 3,3]{KS99}).
We say that an $F$-rational twisted maximal torus $\t\bfS$ of $\t\G$ is \textit{elliptic} if $\bfS^{\nat}$ is anisotropic modulo $\bfA_{\t{\G}}$\symdef{A-G-tilde}{$\bfA_{\t{\G}}$}, where $\bfA_{\t{\G}}$ denotes the maximal split subtorus of $\bfZ_{\G}^{\theta}$ (note that this condition is weaker than that $\bfS$ is elliptic; see \cite[Remark 3.7]{Oi25}).
We say that a semisimple element $\delta\in\t{G}$ is \textit{elliptic} if there exists an $F$-rational elliptic twisted maximal torus $\t\bfS$ of $\t\G$ containing $\delta$.

Let us summarize some basic properties of twisted maximal tori below.

\begin{lem}[{\cite[Lemma 3.8]{Oi25}}]\label{lem:twisted-torus}
Let $\bfS$ be an $F$-rational maximal torus of $\G$.
If there exists a semisimple element $\eta\in\t{G}$ and a Borel subgroup $\bfB_{\bfS}$ containing $\bfS$ such that $(\bfB_{\bfS},\bfS)$ is preserved by $[\eta]$, then $(\t\bfS,\bfS)\colonequals (\bfS\eta,\bfS)$ is an $F$-rational twisted maximal torus of $(\t\G,\G)$.
\end{lem}

\begin{prop}[{\cite[Proposition 3.5]{Oi25}}]\label{prop:twisted-tori}
For any $F$-rational twisted maximal torus $\t\bfS$ of $\t\G$, we have $\bfZ_{\G}(\bfS^{\nat})=\bfS$ and $\bfZ_{\G}(\t\bfS)^{\circ}=\bfS^{\nat}$.
Furthermore, for any $\eta\in\t{\bfS}$, $\G_{\eta}$ is a connected reductive group with a maximal torus $\bfS^{\nat}$ (when $\eta\in\t{S}$, both $\bfS$ and $\G_{\eta}$ are $F$-rational).
\end{prop}

Here, the structure of the connected centralizer $\G_{\eta}$ for $\eta\in\t{S}$ is described as follows (see \cite[Section 3.3]{Wal08} for more details).
Let $\t\bfS$  be an $F$-rational twisted maximal torus of $\t\G$ and $\bfB_{\bfS}$ a Borel subgroup which contains $\bfS$ and is preserved by the action of $\t{\bfS}$.
By fixing an element $g_{\bfS}\in\G$ satisfying $[g_{\bfS}](\bfB_{\bfS},\bfS)=(\bfB,\T)$, we get an isomorphism $[g_{\bfS}]\colon(\t\bfS,\bfS)\xrightarrow{\sim}(\t\T,\T)$.
Note that the isomorphism $[g_{\bfS}]\colon\bfS\xrightarrow{\sim}\T$ is independent of the choice of $g_{\bfS}\in\G$ and that the automorphism $\theta_{\bfS}$ of $\bfS$ is transported to $\theta$ on $\T$ via $[g_{\bfS}]$, i.e., $\theta\circ[g_{\bfS}]=[g_{\bfS}]\circ\theta_{\bfS}$.

We put 
\begin{itemize}
\item
$Y^{\ast}(\T)\colonequals X^{\ast}(\T)/(X^{\ast}(\T)\cap(1-\theta)X^{\ast}(\T)_{\Q})$\symdef{Y-ast-T}{$Y^{\ast}(\T)$} and
\item
$Y_{\ast}(\T)\colonequals X_{\ast}(\T)/(X_{\ast}(\T)\cap(1-\theta)X_{\ast}(\T)_{\Q})$\symdef{Y-sub-ast-T}{$Y_{\ast}(\T)$}.
\end{itemize}
We write $p^{\ast}\colon X^{\ast}(\T)\rightarrow Y^{\ast}(\T)$\symdef{p-ast}{$p^{\ast}$} and $p_{\ast}\colon X_{\ast}(\T)\rightarrow Y_{\ast}(\T)$\symdef{p-sub-ast}{$p_{\ast}$} for the natural surjections.
Then we have the following:
\begin{enumerate}
\item
$Y^{\ast}(\T)\cong X^{\ast}(\T^{\theta,\circ})$ is the $\Z$-dual to $X_{\ast}(\T)^{\theta}\cong X_{\ast}(\T^{\theta,\circ})$;
\item
$Y_{\ast}(\T)$ is the $\Z$-dual to $X^{\ast}(\T)^{\theta}$.
\end{enumerate}

We put $\Theta\colonequals \langle\theta\rangle$.
Note that the action of $\Theta$ on $(\G,\T)$ induces an action on $\Phi(\G,\T)$.
For any $\alpha\in\Phi(\G,\T)$, we let $l_{\alpha}$\symdef{l-alpha}{$l_{\alpha}$} be the cardinality of the $\Theta$-orbit of $\alpha$ in $\Phi(\G,\T)$ and define an element $N(\alpha)\in\Phi(\G,\T)$ by
\[
N(\alpha)
\colonequals \sum_{i=0}^{l_{\alpha}-1} \theta^{i}(\alpha)
=\sum_{\beta\in\Theta(\alpha)}\beta.
\]\symdef{N-alpha}{$N(\alpha)$}
We also define $l_{\alpha^{\vee}}$ and $N(\alpha^{\vee})$ for any $\alpha^{\vee}\in\Phi^{\vee}(\G,\T)$ in the same manner.
For $\alpha\in\Phi(\G,\T)$, we shortly write $\alpha_{\res}\colonequals p^{\ast}(\alpha)$.
We define a set $\Phi_{\res}(\G,\T)$\symdef{Phi-res-G-T}{$\Phi_{\res}(\G,\T)$} by
\[
\Phi_{\res}(\G,\T)
=\{p^{\ast}(\alpha) \mid \alpha\in \Phi(\G,\T)\}\subset Y^{\ast}(\T)\cong X^{\ast}(\T^{\theta,\circ}).
\]
Then $\Phi_{\res}(\G,\T)$ forms a (possibly non-reduced) root system in general.
We call elements of $\Phi_{\res}(\G,\T)$ \textit{restricted roots}.
Recall that we have supposed that $(\G,\T,\theta)$ has no restricted root of type $2$ or $3$; this means that $\Phi_{\res}(\G,\T)$ is a reduced root system.

We also define a set $\Phi_{\res}^{\vee}(\G,\T)$ by
\[
\Phi_{\res}^{\vee}(\G,\T)
=\{N(\alpha^{\vee}) \mid \alpha\in \Phi^{\vee}(\G,\T)\}\subset X_{\ast}(\T)^{\theta}\cong X_{\ast}(\T^{\theta,\circ}).
\]
Then we have bijections
\begin{align*}
\Phi(\G,\T)/\Theta&\xrightarrow{1:1}\Phi_{\res}(\G,\T)\colon \alpha\mapsto \alpha_{\res} \,(\colonequals p^{\ast}(\alpha)),\\
\Phi^{\vee}(\G,\T)/\Theta&\xrightarrow{1:1}\Phi_{\res}^{\vee}(\G,\T)\colon \alpha\mapsto N(\alpha).
\end{align*}
(We note that $\Phi_{\res}(\G,\T)$ and $\Phi_{\res}^{\vee}(\G,\T)$ are denoted by $\Sigma^{\res}$ and $\check{\Sigma}^{\res}$ in \cite[Section 3.3]{Wal08}, respectively.)

Now let $\eta$ be an element of $\t{S}$ and let $\nu\in\T$ be the element such that
\[
[g_{\bfS}](\eta)=\nu\theta\in\t\T=\T\theta.
\]
Then $[g_{\bfS}]\colon \G\rightarrow\G$ induces an isomorphism between $(\G_{\eta},\bfS^{\nat})$ and $(\G_{\nu\theta},\T^{\theta,\circ})$.
In particular, the sets $\Phi(\G_{\eta},\bfS^{\nat})$ and $\Phi^{\vee}(\G_{\eta},\bfS^{\nat})$ can be identified with $\Phi(\G_{\nu\theta},\T^{\theta,\circ})$ and $\Phi^{\vee}(\G_{\nu\theta},\T^{\theta,\circ})$, respectively (note that here we ignore the Galois actions).
The latter sets are described in terms of the restricted roots and coroots as follows:
\begin{align*}
\Phi(\G_{\nu\theta},\T^{\theta,\circ})
&=\{p^{\ast}(\alpha) \mid \alpha\in \Phi(\G,\T); N(\alpha)(\nu)=1\}
\subset\Phi_{\res}(\G,\T),\\
\Phi^{\vee}(\G_{\nu\theta},\T^{\theta,\circ})
&=\{N(\alpha^{\vee}) \mid \alpha^{\vee}\in \Phi^{\vee}(\G,\T); N(\alpha)(\nu)=1\}
\subset\Phi_{\res}^{\vee}(\G,\T).
\end{align*}
Note that these sets are thought of as subsets of $X^{\ast}(\T^{\theta,\circ})$ and $X_{\ast}(\T^{\theta,\circ})$.
Under the isomorphism $[g_{\bfS}]$, we may identify $\Phi(\G_{\eta},\bfS^{\nat})$ and $\Phi^{\vee}(\G_{\eta},\bfS^{\nat})$ with subsets of $\Phi_{\res}(\G,\bfS)$ and $\Phi_{\res}^{\vee}(\G,\bfS)$, respectively, where restricted roots and coroots with respect to $(\G,\bfS)$ are defined in the same manner as above.

\subsection{Structure of the Heisenberg quotient}\label{subsec:Heisen-structure}

From now on (until the end of Section \ref{sec:TCF}), suppose that we have the following:
\begin{itemize}
\item
an $F$-rational tame elliptic twisted maximal torus $(\t\bfS,\bfS)$ of $(\t{\G},\G)$, and
\item 
a tame elliptic toral pair $(\bfS,\vartheta)$ of depth $r\in\R_{>0}$ which is $\eta$-invariant, i.e., $(\bfS,\vartheta)=(\bfS,\vartheta)^{\eta}\colonequals ([\eta]^{-1}(\bfS),\vartheta\circ[\eta])$, for $\eta\in\t{S}$.
(Note that $(\bfS,\vartheta)$ is $\eta$-invariant for some $\eta\in\t{S}$ if and only if it is for any $\eta\in\t{S}$.)
\end{itemize}
As we will see in the next subsection, the toral supercuspidal representation $\pi_{(\bfS,\vartheta)}$ of $G$ associated with such $(\bfS,\vartheta)$ is $\theta$-stable (or equivalently, $\eta$-stable for any $\eta\in\t{S}$).
Our aim of this section is to describe an explicit formula for the twisted character of $\pi_{(\bfS,\vartheta)}$.
For this, we first need to specify the choice of an intertwiner between $\pi_{(\bfS,\vartheta)}$ and its $\eta$-twist $\pi_{(\bfS,\vartheta)}^{\eta}$ because the definition of the twisted character depends on it (see Section \ref{subsec:twisted-spaces}).
The choice of such an intertwiner given in \cite[Section 6]{Oi25} is based on the structure of the Heisenberg group $J/N$ associated with $(\bfS,\vartheta)$ (see Section \ref{subsec:rsc} for the notation), so let us review it here.

By fixing a finite tamely ramified extension $E$ of $F$ splitting $\bfS$, we put
\[
\bfV\colonequals \Lie(\bfS,\G)(E)_{\x,(r,s):(r,s+)}
\quad
\text{and}
\quad
V\colonequals \bfV^{\Gamma}.
\]\symdef{V}{$V$}
Then we have
\[
J/J_{+}
\cong
(S,G)_{\x,(r,s):(r,s+)}
=
(\bfS,\G)(E)_{\x,(r,s):(r,s+)}^{\Gamma}.
\]
Thus the exponential map $\Lie(\bfS,\G)(E)_{\x,(r,s):(r,s+)}\xrightarrow{\sim}(\bfS,\G)(E)_{\x,(r,s):(r,s+)}$ induces an identification
\[
V\xrightarrow{\sim}J/J_{+}.
\]

Let us investigate the space $V$ by using the root space decomposition of $\bmfg$ with respect to the maximal torus $\bfS$ in $\G$.
For $\alpha\in\Phi(\G,\bfS)$, we put $\bfV_{\alpha}$ to be the image of $\bmfg_{\alpha}(E)\cap\Lie(\bfS,\G)(E)_{\x,(r,s)}$ in $\Lie(\bfS,\G)(E)_{\x,(r,s):(r,s+)}$.
Then the root space decomposition $\bmfg\cong\bmfs\oplus\bigoplus_{\alpha\in\Phi(\G,\bfS)}\bmfg_{\alpha}$ naturally induces a decomposition
\[
\bfV
=
\bigoplus_{\alpha\in\Phi(\G,\bfS)}\bfV_{\alpha}.
\]
For each $\alpha\in\Phi(\G,\bfS)$, we put $V_{\alpha}\colonequals \bfV_{\alpha}^{\Gamma_{\alpha}}$ (recall that $\Gamma_{\alpha}$ is the stabilizer of $\alpha$ in $\Gamma$).
Here, note that $\bfV_{\alpha}$ and $V_{\alpha}$ might be zero depending on $\alpha\in\Phi(\G,\bfS)$.
We define a subset $\Xi(\G,\bfS)$ of $\Phi(\G,\bfS)$ by
\[
\Xi(\G,\bfS)\colonequals \{\alpha\in\Phi(\G,\bfS)\mid V_{\alpha}\neq0\}.
\]\symdef{Xi-G-S}{$\Xi(\G,\bfS)$}
Note that, for any $\alpha\in\Xi(\G,\bfS)$, the space $V_{\alpha}$ is (noncanonically) isomorphic to the residue field $k_{\alpha}$ of $F_{\alpha}$.
Also note that $\Xi(\G,\bfS)$ is preserved by the action of $\Sigma=\Gamma\times\{\pm1\}$ on $\Phi(\G,\bfS)$.
In the following, we simply write $\Xi$ for $\Xi(\G,\bfS)$.\symdef{Xi}{$\Xi$}
It can be checked that the set $\Xi$ does not contain any symmetric ramified root (see \cite[Lemma 6.2]{Oi25}).

For $\dot{\alpha}\in\dot{\Phi}(\G,\bfS)$, we put
\[
\bfV_{\dot{\alpha}}\colonequals \bigoplus_{\beta\in\dot{\alpha}}\bfV_{\beta}
\quad\text{and}\quad
V_{\dot{\alpha}}\colonequals \bfV_{\dot{\alpha}}^{\Gamma}.
\]
Then, for any $\dot{\alpha}\in\dot{\Phi}(\G,\bfS)$, we have
\[
V_{\alpha}
\xrightarrow{\sim}
V_{\dot{\alpha}}=\Bigl(\bigoplus_{\beta\in\dot{\alpha}}\bfV_{\beta}\Bigr)^{\Gamma}
\colon X_{\alpha} \mapsto \sum_{\sigma\in\Gamma/\Gamma_{\alpha}}\sigma(X_{\alpha}).
\]\symdef{V-alpha}{$V_{\alpha}$}\symdef{V-alpha-dot}{$V_{\dot{\alpha}}$}
Therefore we get
\begin{align}\label{eq:decomp}
V
\cong
\bigoplus_{\dot{\alpha}\in\dot{\Xi}}V_{\dot{\alpha}}
=
\bigoplus_{\ddot{\alpha}\in\ddot{\Xi}}V_{\ddot{\alpha}},
\end{align}
where we put $V_{\ddot{\alpha}}\colonequals V_{\dot{\alpha}}\oplus V_{-\dot{\alpha}}$ for $\alpha\in\Xi_{\asym}$ and $V_{\ddot{\alpha}}\colonequals V_{\dot{\alpha}}$ for $\alpha\in\Xi_{\sym}$.\symdef{V-alpha-ddot}{$V_{\ddot{\alpha}}$}

Recall that $V\cong J/J_{+}$ has a structure of a symplectic $\F_{p}$-vector by 
\[
(J/J_{+})\times (J/J_{+}) \rightarrow \mu_{p}\cong\F_{p}\colon (g,g')\mapsto\h\vartheta([g,g'])
\]
(see Section \ref{subsec:rsc}).
In fact, the above decomposition (\ref{eq:decomp}) gives an orthogonal decomposition of $V$ into symplectic subspaces.
See \cite[Section 6.1]{Oi25} for a detailed description of each symplectic subspace $V_{\ddot{\alpha}}$.

Now let $\eta\in\t{S}$ denote any topologically semisimple element in the sense that the order of $\eta$ in $\t{S}/A_{\t\G}$ is finite of prime-to-$p$.

Recall that, in Section \ref{subsec:twisted-tori}, we introduced the notion of a restricted root.
More precisely, the action of $[\eta]$ on $\bmfg$ (for any $\eta\in \t{S}$) induces an action on the set $\Phi(\G,\bfS)$ of order $2$.
Since this action does not depend on the choice of $\eta\in\t{S}$, let us write $\theta_{\bfS}$ for this action by abuse of notation and $\Theta_{\bfS}$ for the group $\langle\theta_{\bfS}\rangle$ generated by $\theta_{\bfS}$.
Concretely, for any $\alpha\in\Phi(\G,\bfS)$, $\theta_{\bfS}(\alpha)$ is the root given by $\theta_{\bfS}(\alpha)=\alpha\circ[\eta]^{-1}$.
Whenever there is no risk of confusion, we abbreviate $\theta_{\bfS}(\alpha)$ even as $\theta(\alpha)$ (and also write $\Theta$ for $\Theta_{\bfS}$).
We note that, for $X_{\alpha}\in\mfg_{\alpha}$, $[\eta](X_{\alpha})$ belongs to $\mfg_{\theta(\alpha)}$.
We also note that, as $\eta$ is $F$-rational, the actions of $\Theta_{\bfS}$ and $\Sigma=\Gamma\times\{\pm1\}$ on $\Phi(\G,\bfS)$ commute.
Especially, the symmetry of $\Phi(\G,\bfS)$ is preserved by $\Theta_{\bfS}$.
Each restricted root $\alpha_{\res}\in\Phi_{\res}(\G,\bfS)$ is naturally identified with the $\Theta_{\bfS}$-orbit $\Theta_{\bfS}\alpha$ of $\alpha$ in $\Phi(\G,\bfS)$:
\[
\Phi(\G,\bfS)
\twoheadrightarrow\Phi(\G,\bfS)/\Theta_{\bfS}
\xrightarrow{1:1}\Phi_{\res}(\G,\bfS)
\colon \alpha\mapsto\alpha_{\res}.
\]
Note that, since $\Phi_{\res}(\G,\bfS)$ carries a Galois action induced from that of $\Phi(\G,\bfS)$, we can also discuss the symmetry of a restricted root.
For any $\alpha\in\Phi(\G,\bfS)$, we put $\Gamma_{\alpha_{\res}}$ to be the stabilizer in $\Gamma$ of the restricted root $\alpha_{\res}$ (or, equivalently, the $\Theta_{\bfS}$-orbit $\Theta_{\bfS}\alpha$ of $\alpha$):
\[
\Gamma_{\alpha_{\res}}
\colonequals \{\sigma\in\Gamma\mid \sigma(\alpha_{\res})=\alpha_{\res}\}
=\{\sigma\in\Gamma\mid \sigma(\Theta_{\bfS}\alpha)=\Theta_{\bfS}\alpha\}.
\]
Similarly, we put $\Gamma_{\pm\alpha_{\res}}$ to be the stabilizer in $\Gamma$ of the set $\{\pm\alpha_{\res}\}$:
\[
\Gamma_{\pm\alpha_{\res}}
\colonequals \{\sigma\in\Gamma\mid \sigma(\{\pm\alpha_{\res}\})=\{\pm\alpha_{\res}\}\}
=\{\sigma\in\Gamma\mid \sigma(\{\pm\Theta_{\bfS}\alpha\})=\{\pm\Theta_{\bfS}\alpha\}\}.
\]
Let $F_{\alpha_{\res}}$ and $F_{\pm\alpha_{\res}}$ denote the subfields of $\ol{F}$ fixed by $\Gamma_{\alpha_{\res}}$ and $\Gamma_{\pm\alpha_{\res}}$, respectively.
\[
\xymatrix@R=10pt{
F_{\alpha_{\res}}\ar@{}[r]|*{\subset}&F_{\alpha}&&\Gamma_{\alpha_{\res}}\ar@{}[d]|{\bigcap}&\Gamma_{\alpha}\ar@{}[l]|*{\supset}\ar@{}[d]|{\bigcap}\\
F_{\pm\alpha_{\res}}\ar@{}[r]|*{\subset}\ar@{}[u]|{\bigcup}&F_{\pm\alpha}\ar@{}[u]|{\bigcup}&&\Gamma_{\pm\alpha_{\res}}&\Gamma_{\pm\alpha}\ar@{}[l]|*{\supset}
}
\]

As reviewed in Section \ref{subsec:twisted-tori}, the group $\G_{\eta}$ is a connected reductive group with a maximal torus $\bfS^{\nat}$.
Furthermore, $\Phi(\G_{\eta},\bfS^{\nat})$ is regarded as a subset of the set $\Phi_{\res}(\G,\bfS)$.
By \cite[Proposition 5.6]{Oi25}, the point $\x$ can be regarded as a point of $\mcA(\bfS^{\nat},F)\subset\mcB(\G_{\eta},F)$.
We introduce the subgroups $J_{\eta}$ (resp.\ $J_{\eta,+}$) in the same way as $J$ (resp.\ $J_{+}$) by using $(\G_{\eta},\bfS^{\nat},\x,r,s(+))$ instead of $(\G,\bfS,\x,r,s(+))$, i.e.,
\[
J_{\eta}\colonequals (S^{\nat},G_{\eta})_{\x,(r,s)}
\quad\text{and}\quad
J_{\eta,+}\colonequals (S^{\nat},G_{\eta})_{\x,(r,s+)}.
\]
By the same discussion as above, if we put
\[
\bfV_{\eta}\colonequals \Lie(\bfS_{\eta},\G_{\eta})(E)_{\x,(r,s):(r,s+)}
\quad
\text{and}
\quad
V_{\eta}\colonequals \bfV_{\eta}^{\Gamma}
\]
then we have $J_{\eta}/J_{\eta,+}\cong V_{\eta}$ and a root space decomposition similar to (\ref{eq:decomp}):
\[
V_{\eta}
\cong
\bigoplus_{\ddot{\alpha}_{\res}\in\ddot{\Xi}_{\eta}}V_{\eta,\ddot{\alpha}_{\res}},
\]
where we use the notation defined in the same way as in Section \ref{subsec:Heisen-structure}, e.g., 
\[
\Xi_{\eta}
\colonequals \Xi(\bfG_{\eta},\bfS^{\nat})
\colonequals \{\alpha_{\res}\in\Phi(\G_{\eta},\bfS^{\nat}) \mid V_{\eta,\alpha_{\res}}\neq0\}.
\]

We have a natural identification
\[
V_{\eta}
\cong(S^{\nat},G_{\eta})_{\x,(r,s):(r,s+)}
\cong(S,G)_{\x,(r,s):(r,s+)}^{\eta}
\cong V^{\eta},
\]
where $(S,G)_{\x,(r,s):(r,s+)}^{\eta}$ and $V^{\eta}$ denote the set of $[\eta]$-fixed points in $(S,G)_{\x,(r,s):(r,s+)}$ and $V$, respectively.
Let us investigate this identification more precisely.

The Lie algebra $\bmfg_{\eta}$ of $\G_{\eta}$ is naturally identified with the $[\eta]$-fixed points $\bmfg^{\eta}$ of the Lie algebra $\bmfg$ of $\G$.
If $\alpha_{\res}\in\Phi(\G_{\eta},\bfS^{\nat})$ is a restricted root obtained from $\alpha\in\Phi(\G,\bfS)$, then the root subspace $\bmfg_{\eta,\alpha_{\res}}$ of $\bmfg_{\eta}$ is identified with the $[\eta]$-fixed points in the sum $\bigoplus_{\alpha'\in\Theta\alpha}\bmfg_{\alpha'}$ of root subspaces of $\bmfg$:
\[
\bmfg_{\eta,\alpha_{\res}}
\cong
\Bigl(\bigoplus_{\alpha'\in\Theta\alpha}\bmfg_{\alpha}\Bigr)^{\eta}.
\]
This induces an identification
\[
V_{\eta,\ddot{\alpha}_{\res}}
\cong
\Bigl(\bigoplus_{\ddot{\beta}\in\Sigma\backslash(\Sigma\times\Theta)\alpha} V_{\ddot{\beta}}\Bigr)^{\eta}
\]
for any $\alpha_{\res}\in\Xi_{\eta}$.
Let us put $V_{\Theta(\ddot{\alpha})}\colonequals \bigoplus_{\ddot{\beta}\in\Sigma\backslash(\Sigma\times\Theta)\alpha} V_{\ddot{\beta}}$.
In particular, by letting $\Xi_{\res}$ be the set of restricted roots associated to $\Xi$, the set $\Xi_{\eta}$ can be thought of as the set of restricted roots $\alpha_{\res}\in\Xi_{\res}$ such that the $[\eta]$-action has a nonzero fixed point in $V_{\Theta(\ddot{\alpha})}$.
\[
\xymatrix@R=10pt{
\Phi(\G,\bfS)\ar@{->>}[r]&\Phi(\G,\bfS)/\Theta_{\bfS} \ar^-{1:1}[r]&\Phi_{\res}(\G,\bfS)&\Phi(\G_{\eta},\bfS^{\nat})\ar@{}[l]|*{\supset}\\
\Xi\ar@{->>}[r]\ar@{}[u]|{\bigcup}&\Xi/\Theta_{\bfS} \ar^-{1:1}[r]\ar@{}[u]|{\bigcup}&\Xi_{\res}\ar@{}[u]|{\bigcup}&\Xi_{\eta}\ar@{}[l]|*{\supset}\ar@{}[u]|{\bigcup}
}
\]

\subsection{Intertwiner for toral supercuspidal representations}\label{subsec:twist-int}

In the previous subsection, we have fixed
\begin{itemize}
\item
an $F$-rational tame elliptic twisted maximal torus $(\t\bfS,\bfS)$ of $(\t{\G},\G)$, and
\item 
a tame elliptic toral pair $(\bfS,\vartheta)$ of depth $r\in\R_{>0}$ which is $\eta$-invariant, i.e., $(\bfS,\vartheta)=(\bfS,\vartheta)^{\eta}$, for $\eta\in\t{S}$.
\end{itemize}

Recall from Section \ref{subsec:rsc} that $\pi_{(\bfS,\vartheta)}$ is defined to be the compact induction $\cInd_{K}^{G}\rho_{(\bfS,\vartheta)}$ of a representation $\rho_{(\bfS,\vartheta)}$ of an open compact-mod-center subgroup $K=SJ$ of $G$.
Hence the $\eta$-twisted representation $\pi_{(\bfS,\vartheta)}^{\eta}$ is isomorphic to the compact induction of $\rho_{(\bfS,\vartheta)}^{\eta}$ from $\eta^{-1}K\eta=K$ to $G$ by the following explicit intertwiner:
\[
\cInd_{K}^{G}\rho_{(\bfS,\vartheta)}^{\eta}\xrightarrow{\sim}(\cInd_{K}^{G}\rho_{(\bfS,\vartheta)})^{\eta}=\pi_{(\bfS,\vartheta)}^{\eta}\colon f\mapsto f\circ[\eta]^{-1}.
\tag{1}
\]

Let us consider the relationship between the representations $\rho_{(\bfS,\vartheta)}^{\eta}$ and $\rho_{(\bfS,\vartheta)}$ of $K$.
The representation $\rho_{(\bfS,\vartheta)}$ of $K$ is defined to be the push-out of the representation $\omega_{(\bfS,\vartheta)}\otimes(\vartheta\ltimes\mathbbm{1})$ of $S\ltimes J$ along the natural multiplication map $S\ltimes J\twoheadrightarrow SJ$.
Hence $\rho_{(\bfS,\vartheta)}^{\eta}$ is the push-out of $\omega_{(\bfS,\vartheta)}^{\eta}\otimes(\vartheta^{\eta}\ltimes\mathbbm{1})$ along $S\ltimes J\twoheadrightarrow SJ$.
We note that both of $\omega_{(\bfS,\vartheta)}^{\eta}$ and $\omega_{(\bfS,\vartheta)}$ are Heisenberg--Weil representations with central character $\h\vartheta^{\eta}=\h\vartheta$.
Hence, by the Stone--von Neumann theorem, $\omega_{(\bfS,\vartheta)}^{\eta}$ and $\omega_{(\bfS,\vartheta)}$ are isomorphic.
Let us fix an intertwiner
\[
I_{\omega_{(\bfS,\vartheta)}}^{\eta}\colon\omega_{(\bfS,\vartheta)}\xrightarrow{\sim}\omega_{(\bfS,\vartheta)}^{\eta},
\]\symdef{I-eta-omega}{$I_{\omega_{(\bfS,\vartheta)}}^{\eta}$}
which naturally induces an intertwiner
\[
I_{\rho_{(\bfS,\vartheta)}}^{\eta}\colon\rho_{(\bfS,\vartheta)}\xrightarrow{\sim}\rho_{(\bfS,\vartheta)}^{\eta}.
\]\symdef{I-eta-rho}{$I_{\rho_{(\bfS,\vartheta)}}^{\eta}$}
Then we get an intertwiner between $\cInd_{K}^{G}\rho_{(\bfS,\vartheta)}$ and $\cInd_{K}^{G}\rho_{(\bfS,\vartheta)}^{\eta}$ given by 
\[
\cInd_{K}^{G}\rho_{(\bfS,\vartheta)}\xrightarrow{\sim}\cInd_{K}^{G}\rho_{(\bfS,\vartheta)}^{\eta}\colon f\mapsto I_{\rho_{(\bfS,\vartheta)}}^{\eta}\circ f.
\tag{2}
\]
Therefore, combining (1) with (2), we obtain an intertwiner $I_{\pi_{(\bfS,\vartheta)}}^{\eta}$ between $\pi_{(\bfS,\vartheta)}$ and $\pi_{(\bfS,\vartheta)}^{\eta}$ given by $f\mapsto I_{\rho_{(\bfS,\vartheta)}}^{\eta}\circ f\circ [\eta]^{-1}$:
\[
\pi_{(\bfS,\vartheta)}
=\cInd_{K}^{G}\rho_{(\bfS,\vartheta)}
\xrightarrow{(2)}
\cInd_{K}^{G}\rho_{(\bfS,\vartheta)}^{\eta}
\xrightarrow{(1)}
(\cInd_{K}^{G}\rho_{(\bfS,\vartheta)})^{\eta}
=\pi_{(\bfS,\vartheta)}^{\eta}.
\]

We emphasize that we have not specified yet the choice of an intertwiner $I_{\omega_{(\bfS,\vartheta)}}^{\eta}$ of Heisenberg--Weil representations.

From now on, we fix a ``base point'' $\ul{\eta}\in\t{S}$ which is topologically semisimple.\symdef{eta-ul}{$\ul{\eta}$}
Note that we can always find such an element by choosing arbitrary element $\ul{\eta}\in\t{S}$ first and then replacing it with its topologically semisimple part in the sense of a topological Jordan decomposition.
In the following, we explain our choice of an intertwiner $I_{\omega_{(\bfS,\vartheta)}}^{\ul{\eta}}$ for this element $\ul{\eta}$.

Let us investigate the action $[\ul{\eta}]$ on $J/J_{+}$ through the isomorphism $V\cong J/J_{+}$ and the decomposition (\ref{eq:decomp}) of $V$.
Note that $[\ul{\eta}]$ preserves the symplectic structure of $V$.
Indeed, for any $g,g'\in J/J_{+}$, we have
\[
\h\vartheta([[\ul{\eta}](g),[\ul{\eta}](g')])
=\h\vartheta([\ul{\eta} g\ul{\eta}^{-1},\ul{\eta} g'\ul{\eta}^{-1}])
=\h\vartheta([\ul{\eta}]([g,g']))
=\h\vartheta([g,g']).
\]
Moreover, each $V_{\ddot{\alpha}}$ is mapped onto $V_{\theta(\ddot{\alpha})}$, respectively.
In particular, the action of $\Theta_{\bfS}$ on $\Phi(\G,\bfS)$ preserves $\Xi$.

As in Section \ref{subsec:twisted-tori}, for any $\alpha\in\Phi(\G,\bfS)$, we let $l_{\alpha}$ be the cardinality of the $\Theta_{\bfS}$-orbit of $\alpha$.
We furthermore introduce a number denoted by $m_{\alpha}$ as follows:\symdef{m-alpha}{$m_{\alpha}$}
\begin{defn}\label{defn:l'}
For $\alpha\in\Phi(\G,\bfS)$, let $m_{\alpha}$ be the order of $\Sigma\backslash(\Sigma\times\Theta_{\bfS})\alpha$.
In other words, $m_{\alpha}$ is the smallest positive integer such that $\theta^{m_{\alpha}}(\ddot{\alpha})=\ddot{\alpha}$ (hence $m_{\alpha}\mid l_{\alpha}$).
\end{defn}

For $\ddot{\alpha}\in\ddot{\Xi}$, let us write $(\omega_{\ddot{\alpha}},W_{\ddot{\alpha}})$ for a Heisenberg--Weil representation of $\Sp(V_{\ddot{\alpha}})\ltimes\bbH(V_{\ddot{\alpha}})$ with central character given by $\hat{\vartheta}$, which is unique up to isomorphism.
Since the action $[\ul{\eta}]$ on $V$ induces an symplectic isomorphism from $V_{\ddot{\alpha}}$ to $V_{\theta(\ddot{\alpha})}$, an isomorphism from $\Sp(V_{\ddot{\alpha}})\ltimes\bbH(V_{\ddot{\alpha}})$ to $\Sp(V_{\theta(\ddot{\alpha})})\ltimes\bbH(V_{\theta(\ddot{\alpha})})$ is induced (for which we write $[\ul{\eta}]_{\ast}$).
Then the pull back $(\omega_{\theta(\ddot{\alpha})}^{\ul{\eta}},W_{\theta(\ddot{\alpha})})$ of the Heisenberg--Weil representation $W_{\theta(\ddot{\alpha})}$ of $\Sp(V_{\theta(\ddot{\alpha})})\ltimes\bbH(V_{\theta(\ddot{\alpha})})$ to $\Sp(V_{\ddot{\alpha}})\ltimes\bbH(V_{\ddot{\alpha}})$ via $[\ul{\eta}]_{\ast}$ is isomorphic to $(\omega_{\ddot{\alpha}},W_{\ddot{\alpha}})$.
For each $\ddot{\alpha}\in \ddot{\Xi}$, we fix an intertwiner
\[
I_{\ddot{\alpha}}^{\ul{\eta}}\colon \omega_{\ddot{\alpha}}\xrightarrow{\sim} \omega_{\theta(\ddot{\alpha})}^{\ul{\eta}}
\]\symdef{I-eta-alpha}{$I_{\ddot{\alpha}}^{\ul{\eta}}$}
as follows.
If we put
\[
I^{\ul{\eta}}_{\Theta(\ddot{\alpha})}
\colonequals 
I_{\theta^{m_{\alpha}-1}(\ddot{\alpha})}^{\ul{\eta}} \circ\cdots\circ I_{\ddot{\alpha}}^{\ul{\eta}},
\]\symdef{I-eta-theta-alpha}{$I^{\ul{\eta}}_{\Theta(\ddot{\alpha})}$}
then $I^{\ul{\eta}}_{\Theta(\ddot{\alpha})}$ is an automorphism of $W_{\ddot{\alpha}}$ such that 
\[
\xymatrix{
W_{\ddot{\alpha}}\ar^-{I^{\ul{\eta}}_{\Theta(\ddot{\alpha})}}[rr]\ar_-{\omega_{\ddot{\alpha}}(g,h)}[d]&&W_{\ddot{\alpha}}\ar^-{\omega_{\ddot{\alpha}}([\ul{\eta}]^{m_{\alpha}}_{\ast}(g,h))}[d]\\
W_{\ddot{\alpha}}\ar_-{I^{\ul{\eta}}_{\Theta(\ddot{\alpha})}}[rr]&&W_{\ddot{\alpha}}
}
\]
is commutative for any $(g,h)\in \Sp(V_{\ddot{\alpha}})\ltimes\bbH(V_{\ddot{\alpha}})$.
Note that $[\ul{\eta}]^{m_{\alpha}}$ is a symplectic automorphism of $V$ preserving $V_{\ddot{\alpha}}$.
Hence $I^{\ul{\eta}}_{\Theta(\ddot{\alpha})}$ must be a scalar multiple of the Heisenberg--Weil action $\omega_{\ddot{\alpha}}([\ul{\eta}]^{m_{\alpha}})$.
We choose $I^{\ul{\eta}}_{\ddot{\alpha}}$ for each $\ddot{\alpha}\in\ddot{\Xi}$ so that we have
\[
I^{\ul{\eta}}_{\Theta(\ddot{\alpha})}
=
\omega_{\ddot{\alpha}}([\ul{\eta}]^{m_{\alpha}}).
\]
(This choice still does not specify each $I_{\ddot{\alpha}}^{\ul{\eta}}$ uniquely, but this ambiguity does not matter in the following discussion.)

Recall that the representation $\omega=\omega_{(\bfS,\vartheta)}$ is a Heisenberg--Weil representation of $\Sp(V)\ltimes\bbH(V)$ with central character $\hat{\vartheta}$.
Since we have the decomposition \eqref{eq:decomp}, $\omega$ can be realized by tensoring Heisenberg--Weil representations $\omega_{\ddot{\alpha}}$ for $\ddot{\alpha}\in\ddot{\Xi}$.
Furthermore, by tensoring the fixed intertwiners $I_{\ddot{\alpha}}^{\ul{\eta}}$, we get an intertwiner between $\omega$ and its $[\ul{\eta}]_{\ast}$-twist $\omega^{\ul{\eta}}$.
Let us write $I_{\omega}^{\ul{\eta}}$ for the intertwiner obtained in this way:
\[
I_{\omega}^{\ul{\eta}}=I_{\omega_{(\bfS,\vartheta)}}^{\ul{\eta}}\colon\omega_{(\bfS,\vartheta)}\xrightarrow{\sim}\omega_{(\bfS,\vartheta)}^{\ul{\eta}}.
\]\symdef{I-omega-eta-ul}{$I_{\omega}^{\ul{\eta}}$}

Now we have obtained an intertwiner $I_{\pi_{(\bfS,\vartheta)}}^{\ul{\eta}}$ between $\pi_{(\bfS,\vartheta)}$ and its $[\ul{\eta}]_{\ast}$-twist $\pi_{(\bfS,\vartheta)}^{\ul{\eta}}$.
Let $\Theta_{\t{\pi}_{(\bfS,\vartheta)}}$ be the twisted character of $\pi_{(\bfS,\vartheta)}$ with respect to $I_{\pi_{(\bfS,\vartheta)}}^{\ul{\eta}}$.

\subsection{Good product expansion in twisted spaces}\label{subsec:approx}
Before we start talking about an explicit formula for the twisted character $\Theta_{\t{\pi}_{(\bfS,\vartheta)}}$, we also introduce a twisted version of the theory of good product expansion of Adler and Spice \cite{AS08} (see \cite[Section 3.4]{Oi25} for more details).

We first recall the definition of a good product expansion of elements of $p$-adic groups in the untwisted case.
We temporarily let $\G$ be any tamely ramified connected reductive group over $F$.
Let $\bar{\G}$ be the quotient $\G/\bfZ_{\G}^{\circ}$ of $\G$ by the identity component of the center $\bfZ_{\G}$ of $\G$.

\begin{defn}[{\cite[Definitions 4.11 and 6.1]{AS08}}]\label{defn:good-element}
\begin{enumerate}
\item
We say that an element $\gamma\in G$ is \textit{good of depth zero} if $\gamma$ is semisimple and its image $\bar{\gamma}$ in $\bar{G}$ is \textit{absolutely semisimple}, i.e., every character value of $\bar{\gamma}$ (see \cite[Definition A.4]{AS08}) is of finite prime-to-$p$ order.
\item
For $d\in\R_{>0}$, an element $\gamma\in G$ is said to be \textit{good of depth $d$} if there exists an $F$-rational tame-modulo-center torus $\bfS$ in $\G$ such that
\begin{itemize}
\item
$\gamma\in S_{d}\smallsetminus S_{d+}$, and
\item
for every $\alpha\in\Phi(\G,\bfS)$, $\alpha(\gamma)=1$ or $\val_{F}(\alpha(\gamma)-1)=d$.
\end{itemize}
\end{enumerate}
\end{defn}

\begin{defn}[{\cite[Definition 6.4]{AS08}}]\label{defn:good-seq}
For $r\in\tilde{\R}$, a sequence $\underline{\gamma}=(\gamma_{i})_{0\leq i<r}$ of elements of $G$ indexed by real numbers $0\leq i <r$ is called a \textit{good sequence} if
\begin{itemize}
\item
$\gamma_{i}$ is $1$ or a good of depth $i$ for each $0\leq i<r$, and
\item
there exists an $F$-rational tame torus $\bfS$ of $\G$ such that $\gamma_{i}\in S$ for every $0\leq i<r$.
\end{itemize}
\end{defn}

To a good sequence $\underline{\gamma}=(\gamma_{i})_{0\leq i<r}$, we associate subgroups of $\G$ to $\underline{\gamma}$ as follows:
\[
C_{\G}^{(r)}(\underline{\gamma})
\colonequals 
\Bigl(\bigcap_{0\leq i<r} \bfZ_{\G}(\gamma_{i})\Bigr)^{\circ},
\quad
Z_{\G}^{(r)}(\underline{\gamma})
\colonequals 
\bfZ_{C_{\G}^{(r)}(\underline{\gamma})}.
\]

\begin{defn}[{\cite[Definition 6.8]{AS08}}]\label{defn:approx}
For $\gamma\in G$, a good sequence $\underline{\gamma}=(\gamma_{i})_{0\leq i<r}$ $(r\in \tilde{\R}_{>0})$ is called an \textit{$r$-approximation to $\gamma$} if there exists a point $\x\in\mcB(C_{\G}^{(r)}(\underline{\gamma}),F)$ satisfying $\gamma\in(\prod_{0\leq i<r}\gamma_{i})G_{\x,r}$.
When we have $\gamma\in C_{\G}^{(r)}(\underline{\gamma})$, we say that $\underline{\gamma}$ is a \textit{normal $r$-approximation to $\gamma$}.
\end{defn}

When we have a normal $r$-approximation $\underline{\gamma}$ to $\gamma$, we put
\[
\gamma_{<r}\colonequals \prod_{0\leq i<r}\gamma_{i},\quad
\gamma_{\geq r}\colonequals \gamma\cdot\gamma_{<r}^{-1}
\]\symdef{gamma-less-r}{$\gamma_{<r}$}\symdef{gamma-geq-r}{$\gamma_{\geq r}$}
and simply say that ``$\gamma=\gamma_{<r}\cdot\gamma_{\geq r}$ is a normal $r$-approximation.''
Note that $\gamma_{\geq r}$ commutes with $\gamma_{<r}$ when $\gamma=\gamma_{<r}\cdot\gamma_{\geq r}$ is a normal $r$-approximation.

Now we move on to the setting of twisted spaces.
Let $(\G,\theta)$ be as in Section \ref{subsec:twisted-spaces}.
In particular, we have a twisted space $\t{\G}=\G\theta$.
We put $\G^{\dagger}\colonequals \G\rtimes\langle\theta\rangle$.
Note that this is a disconnected reductive group whose identity component is $\G$ and non-identity component is given by $\t\G$.
Recall that $\bfA_{\t{\G}}$ is the maximal split subtorus of $\bfZ_{\G}^{\theta}$.
To extend the theory of Adler--Spice to $\G^{\dagger}$, we utilize Spice's topological Jordan decomposition.

\begin{defn}[{\cite[Definition 1.6]{Spi08}}]\label{defn:Jordan}
For $\gamma\in G^{\dagger}$, we say that a pair $(\gamma_{0},\gamma_{+})$ of elements of $G^{\dagger}$ is a \textit{topological $p$-Jordan decomposition modulo $A_{\t{\G}}$} of $\gamma$ if 
\begin{itemize}
\item
$\gamma=\gamma_{0}\gamma_{+}=\gamma_{+}\gamma_{0}$, 
\item
$\gamma_{0}$ is absolutely $p$-semisimple modulo $A_{\t{\G}}$, i.e., the image $\bar{\gamma}_{0}$ of $\gamma_{0}$ in $G^{\dagger}/A_{\t{\G}}$ is of finite prime-to-$p$ order, and
\item
$\gamma_{+}$ is topologically $p$-unipotent modulo $A_{\t{\G}}$, i.e., the image $\bar{\gamma}_{+}$ of $\gamma_{+}$ in $G^{\dagger}/A_{\t{\G}}$ satisfies $\lim_{n\rightarrow\infty}\bar{\gamma}_{+}^{p^{n}}=1$.
\end{itemize}
\end{defn}

In this paper, we refer to a topological $p$-Jordan decomposition modulo $A_{\t{\G}}$ simply as a \textit{topological Jordan decomposition}.
Similarly, when an element $\gamma$ is absolutely $p$-semisimple modulo $A_{\t\G}$ (resp.\ topologically $p$-unipotent modulo $A_{\t\G}$), we often simply say that $\gamma$ is topologically semisimple (resp.\ topologically unipotent) as long as there is no risk of confusion.

\begin{defn}\label{defn:twisted-approx}
Let $\delta\in\t{G}$ be an elliptic semisimple element.
A \textit{normal $r$-approximation to $\delta$} ($r\in\tilde{\R}_{>0}$) is a pair $(\delta=\delta_{0}\delta_{+}, \underline{\delta}_{+})$ of 
\begin{itemize}
\item
a topological Jordan decomposition $\delta=\delta_{0}\delta_{+}$ of $\delta$ such that $\delta_{0}\in\t{G}$ and $\delta_{+}\in G_{\delta_{0},0+}$, and
\item
a normal $r$-approximation $\underline{\delta}_{+}=(\delta_{i})_{0<i<r}$ to $\delta_{+}$ in $G_{\delta_{0}}$.
\end{itemize}
\end{defn}

For a normal $r$-approximation $(\delta=\delta_{0}\delta_{+}, \underline{\delta}_{+})$ to a tame elliptic semisimple element $\delta\in\t{G}$, we put
\[
\delta^{+}_{<r}\colonequals \prod_{0<i<r}\delta_{i},\quad
\delta_{<r}\colonequals \prod_{0\leq i<r}\delta_{i},\quad
\delta_{\geq r}\colonequals \delta_{<r}^{-1}\delta,
\]\symdef{delta-plus-less-r}{$\delta^{+}_{<r}$}\symdef{delta-less-r}{$\delta_{<r}$}\symdef{delta-geq-r}{$\delta_{\geq r}$}
\[
C_{\G}^{(r)}(\delta)
\colonequals 
C_{\G_{\delta_{0}}}^{(r)}(\delta_{+}).
\]
When $(\delta=\delta_{0}\delta_{+}, \underline{\delta}_{+})$ is a normal $r$-approximation to $\delta$, we often simply say that ``$\delta=\delta_{0}\delta^{+}_{<r}\delta_{\geq r}$ is a normal $r$-approximation to $\delta$''.

\begin{prop}[{\cite[Proposition 3.20]{Oi25}}]\label{prop:twisted-Jordan}
Suppose that $\G$ is tamely ramified and $p$ does not divide the order of the absolute Weyl group of $\G$.
Then any elliptic semisimple element $\delta\in\t{G}$ has a normal $r$-approximation.
\end{prop}

\begin{lem}[{\cite[Lemma 3.22]{Oi25}}]\label{lem:cent-cent}
Let $\delta\in\t{G}$ be a tame elliptic semisimple element having a normal $r$-approximation $\delta=\delta_{0}\delta^{+}_{<r}\delta_{\geq r}$.
Then we have
\[
(\G_{\delta_{0}})_{\delta^{+}_{<r}}
=
\G_{\delta_{<r}}
\]
\end{lem}

\subsection{Twisted character formula}\label{subsec:TCF}

Now we explain an explicit formula for the twisted character $\Theta_{\t{\pi}_{(\bfS,\vartheta)}}$ of the toral supercuspidal representation $\pi_{(\bfS,\vartheta)}$.

Let us fix an elliptic regular semisimple element $\delta\in\t{G}$ having a normal $r$-approximation $\delta=\delta_{0}\delta^{+}_{<r}\delta_{\geq r}$.
We put $\eta\colonequals \delta_{<r}\in\t{G}$.
Our aim is to describe $\Theta_{\t{\pi}_{(\bfS,\vartheta)}}(\delta)$ in terms of the $r$-approximation $\delta=\delta_{0}\delta^{+}_{<r}\delta_{\geq r}$.
To get a cleaner formula, we normalize $\Theta_{\t{\pi}_{(\bfS,\vartheta)}}$ using the twisted Weyl discriminant (a.k.a.\ the fourth absolute transfer factor; see Section \ref{subsec:tran-IV} or \cite[Section 4.5]{KS99}) as follows:
\[
\Phi_{\t{\pi}}(\delta)\colonequals \Delta_{\IV}^{\t{\G}}(\delta)\cdot\Theta_{\t{\pi}}(\delta).
\]\symdef{Phi-pi-tilde}{$\Phi_{\t{\pi}}$}

Recall that, by the torality of $\vartheta$, there exists a $\G$-generic element $X^{\ast}\in\mfs^{\ast}_{-r}$ of depth $r$ which lifts a unique element of $\mfs^{\ast}_{-r:-r+}$ satisfying $\vartheta(\exp(Y))=\psi_{F}(\langle Y,X^{\ast}\rangle)$ for any $Y\in\mfs_{s+:r+}$.
We may and do choose $X^{\ast}\in\mfs_{-r}^{\ast}$ so that it is $[\ul{\eta}]$-invariant ({\cite[Lemma 5.4]{Oi25}}).

Suppose that $g\in G$ satisfies that ${}^{g}\eta\in\t{S}$, which implies that $\bfS^\nat\subset \bfG_{{}^{g}\eta}$.
Then $X^\ast$ can be regarded as an element of the dual Lie algebra of $\G_{{}^{g}\eta}$.
Since ${}^{g}\delta_{\geq r}$ belongs to $G_{{}^{g}\eta}$, its logarithm image $\log({}^{g}\delta_{\geq r})$ belongs to the Lie algebra of $\G_{{}^{g}\eta}$.
Therefore, it makes sense to consider the normalized Fourier transform of the Lie algebra orbital integral $\hat{\iota}^{\G_{{}^{g}\eta}}_{X^{\ast}}(\log({}^{g}\delta_{\geq r}))$.
To be more precise, we let $\hat{\mu}^{\G_{\eta}}_{\Wal,X^{\ast}}$ be the Fourier transform of the Lie algebra orbital integral defined via Waldspurger's normalized Haar measure.\symdef{mu-hat-G-eta-Wal-X-star}{$\hat{\mu}^{\G_{\eta}}_{\Wal,X^{\ast}}$}
Then we define
\[
\hat{\iota}^{\G_{\eta}}_{X^{\ast}}(-)
\colonequals 
|D^{\red}_{G_{\eta}}(X^{\ast})|^{\frac{1}{2}}\cdot|D^{\red}_{G_{\eta}}(-)|^{\frac{1}{2}}\cdot\hat{\mu}^{\G_{\eta}}_{\Wal,X^{\ast}}(-).
\]\symdef{iota-hat-G-eta-X-star}{$\hat{\iota}^{\G_{\eta}}_{X^{\ast}}$}
See \cite[Section 4.2]{Kal19} and also \cite[Section 7]{Oi25} for the details.

We also note that, for $\alpha\in\Xi_{\asym}$, its restricted root $\alpha_{\res}$ is
\[
\begin{cases}
\text{asymmetric} & \text{if $\theta(\alpha)=\alpha$,}\\
\text{asymmetric} & \text{if $\theta(\alpha)\neq\alpha$, $\theta(\alpha)\notin\ddot{\alpha}$,}\\
\text{asymmetric} & \text{if $\theta(\alpha)\neq\alpha$, $\theta(\alpha)\in\dot{\alpha}$,}\\
\text{unramified or ramified} & \text{if $\theta(\alpha)\neq\alpha$, $\theta(\alpha)\in-\dot{\alpha}$,}
\end{cases}
\]
and that, for $\alpha\in\Xi_{\sym}$, its restricted root $\alpha_{\res}$ is
\[
\begin{cases}
\text{unramified} & \text{if $\theta(\alpha)=\alpha$,}\\
\text{unramified} & \text{if $\theta(\alpha)\neq\alpha$, $\theta(\alpha)\notin\dot{\alpha}$,}\\
\text{unramified or ramified} & \text{if $\theta(\alpha)\neq\alpha$, $\theta(\alpha)\in\dot{\alpha}$.}
\end{cases}
\]
We define a sign character $\t\epsilon_{\Xi}\colon S\rightarrow\C^\times$ by
\[
\t\epsilon_{\Xi}(s)
\colonequals 
\prod_{\begin{subarray}{c} \ddot{\alpha}\in\ddot{\Xi}\\ \alpha_{\res}: \, \asym/\ur\end{subarray}}\epsilon_{\alpha}(s).
\]\symdef{epsilon-tilde-Xi}{$\t\epsilon_{\Xi}$}

The following is the main result of \cite{Oi25}.

\begin{thm}[{\cite[Theorem 7.3]{Oi25}}]\label{thm:TCF-final}
    We write ${}^{g}\eta=s_{g}\cdot\ul{\eta}\in\t{S}$ for each $g\in G$ satisfying ${}^{g}\eta\in\t{S}$.
    Then 
    \[
    \Phi_{\t{\pi}}(\delta)
    =
    C_{\ul{\eta}}\cdot(-1)^{|\ddot{\Xi}_{\eta_{0},\ur}|}\cdot 
    \!\!\!\!\!
    \sum_{\begin{subarray}{c} g\in S\backslash G/G_{\eta} \\ {}^{g}\eta\in\t{S}\end{subarray}}
    \t\epsilon_{\Xi}(s_{g})\cdot\vartheta(s_{g})
    \cdot\mfG_{\G_{{}^{g}\eta_{0}}}(\vartheta,{}^{g}\eta_{+})
    \cdot\hat{\iota}^{\G_{{}^{g}\eta}}_{X^{\ast}}(\log({}^{g}\delta_{\geq r})),
    \]
    where $C_{\ul{\eta}}$ is a constant independent of $\delta$ and $\mfG_{\G_{{}^{g}\eta_{0}}}(\vartheta,{}^{g}\eta_{+})$ is the constant defined in \cite[Definition 5.3.2]{DS18}.
\end{thm}

By using the second absolute transfer factor (see Section \ref{subsec:invariants-for-LLC}) we can furthermore rewrite the above formula as follows.
We first note that, for any $\eta\in\t{S}$, the restriction $(\bfS^{\nat},\vartheta^{\nat}\colonequals \vartheta|_{S^{\nat}})$ gives a tame elliptic toral pair of $\G_{\eta}$ of depth $r$ (see {\cite[Lemma 5.5]{Oi25}}).
We let $a_{\vartheta}^{\nat}$ and $\chi_{\vartheta}^{\nat}$ be the sets of Kaletha's $a$-data and $\chi$-data associated to the tame elliptic toral pair $(\bfS^{\nat},\vartheta^{\nat})$ of $\bfG_{{}^{g}\eta_{0}}$ as in Section \ref{subsec:invariants-for-LLC}.

\begin{prop}\label{prop:TCF-normal}
With the notation as above, we have
\begin{multline*}
\Phi_{\t{\pi}}(\delta)
=
C_{\ul{\eta}}
\cdot(-1)^{|\ddot{\Xi}_{\eta_{0},\ur}|}
\cdot e(\G_{\eta_{0}})\cdot e(\G_{\eta})
\cdot \varepsilon(\T_{\G_{\eta_{0}}^{\ast}})\cdot\varepsilon(\T_{\G_{\eta}^{\ast}})^{-1}\\
\sum_{\begin{subarray}{c} g\in S\backslash G/G_{\eta} \\ {}^{g}\eta\in\t{S}\end{subarray}}
\t\epsilon_{\Xi}(s_{g})\cdot\vartheta(s_{g})
\cdot\Delta_{\II}^{\G_{{}^{g}\eta_{0}}}[a_{\vartheta}^{\nat},\chi_{\vartheta}^{\nat}]({}^{g}\eta_{+})
\cdot\hat{\iota}^{\G_{{}^{g}\eta}}_{X^{\ast}}(\log({}^{g}\delta_{\geq r})),
\end{multline*}
where $e(-)$ denotes the Kottwitz sign and $\bfG^{\ast}_{\eta_0}$ and (resp.\ $\bfG^{\ast}_{\eta}$) denotes the quasi-split inner form of $\bfG_{\eta_0}$ (resp.\ $\bfG_{\eta}$).
\end{prop}

\begin{proof}
We investigate the summands of the sum on the right-hand side of Theorem \ref{thm:TCF-final}.
By the proof of \cite[Proposition 4.21]{DS18}, $\mfG_{\G_{{}^{g}\eta_{0}}}(\vartheta,{}^{g}\eta_{+})$ is given by 
\[
\varepsilon_{\G_{{}^{g}\eta_{0}}}(\vartheta,{}^{g}\eta_{+})^{-1}
\cdot\varepsilon_{\G_{{}^{g}\eta_{0}},\ram}(\pi',{}^{g}\eta_{+})
\cdot\varepsilon_{\G_{{}^{g}\eta_{0}}}^{\ram}(\pi',{}^{g}\eta_{+})
\cdot\t{e}(\pi',{}^{g}\eta_{+})
\]
with the notation as in \textit{loc.\ cit.}
Since the depth-zero part of ${}^{g}\eta_{+}$ is trivial, we have $\varepsilon_{\G_{{}^{g}\eta_{0}}}(\vartheta,{}^{g}\eta_{+})=1$ (see \cite[Proposition 3.8]{AS09}) and $\varepsilon_{\G_{{}^{g}\eta_{0}}}^{\ram}(\pi',{}^{g}\eta_{+})=1$ (see \cite[Notation 4.14]{DS18}).
On the other hand, by \cite[Corollary 4.7.6]{Kal19}, the product $\varepsilon_{\G_{{}^{g}\eta_{0}},\ram}(\pi',{}^{g}\eta_{+})\cdot\t{e}(\pi',{}^{g}\eta_{+})$ equals
\[
\varepsilon_{\bfS^{\nat},\ram}({}^{g}\eta_{+})
\cdot e(\G_{{}^{g}\eta_{0}})
\cdot e(\G_{{}^{g}\eta})
\cdot \varepsilon(\T_{\G_{{}^{g}\eta_{0}}^{\ast}})\cdot\varepsilon(\T_{\G_{{}^{g}\eta}^{\ast}})^{-1}
\cdot\Delta_{\II}^{\G_{{}^{g}\eta_{0}}}[a_{\vartheta}^{\nat},\chi_{\vartheta}^{\nat}]({}^{g}\eta_{+}).
\]
Note that $\varepsilon_{\bfS^{\nat},\ram}({}^{g}\eta_{+})$ is trivial as ${}^{g}\eta_{+}$ has no depth zero part.
Moreover, since any $g\in G$ does not change $e(\G_{\eta_{0}})\cdot e(\G_{\eta})\cdot \varepsilon(\T_{\G_{\eta_{0}}^{\ast}})\cdot\varepsilon(\T_{\G_{\eta}^{\ast}})$ by conjugating $\eta$, we get the desired formula.
\end{proof}

\section{Framework of twisted endoscopy}\label{sec:tw-endo}

In this section, we review the framework of twisted endoscopy of Kottwitz--Shelstad \cite{KS99} and Waldspurger \cite{Wal08}.

\subsection{Endoscopic data treated in this paper}\label{subsec:KSW}

We introduce a structure of a twisted space on the $L$-group ${}^{L}\G$ following \cite[Section 1.2]{KS99} and \cite[Section 1.3]{Wal08}.
The automorphism $\theta$ and the fixed splitting $\spl_{\hat{\G}}$ define an automorphism $\hat{\theta}$ of $\hat{\G}$; namely, $\hat{\theta}$ is the unique $\spl_{\hat{\G}}$-preserving automorphism of $\hat{\G}$ which is compatible with $\theta$ under the isomorphism $\Psi(\hat{\G})\cong\Psi(\G)^{\vee}$.
Since $\hat{\theta}$ commutes with the action of $\Gamma$ on $\hat{\G}$, we can extend it to an automorphism ${}^{L}\theta$ of ${}^{L}\G$ by ${}^{L}\theta(x,w)\colonequals (\hat{\theta}(x),w)$ for $(x,w)\in{}^{L}\G=\hat{\G}\rtimes W_{F}$.
We define a twisted space on the dual side by $\L\t\G\colonequals \L\G\L\theta$.

We next review the notion of endoscopic data following \cite[Section 2.1]{KS99} and \cite[Section 1.3]{Wal08}.
We call a quadruple $(\H,\mcH,s,\hat{\xi})$ \textit{endoscopic data for the triple $(\G,\theta,\mathbbm{1})$} if
\begin{itemize}
\item
$\H$ is a quasi-split connected reductive group over $F$,
\item
$\mcH$ is a split extension $1\rightarrow\hat{\H}\rightarrow\mcH\rightarrow W_{F}\rightarrow1$ such that the induced action of $W_{F}$ on $\hat{\H}$ coincides with the action of $W_{F}$ on $\hat{\H}$ induced from the $F$-rational structure of $\H$ up to inner automorphisms of $\hat{\H}$,
\item
$s\in\hat{\G}$ such that the automorphism $[s]\circ\hat{\theta}$ is quasi-semisimple, and
\item
$\hat{\xi}\colon \mcH\hookrightarrow\L\G$ is an $L$-homomorphism (i.e., continuous and commuting with projections to $W_{F}$) satisfying the following conditions:
\begin{itemize}
\item
$[s]\circ\L{\theta}\circ\hat{\xi}=\hat{\xi}$,
\item
$\hat{\xi}|_{\hat{\H}}\colon\hat{\H}\xrightarrow{\sim} \hat{\G}_{s\L\theta}=\bfZ_{\hat{\G}}(s\L\theta)^{\circ}$.
\end{itemize}
\end{itemize}

When a set of endoscopic data $(\H,\mcH,s,\hat{\xi})$ is given, by replacing it with an equivalent data (see \cite[Section 3.1]{KS99} for the definition of the equivalence relation on endoscopic data), we may suppose that 
\begin{itemize}
\item
$s$ belongs to the pinned maximal torus $\hat{\T}$, and
\item
$(\hat{\bfB}_{\H},\hat{\T}_{\H})\colonequals \hat{\xi}^{-1}(\hat{\bfB},\hat{\T})$ is a $\Gamma$-stable Borel pair of $\hat{\H}$.
\end{itemize}
In this paper, we assume that
\[
\textbf{$\mcH$ in the endoscopic data $(\H,\mcH,s,\hat{\xi})$ is equal to $\L\H$.}
\]
Let us fix such an endoscopic data $(\H,\L\H,s,\hat{\xi})$ in the following.

We note that, the absolute Weyl group $\Omega_{\H}(\bfT_\bfH)$ of $\H$ can be identified with a subgroup of that $\Omega_{\G}(\bfT)$ of $\G$ (see Section \ref{subsec:norm}).
In particular, our assumption that $p\nmid|\Omega_{\G}(\bfT)|$ implies that $p\nmid|\Omega_{\H}(\bfT_\bfH)|$.

\subsection{Norm correspondence in twisted endoscopy}\label{subsec:norm}
Let us briefly review the notion of a norm in twisted endoscopy.
(See \cite[Section 3]{KS99} for details.)

We fix a Borel pair $(\bfB_{\H},\T_{\H})$ of $\H$ defined over $F$ so that the Langlands dual group $\hat{\H}$ of $\H$ is equipped with an isomorphism $\Psi(\hat{\H})\cong\Psi(\H)^{\vee}$.
In particular, we have isomorphisms $X^{\ast}(\T_{\H})\cong X_{\ast}(\hat{\T}_{\H})$ and $X_{\ast}(\T_{\H})\cong X^{\ast}(\hat{\T}_{\H})$.
Since the restriction of $\hat{\xi}$ to $\hat{\T}_{\H}$ induces
\[
\hat{\xi}|_{\hat{\T}_{\H}} \colon \hat{\T}_{\H}\xrightarrow{\sim}\hat{\T}^{\hat{\theta},\circ},
\]
by taking the dual of $\hat{\xi}|_{\hat{\T}_{\H}}$, we get an $F$-rational isomorphism
\[
\xi\colon \T_{\theta}\xrightarrow{\sim}\T_{\H}.
\]
By abuse of notation, we often write $\xi$ also for the homomorphism $\T\twoheadrightarrow\T_{\theta}\xrightarrow{\xi}\T_{\H}$.
Via the isomorphism $\xi$, the absolute Weyl group $\Omega_{\H}$ is identified with a subgroup of $\Omega_{\G}^{\theta}$ (see \cite[Section 1.1]{KS99}).
Therefore $\xi^{-1}$ induces a surjective map
\begin{align}\label{eq:norm-map}
\T_{\H}/\Omega_{\H} \twoheadrightarrow \T_{\theta}/\Omega_{\G}^{\theta}.
\end{align}
Note that $\T_{\H}/\Omega_{\H}$ and $\T_{\theta}/\Omega_{\G}^{\theta}$ parametrize the semisimple conjugacy classes of $\H$ and $\t\G$, respectively.
Moreover, the map \eqref{eq:norm-map} is $\Gamma$-equivariant.

We let $\t\G_{\ss}$ (resp.\ $\H_{\ss}$) denote the semisimple locus of $\t\G$ (resp.\ $\H$).
For $\gamma\in\H_{\ss}$ and $\delta\in\t\G_{\ss}$, we say that \textit{$\gamma$ and $\delta$ correspond} if the conjugacy classes of $\gamma$ and $\delta$ correspond under the map \eqref{eq:norm-map}.
We say that $\gamma\in\H_{\ss}$ is \textit{$\t\G$-strongly regular} if it corresponds to a strongly regular semisimple conjugacy class in $\t\G$.
Note that if $\gamma\in\H_{\ss}$ is $\t\G$-strongly regular, then it is strongly regular.
We let $\t\G_{\srs}$ (resp.\ $\H_{\t\G\-\srs}$) denote the strongly regular semisimple locus of $\t\G$ (resp.\ $\t\G$-strongly regular semisimple locus of $\H$).

For two $F$-rational elements $\delta$ and $\delta'$ of $\t{G}_{\srs}$ (resp.\ $\gamma$ and $\gamma'$ of $H_{\t\G\-\srs}$), we say that they are \textit{stably conjugate} if they are conjugate by an element of $\G$ (resp.\ $\H$).

When an $F$-rational element $\gamma\in H_{\t\G\-\srs}$ corresponds to an $F$-rational element $\delta\in\t{G}_{\srs}$, we say that \textit{$\gamma$ is a norm of $\delta$}.
We define $\mcD$\symdef{mathcal-D}{$\mcD$} to be the subset of $H_{\t\G\-\srs}\times\t{G}_{\srs}$ consisting of pairs $(\gamma,\delta)$ such that $\gamma$ is a norm of $\delta$.

\subsection{Transfer factor}\label{subsec:tran}
We have a function
\[
\Delta_{\H,\t{\G}}\colon H_{\G\-\srs}\times \t{G}_{\srs}\rightarrow\C
\]
called the \textit{(geometric) absolute transfer factor of Langlands--Kottwitz--Shelstad} (introduced in \cite{LS87,KS99,KS12}).
When the groups $\H$ and $\t{\G}$ are obvious from the context, we often omit the subscript from the notation and simply write $\Delta$ for $\Delta_{\H,\t{\G}}$.
Instead of reviewing its precise definition, we just give several comments on its basic properties; we refer the readers to \cite[Sections 4, 5]{KS99}, \cite[Chapitre 7]{Wal08}, and \cite{KS12} for the details.
\begin{enumerate}
\item
For any $(\gamma,\delta)$, $\Delta(\gamma,\delta)\neq0$ if and only if $(\gamma,\delta)\in\mcD$, i.e., $\gamma$ is a norm of $\delta$.
\item
The transfer factor $\Delta(\gamma,\delta)$ depends on the choice of a $\theta$-stable Whittaker datum of $\G$.
In this paper, we choose a $\theta$-stable Whittaker datum $\mfw$ of $\G$ determined by the fixed $\theta$-stable splitting $\spl_{\G}$ of $\G$ and the nontrivial additive character $\psi_F$ (see \cite[Section 5.3]{KS99} for how to produce $\mfw$ from $\spl_{\G}$ and $\psi_F$).
\item
The transfer factor $\Delta(\gamma,\delta)$ is defined to be the product of the ratio of root numbers $\varepsilon(\T^{\theta,\circ})\cdot\varepsilon(\T_{\H})^{-1}$ and four factors $\Delta_{\I}(\gamma,\delta)$, $\Delta_{\II}(\gamma,\delta)$, $\Delta_{\III}(\gamma,\delta)$, and $\Delta_{\IV}(\gamma,\delta)$.
Among these factors, $\Delta_{\I}(\gamma,\delta)$, $\Delta_{\II}(\gamma,\delta)$, and $\Delta_{\III}(\gamma,\delta)$ depend on the choice of $a$-data and $\chi$-data for the restricted root system of the $F$-rational maximal torus $\bfZ_{\G}(\bfZ_{\G}(\delta))$  in $\G$ although the whole product does not.
For this reason, we write $\Delta_{\bullet}[a,\chi](\gamma,\delta)$ (where $\bullet\in\{\I,\II,\III\}$) when we want to emphasize the dependence on $a$-data $a$ and $\chi$-data $\chi$.
\item
Following \cite{Kal19}, we let $\mr{\Delta}$ denote the transfer factor $\Delta$ without the fourth factor $\Delta_{\IV}$.
\item
The ratio of absolute transfer factors 
\[
\Delta(\gamma,\delta; \gamma',\delta')\colonequals \Delta(\gamma,\delta)/\Delta(\gamma',\delta')
\]
is called the relative transfer factor.
We also define the relative versions of $\Delta_{\bullet}$ for $\bullet\in\{\I,\II,\III,\IV\}$ in the same way.
\item
The definition of the transfer factor given in \cite{KS99} must be modified as announced in \cite{KS12} (see also \cite[Section 2]{Wal-Errata} or \cite[Appendix]{Kal21-STF}).
We adopt the modified version ``$\Delta'$'' which is compatible with the classical normalization of the local class field theory (hence consistent with, especially, \cite{LS87},\cite{MW16-1},\cite{Kal19}).
More precisely, the factor $\Delta'$ is the defined to be the product of $\Delta_{\I}^{\new,-1}$, $\Delta_{\II}^{\KS}$, $\Delta_{\III}^{\KS,-1}$, and $\Delta_{\IV}^{\KS}$ (and also the epsilon factors), where $\Delta_{\I}^{\new}$ is the factor defined in \cite[Section 3.4]{KS12} and $\Delta_{\II}^{\KS}$, $\Delta_{\III}^{\KS}$ and $\Delta_{\IV}^{\KS}$ are the factors defined in \cite{KS99}.
We note that $\Delta_{\I}^{\new}$ equals the factor $\Delta_{\I}^{\KS}$ defined in \cite{KS99} when there is no restricted roots of type $2$ or $3$, which is the case for the setting treated in this paper.
The symbols $\Delta_{\I}$, $\Delta_{\II}$, $\Delta_{\III}$, and $\Delta_{\IV}$ in this paper denote $\Delta_{\I}^{\new,-1}$, $\Delta_{\II}^{\KS}$, $\Delta_{\III}^{\KS,-1}$, and $\Delta_{\IV}^{\KS}$, respectively.
\end{enumerate}

\section{Analysis of $\theta$-stable regular supercuspidal $L$-packets}\label{sec:theta-stable}

\subsection{Twist of regular supercuspidal $L$-packets and $L$-parameters}\label{subsec:theta-twist-L}
Let $(\bfS,\hat{\jmath},\chi,\vartheta)$ be a regular supercuspidal $L$-packet datum of $\G$.
Let $\phi$ be the $L$-parameter of $\G$ associated to $(\bfS,\hat{\jmath},\chi,\vartheta)$, i.e., $\phi\colonequals {}^{L}j_{\chi}\circ\phi_{\vartheta}$ (see Section \ref{subsec:construction}).

We put $\j\colon\bfT\rightarrow\bfS$ to be the dual isomorphism of $\h{\j}$.
Recall that $\mcJ^{\G}\colonequals \{\text{$\hat{\jmath}$-admissible $F$-rational embeddings $\bfS\hookrightarrow\G$}\}$.
Let us investigate the $\hat{\jmath}$-admissibility condition.
By definition, an embedding $j\colon \bfS\hookrightarrow\G$ is $\hat{\jmath}$-admissible if and only if there exists an element $g\in\G$ such that $[g]\circ j(\bfS)=\bfT$ and the inverse of the dual of $[g]\circ j$ is $\hat{\G}$-conjugate to $\hat{\jmath}$.
Since the image of $\hat{\jmath}$ is assumed to be $\hat{\T}$, this is equivalent to that there exists an element $\h{w}\in\Omega_{\hat{\G}}\colonequals \Omega_{\hat{\G}}(\hat{\T})$ such that the inverse of the dual of $[g]\circ j$ is $[\h{w}]\circ\h{\j}$.
By letting $w\in\Omega_{\G}$ be the element corresponding to $\h{w}\in\Omega_{\h{\G}}$, this condition is equivalent to that $[g]\circ j=[w]^{-1}\circ \j^{-1}$.
Therefore, we see that the $\hat{\jmath}$-admissibility condition simply says that $j$ and $\j^{-1}$ are $\G$-conjugate.

The following lemma follows from this observation.

\begin{lem}\label{lem:adm-emb-twist}
If we put $\theta^{-1}\circ\mcJ^{\G}\colonequals \{\theta^{-1}\circ j\mid j\in\mcJ^{\G}\}$, then we have
\[
\theta^{-1}\circ\mcJ^{\G}
=
\{\text{$\hat{\theta}\circ\hat{\jmath}$-admissible $F$-rational embeddings $\bfS\hookrightarrow\G$}\}.
\]
\end{lem}

Recall that $\chi=\{\chi_{\alpha}\}_{\alpha}$ is a set of $\chi$-data for $\Phi(\G,\bfS_{\h{\j}})\cong\Phi(\G,\bfS_{j})$ (for any $j\in\mcJ^{\G}_{G}$).
As we have $\Phi(\G,\bfS_{j})\xrightarrow{\sim}\Phi(\G,\bfS_{\theta^{-1}\circ j})$, we can transport $\chi$ to a set of $\chi$-data for $\Phi(\G,\bfS_{\theta^{-1}\circ j})$, for which we write $\theta(\chi)$.
Then we get a regular supercuspidal $L$-packet datum $(\bfS,\hat{\theta}\circ\hat{\jmath},\theta(\chi),\vartheta)$.

The following lemma can be also found in \cite[Lemma 4.9]{Zha20}.

\begin{lem}\label{lem:packet-twist}
The $L$-parameter $\L\theta\circ\phi$ corresponds to the regular supercuspidal $L$-packet datum $(\bfS,\hat{\theta}\circ\hat{\jmath},\theta(\chi),\vartheta)$.
\end{lem}

\begin{proof}
By tracking the Langlands--Shelstad construction (\cite[Section 2.6]{LS87}) of the $L$-embedding $\Lj_{\chi}\colon\L\bfS\hookrightarrow\L\G$, we can check that $\L\theta\circ\Lj_{\chi}$ is nothing but the $L$-embedding obtained by applying the Langlands--Shelstad construction to the embedding $\hat{\theta}\circ\hat{\jmath}\colon\hat{\bfS}\hookrightarrow\hat{\G}$ with the $\chi$-data $\theta(\chi)$.
In other words, the $L$-parameter $\L\theta\circ\Lj_{\chi}\circ\phi_{\vartheta}$ is associated to the regular supercuspidal $L$-packet datum $(\bfS,\hat{\theta}\circ\hat{\jmath},\theta(\chi),\vartheta)$.
\end{proof}

\begin{lem}\label{lem:rep-twist}
The $L$-packet $\Pi_{\L\theta\circ\phi}^{\G}$ consists of regular supercuspidal representations whose regular supercuspidal data are given by $(\bfS,\hat{\theta}\circ\hat{\jmath},\theta(\chi),\vartheta,\theta^{-1}\circ j)$ for $j\in\mcJ^{\G}_{G}$.
\end{lem}

\begin{proof}
This follows from Lemma \ref{lem:packet-twist} by noting that the equality of Lemma \ref{lem:adm-emb-twist}
\[
\theta^{-1}\circ\mcJ^{\G}
=\{\text{$\hat{\theta}\circ\hat{\jmath}$-admissible $F$-rational embeddings $\bfS\hookrightarrow\G$}\}
\]
induces an identification
\[
\theta^{-1}\circ\mcJ^{\G}_{G}
\cong\{\text{$\hat{\theta}\circ\hat{\jmath}$-admissible $F$-rational embeddings $\bfS\hookrightarrow\G$}\}/{\sim_{G}}.
\]
\end{proof}

\begin{lem}\label{lem:rep-twist2}
Let $\pi\in\Pi_{\phi}^{\G}$ be a regular supercuspidal representation whose regular supercuspidal datum is $(\bfS,\hat{\jmath},\chi,\vartheta,j)$.
Then its $\theta$-twist $\pi^{\theta}\colonequals \pi\circ\theta$ arises from the regular supercuspidal datum $(\bfS,\hat{\theta}\circ\hat{\jmath},\theta(\chi),\vartheta,\theta^{-1}\circ j)$.
\end{lem}

\begin{proof}
We first note that, for any tame elliptic regular pair $(\bfS_{0},\vartheta_{0})$ of $\G$, the $\theta$-twist $\pi_{(\bfS_{0},\vartheta_{0})}^{\theta}$ of the associated regular supercuspidal representation $\pi_{(\bfS_{0},\vartheta_{0})}$ is equivalent to $\pi_{(\theta^{-1}(\bfS_{0}),\vartheta_{0}\circ\theta)}$.
(This can be easily checked in the same way as in Section \ref{subsec:twist-int}, where the toral case is treated.)
Thus we have
\[
\pi^{\theta}
=
\pi_{(\bfS_{j}, \vartheta'_{j})}^{\theta}
\cong
\pi_{(\theta^{-1}(\bfS_{j}),\vartheta'_{j}\circ\theta)}.
\]
Here recall that $\vartheta'_{j}$ is a character of $S_{j}=j(S)$ given by
\[
\vartheta'_{j}
\colonequals 
\epsilon_{\vartheta_{j}}\cdot(\vartheta\cdot\zeta_{\chi_{\vartheta_{\hat{\jmath}}}/\chi,S}^{-1})\circ j^{-1}.
\]

On the other hand, the regular supercuspidal representation associated to the datum $(\bfS,\hat{\theta}\circ\hat{\jmath},\theta(\chi),\vartheta,\theta^{-1}\circ j)$ is given by $\pi_{(\theta^{-1}(\bfS_{j}),\vartheta'_{\theta^{-1}\circ j})}$, where 
\[
\vartheta'_{\theta^{-1}\circ j}
=\epsilon_{\vartheta_{\theta^{-1}\circ j}}\cdot(\vartheta\cdot\zeta_{\chi_{\vartheta_{\hat{\theta}\circ\hat{\jmath}}}/\theta(\chi),S}^{-1})\circ (\theta^{-1}\circ j)^{-1}.
\]

By the definition of the character $\epsilon$, we easily see that $\epsilon_{\vartheta_{j}}\circ\theta=\epsilon_{\vartheta_{j}\circ\theta}$.
Moreover, it can be also easily checked that $\zeta_{\chi_{\vartheta_{\hat{\jmath}}}/\chi,S}$ equals $\zeta_{\chi_{\vartheta_{\hat{\theta}\circ\hat{\jmath}}}/\theta(\chi),S}$.
Hence we conclude that the characters $\vartheta'_{j}\circ\theta$ and $\vartheta'_{\theta^{-1}\circ j}$ are equal, which implies that $\pi^{\theta}\cong\pi_{(\theta^{-1}(\bfS_{j}),\vartheta'_{\theta^{-1}\circ j})}$.
\end{proof}

Lemmas \ref{lem:rep-twist} and \ref{lem:rep-twist2} imply the following:
\begin{prop}\label{prop:packet-twist}
We have $\Pi_{\phi}^{\G}\circ\theta=\Pi_{\L\theta\circ\phi}^{\G}$.
\end{prop}

\subsection{Structure of $\theta$-stable $L$-packets}\label{subsec:str-theta-L}
Let us keep the notation as in Section \ref{subsec:theta-twist-L}.
Thus $\phi$ denotes the $L$-parameter attached to a regular supercuspidal $L$-packet datum $(\bfS,\hat{\jmath},\chi,\vartheta)$, i.e., $\phi=\Lj_{\chi}\circ\phi_{\vartheta}$.
In the following, we suppose that $\phi$ factors through the $L$-embedding $\h{\xi}\colon\L\H\hookrightarrow\L\G$.
As we have $[s]\circ\L\theta\circ\h{\xi}=\h{\xi}$, this assumption implies that we have $[s]\circ\L\theta\circ\phi=\phi$.
In particular, the $L$-parameters $\L\theta\circ\phi$ and $\phi$ are $\h\bfG$-conjugate and the conjugation is given by $s$.
Thus, by Lemma \ref{lem:rep-twist} and Proposition \ref{prop:rsc-LLC}, there exists an isomorphism between the regular supercuspidal $L$-packet data $(\bfS,\hat{\jmath},\chi,\vartheta)$ and $(\bfS,\hat{\theta}\circ\hat{\jmath},\theta(\chi),\vartheta)$.
Let us investigate how such an isomorphism can be constructed explicitly.

In the following, we put $\phi'\colonequals \L\theta\circ\phi$, $\hat{\jmath}'\colonequals \hat{\theta}\circ\hat{\jmath}$, and $\chi'\colonequals \theta(\chi)$.
We may and do assume that the image of $\hat{\jmath}$ is given by $\hat{\T}$.
We define an automorphism $\hat{\bfS}$ of $\h\bfS$ by $\hat{\theta}_{\bfS}\colonequals \hat{\jmath}^{-1}\circ\hat{\jmath}'=\h{\j}^{-1}\circ\h\theta\circ\h{\j}$.
Let $\theta_{\bfS}$ be the automorphism of $\bfS$ which is dual to $\hat{\theta}_{\bfS}$.
Note that $\h\theta_{\bfS}$ and $\theta_{\bfS}$ are involutive.

\begin{lem}\label{lem:rationality-of-theta}
The automorphism $\h\theta_{\bfS}$ of $\h\bfS$ is $\Gamma$-equivariant, hence $\theta_{\bfS}$ is $F$-rational.
\end{lem}

\begin{proof}
As the $\Gamma$-actions on $\h\bfS$ and $\h\bfT$ factor through a finite quotient, we may discuss the equivariance for $W_{F}$ instead of $\Gamma$.
Recall that we have $\phi=\Lj_{\chi}\circ\phi_{\vartheta}$.
As $\phi_{\vartheta}\colon W_{F}\rightarrow\L\bfS=\h{\bfS}\rtimes W_{F}$ is an $L$-parameter of $\bfS$, the $W_{F}$-action on $\h{\bfS}$ is given by $\sigma(t)=[\phi_{\vartheta}(\sigma)](t)$ for any $\sigma\in W_{F}$ and $t\in\h{\bfS}$.
By noting that $\Lj_{\chi}\colon \L\bfS\hookrightarrow\L\G$ is an $L$-embedding extending $\h\j\colon\h\bfS\xrightarrow{\sim}\h\T$, this implies that $\h\j\circ\sigma(t)=[\phi(\sigma)]\circ\h\j(t)$ for any $\sigma\in W_{F}$ and $t\in\h{\bfS}$.
Similarly, we also have $\h\j'\circ\sigma(t)=[\phi'(\sigma)]\circ\h\j'(t)$ for any $\sigma\in W_{F}$ and $t\in\h{\bfS}$.
Hence, by noting that $[s]\circ\phi'=\phi$ and that $[s]|_{\h\T}=\id_{\h\T}$, we get
\begin{align*}
\h\theta_{\bfS}\circ\sigma
=\h{\j}^{-1}\circ[s]\circ\h{\j}'\circ\sigma
&=\h{\j}^{-1}\circ[s]\circ[\phi'(\sigma)]\circ\h{\j}'\\
&=\h{\j}^{-1}\circ[\phi(\sigma)]\circ[s]\circ\h{\j}'
=\sigma\circ\h{\j}^{-1}\circ[s]\circ\h{\j}'
=\sigma\circ\h\theta_{\bfS}.
\end{align*}
This completes the proof.
\end{proof}

We define a set $\zeta=(\zeta_{\alpha})_{\alpha\in\Phi(\G,\bfS)}$ of $\zeta$-data for $\Phi(\G,\bfS)$ by $\zeta\colonequals \zeta_{\theta_{\bfS,\ast}(\chi)/\chi'}$ (see Definition \ref{defn:isom-rscp-datum} and also Remark \ref{rem:isom-rscp-datum}).

\begin{prop}\label{prop:intertwiner-rscp-data}
The tuple $(\theta_{\bfS},1,\zeta)$ gives an isomorphism of regular supercuspidal $L$-packet data:
\[
(\theta_{\bfS},1,\zeta)\colon (\bfS,\hat{\jmath},\chi,\vartheta)\xrightarrow{\sim}(\bfS,\hat{\jmath}',\chi',\vartheta).
\]
\end{prop}

To show this proposition, we recall the following property of a set of $\zeta$-data, which is essentially discussed in the proof of \cite[Theorem 3.16]{Kal19-cover}:

\begin{lem}\label{lem:zeta-measures-diff}
Let $\chi_{1}$ and $\chi_{2}$ be sets of $\chi$-data for $\Phi(\G,\bfS_{\h\j})$.
Then the $L$-parameter of the character $\zeta_{\chi_{1}/\chi_{2},S}\colon S\rightarrow\C^{\times}$ is represented by the $1$-cocycle $c_{\chi_{1}/\chi_{2}}\colon W_{F}\rightarrow\hat{\bfS}$ defined by $\Lj_{\chi_{1}}=(\h\j\circ c_{\chi_{1}/\chi_{2}})\cdot \Lj_{\chi_{2}}$.
\end{lem}

\begin{proof}[Proof of Proposition \ref{prop:intertwiner-rscp-data}]
Our task is to check that the condition (3) of Definition \ref{defn:isom-rscp-datum} is satisfied, i.e., the equality $(\zeta_{S}^{-1}\cdot\vartheta)\circ\theta_{\bfS}=\vartheta$ holds.
With the notation as in Lemma \ref{lem:zeta-measures-diff}, the $L$-parameter of $\zeta_{S}^{-1}\cdot\vartheta$ is given by $c_{\chi'/\theta_{\bfS,\ast}(\chi)}\cdot\phi_{\vartheta}$.
Thus, by the functoriality of the local Langlands correspondence for tori, the $L$-parameter of $(\zeta_{S}^{-1}\cdot\vartheta)\circ\theta_{\bfS}$ is given by $\L\theta_{\bfS}\circ(c_{\chi'/\theta_{\bfS,\ast}(\chi)}\cdot\phi_{\vartheta})$.
We have to show that this is equivalent to $\phi_{\vartheta}$ as an $L$-parameter of $\bfS$, or equivalently, $c_{\chi'/\theta_{\bfS,\ast}(\chi)}\cdot\phi_{\vartheta}$ and $\L\theta_{\bfS}\circ\phi_{\vartheta}$ are $\hat{\bfS}$-conjugate.
By putting $s'\colonequals \h\j^{\prime-1}(s)$, let us check that $[s']\circ(c_{\chi'/\theta_{\bfS,\ast}(\chi)}\cdot\phi_{\vartheta})$ is equal to $\L\theta_{\bfS}\circ\phi_{\vartheta}$.

Since $\Lj'_{\theta_{\bfS,\ast}(\chi)}$ is injective, it suffices to show the equality after composing them with $\Lj'_{\theta_{\bfS,\ast}(\chi)}$.
By using Lemma \ref{lem:zeta-measures-diff}, we have
\begin{align*}
\Lj'_{\theta_{\bfS,\ast}(\chi)}\circ[s']\circ(c_{\chi'/\theta_{\bfS,\ast}(\chi)}\cdot\phi_{\vartheta})
&=[s]\circ\Lj'_{\theta_{\bfS,\ast}(\chi)}\circ(c_{\chi'/\theta_{\bfS,\ast}(\chi)}\cdot\phi_{\vartheta})\\
&=[s]\circ\Lj'_{\chi'}\circ\phi_{\vartheta}=[s]\circ\phi'.
\end{align*}
On the other hand, by noting that $\Lj_{\chi}=\Lj'_{\theta_{\bfS,\ast}(\chi)}\circ\L\theta_{\bfS}$ (this is essentially the same identity as $\Lj'_{\chi'}=\L\theta\circ\Lj_{\chi}$, which was used in the proof of Lemma \ref{lem:packet-twist}), we have
\[
\Lj'_{\theta_{\bfS,\ast}(\chi)}\circ\L\theta_{\bfS}\circ\phi_{\vartheta}
=\Lj_{\chi}\circ\phi_{\vartheta}=\phi.
\]
As we have $[s]\circ\phi'=\phi$, we get the assertion.
\end{proof}

Since we have $\L\theta\circ\phi\cong\phi$, Proposition \ref{prop:packet-twist} implies that $\Pi_{\phi}^{\G}\circ\theta=\Pi_{\phi}^{\G}$.
The effect of $\theta$-twist on $\Pi_{\phi}^{\G}$ can be described more explicitly as follows.

\begin{prop}\label{prop:packet-rep-twist}
Let $\pi$ be the member of $\Pi_{\phi}^{\G}$ labeled by $(\bfS,\hat{\jmath},\chi,\vartheta,j)$ for $j\in\mcJ^{\G}_{G}$.
Then its $\theta$-twist $\pi^{\theta}$ is labeled by $(\bfS,\hat{\jmath},\chi,\vartheta,\theta^{-1}\circ j\circ\theta_{\bfS})$.
\end{prop}

\begin{proof}
When $\pi$ arises from $(\bfS,\hat{\jmath},\chi,\vartheta,j)$ for $j\in\mcJ^{\G}_{G}$, by Lemma \ref{lem:rep-twist2}, $\pi^{\theta}$ arises from the datum $(\bfS,\hat{\theta}\circ\hat{\jmath},\theta(\chi),\vartheta,\theta^{-1}\circ j)$.
On the other hand, the isomorphism of regular supercuspidal $L$-packet data $(\theta_{\bfS},1,\zeta)$ introduced above induces an isomorphism of regular supercuspidal data
\[
(\theta_{\bfS},1,\zeta,1)\colon (\bfS,\hat{\jmath},\chi,\vartheta,j')\xrightarrow{\sim}(\bfS,\hat{\theta}\circ\hat{\jmath},\theta(\chi),\vartheta,j'\circ\theta_{\bfS}^{-1})
\]
for each $j'\in\mcJ^{\G}_{G}$.
Thus the datum $(\bfS,\hat{\theta}\circ\hat{\jmath},\theta(\chi),\vartheta,\theta^{-1}\circ j)$ is isomorphic to $(\bfS,\hat{\jmath},\chi,\vartheta,\theta^{-1}\circ j\circ\theta_{\bfS})$.
\end{proof}

\begin{cor}\label{cor:theta-stable-member}
Let $\pi$ be the member of $\Pi_{\phi}^{\G}$ labeled by $(\bfS,\hat{\jmath},\chi,\vartheta,j)$ for $j\in\mcJ^{\G}_{G}$.
Then the following are equivalent:
\begin{enumerate}
\item
$\pi$ is $\theta$-stable, i.e., $\pi\cong\pi^{\theta}$,
\item
$j$ equals $\theta^{-1}\circ j\circ\theta_{\bfS}$ in $\mcJ^{\G}_{G}$, i.e., $\theta^{-1}\circ j\circ\theta_{\bfS}$ and $j$ are $G$-conjugate.
\end{enumerate}
\end{cor}

\subsection{Embeddings of twisted tori}\label{subsec:t-tori-emb}
We introduce several notions related to twisted maximal tori of $\t{\G}$.
Suppose that $(\bfS,\t\bfS)$ is a twisted space over $F$ whose $\bfS$ is a torus (here $(\bfS,\t\bfS)$ is not a priori supposed to be a subspace of $(\bfG,\t\bfG)$).
Let $\theta_{\bfS}$ be the automorphism of $\bfS$ given by $\t\bfS$, i.e., for any $s\in\bfS$ and $\eta\in\t{\bfS}$, we have $\theta_{\bfS}(s)=[\eta](s)$.

\begin{defn}\label{defn:tori-emb}
We say that an embedding $j\colon \bfS\hookrightarrow\G$ is an \textit{$F$-rational embedding of a maximal torus} if $j$ is $F$-rational and $\bfS_{j}\colonequals j(\bfS)$ is a maximal torus of $\G$.
\end{defn}

\begin{defn}\label{defn:t-tori-emb}
Let $(j,\t{j})\colon (\bfS,\t\bfS)\hookrightarrow(\G,\t\G)$ be an embedding of a twisted space, i.e., $j\colon\bfS\hookrightarrow\G$ and $\t{j}\colon\t\bfS\hookrightarrow\t\G$ are embeddings such that, for any $s_{1},s_{2}\in\bfS$ and $t\in\t{\bfS}$, we have $\t{j}(s_{1}ts_{2})=j(s_{1})\t{j}(t)j(s_{2})$.
We say that $(j,\t{j})$ is \textit{an $F$-rational embedding of a twisted maximal torus} if the following conditions are satisfied:
\begin{itemize}
\item
$j$ is an $F$-rational embedding of a maximal torus and $\t{j}$ is $F$-rational;
\item
$(\bfS_{j},\t\bfS_{j})\colonequals (j(\bfS),\t{j}(\t\bfS))$ is an $F$-rational twisted maximal torus of $\t\G$.
\end{itemize}
We often simply write ``$(j,\t{j})\colon\t\bfS\hookrightarrow\t\G$ is an $F$-rational embedding of a twisted maximal torus''.
\end{defn}

\begin{rem}\label{rem:theta-eta}
Let $(j,\t{j})\colon \t\bfS\hookrightarrow \t\G$ be an $F$-rational embedding of a twisted maximal torus.
For any $\eta\in\t\bfS$, as we have $\theta_{\bfS}=[\eta]$, we have $j^{-1}\circ[\t{j}(\eta)]\circ j=\theta_{\bfS}$.
\end{rem}

Note that if $(j,\t{j})\colon \t\bfS\hookrightarrow\t\G$ is an $F$-rational embedding of a twisted maximal torus, then $(j,\t{j}_{s})\colon \t\bfS\hookrightarrow\t\G$ is again an $F$-rational embedding of a twisted maximal torus for any $s\in S$, where $\t{j}_{s}$ is defined by $\t{j}_{s}(\eta)\colonequals j(s)\t{j}(\eta)$ for $\eta\in\t\bfS$.
The following lemma says that the converse of this fact is also true:

\begin{lem}\label{lem:S-trans}
Let $(\bfS,\t\bfS)$ be a twisted space defined over $F$.
Let $j\colon\bfS\hookrightarrow\G$ be an $F$-rational embedding of a maximal torus.
If $(j,\t{j}_{1})$ and $(j,\t{j}_{2})$ are $F$-rational embeddings of a twisted maximal torus $\t\bfS\hookrightarrow\t\G$, then there exists an element $s\in S$ satisfying $\t{j}_{2}(\eta)=j(s)\t{j}_{1}(\eta)$ for any $\eta\in\t{\bfS}$.
\end{lem}

\begin{proof}
If we fix an element $\eta'\in\t{S}$, then we have
\[
j^{-1}\circ[\t{j}_{1}(\eta')]\circ j=\theta_{\bfS}=j^{-1}\circ[\t{j}_{2}(\eta')]\circ j
\]
by Remark \ref{rem:theta-eta}.
This implies that $\t{j}_{2}(\eta')^{-1}\cdot\t{j}_{1}(\eta')$ belongs to $S_{j}$, hence there exists an $s\in S$ such that $\t{j}_{2}(\eta')=j(s)\t{j}_{1}(\eta')$.
From this, we can see that $\t{j}_{2}(\eta)=j(s)\t{j}_{1}(\eta)$ for any $\eta\in\t{\bfS}$.
\end{proof}

For two $F$-rational embeddings $(j_{1},\t{j}_{1})$ and $(j_{2},\t{j}_{2})$ of a twisted maximal torus $\t\bfS\hookrightarrow\t\G$, we write $(j_{1},\t{j}_{1})\sim(j_{2},\t{j}_{2})$ when $j_{1}=j_{2}$ (this gives an equivalence relation on the set of $F$-rational embeddings of a twisted maximal torus).
Note that, by Lemma \ref{lem:S-trans}, the image $\t{\bfS}_{j}$ of $\t{j}$ depends only on the equivalence class of $(j,\t{j})$.
When $(j,\t{j})$ is an $F$-rational embedding of a twisted maximal torus, we often let $j$ denote the equivalence classes of $(j,\t{j})$ by abuse of notation.
Also, if we simply say ``$j\colon\t\bfS\hookrightarrow\t\G$ is an $F$-rational embedding of a twisted maximal torus'', then it means that we have an $F$-rational embedding of a twisted maximal torus $(j,\t{j})$ and $j$ is its equivalence class.

As in the untwisted case, we define the stable/rational conjugacy for $F$-rational embeddings of a twisted maximal torus as follows:
\begin{defn}\label{defn:stab-conj}
Let $(j,\t{j})$ and $(j',\t{j}')$ be $F$-rational embeddings of a twisted maximal torus $\t\bfS\hookrightarrow\t\G$.
We say that $(j,\t{j})$ and $(j',\t{j}')$ are \textit{$\G$-conjugate} (resp.\ \textit{$G$-conjugate}) if there exists an element $x\in\G$ (resp.\ $x\in G$) satisfying $j'=[x]\circ j$.
In this case, we write $(j,\t{j})\sim_{\G}(j',\t{j}')$ (resp.\ $(j,\t{j})\sim_{G}(j',\t{j}')$).
When $j$ and $j'$ are equivalence classes of $F$-rational embeddings of a twisted maximal torus, we say $j$ and $j'$ are $\G$-conjugate (resp.\ $G$-conjugate) if some (or, equivalently, any) representatives $(j,\t{j})$ and $(j',\t{j}')$ are $\G$-conjugate (resp.\ $G$-conjugate).
In this case, we write $j\sim_{\G}j'$ (resp.\ $j\sim_{G}j'$).
\end{defn}

\subsection{Parametrization of $\theta$-stable members of a $\theta$-stable packet}\label{subsec:twisted-parametrization}
Let us go back to the setting of Sections \ref{subsec:theta-twist-L} and \ref{subsec:str-theta-L}.
Recall that $(\bfS,\hat{\jmath},\chi,\vartheta)$ is a regular supercuspidal $L$-packet datum whose $L$-parameter satisfies $[s]\circ\L\theta\circ\phi=\phi$.
In Section \ref{subsec:str-theta-L}, we introduced an $F$-rational involutive automorphism $\theta_{\bfS}$ of $\bfS$.
We consider the twisted space $\t{\bfS}\colonequals \bfS\theta_{\bfS}$ associated to the pair $(\bfS,\theta_{\bfS})$.

\begin{prop}\label{prop:theta-stable-member}
Let $\pi_{j}$ be the member of $\Pi_{\phi}^{\G}$ labeled by $(\bfS,\hat{\jmath},\chi,\vartheta,j)$ for $j\in\mcJ^{\G}_{G}$.
Then the following are equivalent:
\begin{enumerate}
\item
$\pi_{j}$ is $\theta$-stable, i.e., $\pi_{j}\cong\pi_{j}^{\theta}$,
\item
$j$ extends to an $F$-rational embedding of a twisted maximal torus $\t{\bfS}\hookrightarrow\t{\G}$.
\end{enumerate}
\end{prop}

\begin{proof}
We first show that (2) implies (1).
Suppose that $j$ extends to an embedding $(j,\t{j})\colon\t{\bfS}\hookrightarrow\t{\G}$ of an $F$-rational twisted maximal torus.
Then, putting $\eta\colonequals \t{j}(\theta_{\bfS})\in\t{G}$, we get $[\eta]\circ j=j\circ\theta_{\bfS}$.
In other words, if we write $\eta=\eta^{\circ}\theta$ with $\eta^{\circ}\in G$, then we have $[\eta^{\circ}]\circ\theta\circ j=j\circ\theta_{\bfS}$.
In particular, $\theta^{-1}\circ j\circ\theta_{\bfS}$ and $j$ are $G$-conjugate.
This implies that $\pi_{j}$ is $\theta$-stable by Corollary \ref{cor:theta-stable-member}.

We next show that (1) implies (2).
Again by Corollary \ref{cor:theta-stable-member}, we may suppose that we have an element $x\in G$ satisfying $\theta^{-1}\circ j\circ\theta_{\bfS}=[x]\circ j$.
Then, by putting $\t{j}(\theta_{\bfS})\colonequals \theta(x)\theta\in \t{G}$, the argument in the previous paragraph shows that $(j,\t{j})$ defines an $F$-rational embedding of a twisted space $(\bfS,\t{\bfS})$ into $(\G,\t{\G})$.
Hence our task is to show that the image $(\bfS_{j},\t{\bfS}_{j})$ is an $F$-rational twisted maximal torus of $(\G,\t{\G})$.
In other words, we have to find a Borel subgroup which contains $\bfS_{j}$ and is preserved by $[\t{j}(\theta_{\bfS})]=[\theta(x)]\circ\theta$ (see Lemma \ref{lem:twisted-torus}).

Recall that we put $\jmath\colon\T\xrightarrow{\sim}\bfS$ to be the dual of $\hat{\jmath}\colon\hat{\bfS}\xrightarrow{\sim}\hat{\T}$ and that $\theta_{\bfS}$ is defined to be the dual of $\h\j^{-1}\circ\h\theta\circ\h\j$.
Hence we have $\theta_{\bfS}=\j\circ\theta\circ\j^{-1}$.
The embedding $j$ is given by $[y]\circ\jmath^{-1}$ for some $y\in\G$ by the $\h\j$-admissibility (see the beginning of Section \ref{subsec:theta-twist-L}).
Thus we get 
\[
j\circ\theta_{\bfS}
=[y]\circ\j^{-1}\circ\j\circ\theta\circ\j^{-1}
=[y\theta(y)^{-1}]\circ\theta\circ[y]\circ\j^{-1}
=[y\theta(y)^{-1}]\circ\theta\circ j.
\]
As we also have $\theta^{-1}\circ j\circ\theta_{\bfS}=[x]\circ j$, or equivalently, $j\circ\theta_{\bfS}=[\theta(x)]\circ\theta\circ j$, we get $[y\theta(y)^{-1}]\circ\theta\circ j=[\theta(x)]\circ\theta\circ j$.
This implies that $\theta(x)^{-1}y\theta(y)^{-1}\in\theta(\bfS_{j})$.
Let us write $\theta(x)^{-1}y\theta(y)^{-1}=t$ with $t\in\theta(\bfS_{j})$.

Since we have $\bfS_{j}={}^{y}\T$, the Borel subgroup ${}^{y}\B$ contains $\bfS_{j}$.
Let us check that ${}^{y}\B$ satisfies the desired condition, i.e., $[\theta(x)]\circ\theta({}^{y}\B)={}^{y}\B$.
By noting that $\B$ is stable under $\theta$ and that $t\in\theta(\bfS_{j})\subset\theta({}^{y}\B)$, we have
\[
[\theta(x)]\circ\theta({}^{y}\B)
={}^{\theta(x)}\theta({}^{y}\B)
={}^{y\theta(y)^{-1}t^{-1}}\theta({}^{y}\B)
={}^{y\theta(y)^{-1}}\theta({}^{y}\B)
={}^{y}\theta(\B)
={}^{y}\B.
\]
This completes the proof.
\end{proof}

Now let us recall that
\[
\mcJ^{\G}_{G}
=
\{
j\colon\bfS\hookrightarrow\G
\mid
\text{$j$ is $F$-rational and $j\sim_{\G}\j^{-1}$}
\}/{\sim_{G}}
\]
(see the discussion at the beginning of Section \ref{subsec:theta-twist-L}).
We put
\[
\t\mcJ_{G}^{\G}
\colonequals 
\{j\colon \t\bfS\hookrightarrow\t\G \mid \text{$j$ is $F$-rational and $j\sim_{\G}\j^{-1}$}\}
/{\sim_{G}},
\]\symdef{J-tilde-G-G}{$\tilde{\mathcal{J}}_{G}^{\G}$}
namely, $\t\mcJ_{G}^{\G}$ is the set of $G$-conjugacy classes of equivalence classes of $F$-rational embeddings of a twisted maximal torus which are $\G$-conjugate to $\j^{-1}$.
Note that the canonical forgetful map $\t\mcJ_{G}^{\G}\hookrightarrow\mcJ_{G}^{\G}\colon(j,\t{j})\mapsto j$ is injective.
Thus, in the following, we regard $\t\mcJ_{G}^{\G}$ as a subset of $\mcJ_{G}^{\G}$.
By Proposition \ref{prop:theta-stable-member}, the set $\t\mcJ_{G}^{\G}$ parametrizes the $\theta$-stable representations in $\Pi_{\phi}^{\G}$.
More precisely, for each $j\in\mcJ^{\G}_{G}$, the corresponding member $\pi_{j}$ is $\theta$-stable if and only if $j$ belongs to $\t{\mcJ}^{\G}_{G}$.

\begin{rem}\label{rem:generic-stable}
Recall that Shahidi's generic packet conjecture \cite[Conjecture 9.4]{Sha90} predicts that every tempered $L$-packet contains a unique $\mfw$-generic member.
If $\Pi_{\phi}^{\G}$ satisfies the generic packet conjecture, then the unique $\mfw$-generic member of $\Pi_{\phi}^{\G}$ (say $\pi_{\mfw}$) is $\theta$-stable.
Indeed, the $\theta$-twist $\pi_{\mfw}^{\theta}$ is again $\mfw$-generic since $\mfw$ is $\theta$-stable.
As we have $\Pi_{\phi}^{\G}\circ\theta=\Pi_{\phi}^{\G}$, both $\pi_{\mfw}$ and $\pi_{\mfw}^{\theta}$ belong to $\Pi_{\phi}^{\G}$.
Hence the uniqueness part of the generic packet conjecture implies that $\pi_{\mfw}$ and $\pi_{\mfw}^{\theta}$ are isomorphic.
We remark that the generic packet conjecture is proved in \cite[Lemma 6.2.2]{Kal19} for toral regular supercuspidal $L$-packets.
\end{rem}

\subsection{Descended regular supercuspidal $L$-packet}\label{subsec:L-packet-descent}
We keep the notation as in the previous subsections.
Recall that $\phi\colon W_{F}\rightarrow\G$ factors through $\h\xi\colon\L\H\hookrightarrow\L\G$.
Let $\phi_{\H}$ be the $L$-parameter of $\H$ such that $\phi=\h\xi\circ\phi_{\H}$.

\begin{prop}\label{prop:toral-toral}
The $L$-parameter $\phi_{\H}$ is regular supercuspidal.
Moreover, if $\phi$ is toral supercuspidal, then so is $\phi_{\H}$.
\end{prop}

\begin{proof}
Let us check that the four conditions of Definition \ref{defn:rsc-L-par} for $\H$ are satisfied.
We first check (0).
Obviously $\phi|_{\SL_{2}(\C)}$ is trivial, hence so is $\phi_{\H}|_{\SL_{2}(\C)}$.
Since $\bfZ_{\hat{\H}}(\phi(W_{F}))^{\circ}$ is contained in $\bfZ_{\hat{\G}}(\phi(W_{F}))^{\circ}$, we have $\bfZ_{\hat{\H}}(\phi(W_{F}))^{\circ}\subset \bfZ_{\hat{\G}}\cap\hat{\H}\subset \bfZ_{\hat{\H}}$.
We next check (1).
Since $\phi=\Lj_{\chi}\circ\phi_{\vartheta}$ and $\bfS$ is tamely ramified, a torus of $\h\G$ containing $\phi(P_{F})$ can be taken to be $\h\T$.
Thus $\phi_{\H}(P_F)$ is contained in $\h\xi^{-1}(\h\T)=\h\T_{\H}$.
We check (2).
Since we have $\bfZ_{\h\H}(\phi(I_{F}))^{\circ}\subset\bfZ_{\h\G}(\phi(I_{F}))^{\circ}\cap\h\H$ and $\bfZ_{\h\G}(\phi(I_{F}))^{\circ}$ is a torus, $\bfZ_{\h\H}(\phi(I_{F}))^{\circ}$ is also a torus.
We finally check (3).
We put $\mcM_{\H}\colonequals \bfZ_{\hat{\H}}(\phi(P_{F}))^{\circ}$, $\mcC_{\H}\colonequals \bfZ_{\hat{\H}}(\phi(I_{F}))^{\circ}$, and $\mcT_{\H}\colonequals \bfZ_{\mcM_{\H}}(\mcC_{\H})$.
Then we have $\bfN_{\mcM_{\H}}(\mcT_{\H})\subset \bfN_{\mcM}(\mcT)$ and this inclusion induces a $\Gamma$-equivariant inclusion of Weyl groups $\Omega_{\mcM_{\H}}(\hat{\bfS}_{\H})\hookrightarrow\Omega_{\mcM}(\hat{\bfS})$.
Thus, if $n\in N_{\mcM_{\H}}(\mcT_{\H})$ maps to a nontrivial element of $\Omega_{\mcM_{\H}}(\hat{\bfS}_{\H})^{\Gamma}$, then we have $n\notin \bfZ_{\hat{\G}}(\phi(I_{F}))$.
This implies that $n\notin \bfZ_{\h\H}(\phi(I_{F}))$.

We next consider the case where $\phi$ is toral supercuspidal.
Let us check that the two conditions of Definition \ref{defn:toral-L-par} for $\H$ are satisfied.
The condition (1) for $\phi_{\H}$ is clearly deduced from the condition (1) for $\phi$.
Thus let us check (2).
By the torality of $\phi$, $\bfZ_{\hat{\G}}(\phi(I_{F}^{r}))$ is a maximal torus of $\hat{\G}$ containing $\phi(P_{F})$ .
Since we have $\phi(I_{F}^{r})\subset\h\T$ (by the construction of $\phi$), we have $\bfZ_{\hat{\G}}(\phi(I_{F}^{r}))\supset \bfZ_{\hat{\G}}(\h\T)=\h\T$.
Thus we get $\bfZ_{\hat{\G}}(\phi(I_{F}^{r}))=\h\T$.
This implies that $\mcT_{\H}\colonequals Z_{\hat{\H}}(\phi(I_{F}^{r}))$ is equal to $\h\H\cap\h\T=\h\T_{\H}$, which is a maximal torus of $\h\H$ and contains $\phi(P_{F})$.
\end{proof}

By applying Proposition \ref{prop:rsc-LLC} to the descended $L$-parameter $\phi_{\H}$, we obtain a regular supercuspidal $L$-packet datum $(\bfS_{\H},\hat{\jmath}_{\H},\chi_{\H},\vartheta_{\H})$ of $\H$ and hence a regular supercuspidal $L$-packet $\Pi_{\phi_{\H}}^{\H}$ of $\H$. 
In particular, we may and do assume $\phi_{\H}=\Lj_{\chi_{\H}}\circ\phi_{\vartheta_{\H}}$, where $\Lj_{\chi_{\H}}$ denotes the Langlands--Shelstad extension of $\h\jmath_{\H}$ to an $L$-embedding via the set of $\chi$-data $\chi_{\H}$:
\begin{align}\label{diag:L-parameters}
\xymatrix{
W_{F}\ar^-{\phi_{\vartheta}}[r]\ar_-{\phi_{\vartheta_{\H}}}[rd]&\hat{\bfS}\rtimes W_{F}\ar@{^{(}->}^-{\Lj_{\chi}}[rr]&&\hat{\G}\rtimes W_{F}\\
&\hat{\bfS}_{\H}\rtimes W_{F}\ar@{^{(}->}^-{\Lj_{\chi_{\H}}}[rr]&&\hat{\H}\rtimes W_{F}\ar@{^{(}->}_-{\h\xi}[u]
}
\end{align}
Let us investigate the relationship between $(\bfS,\hat{\jmath},\chi,\vartheta)$ and $(\bfS_{\H},\hat{\jmath}_{\H},\chi_{\H},\vartheta_{\H})$.

As we saw in the proof of Proposition \ref{prop:toral-toral}, the image of $\hat{\jmath}_{\H}$ is given by $\hat{\jmath}_{\H}(\hat{\bfS}_{\H})= \bfZ_{\hat{\H}}(\phi(I_{F}^{r}))=\hat{\H}\cap\hat{\T}=\hat{\T}_{\H}$.
Recall that $\hat{\jmath}\circ\hat{\theta}_{\bfS}=\hat{\theta}\circ\hat{\jmath}$.
This implies that the embedding $\hat{\jmath}$ induces an isomorphism $\hat{\bfS}^{\h{\theta}_{\bfS},\circ}\cong\hat{\T}^{\hat{\theta},\circ}=\h\xi(\hat{\T}_{\H})$.
Hence, we get an identification of $\hat{\bfS}_{\H}$ with $\hat{\bfS}^{\hat{\theta}_{\bfS},\circ}$:
\[
\hat{\jmath}^{-1}\circ\h\xi\circ\hat{\jmath}_{\H}\colon\hat{\bfS}_{\H}\xrightarrow{\sim}\hat{\bfS}^{\hat{\theta}_{\bfS},\circ}.
\]
As discussed in the proof of Lemma \ref{lem:rationality-of-theta}, we have $\h\j\circ\sigma(t)=[\phi(\sigma)]\circ\h\j(t)$ for any $\sigma\in W_{F}$ and $t\in\h{\bfS}$.
Similarly, we also have $\h\j_{\H}\circ\sigma(t)=[\phi_{\H}(\sigma)]\circ\h\j_{\H}(t)$ for any $\sigma\in W_{F}$ and $t\in\h{\bfS}_{\H}$.
Therefore, since we have $\phi=\h\xi\circ\phi_{\H}$, the isomorphism $\hat{\jmath}^{-1}\circ\xi\circ\hat{\jmath}_{\H}$ is $\Gamma$-equivariant.
Thus, by taking dual, we get an $F$-rational isomorphism
\[
\bfS_{\H}\cong\bfS_{\theta_{\bfS}}.
\]

\begin{prop}\label{prop:vartheta}
The restriction $\vartheta|_{S_{0+}}$ of $\vartheta$ to $S_{0+}$ coincide with a pullback of the restriction $\vartheta_{\H}|_{S_{\H,0+}}$ of $\vartheta_{\H}$ to $S_{\H,0+}$ through the map $S_{0+}\rightarrow S_{\theta_{\bfS},0+}\cong S_{\H,0+}$.
\end{prop}

\begin{proof}
By abuse of notation, we again write $\vartheta_{\H}$ for the pullback of $\vartheta_{\H}$ along the canonical map $S\rightarrow S_{\theta_{\bfS}}\cong S_{\H}$.
Then our task is to show that the depth of the character $\vartheta^{-1}\cdot\vartheta_{\H}$ of $S$ is zero.
Since the local Langlands correspondence for tame tori is multiplicative and preserves the depth (see, e.g., \cite{Yu09-LLC}), the depth of $\vartheta^{-1}\cdot\vartheta_{\H}$ equals that of $\phi_{\vartheta}^{-1}\cdot\phi_{\vartheta_{\H}}$. 
Here $\phi_{\vartheta}^{-1}\cdot\phi_{\vartheta_{\H}}$ denotes the product as $1$-cocycles.

We note that the following diagram is commutative since every object is tamely ramified (more precisely, $\Gamma$-actions on $\hat{\bfS}$, $\hat{\bfS}_{\H}$, $\hat{\G}$, and $\hat{\H}$ are trivial on $I_{F}^{0+}$ and the set of $\chi$-data $\chi$ and $\chi_{\H}$ are minimally ramified):
\[
\xymatrix{
\hat{\bfS}\rtimes I_{F}^{0+}\ar@{^{(}->}^-{\Lj_{\chi}}[r]&\hat{\G}\rtimes I_{F}^{0+}\\
\hat{\bfS}_{\H}\rtimes I_{F}^{0+}\ar@{^{(}->}[u]\ar@{^{(}->}^-{\Lj_{\chi_{\H}}}[r]&\hat{\H}\rtimes I_{F}^{0+}\ar@{^{(}->}_-{\h\xi}[u]
}
\]
Thus, by taking into account the commutativity of the diagram \eqref{diag:L-parameters}, we see that the following diagram commutes:
\[
\xymatrix{
I_{F}^{0+}\ar^-{\phi_{\vartheta}|_{I_{F}^{0+}}}[r]\ar_-{\phi_{\vartheta_{\H}}|_{I_{F}^{0+}}}[rd]&\hat{\bfS}\rtimes I_{F}^{0+}\\
&\hat{\bfS}_{\H}\rtimes I_{F}^{0+}\ar@{^{(}->}[u]
}
\]
This implies that depth of the $L$-parameter $\phi_{\vartheta}^{-1}\cdot\phi_{\vartheta_{\H}}$ is zero.
\end{proof}

Now let us suppose that $\phi$ is toral of depth $r>0$.
Recall that, as proved in the proof of Lemma \ref{lem:zeta-measures-diff}, we have $(\zeta_{S}^{-1}\cdot\vartheta)\circ\theta_{\bfS}=\vartheta$.
As $\zeta_{S}$ is tamely ramified, this implies that $\vartheta|_{S_{0+}}\circ\theta_{\bfS}=\vartheta|_{S_{0+}}$.
This implies that we can take a $\theta_{\bfS}$-invariant element $X^{\ast}\in\mfs_{-r}^{\ast}$ realizing $\vartheta|_{S_{s+:r+}}$ (see the discussion in Section \ref{subsec:TCF} using Lemma \cite[Lemma 5.4]{Oi25}).
By Proposition \ref{prop:vartheta}, we furthermore have the following (note that $\mfs_{\theta_{\bfS}}^{\ast}$ is identified with the $\theta_{\bfS}$-fixed subspace of $\mfs^{\ast}$):

\begin{cor}\label{cor:vartheta}
We can take elements $X^{\ast}\in\mfs_{-r}^{\ast}$ and $Y^{\ast}\in\mfs_{\H,-r}^{\ast}$ realizing $\vartheta|_{S_{s+:r+}}$ and $\vartheta_{\H}|_{S_{\H,s+:r+}}$, respectively, so that $Y^{\ast}$ maps to $X^{\ast}$ under the natural map $\mfs_{\H}^{\ast}\cong\mfs_{\theta_{\bfS}}^{\ast}\hookrightarrow\mfs^{\ast}$.
\end{cor}

\begin{rem}
We caution that the diagram
\[
\xymatrix{
\hat{\bfS}\rtimes W_{F}\ar@{^{(}->}^-{\Lj_{\chi}}[r]&\hat{\G}\rtimes W_{F}\\
\hat{\bfS}_{\H}\rtimes W_{F}\ar@{^{(}->}[u]\ar@{^{(}->}^-{\Lj_{\chi_{\H}}}[r]&\hat{\H}\rtimes W_{F}\ar@{^{(}->}_-{\h\xi}[u]
}
\]
is not commutative in general although it is commutative at the positive depth level as observed in the proof of Proposition \ref{prop:vartheta}.
The non-commutativity of this diagram is crucially related to the computation of the transfer factor (Section \ref{subsubsec:tran-III-revisited}).
\end{rem}

\section{Twisted version of Kaletha's descent lemma}\label{sec:descent}

\subsection{Waldspurger's diagram}\label{subsec:diag}

We recall the notion of a ``diagram''.
This was introduced by Waldspurger in \cite[Section 3.2]{Wal08} first .
Then, a modified definition was introduced in \cite[I.1.10]{MW16-1}, which we follow in this paper.

\begin{defn}\label{defn:diag}
For $(\epsilon,\eta)\in H_{\ss}\times\t{G}_{\ss}$, \textit{a diagram associated to $(\epsilon,\eta)$} is a quadruple $D=(\B^{\flat},\T^{\flat},\B^{\dia},\T^{\dia})$ satisfying the following:
\begin{itemize}
\item
$\T^{\flat}$ is an $F$-rational maximal torus of $\H$,
\item
$(\B^{\flat},\T^{\flat})$ is a Borel pair of $\H$ which is $[\epsilon]$-stable,
\item
$\T^{\dia}$ is an $F$-rational maximal torus of $\G$,
\item
$(\B^{\dia},\T^{\dia})$ is a Borel pair of $\G$ which is $[\eta]$-stable,
\item
The Borel pairs $(\B^{\flat},\T^{\flat})$ and $(\B_{\H},\T_{\H})$ induce a unique isomorphism $\xi_{\flat}\colon\T^{\flat}\xrightarrow{\sim}\T_{\H}$ given by $\H$-conjugation.
Similarly, $(\B^{\dia},\T^{\dia})$ and $(\B,\T)$ induce a unique isomorphism $\xi_{\dia}\colon\T^{\dia}\xrightarrow{\sim}\T$ given by $\G$-conjugation.
Then the composition $\xi_{\flat}^{-1}\circ\xi\circ\xi_{\dia}$ is defined over $F$.
(We write $\xi_{D}$ for this composition.)
\item
If we let $g\in\G$ be an element such that $[g]=\xi_{\dia}$, then $\eta$ belongs to $[g]^{-1}(\t\T)$.
(We put $\t\T^{\dia}\colonequals \T^{\dia}\eta$ and write $\t\xi_{\dia}$ for the map $[g]\colon\t\T^{\dia}\rightarrow\t\T$.)
\item
Let $\nu_{D}\in\T$ be an element such that $[g](\eta)=\nu_{D}\theta$.
Then we have $\xi(\nu_{D})=\xi_{\flat}(\epsilon)$.
(We write $\mu_{D}$ for this element.)
In other words, if we define a map $\t\xi_{D}\colon\t\T^{\dia}\rightarrow\T^{\flat}$ by $\t{\xi}_{D}\colonequals \xi_{\flat}^{-1}\circ\xi\circ((-)\cdot\theta^{-1})\circ\t\xi_{\dia}$, then we have $\t{\xi}_{D}(\eta)=\epsilon$.
\end{itemize}
\[
\xymatrix{
\T^{\flat}\ar_-{\xi_{\flat}}[d]&&\T^{\dia}\ar^-{\xi_{\dia}}[d]\ar_-{\xi_{D}}@{->>}[ll]&\t\T^{\dia}=\T^{\dia}\eta\ar^-{\t\xi_{\dia}}[d]&\epsilon\ar^-{\xi_{\flat}}@{|->}[d]&&&\eta\ar^-{\t\xi_{\dia}}@{|->}[d]\ar_-{\t\xi_{D}}@{|->}[lll]\\
\T_{\H}\ar@{<-}_-{\xi}^-{\cong}[r]&\T_{\theta}&\T\ar@{->>}[l]&\t\T=\T\theta\ar^-{(-)\cdot\theta^{-1}}[l]&\mu_{D}&&\nu_{D}\ar@{|->}[ll]&\nu_{D}\theta\ar@{|->}[l]
}
\]
\end{defn}
\symdef{xi-D}{$\xi_{D}$}\symdef{xi-dia}{$\xi_{\dia}$}\symdef{xi-flat}{$\xi_{\flat}$}\symdef{xi-tilde-D}{$\tilde{\xi}_{D}$}\symdef{nu-D}{$\nu_{D}$}\symdef{mu-D}{$\mu_{D}$}

For $(\epsilon,\eta)\in H_{\ss}\times\t{G}_{\ss}$, let $\bfD(\epsilon,\eta)$\symdef{D-epsilon-eta}{$\bfD(\epsilon,\eta)$} denote the set of diagrams associated to $(\epsilon,\eta)$.

\begin{rem}
\begin{enumerate}
\item
We often simply write $\nu$ and $\mu$ for $\nu_{D}$ and $\mu_{D}$, respectively.
\item
The condition that $(\B^{\flat},\T^{\flat})$ is $[\epsilon]$-stable is equivalent to that $\epsilon\in\T^{\flat}$, which is furthermore equivalent to that $\T^{\flat}\subset\H_{\epsilon}$.
\item
Let $D=(\B^{\flat},\T^{\flat},\B^{\dia},\T^{\dia})\in\bfD(\epsilon,\eta)$.
Then $(\T^{\dia},\t\T^{\dia})$ is an $F$-rational twisted maximal torus of $(\G,\t\G)$ by Lemma \ref{lem:twisted-torus}.
In particular, $\T^{\nat}\colonequals \T^{\eta,\circ}$ is a maximal torus of $\G_{\eta}$ and $\T^{\dia}$ is recovered from $\T^{\nat}$ by $\T^{\dia}=\bfZ_{\G}(\T^{\nat})$ (see Proposition \ref{prop:twisted-tori}).
\item
If the set $\bfD(\epsilon,\eta)$ is not empty, then the stable conjugacy classes of $\epsilon$ and $\eta$ correspond in the sense of twisted endoscopy (see Section \ref{subsec:norm}).
\item
In general, even if the stable conjugacy classes of $\epsilon$ and $\eta$ correspond, the set $\bfD(\epsilon,\eta)$ might be empty.
See \cite[I.1.10]{MW16-1}.
\item
When $\epsilon$ is strongly $\t\G$-regular semisimple and $\eta$ is strongly regular semisimple, the set $\bfD(\epsilon,\eta)$ is not empty if and only if $\epsilon$ is a norm of $\eta$ by \cite[I.1.10, Lemme]{MW16-1}.
Furthermore, a diagram associated to $(\epsilon,\eta)$ is essentially unique.
We will investigate these facts later (Lemma \ref{lem:unique-diag}) in detail.
\end{enumerate}
\end{rem}

\begin{rem}\label{rem:Weyl}
Recall that the map $\xi\colon\T\twoheadrightarrow\T_{\theta}\xrightarrow{\sim}\T_{\H}$ induces an identification of the Weyl group $\Omega_{\H}(\T_{\H})$ with a subgroup of $\Omega_{\G}(\T)^{\theta}$.
For any diagram $D=(\B^{\flat},\T^{\flat},\B^{\dia},\T^{\dia})\in\bfD(\epsilon,\eta)$, the maps $\xi_{\dia}$ and $\xi_{\flat}$ induce isomorphisms $\Omega_{\H}(\T^{\flat})\xrightarrow{\sim}\Omega_{\H}(\T_{\H})$ and $\Omega_{\G}(\T^{\dia})\xrightarrow{\sim}\Omega_{\G}(\T)$, respectively.
The image of $\Omega_{\H_{\epsilon}}(\T^{\flat})\subset\Omega_{\H}(\T^{\flat})$ is contained in the image of $\Omega_{\G_{\eta}}(\T^{\nat})\subset\Omega_{\G}(\T^{\dia})$, hence we get an identification $\Omega_{\H_{\epsilon}}(\T^{\flat})\hookrightarrow \Omega_{\G_{\eta}}(\T^{\nat})$.
Since $\xi_{D}$ is $F$-rational, this identification is $F$-rational.
\end{rem}

\begin{lem}\label{lem:key}
Let $D\in\bfD(\epsilon,\eta)$.
Then the map $\t\xi_{D}$ is $F$-rational.
\end{lem}

\begin{proof}
Any element of $\t\T^{\dia}$ can be written as $t\eta$ with $t\in\T^{\dia}$.
Then the image of $t\eta$ under $\t\xi_{D}$ is given by $\xi_{D}(t)\epsilon$.
In other words, the map $\t\xi_{D}$ equals the composition $((-)\cdot\epsilon)\circ\xi_{D}\circ((-)\cdot\eta^{-1})$.
Since $\eta$, $\epsilon$, and $\xi_{D}$ are $F$-rational, so is $\t\xi_{D}$.
\end{proof}

\subsection{An equivalence relation on diagrams}\label{subsec:diag-eq}

Let $(\epsilon,\eta)\in H_{\ss}\times\t{G}_{\ss}$.
We introduce an equivalence relation on $\bfD(\epsilon,\eta)$ as follows:

\begin{defn}\label{defn:diag-eq}
We define $\sim$ to be the equivalence relation on $\bfD(\epsilon,\eta)$ generated by the following two equivalence relations:
Let $D=(\B^{\flat},\T^{\flat},\B^{\dia},\T^{\dia}), \bar{D}=(\bar{\B}^{\flat},\bar{\T}^{\flat},\bar{\B}^{\dia},\bar{\T}^{\dia})\in\bfD(\epsilon,\eta)$.
\begin{description}
\item[(i) $(\H_{\epsilon},\G_{\eta})$-conjugacy]
We say that $D$ and $\bar{D}$ are $(\H_{\epsilon},\G_{\eta})$-conjugate if there exists elements $h\in\H_{\epsilon}$ and $g\in\G_{\eta}$ such that 
\[
\bar{D}=({}^{h}\B^{\flat},{}^{h}\T^{\flat},{}^{g}\B^{\dia},{}^{g}\T^{\dia}).
\]
\item[(ii) $\Omega_{\H}$-conjugacy]
We say that $D$ and $\bar{D}$ are $\Omega_{\H}$-conjugate if there exists elements $w\in\Omega_{\H}(\T^{\flat})$ such that 
\[
\bar{D}=({}^{w}\B^{\flat},\T^{\flat},{}^{w}\B^{\dia},\T^{\dia}).
\]
Here, $w$ is also regarded as an element of $\Omega_{\G}(\T^{\dia})$ (see Remark \ref{rem:Weyl}).
\end{description}
\end{defn}

We write $\bbD(\epsilon,\eta)$\symdef{D-epsilon-eta-eq}{$\bbD(\epsilon,\eta)$} for the set $\bfD(\epsilon,\eta)/{\sim}$ of equivalence classes of diagrams associated to $(\epsilon,\eta)$.


\begin{rem}
When $\G$ is untwisted ($\theta$ is trivial) and $\H_{\epsilon}$ is quasi-split, the set $\bbD(\epsilon,\eta)$ is nothing but the set $\Xi(\H_{\epsilon},\G_{\eta})$ used in the proof of \cite[Theorem 6.3.4]{Kal19}.
\end{rem}

\begin{lem}\label{lem:unique-diag}
If $(\epsilon,\eta)\in\mcD$, then the set $\bbD(\epsilon,\eta)$ is a singleton.
Moreover, the maps $\xi_{D}$ and $\t{\xi}_{D}$ are independent of a diagram $D\in\bfD(\epsilon,\eta)$.
\end{lem}

\begin{proof}
The non-emptiness of $\bfD(\epsilon,\eta)$ follows from \cite[Lemma 3.3.B]{KS99} (with the argument in the final paragraph in 29 page of \cite{KS99}).
See also \cite[I.1.10, Lemme]{MW16-1}.

We show that any two diagrams associated to $(\epsilon,\eta)$ are equivalent.
Let $D=(\B^{\flat},\T^{\flat},\B^{\dia},\T^{\dia}), \bar{D}=(\bar{\B}^{\flat},\bar{\T}^{\flat},\bar{\B}^{\dia},\bar{\T}^{\dia})\in\bfD(\epsilon,\eta)$.
Since $\epsilon$ is strongly regular semisimple, we have $\T^{\flat}=\H_{\epsilon}=\bar{\T}^{\flat}$.
Similarly we have $\T^{\nat}=\G_{\eta}=\bar{\T}^{\nat}$ (recall that both $\T^{\nat}$ and $\bar{\T}^{\nat}$ are maximal tori of $\G_{\eta}$ and that $\G_{\eta}$ is a torus by the strong regularity of $\eta$).
As we have $\T^{\dia}=Z_{\G}(\T^{\nat})$ and $\bar{\T}^{\dia}=Z_{\G}(\bar{\T}^{\nat})$, we get $\T^{\dia}=\bar{\T}^{\dia}$.

Since both $\B^{\flat}$ and $\bar{\B}^{\flat}$ are Borel subgroups of $\H$ containing $\T^{\flat}$, there exists an element $w\in\Omega_{\H}(\T^{\flat})$ such that ${}^{w}\B^{\flat}=\bar{\B}^{\flat}$.
Hence, by replacing $D$ with its $\Omega_{\H}(\T^{\flat})$-conjugate diagram $({}^{w}\B^{\flat},\T^{\flat},{}^{w}\B^{\dia},\T^{\dia})$, we may suppose that $\B^{\flat}=\bar{\B}^{\flat}$.

Let $g\in\G$ be an element satisfying $({}^{g}\B^{\dia},{}^{g}\T^{\dia})=(\B,\T)$.
Similarly, let $\bar{g}\in\G$ be an element satisfying $({}^{\bar{g}}\bar{\B}^{\dia},{}^{\bar{g}}\T^{\dia})=(\B,\T)$.
Since $\t\xi_{D}(\eta)=\epsilon=\t\xi_{\bar{D}}(\eta)$, we get
\[
\xi_{\flat}^{-1}\circ\xi({}^{g}\eta\cdot\theta^{-1})
=\xi_{\flat}^{-1}\circ\xi({}^{\bar{g}}\eta\cdot\theta^{-1}).
\]
In other words, ${}^{g}\eta\cdot\theta^{-1},{}^{\bar{g}}\eta\cdot\theta^{-1}\in\T$ map to the same element of $\T_{\theta}$ under the natural quotient map $\T\twoheadrightarrow\T_{\theta}$.
Hence we can find an element $t\in\T$ such that ${}^{\bar{g}}\eta\cdot\theta^{-1}=t\cdot({}^{g}\eta\cdot\theta^{-1})\cdot\theta(t)^{-1}$, equivalently, ${}^{\bar{g}}\eta=t\cdot({}^{g}\eta)\cdot t^{-1}$.
Thus $\bar{g}^{-1}tg$ belongs to the (full) centralizer $\G^{\eta}$ of $\eta$ in $\G$.
By \cite[Section 3.1]{Wal08}, the strong regularity of $\eta$ implies $\G^{\eta}=\bfZ_{\G}^{\eta}\G_{\eta}=\bfZ_{\G}^{\eta}\T^{\nat}$.
In particular, we can take an element $z\in \bfZ_{\G}^{\eta}$ such that $z\bar{g}^{-1}tg$ belongs to $\G_{\eta}$.
If we put $g'\colonequals z\bar{g}^{-1}tg\in\G_{\eta}$, then we have $({}^{g'}\B^{\dia},{}^{g'}\T^{\dia})=(\bar{\B}^{\dia},\bar{\T}^{\dia})$.
Hence $D$ and $\bar{D}$ are $(\H_{\epsilon},\G_{\eta})$-conjugate.

Finally, noting that $\H_{\epsilon}$ and $\G_{\eta}$ are tori, we see that $(\H_{\epsilon},\G_{\eta})$-conjugacy does not change the maps $\xi_{D}$ and $\t{\xi}_{D}$.
We also see that $\Omega_{\H}(\T^{\flat})$-conjugacy does not change $\xi_{D}$ and $\t{\xi}_{D}$.
Hence $\xi_{D}$ and $\t{\xi}_{D}$ are independent of the choice of $D\in\bfD(\epsilon,\eta)$.
\end{proof}

\begin{lem}\label{lem:diag-H-conj}
Let $D=(\B^{\flat},\T^{\flat},\B^{\dia},\T^{\dia})\in\bfD(\epsilon,\eta)$.
For any $F$-rational elliptic maximal torus $\bar{\T}^{\flat}$ of $\H_{\epsilon}$, there exists a diagram $(\bar{\B}^{\flat},\bar{\T}^{\flat},\bar{\B}^{\dia},\bar{\T}^{\dia})\in\bfD(\epsilon,\eta)$ which is equivalent to $D$.
\end{lem}

\begin{proof}
Let $g\in\G$ be an element satisfying $({}^{g}\B^{\dia},{}^{g}\T^{\dia})=(\B,\T)$ (i.e., $\xi_{\dia}=[g]$).
Similarly, let $h\in\H$ be an element satisfying $({}^{h}\B^{\flat},{}^{h}\T^{\flat})=(\B_{\H},\T_{\H})$ (i.e., $\xi_{\flat}=[h]$).

Since both $\T^{\flat}$ and $\bar{\T}^{\flat}$ are maximal tori of $\H_{\epsilon}$, there exists an element $h'\in\H_{\epsilon}$ satisfying ${}^{h'}\T^{\flat}=\bar{\T}^{\flat}$.
We put $\bar{h}\colonequals hh'^{-1}$ (hence $^{\bar{h}}\bar{\T}^{\flat}=\T_{\H}$).
We define a Borel subgroup $\bar{\B}^{\flat}$ containing $\bar{\T}^{\flat}$ by $\bar{\B}^{\flat}\colonequals {}^{\bar{h}^{-1}}\B_{\H}$ (hence $^{\bar{h}}\bar{\B}^{\flat}=\B_{\H}$).

Let us construct $\bar{\T}^{\dia}$.
For this, we first take a quasi-split inner form $\G_{\eta}^{\ast}$ of $\G_{\eta}$ and an inner twist $\psi_{\eta}\colon\G_{\eta}\rightarrow\G_{\eta}^{\ast}$.
Since $\G_{\eta}^{\ast}$ is quasi-split, the maximal torus $\T^{\nat}$ of $\G_{\eta}$ transfers to an $F$-rational maximal torus $\T^{\nat\ast}$ of $\G_{\eta}^{\ast}$ (see, e.g., \cite[Lemma 3.2.2]{Kal19}).
More precisely, by composing a $\G_{\eta}^{\ast}$-conjugation with $\psi_{\eta}$ if necessary, we may assume that $\psi_{\eta}|_{\T^{\nat}}\colon\T^{\nat}\rightarrow\T^{\nat\ast}\colonequals \psi_{\eta}(\T^{\nat})$ is an $F$-rational isomorphism.
Then the inner twist $\psi_{\eta}$ induces a $\Gamma$-equivariant isomorphism $\Omega_{\G_{\eta}}(\T^{\nat})\cong \Omega_{\G_{\eta}^{\ast}}(\T^{\nat\ast})$.

Since we have ${}^{h'}\T^{\flat}=\bar{\T}^{\flat}$ and $h'\in\H_{\epsilon}$, the map $\sigma\mapsto[\sigma(h')^{-1}h']$ gives a $1$-cocycle of $\Gamma$ valued in $\Omega_{\H_{\epsilon}}(\T^{\flat})$.
Then, by the $\Gamma$-equivariant identifications of Weyl groups $\Omega_{\H_{\epsilon}}(\T^{\flat})\subset\Omega_{\G_{\eta}}(\T^{\nat})$ (see Remark \ref{rem:Weyl}) and $\Omega_{\G_{\eta}}(\T^{\nat})\cong \Omega_{\G_{\eta}^{\ast}}(\T^{\nat\ast})$, we may regard $\sigma\mapsto[\sigma(h')^{-1}h']$ as a $1$-cocycle of $\Gamma$ valued in $\Omega_{\G_{\eta}^{\ast}}(\T^{\nat\ast})$.
By applying \cite[Lemma 2.1]{Kot82} to $(\T^{\nat\ast},\G_{\eta}^{\ast})$, we take an element $g^{\ast}\in\G_{\eta}^{\ast}$ such that $[\sigma(g^{\ast})^{-1}g^{\ast}]=[\sigma(h')^{-1}h']$.
We put $\bar{\T}^{\nat\ast}\colonequals {}^{g^{\ast}}\T^{\nat\ast}$.
Then the map 
\[
\bar{\T}^{\nat\ast}
\xrightarrow{[g^{\ast}]^{-1}}
\T^{\nat\ast}
\xrightarrow{\psi_{\eta}^{-1}}
\T^{\nat}
\subset
\T^{\dia}
\xrightarrow{\xi_{D}}
\T^{\flat}
\xrightarrow{[h']}
\bar{\T}^{\flat}
\]
is defined over $F$.

Note that the maximal torus $\bar{\T}^{\nat\ast}$ is elliptic in $\G_{\eta}^{\ast}$.
Indeed, the above homomorphism $\bar{\T}^{\nat\ast}\rightarrow\bar{\T}^{\flat}$ is locally isomorphic (isomorphic at the Lie algebra level) since $\T^{\nat}$ is the identity component of the $[\eta]$-invariant of $\T^{\dia}$ and the map $\T^{\dia}\rightarrow\T^{\flat}$ induces an isomorphism between the $[\eta]$-coinvariant of $\T^{\dia}$ and $\T^{\flat}$.
Thus, since $\bar{\T}^{\flat}$ is elliptic in $\H$ and the center of $\H$ is smaller than that of $\G_{\eta}^{\ast}$, $\bar{\T}^{\nat\ast}$ is elliptic in $\G_{\eta}^{\ast}$ (later, we will review a description of the relation between these centers; see Section \ref{subsec:decomposition}).
Therefore $\bar{\T}^{\nat\ast}$ transfers to $\G_{\eta}$ (see \cite[Section 10]{Kot86} or \cite[Lemma 3.2.1]{Kal19}).
In other words, there exists an element $g^{\ast\prime}\in\G_{\eta}^{\ast}$ such that $\psi_{\eta}^{-1}\circ[g^{\ast\prime}]\colon\bar{\T}^{\nat\ast}\rightarrow\bar{\T}^{\nat}\colonequals \psi_{\eta}^{-1}\circ[g^{\ast\prime}](\bar{\T}^{\nat\ast})$ is an $F$-rational isomorphism.
Note that then, by putting $g'\colonequals \psi_{\eta}^{-1}(g^{\ast\prime}g^{\ast})\in\G_{\eta}$, we have $\bar{\T}^{\nat}={}^{g'}\T^{\nat}$.

We define $\bar{\T}^{\dia}$ by $\bar{\T}^{\dia}\colonequals \bfZ_{\G}(\bar{\T}^{\nat})={}^{g'}\T^{\dia}$.
We put $\bar{\B}^{\dia}={}^{g'}\B^{\dia}$.
Let $\bar{D}\colonequals (\bar{\B}^{\flat},\bar{\T}^{\flat},\bar{\B}^{\dia},\bar{\T}^{\dia})$.
By construction, it can be easily seen that the map $\xi_{\bar{D}}$ determined by $\bar{D}$, which is given by $[h']\circ\xi_{D}\circ[g']^{-1}$, is $F$-rational.
\[
\xymatrix{
\T^{\nat}\ar^-{[g']}[d]&&\T^{\nat\ast}\ar_-{\psi_{\eta}^{-1}}[ll]\ar^-{[g^{\ast}]}[d]&&\T^{\flat}\ar_-{[h']}[d]&&\T^{\dia}\ar^-{[g']}[d]\ar_-{\xi_{D}}[ll]\\
\bar{\T}^{\nat}&&\bar{\T}^{\nat\ast}\ar_-{\psi_{\eta}^{-1}\circ[g^{\ast\prime}]}[ll]&&\bar{\T}^{\flat}&&\bar{\T}^{\dia}\ar^-{\xi_{\bar{D}}}[ll]
}
\]

Moreover, since $h'\in\H_{\epsilon}$ and $g'\in\G_{\eta}$, we have $\t\xi_{\bar{D}}(\eta)=\epsilon$.
Thus $\bar{D}$ is a diagram associated to $(\epsilon,\eta)$ and $(\H_{\epsilon},\G_{\eta})$-conjugate to $D$.
\end{proof}

\subsection{Kaletha's descent lemma}\label{subsec:descent-lem}
Suppose that we are in the situation of Section \ref{sec:theta-stable}.
In particular, we have the sets $\t{\mcJ}^{\G}_{G}$ and $\mcJ^{\H}_{H}$ parametrizing the ($\theta$-stable) members of our $L$-packets $\Pi_{\phi}^{\G}$ and $\Pi_{\phi_{\H}}^{\H}$.

Let $\jmath\colon\T\rightarrow\bfS$ and $\jmath_{\H}\colon\T_{\H}\rightarrow\bfS_{\H}$ be the duals to $\hat{\jmath}\colon\hat{\bfS}\rightarrow\hat{\T}$ and $\hat{\jmath}_{\H}\colon\hat{\bfS}_{\H}\rightarrow\hat{\T}_{\H}$, respectively.
Since both $\T$ and $\bfS$ are $F$-rational, for any $\sigma\in\Gamma$, the map $a_{\jmath,\sigma}\colonequals \sigma(\jmath)^{-1}\circ\jmath$ is an automorphism of $\T$.
Hence we get a $1$-cocycle $a_{\jmath}\colon\Gamma\rightarrow\Aut(\T)\colon\sigma\mapsto a_{\jmath,\sigma}$.\symdef{a-j-math}{$a_{\jmath}$ ($a_{\jmath,\sigma}$)}
We define a $1$-cocycle $a_{\jmath_{\H}}\colon\Gamma\rightarrow\Aut(\T_{\H})$ in a similar way.\symdef{a-j-math-H}{$a_{\jmath_{\H}}$ ($a_{\jmath_{\H},\sigma}$)}

Recall that any $j\in\t\mcJ^{\G}_{G}$ can be written as $j=[g]\circ\jmath^{-1}$ for some $g\in\G$.
If we define a $1$-cocycle $a_{j}\colon\Gamma\rightarrow\Omega_{\G}$ by $\sigma\mapsto a_{j,\sigma}\colonequals [\sigma(g)^{-1}g]$, then this does not depend on the choice of $g\in\G$.\symdef{a-j}{$a_{j}$ ($a_{j,\sigma}$)}
Similarly, for any $j_{\H}\in\mcJ^{\H}_{H}$, we can define a $1$-cocycle $a_{j_{\H}}\colon\Gamma\rightarrow\Omega_{\H}$.\symdef{a-j-H}{$a_{j_{\H}}$ ($a_{j_{\H},\sigma}$)}

\begin{lem}\label{lem:cocycle}
For any $j_{\H}\in\mcJ^{\H}_{H}$ and $j\in\t\mcJ^{\G}_{G}$, we have $a_{j}=a_{\jmath}=a_{\jmath_{\H}}=a_{j_{\H}}$.
Here, we regard $\Omega_{\H}$, $\Omega_{\G}$, and $\Aut(\T_{\H})$ as subgroups of $\Aut(\T)$ so that the equalities make sense.
\end{lem}

\begin{proof}
Let $g\in\G$ be an element satisfying $j=[g]\circ\jmath^{-1}$.
As $j$ is defined over $F$, for any $\sigma\in\Gamma$, we have $\sigma([g]\circ\jmath^{-1})=[g]\circ\jmath^{-1}$, which implies that $a_{\jmath,\sigma}=a_{j,\sigma}$, hence $a_{\jmath}=a_{j}$.
Similarly, we also have $a_{\jmath_{\H}}=a_{j_{\H}}$.
Thus it is enough to show that $a_{\jmath}=a_{\jmath_{\H}}$.
By construction, the map $\bfS\twoheadrightarrow\bfS_{\theta_{\bfS}}\cong\bfS_{\H}$ (say $\xi_{\bfS}$) is the dual to $\hat{\jmath}^{-1}\circ\hat{\xi}\circ\hat{\jmath}_{\H}$, hence given by $\jmath_{\H}\circ\xi\circ\jmath^{-1}$.
Since $\xi_{\bfS}$ is $F$-rational, for any $\sigma\in\Gamma$, we have $\sigma(\jmath_{\H}\circ\xi\circ\jmath^{-1})=\jmath_{\H}\circ\xi\circ\jmath^{-1}$, which implies the desired assertion (recall that the identification $\Omega_{\H}\subset\Omega_{\G}^{\theta}$ is given through $\xi$).
\end{proof}

For a semisimple element $\eta\in\t{G}_{\ss}$, we define $\t\mcJ_{\G_{\eta}}^{\G}$ to be the set 
\[
\t\mcJ_{\G_{\eta}}^{\G}
\colonequals 
\{j\colon\t\bfS\hookrightarrow\t\G \mid \text{$j$ is $F$-rational, $j\sim_{\G}\jmath^{-1}$, and $\eta\in\t{S}_{j}$}\} / {\sim_{\G_{\eta}}},
\]\symdef{J-tilde-G-G-eta}{$\tilde{\mathcal{J}}_{\G_{\eta}}^{\G}$}
i.e., the set of $\G_{\eta}$-conjugacy classes of $F$-rational $\h{\j}$-admissible embeddings $j$ of a twisted maximal torus satisfying $\eta\in\t{S}_{j}$.

Similarly, for a semisimple element $\epsilon\in H_{\ss}$, we define $\mcJ_{\H_{\epsilon}}^{\H}$ to be the set 
\[
\mcJ_{\H_{\epsilon}}^{\H}
\colonequals 
\{\text{$j_{\H}\colon\bfS_{\H}\hookrightarrow\H$}\mid \text{$j_{\H}$ is $F$-rational, $j_{\H}\sim_{\H}\jmath_{\H}^{-1}$, and $\epsilon\in S_{j_{\H}}$}\} / {\sim_{\H_{\epsilon}}},
\]\symdef{J-H-H-epsilon}{$\mathcal{J}_{\H_{\epsilon}}^{\H}$}
i.e., the set of $\H_{\epsilon}$-conjugacy classes of $F$-rational $\h\j_{\H}$-admissible embeddings of $\bfS_{\H}$ into $\H$ satisfying $\epsilon\in S_{j_{\H}}$ (or equivalently, $\bfS_{j_{\H}}$ is contained in $\H_{\epsilon}$).
Here, to make the notation lighter, we write $\bfS_{j_{\H}}\colonequals \bfS_{\H,j_{\H}}=j_{\H}(\bfS_{\H})$.

In the following, we fix a semisimple element $\eta\in\t{G}_{\ss}$.
Let $\mfH_{\eta}\subset H_{\ss}$ be a set of representatives for the stable conjugacy classes of semisimple elements of $H$ corresponding to $\eta$ such that $\H_{y}$ is quasi-split for any $y\in\mfH_{\eta}$.\symdef{H-eta}{$\mfH_{\eta}$}

Now we define a map
\[
\ul{\tran}\colon\bigsqcup_{y\in \mfH_{\eta}} \bbD(y,\eta)\times\mcJ_{\H_{y}}^{\H}\rightarrow\t\mcJ_{\G_{\eta}}^{\G}
\]
\symdef{tran-ul}{$\underline{\mathfrak{tran}}$}
in the following manner.
Let $D=(\B^{\flat},\T^{\flat},\B^{\dia},\T^{\dia})\in\bbD(y,\eta)$ for $y\in\mfH_{\eta}$ and $j_{\H}=[h]\circ \jmath_{\H}^{-1}\in\mcJ^{\H}_{\H_{y}}$ ($h\in\H$).
Since the torus $\bfS_{j_{\H}}$ is elliptic in $\H$, we may assume that $\T^{\flat}=\bfS_{j_{\H}}$ by replacing $D$ with its equivalent diagram by Lemma \ref{lem:diag-H-conj}.
We take an element $h^{\flat}\in\H$ and $g^{\dia}\in\G$ such that $\xi_{\flat}=[h^{\flat}]$ and $\xi_{\dia}=[g^{\dia}]$, respectively.
Then $n_{\H}\colonequals h^{\flat}h\in\H$ belongs to $\bfN_{\H}(\T_{\H})$.
We take an element $n\in \bfN_{\G^{\theta,\circ}}(\T)$ such that $[n]\in\Omega_{\G}^{\theta}$ is equal to $[n_{\H}]\in\Omega_{\H}\subset\Omega_{\G}^{\theta}$ (we can take $n$ from $\G^{\theta,\circ}$; see \cite[Section 1.1]{KS99}).
We define an element $\ul{\tran}(D,j_{\H})$ of $\t\mcJ^{\G}_{\G_{\eta}}$ to be the following embedding $(j,\t{j})$ of $(\bfS,\t\bfS)$ into $(\G,\t\G)$:
\[
j\colonequals [g^{\dia}]^{-1}\circ[n]\circ\jmath^{-1},\quad
\t{j}\colonequals [g^{\dia}]^{-1}\circ[n]\circ\t\jmath^{-1},
\]
where $\t\jmath^{-1}\colon\t\bfS\rightarrow\t\T$ is given by $s\theta_{\bfS}\mapsto \j^{-1}(s)\theta$ for any $s\in\bfS$.
\[
\xymatrix{
\bfS \ar^-{\jmath^{-1}}[r] \ar_-{\xi_{\bfS}}[d]&\T \ar_-{\xi}[d] \ar^-{[g^{\dia}]^{-1}\circ[n]}[rr]&&\T^{\dia} \ar_-{\xi_{D}}[d] \ar^-{[g^{\dia}]}[r]&\T\ar_-{\xi}[d] \\
\bfS_{\H} \ar_-{\jmath_{\H}^{-1}}[r]&\T_{\H} \ar_-{[h]}[rr]&&\T^{\flat}\ar_-{[h^{\flat}]}[r]&\T_{\H}
}
\]

\begin{prop}\label{prop:tran-well-def}
The above procedure gives a well-defined map.
In other words, 
\begin{enumerate}
\item
$(j,\t{j})$ is an $F$-rational embedding of a twisted maximal torus,
\item
$j$ and $\jmath^{-1}$ are $\G$-conjugate,
\item
$\eta\in\t{S}_{j}$, and
\item
the $\G_{\eta}$-conjugacy class of $j$ is independent of the choices of auxiliary data.
\end{enumerate}
\end{prop}

\begin{proof}
The assertion (2) is obvious by construction.

Let us check that $j$ is $F$-rational.
For any $\sigma\in\Gamma$, we have $\sigma(j)=j$ if and only if $[\sigma(g^{\dia})]^{-1}\circ[\sigma(n)]\circ\sigma(\jmath)^{-1}
=[g^{\dia}]^{-1}\circ[n]\circ\jmath^{-1}$, or equivalently, 
\begin{align}\label{eq:F-rat}
\sigma(\jmath)^{-1}\circ\jmath
=
[\sigma(n)]^{-1}\circ[\sigma(g^{\dia})]\circ[g^{\dia}]^{-1}\circ[n].
\end{align}
If we put $j'_{\H}\colonequals [n_{\H}]\circ\jmath_{\H}^{-1}=[h^{\flat}]\circ j_{\H}$, then we have
\[
\sigma(j'_{\H})\circ j_{\H}^{\prime-1}
=[\sigma(n_{\H})]\circ\sigma(\jmath_{\H})^{-1}\circ\jmath_{\H}\circ[n_{\H}]^{-1}
=[\sigma(h^{\flat})]^{-1}\circ[h^{\flat}],
\]
hence $\sigma(\jmath_{\H})^{-1}\circ\jmath_{\H}
=[\sigma(n_{\H})]^{-1}\circ[\sigma(h^{\flat})]^{-1}\circ[h^{\flat}]\circ[n_{\H}]$.
Since we have 
\begin{itemize}
\item
$\sigma(\jmath)^{-1}\circ\jmath=a_{\jmath,\sigma}=a_{\jmath_{\H},\sigma}=\sigma(\jmath_{\H})^{-1}\circ\jmath_{\H}$ (Lemma \ref{lem:cocycle}), 
\item
$[n]^{-1}=[n_{\H}]^{-1}$ and $[\sigma(n)]^{-1}=[\sigma(n_{\H})]^{-1}$, and
\item
$[\sigma(g^{\dia})]\circ[g^{\dia}]^{-1}=[\sigma(h^{\flat})]\circ[h^{\flat}]^{-1}$ (by the $F$-rationality of $\xi_{D}$),
\end{itemize}
(all the equalities are considered in $\Omega_{\H}\subset\Omega_{\G}^{\theta}$), we get the equality \eqref{eq:F-rat}.

By noting that $n$ is $\theta$-invariant and $D$ is a diagram associated to $(y,\eta)$, we see that $\t\bfS_{j}=\t{j}(\t\bfS)$ contains $\eta\in \t{G}$.
Combined with the $F$-rationality of $j$, this shows that $\t\bfS_{j}$ is $F$-rational and $(\bfS_{j},\t\bfS_{j})$ is an $F$-rational twisted maximal torus of $(\G,\t\G)$.
Hence we get the assertions (1) and also (3).

We consider (4).
As long as $D$ and $j_{\H}$ are fixed, the embedding $(j,\t{j})$ is obviously independent of the choices of $n_{\H}$, $n$, $h^{\flat}$, and $g^{\dia}$.
Moreover, it is also easy to see that $(j,\t{j})$ does not change even if we replace $D$ with a $\Omega_{\H}(\T^{\flat})$-equivalent diagram.
Thus our task is to show that, if we take 
\begin{itemize}
\item
another embedding $\bar{j}_{\H}\in\mcJ^{\H}_{\H_{y}}$ which is $\H_{y}$-conjugate to $j_{\H}$ and
\item
another diagram $\bar{D}=(\bar{\B}^{\flat},\bar{\T}^{\flat},\bar{\B}^{\dia},\bar{\T}^{\dia})\in\bfD(y,\eta)$ which is $(\H_{y},\G_{\eta})$-conjugate to $D$ and satisfies $\bar{\T}^{\flat}=\bfS_{\bar{j}_{\H}}$,
\end{itemize}
then $j$ and $\bar{j}$ (which is constructed from $\bar{D}$ and $\bar{j}_{\H}$) are $\G_{\eta}$-conjugate.

We take $h_{y}\in\H_{y}$ such that $\bar{j}_{\H}=[h_{y}]\circ j_{\H}$ (hence $\bar{j}_{\H}=[\bar{h}]\circ\jmath_{\H}^{-1}$, where $\bar{h}=h_{y}h$).
Let $h'\in\H_{y}$ and $g'\in\G_{\eta}$ be elements such that 
\[
({}^{h'}\bar{\B}^{\flat},{}^{h'}\bar{\T}^{\flat},{}^{g'}\bar{\B}^{\dia},{}^{g'}\bar{\T}^{\dia})
=(\B^{\flat},\T^{\flat},\B^{\dia},\T^{\dia}).
\]
Then the element $\bar{h}^{\flat}\colonequals h^{\flat}h'\in\H$ satisfies $\bar{\xi}_{\flat}=[\bar{h}^{\flat}]$.
Similarly, the element $\bar{g}^{\dia}\colonequals g^{\dia}g'\in\G$ satisfies $\bar{\xi}_{\dia}=[\bar{g}^{\dia}]$.
We take an element $\bar{n}\in\G^{\theta,\circ}$ such that $[\bar{n}]\in\Omega_{\G}^{\theta}$ is equal to $[\bar{n}_{\H}]\in\Omega_{\H}\subset\Omega_{\G}^{\theta}$, where $\bar{n}_{\H}\colonequals \bar{h}^{\flat}\bar{h}\in\bfN_{\H}(\T_{\H})$.
Then, by construction, $\bar{j}$ is given by $[\bar{g}^{\dia-1}]\circ[\bar{n}]\circ\jmath^{-1}$.

In the following, we simply write $\nu$ and $\mu$ for $\nu_{D}$ and $\mu_{D}$ associated to $D$, respectively (see Definition \ref{defn:diag}).
As we have $\bar{g}^{\dia}\colonequals g^{\dia}g'$ and $g'\in\G_{\eta}$, $\bar{j}$ is $\G_{\eta}$-conjugate to $[g^{\dia}]^{-1}\circ[\bar{n}]\circ\j^{-1}$.
Since $j=[g^{\dia}]^{-1}\circ[n]\circ\jmath^{-1}$, it suffices to show that $[g^{\dia}]^{-1}\circ[\bar{n}]\circ[n]^{-1}\circ[g^{\dia}]\in\Aut(\T^{\dia})$ is realized by an element of $\Omega_{\G_{\eta}}(\T^{\nat})\subset\Omega_{\G}(\T^{\dia})$.
Since $\xi_{\dia}=[g^{\dia}]$ induces an identification $\Omega_{\G_{\eta}}(\T^{\nat})\cong\Omega_{\G_{\nu\theta}}(\T^{\theta,\circ})$, it is equivalent to showing that $[\bar{n}]\circ[n]^{-1}\in\Aut(\T)$ is realized by an element of $\Omega_{\G_{\nu\theta}}(\T^{\theta,\circ})\subset\Omega_{\G}(\T)$.
By noting that $\Omega_{\H_{\mu}}(\T_{\H})$ is identified with a subgroup of $\Omega_{\G_{\nu\theta}}(\T^{\theta,\circ})$ (both regarded as subgroups of $\Omega_{\G}(\T)$), let us show a slightly stronger statement that $[\bar{n}]\circ[n]^{-1}\in\Aut(\T)$ is realized by an element of $\Omega_{\H_{\mu}}(\T_{\H})$.
By construction, we have $[n]=[n_{\H}]=[h^{\flat}h]$ and $[\bar{n}]=[\bar{n}_{\H}]=[\bar{h}^{\flat}\bar{h}]=[h^{\flat}h'h_{y}h]$.
Thus we get $[\bar{n}]\circ[n]^{-1}=[h^{\flat}]\circ[h'h_{y}]\circ[h^{\flat}]^{-1}$.
Since $\xi_{\flat}=[h^{\flat}]$ induces an identification $\Omega_{\H_{y}}(\T^{\flat})\cong\Omega_{\H_{\mu}}(\T_{\H})$ and $[h'h_{y}]$ belongs to $\Omega_{\H_{y}}(\T^{\flat})$, we get the assertion.
\end{proof}

The following is the twisted version of Kaletha's ``descent lemma'' \cite[Lemma 6.5]{Kal15}:

\begin{prop}\label{prop:descent-lemma}
For each $y\in\mfH_{\eta}$, the restriction of $\ul{\tran}$ to $\bbD(y,\eta)\times\mcJ_{\H_{y}}^{\H}$ is a $\pi_{0}(\H^{y})(F)$-torsor onto its image.
Furthermore, $\ul{\tran}$ induces a bijective map
\[
\tran\colon\bigsqcup_{y\in \mfH_{\eta}} \bigl(\bbD(y,\eta)\times\mcJ_{\H_{y}}^{\H}\bigr)/\pi_{0}(\H^{y})(F) \rightarrow\t\mcJ_{\G_{\eta}}^{\G}.
\]
\end{prop}\symdef{tran}{$\mathfrak{tran}$}

\begin{proof}
We first show the surjectivity.
Suppose that an element $j=[g]\circ\jmath^{-1}$ of $\t\mcJ^{\G}_{\G_{\eta}}$ is given, where $g\in\G$.
We take an(y) element $j_{\H}=[h]\circ\jmath_{\H}^{-1}$ of $\mcJ^{\H}_{H}$, where $h\in\H$.
We put $\T^{\dia}\colonequals \bfS_{j}={}^{g}\T$ and $\B^{\dia}\colonequals {}^{g}\B$.
Then, by putting $[g]^{-1}(\eta)=\nu\theta\in\t\T$, $\mu\colonequals \xi(\nu)$, and $\epsilon\colonequals [h](\mu)$, we can check that $\epsilon\in H_{\ss}$ and that $D'\colonequals ({}^{h}\B_{\H},{}^{h}\T_{\H},\B^{\dia},\T^{\dia})$ is a diagram associated to $(\epsilon,\eta)$ (note that $\bfS_{j_{\H}}={}^{h}\T_{\H}$ and use Lemma \ref{lem:cocycle} to check the $F$-rationality of $\xi_{D'}$).
By the definition of the set $\mfH_{\eta}$, there exists a unique element $y\in\mfH_{\eta}$ which is stably $\H$-conjugate to $\epsilon$.
Since $\H_{y}$ is the quasi-split inner form of $\H_{\epsilon}$, the maximal torus ${}^{h}\T_{\H}$ of $\H_{\epsilon}$ transfers to $\H_{y}$ (see, e.g., \cite[Lemma 3.2.2]{Kal19}).
More precisely, we can find an element $h'\in\H$ such that $[h'](\epsilon)=y$ and $[h']$ gives an $F$-rational isomorphism from ${}^{h}\T_{\H}$ to ${}^{h'h}\T_{\H}$.
Hence, by putting $\T^{\flat}\colonequals {}^{h'h}\T_{\H}$ and $\B^{\flat}\colonequals {}^{h'h}\B_{\H}$, we get a diagram $D\colonequals (\B^{\flat},\T^{\flat},\B^{\dia},\T^{\dia})$ associated to $(y,\eta)$.
If we put $j'_{\H}\colonequals [h']\circ j_{\H}$, then $j'_{\H}$ belongs to $\mcJ^{\H}_{\H_{y}}$.
Furthermore, by going back to the construction of the map $\ul{\tran}$, we can easily check that $\ul{\tran}(D,j'_{\H})=j$.


We next investigate the fibers of $\ul{\tran}$.
For this, let us take two diagrams $D=(\B^{\flat},\T^{\flat},\B^{\dia},\T^{\dia})\in\bfD(y,\eta)$, $\bar{D}=(\bar{\B}^{\flat},\bar{\T}^{\flat},\bar{\B}^{\dia},\bar{\T}^{\dia})\in\bfD(\bar{y},\eta)$ for $y,\bar{y}\in \mfH_{\eta}$ and two embeddings $j_{\H}\in\mcJ^{\H}_{\H_{y}}$, $\bar{j}_{\H}\in\mcJ^{\H}_{\H_{\bar{y}}}$ satisfying $\ul{\tran}(D,j_{\H})=\ul{\tran}(\bar{D},\bar{j}_{\H})$.
We may suppose that $\T^{\flat}=\bfS_{j_{\H}}$ and $\bar{\T}^{\flat}=\bfS_{\bar{j}_{\H}}$ by Lemma \ref{lem:diag-H-conj}.
Let $h\in\H$ and $\bar{h}\in\H$ be elements satisfying $j_{\H}=[h]\circ \jmath_{\H}^{-1}$ and $\bar{j}_{\H}=[\bar{h}]\circ \jmath_{\H}^{-1}$, respectively.
By replacing $\bar{D}$ with its $\G_{\eta}$-equivalent diagram if necessary, we may suppose that $(D,j_{\H})$ and $(\bar{D},\bar{j}_{\H})$ produce exactly the same embedding $j$.
(Note that then $\T^{\dia}=\bfS_{j}=\bar{\T}^{\dia}$.)
We take $h^{\flat}\in\H$, $g^{\dia}\in\G$, and $n\in\G^{\theta,\circ}$ (which corresponds to $n_{\H}\colonequals h^{\flat}h$) for $D$ as in the definition of $\ul{\tran}$.
Similarly, for $\bar{D}$, we take $\bar{h}^{\flat}\in\H$, $\bar{g}^{\dia}\in\G$, and $\bar{n}\in\G^{\theta,\circ}$ (which corresponds to $\bar{n}_{\H}\colonequals \bar{h}^{\flat}\bar{h}$) for $\bar{D}$ as in the definition of $\ul{\tran}$.
Then we have $[g^{\dia}]^{-1}\circ[n]\circ \jmath^{-1}=[\bar{g}^{\dia}]^{-1}\circ[\bar{n}]\circ \jmath^{-1}$.

Thus we have $[n\bar{n}^{-1}]=[g^{\dia}\bar{g}^{\dia-1}]$, which is an equality as elements of the Weyl group $\Omega_{\H}\subset\Omega_{\G}$.
We write $w$ for this element.
Recall that $[\bar{h}^{\flat}]$ and $[\bar{g}^{\dia}]$ induce an identifications $\Omega_{\H}(\bar{\T}^{\flat})\cong\Omega_{\H}$ and $\Omega_{\G}(\bar{\T}^{\dia})\cong\Omega_{\G}$.
If we put $w^{\flat}\in\Omega_{\H}(\bar{\T}^{\flat})$ and $w^{\dia}\in\Omega_{\G}(\bar{\T}^{\dia})$ to be the images of $w\in\Omega_{\H}$ under these identifications, respectively, then $w^{\flat}$ and $w^{\dia}$ are identified through $\xi_{\bar{D}}$ (see Remark \ref{rem:Weyl}).
By replacing the diagram $\bar{D}$ with its $\Omega_{\H}$-equivalent diagram $({}^{w^{\flat}}\bar{\B}^{\flat},\bar{\T}^{\flat},{}^{w^{\dia}}\bar{\B}^{\dia},\bar{\T}^{\dia})$, we may assume that $(\bar{\B}^{\dia},\bar{\T}^{\dia})=(\B^{\dia},\T^{\dia})$.
Note that then we have $g^{\dia}=\bar{g}^{\dia}$ and $[n]=[\bar{n}]$.

Recall that $\nu_{D}\in\T$ (resp.\ $\nu_{\bar{D}}\in\T$) is the element such that $[g^{\dia}](\eta)=\nu_{D}\theta$ (resp.\ $[\bar{g}^{\dia}](\eta)=\nu_{\bar{D}}\theta$), hence we have $\nu_{\bar{D}}=\nu_{D}$.
This implies that $\mu_{\bar{D}}=\mu_{D}$.
As we have $[h^{\flat}](y)=\mu_{D}$ and $[\bar{h}^{\flat}](\bar{y})=\mu_{\bar{D}}$, we get $[h^{\flat-1}\bar{h}^{\flat}](\bar{y})=y$.
Note that the equality $[n]=[\bar{n}]$ is equivalent to the equality $[h^{\flat-1}\bar{h}^{\flat}]=[h\bar{h}^{-1}]$.
Since $[h\bar{h}^{-1}]=j_{\H}\circ \bar{j}_{\H}^{-1}$ gives an $F$-rational isomorphism from $\bar{\T}^{\flat}$ to $\T^{\flat}$ (i.e., stable conjugacy between $\bar{\T}^{\flat}$ and $\T^{\flat}$), this implies that $y$ and $\bar{y}$ are stably conjugate.
Thus the definition of the set $\mfH_{\eta}$ implies that $y=\bar{y}$.
We also get $h\bar{h}^{-1}\in\H^{y}$.
Therefore, by putting $h_{y}\colonequals h\bar{h}^{-1}\in\H^{y}$, we get $(\B^{\flat},\T^{\flat},\B^{\dia},\T^{\dia})=({}^{h_{y}}\bar{\B}^{\flat},{}^{h_{y}}\bar{\T}^{\flat},\bar{\B}^{\dia},\bar{\T}^{\dia})$ and $j_{\H}=[h_{y}]\circ \bar{j}_{\H}$.

Thus the remaining task is to show that, by replacing $\bar{D}$ with its $\H_{y}$-equivalent one if necessary, we can take $h_{y}$ to be $F$-rational.
This follows from \cite[Lemma 6.3]{Kal15} (cf.\ the proof of the descent lemma in the untwisted case; \cite[Lemma 6.5]{Kal15}).
\end{proof}

\begin{rem}\label{rem:descent-lemma}
Note that $\mcJ^{\H}_{\H_{y}}$ is not empty for any $y\in\mfH_{\eta}$.
Hence, in particular, Proposition \ref{prop:descent-lemma} implies the following: $\t{\mcJ}^{\G}_{\G_{\eta}}$ is empty if and only if $\bbD(y,\eta)$ is empty for any $y\in\mfH_{\eta}$.
\end{rem}

\section{Waldspurger's descent theorems on twisted endoscopy}\label{sec:Waldspurger}
In this section, we review part of Waldspurger's framework ``l'endoscopie tordue n'est pas si tordue''.

Note that, in the following of this paper, we need to require that our exponential map is invariant under conjugation.
However, this property might not be satisfied by a mock exponential map in the sense of \cite[Appendix A]{AS09}.
So, from now on, we furthermore assume that 
\[
p\geq(2+e_{F})n,
\]
where $e_{F}$ is the ramification index of $F/\Q_{p}$ and $n$ is the minimum dimension of a faithful representation of $\G$.
It is known that the ``traditional'' exponential map converges on the topologically nilpotent loci under this assumption, thus we can choose it as our exponential map (see \cite[Appendix B]{DR09} and also \cite[Appendice B]{Wal08}).

\subsection{Non-standard endoscopy}\label{subsec:non-std}
Let us start with recalling the formalism of non-standard endoscopy following \cite[Sections 1.7, 1.8]{Wal08}.

Let $\G_{1}$ and $\G_{2}$ be quasi-split semisimple simply-connected groups over $F$.
For each $\G_{i}$, we fix a Borel pair $(\bfB_{i},\T_{i})$ defined over $F$.
Let $\Omega_{i}$ denote the Weyl group of $\T_{i}$ in $\G_{i}$.
We write $\Phi_{i}$ and $\Phi_{i}^{\vee}$ for the set of roots and coroots of $\T_{i}$ in $\G_{i}$, respectively.
Suppose that we have an isomorphism $j_{\ast}\colon X_{\ast}(\T_{1})_{\Q}\xrightarrow{\sim}X_{\ast}(\T_{2})_{\Q}$.
Let $j^{\ast}\colon X^{\ast}(\T_{2})_{\Q}\xrightarrow{\sim}X^{\ast}(\T_{1})_{\Q}$ denote the dual to $j_{\ast}$.

Then the triple $(\G_{1},\G_{2}, j_{\ast})$ is called a \textit{non-standard endoscopic triple} if the following conditions are satisfied:

\begin{enumerate}
\item
There exist bijections $\tau^{\vee}\colon\Phi_{1}^{\vee}\xrightarrow{\sim}\Phi_{2}^{\vee}$ and $\tau\colon\Phi_{2}\xrightarrow{\sim}\Phi_{1}$ and functions $b^{\vee}\colon \Phi_{1}^{\vee}\rightarrow\Q_{>0}$ and $b \colon \Phi_{2}\rightarrow\Q_{>0}$ such that
\begin{enumerate}
\item
$\alpha_{2}^{\vee}=\tau^{\vee}(\tau(\alpha_{2})^{\vee})$ for any $\alpha_{2}\in\Phi_{2}$;
\item
we have $j_{\ast}(\alpha_{1}^{\vee})=b^{\vee}(\alpha_{1}^{\vee})\cdot \tau^{\vee}(\alpha^{\vee}_{1})$ for any $\alpha_{1}^{\vee}\in\Phi_{1}^{\vee}$ and $j^{\ast}(\alpha_{2})=b(\alpha_{2})\cdot \tau(\alpha_{2})$ for any $\alpha_{2}\in\Phi_{2}$.
\end{enumerate}
\item
The isomorphisms $j_{\ast}$ and $j^{\ast}$ are $\Gamma$-equivariant.
\end{enumerate}

For a non-standard endoscopic triple $(\G_{1},\G_{2}, j^{\ast})$, the isomorphism $j_{\ast}$ induces an isomorphism between the Lie algebras $\bmft_{1}\colonequals \Lie\T_{1}$ and $\bmft_{2}\colonequals \Lie\T_{2}$:
\[
\bmft_{1}\cong X_{\ast}(\T_{1})\otimes_{\Z}\overline{F}
\xrightarrow{j_{\ast}}
X_{\ast}(\T_{2})\otimes_{\Z}\overline{F}\cong \bmft_{2},
\]
which induces a bijection
\[
(\bmft_{1}/\Omega_{1})^{\Gamma}
\cong
(\bmft_{2}/\Omega_{2})^{\Gamma}.
\]
Thus, through this bijection, we can define a bijective correspondence between the sets of stable conjugacy classes of semisimple elements of $\mathfrak{g}_{1}$ and $\mathfrak{g}_{2}$, which preserves the regular semisimplicity.

\subsection{Decomposition of twisted endoscopy}\label{subsec:decomposition}
We next briefly review Waldspurger's decomposition result on twisted endoscopy established in \cite[Section 3]{Wal08}.

Let $(y,\eta)\in H_{\ss}\times\t{G}_{\ss}$.
In the following, we assume that
\begin{itemize}
\item
the connected centralizer $\H_{y}$ of $y$ in $\H$ is quasi-split, and
\item
the set $\bfD(y,\eta)$ of diagrams associated to $(y,\eta)$ is not empty.
\end{itemize}
We fix a diagram $D=(\B^{\flat},\T^{\flat},\B^{\dia},\T^{\dia})\in\bfD(y,\eta)$.
In \cite[Sections 3.5 and 3.6]{Wal08}, Waldspurger constructed a quasi-split connected reductive group $\bar{\H}$ over $F$ equipped with 
\begin{itemize}
\item
standard endoscopic data $(\bar{\H}, \bar{\mathcal{H}}, \bar{s}, \hat{\bar{\xi}})$ of $\G_{\eta,\sc}$, and
\item
a non-standard endoscopic triple $(\H_{y,\sc},\bar{\H}_{\sc},j_{\ast})$,
\end{itemize}
where the subscript ``$\sc$'' denotes the simply-connected cover of the derived subgroup.
Here we emphasize that the construction of these objects depends on the choice of $D\in\bfD(y,\eta)$.
\[
\xymatrix@R=10pt{
\t\G\ar@{-}_{\text{twisted endoscopy}}[dd]\ar@{~>}^{\text{descent}}[rr]&&\G_{\eta}&\G_{\eta,\sc}\ar[l]\ar@{-}^-{\text{standard endoscopy}}[d]&\\
&&&\bar{\H}&\bar{\H}_{\sc}\ar[l]\ar@{-}^-{\text{non-standard endoscopy}}[d]\\
\H\ar@{~>}_{\text{descent}}[rr]&&\H_{y}&&\H_{y,\sc}\ar[ll]
}
\]

Let us review how the stable conjugacy classes correspond under this picture (\cite[Section 3.8]{Wal08}).
We first consider the decompositions of the Lie algebras
\begin{align*}
&\bmfg_{\eta}=\bmfg_{\eta,\sc}\oplus\bmfz_{\G_{\eta}}=\Lie\G_{\eta,\sc}\oplus\Lie\bfZ_{\G_{\eta}},\\
&\bar{\bmfh}=\bar{\bmfh}_{\sc}\oplus\bmfz_{\bar{\H}}=\Lie\bar{\H}_{\sc}\oplus\Lie\bfZ_{\bar{\H}},\\
&\bmfh_{y}=\bmfh_{y,\sc}\oplus\bmfz_{\H_{y}}=\Lie\H_{y,\sc}\oplus\Lie\bfZ_{\H_{y}}.
\end{align*}
For any $X\in\bmfg_{\eta}$, $\bar{Y}\in\bar{\bmfh}$, and $Y\in\bmfh_{y}$, we write $X=X_{\sc}+X_{Z}$, $\bar{Y}=\bar{Y}_{\sc}+\bar{Y}_{Z}$, and $Y=Y_{\sc}+Y_{Z}$ for their decompositions according to the above direct sum decompositions, respectively.
We note that we have an $F$-rational isomorphism $\mfz_{\H_{y}}\cong\mfz_{\bar{\H}}\oplus\mfz_{\G_{\eta}}$ (see \cite[Section 3.8]{Wal08}).

For our convenience, let us introduce the following terminology:
\begin{defn}\label{defn:norm-pair}.
We say that $(Y,X)\in\mfh_{y,0+}\times\mfg_{\eta,0+}$ is a \textit{$D$-norm pair} if 
\begin{itemize}
\item
$\eta\exp(X)\in\t{G}$ is strongly regular semisimple,
\item
$y\exp(Y)\in H$ is strongly $\t\G$-regular semisimple,
\end{itemize}
and there exists an element $\bar{Y}\in\bar{\mfh}$ satisfying the following:
\begin{itemize}
\item
$\bar{Y}\in\bar{\mfh}$ is a norm of $X_{\sc}\in\mfg_{\eta,\sc}$ in the sense of standard endoscopy,
\item
the stable conjugacy classes of $\bar{Y}_{\sc}\in\bar{\mfh}_{\sc}$ and $Y_{\sc}\in\mfh_{y,\sc}$ correspond in the sense of non-standard endoscopy (see Section \ref{subsec:non-std}),
\item
$Y_{Z}\in\mfz_{\H_{y}}$ corresponds to $\bar{Y}_{Z}+X_{Z}\in\mfz_{\bar{\H}}\oplus\mfz_{\G_{\eta}}$ under the identification $\mfz_{\H_{y}}\cong\mfz_{\bar{\H}}\oplus\mfz_{\G_{\eta}}$.
\end{itemize}
\end{defn}

The following is part of \cite[Section 3.8, Lemme]{Wal08}:

\begin{prop}\label{prop:norm-decomp}
For any $D$-norm pair $(Y,X)$, $(y\exp(Y),\eta\exp(X))\in\mcD$.
\end{prop}


\subsection{Descent of transfer factor}

We write $\Delta\!^{D}$ for the (absolute or relative) Lie algebra transfer factor for the pair $(\bar{\H},\G_{\eta,\sc})$.
Note that we put the symbol $D$ on the exponent in order to emphasize that the endoscopic structure of $(\bar{\H},\G_{\eta,\sc})$ depends on the choice of a diagram $D\in\bfD(y,\eta)$.

\begin{thm}[{\cite[Section 3.9, Th\'eor\`eme]{Wal08}}]\label{thm:NST}
There exists a neighborhood $\mfV$ of $0$ in $\mfh_{y,0+}$ such that, for any $D$-norm pairs $(Y,X), (\ul{Y},\ul{X})\in \mfV\times\mfg_{\eta,0+}$, we have
\[
\Delta\bigl(y\exp(Y),\eta\exp(X); y\exp(\underline{Y}),\eta\exp(\underline{X})\bigr)
=
\Delta\!^{D}(\bar{Y},X_{\sc}; \bar{\underline{Y}},\underline{X}_{\sc}),
\]
where $\bar{Y}$ and $\bar{\ul{Y}}$ are the elements of $\bar{\mfh}$ associated to $(Y,X)$ and $(\ul{Y},\ul{X})$ as in Definition \ref{defn:norm-pair}, respectively.
\end{thm}

\begin{cor}\label{cor:NST}
The absolute Lie algebra transfer factor $\Delta\!^{D}(-,-)$ can be normalized so that there exists a neighborhood $\mfV$ of $0$ in $\mfh_{y,0+}$ such that, for any $D$-norm pair $(Y,X)\in\mfV\times\mfg_{\eta,0+}$, we have
\[
\Delta\bigl(y\exp(Y),\eta\exp(X)\bigr)
=
\Delta\!^{D}(\bar{Y},X_{\sc}).
\]
\end{cor}

\subsection{Transfer of Fourier transforms of orbital integrals}\label{subsec:Lie-tran}

In this section, we summarize the results on the transfer of the Fourier transforms of orbital integrals on Lie algebras, which were established by Waldspurger and Ng\^o.

For any connected reductive group $\J$ over $F$ equipped with an invariant symmetric non-degenerate bilinear form $B_{\mfj}$ on $\mfj=\Lie\J(F)$, we let $\gamma(\mfj)$ denote the Weil constant of $(\mfj,B_{\mfj})$ with respect to the fixed non-trivial additive character $\psi_{F}$ of $F$ (see \cite[Section 3.1]{Wal97}).\symdef{gamma-j}{$\gamma(\mfj)$}
For regular semisimple elements $X\in\mfj$ and $X^{\ast}\in\mfj$, we put
\[
D_{X^{\ast}}^{\J}(X)\colonequals \gamma(\mfj)\cdot\hat{\iota}^{\J}_{X^{\ast}}(X).
\]\symdef{D-J-X-ast}{$D_{X^{\ast}}^{\J}$}
Here, $\hat{\iota}^{\J}_{X^{\ast}}(X)$ is the normalized Fourier transform of the orbital integral of $X^{\ast}$ (see Section \ref{subsec:TCF}; note that $X^{\ast}$ is regarded also as an element of $\mfj^{\ast}$ via $B_{\mfj}$).
We also put
\[
D^{\J,\st}_{X^{\ast}}(X)
\colonequals 
\sum_{X^{\ast\prime}\sim_{\J}X^{\ast}/{\sim_{J}}}
D^{\J}_{X^{\ast\prime}}(X),
\]\symdef{D-J-st-X-ast}{$D_{X^{\ast}}^{\J,\st}$}
where the index set is over the $J$-conjugacy classes within the stable conjugacy class of $X^{\ast}$ in $\mfj$.

\subsubsection{The case of standard endoscopy}
Let $\J$ be a connected reductive group over $F$ and $\J'$ a standard endoscopic group of $\J$.
We fix an invariant symmetric non-degenerate bilinear form $B_{\mfj}$ on $\mfj$.
Then $B_{\mfj}$ induces an invariant symmetric non-degenerate bilinear form on $\mfj'$ (see \cite[Section VIII.6]{Wal95}).
Let us write $B_{\mfj'}$ for this bilinear form.
We remark that these bilinear forms satisfy the following consistency property on the maximal tori.
Let $\T_{\J}$ and $\T_{\J'}$ be maximal tori of $\J$ and $\J'$ belonging to the (implicitly fixed) pinnings of $\J$ and $\J'$, respectively.
Then the endoscopic structure of $\J'$ in $\J$ gives an isomorphism $\xi_{\J}\colon\T_{\J}\cong\T_{\J'}$, which induces an isomorphism $\xi_{\J}\colon\bmft_{\J}\cong\bmft_{\J'}$ on the Lie algebras.
With these notation, for any $X,X'\in\mft_{\J}$, we have $B_{\mfj}(X,X')=B_{\mfj'}(\xi_{\J}(X),\xi_{\J}(X'))$.

For a strongly $\J$-regular semisimple element $Y^{\ast}\in\mfj'$ and a strongly regular semisimple element $X\in\mfj$, we put
\[
D_{\J',\J}(Y^{\ast},X)
\colonequals 
\sum_{X^{\ast}\leftrightarrow Y^{\ast}/{\sim_{J}}}
\mr{\Delta}_{\J',\J}(Y^{\ast},X^{\ast})
D^{\J}_{X^{\ast}}(X),
\]\symdef{D-J-prime-J-Y-ast-X}{$D_{\J',\J}(Y^{\ast},X)$}
where the index set is over the $J$-conjugacy classes of strongly regular semisimple elements of $\mfj$ which correspond to $Y^{\ast}$, and $\mr{\Delta}_{\J',\J}(Y^{\ast},X^{\ast})$ denotes the Lie algebra transfer factor without the fourth factor.
We also put
\[
\t{D}_{\J',\J}(Y^{\ast},X)
\colonequals 
\sum_{Y'\leftrightarrow X/{\sim_{\J'}}}
\mr{\Delta}_{\J',\J}(Y',X)
D^{\J',\st}_{Y^{\ast}}(Y'),
\]\symdef{D-tilde-J-prime-J-Y-ast-X}{$\t{D}_{\J',\J}(Y^{\ast},X)$}
where the index set is over the stable conjugacy classes of the elements of $\mfj'$ which correspond to $X$.

With these notation, the following holds:
\begin{thm}[{\cite[1.2. Conjecture]{Wal97}}; {\cite{Wal06},\cite{Ngo10}}]\label{thm:tran-std}
We have
\[
\t{D}_{\J',\J}(Y^{\ast},X)
=
D_{\J',\J}(Y^{\ast},X).
\]
\end{thm}

\subsubsection{The case of non-standard endoscopy}
Let $(\G_{1},\G_{2},j^{\ast})$ be a non-standard endoscopic triple.
We fix an invariant symmetric non-degenerate bilinear form $B_{i}$ on $\mfg_{i}$ for each $i$ such that we have $B_{1}(X,X')=B_{2}\bigl(j_{\ast}(X),j_{\ast}(X')\bigr)$ for any $X, X'\in\mft_{1}$.

\begin{thm}[{\cite[Proposition 1.8]{Wal08}}]\label{thm:tran-nonst}
For any regular semisimple elements $Y_{1}\in\mfg_{1}$ and $Y_{2}\in\mfg_{2}$ which correspond (resp.\ $X_{1}^{\ast}\in\mfg_{1}$ and $X_{2}^\ast\in\mfg_{2}$ which correspond), we have 
\[
D^{\G_{1},\st}_{X_{1}^{\ast}}(Y_{1})
=
D^{\G_{2},\st}_{X_{2}^{\ast}}(Y_{2}).
\]
\end{thm}


\subsubsection{The case of isogeny}

Let $\J$ be a connected reductive group over $F$.
We fix a $J$-invariant symmetric non-degenerate bilinear form $B_{\mfj}$ on $\mfj$.
Then we get an identification $\mfj\cong\mfj^{\ast}$ which also induces identifications $\mfj_{\sc}\cong\mfj_{\sc}^{\ast}$ and $\mfz_{\J}\cong\mfz_{\J}^{\ast}$.

\begin{lem}\label{lem:isog}
For a strongly regular semisimple element $X\in\mfj$ with decomposition $X=X_{\sc}+X_{Z}\in\mfj_{\sc}\oplus\mfz_{\J}$ and a strongly regular semisimple element $X^{\ast}\in\mfj$ with decomposition $X^{\ast}=X^{\ast}_{\sc}+X^{\ast}_{Z}\in\mfj_{\sc}\oplus\mfz_{\J}$, we have
\[
D^{\J,\st}_{X^{\ast}}(X)
=
\gamma(\mfz_{\J})\cdot\psi_{F}(B_{\mfj}(X^{\ast}_{Z},X_{Z}))\cdot D^{\J_{\sc},\st}_{X^{\ast}_{\sc}}(X_{\sc}).
\]
\end{lem}

\begin{proof}
According to \cite[4.4 (1)]{Wal97}, we have
\[
\hat{\iota}^{\J}_{X^{\ast}}(X)
=
\psi_{F}(B_{\mfj}(X^{\ast}_{Z},X_{Z}))\cdot \hat{\iota}^{\J_{\sc}}_{X^{\ast}_{\sc}}(X_{\sc}).
\]
In general, for the orthogonal sum $V_{1}\oplus V_{2}$ of any finite-dimensional quadratic spaces $V_{1}$ and $V_{2}$, we have $\gamma(V_{1}\oplus V_{2})=\gamma(V_{1})\cdot\gamma(V_{2})$.
Hence we have $\gamma(\mfj)=\gamma(\mfj_{\sc})\cdot\gamma(\mfz_{\J})$.
This implies that
\[
D^{\J}_{X^{\ast}}(X)
=
\gamma(\mfz_{\J})\cdot\psi_{F}(B_{\mfj}(X^{\ast}_{Z},X_{Z}))\cdot D^{\J_{\sc}}_{X^{\ast}_{\sc}}(X_{\sc}).
\]
For any $X^{\prime\ast}\in\mfj$ with decomposition $X^{\prime\ast}_{\sc}+X^{\prime\ast}_{Z}\in\mfj_{\sc}\oplus\mfz_{\J}$, $X^{\prime\ast}$ is stably $\J$-conjugate (resp.\ $J$-conjugate) to $X^{\ast}$ if and only if $X^{\prime\ast}_{\sc}$ is stably $\J$-conjugate (resp.\ $J$-conjugate) to $X^{\ast}_{\sc}$ and $X^{\prime\ast}_{Z}=X^{\ast}_{Z}$.
Thus we get the assertion.
\end{proof}

\subsubsection{Combined form}
Now let us go back to our situation; $\H$ is a twisted endoscopic group of $\t\G$.
Suppose that we have $(y,\eta)\in \t{G}_{\ss}\times H_{\ss}$ satisfying $y\in\mfH_{\eta}$ (see Section \ref{subsec:descent-lem}) and that we have a diagram $D\in\bfD(y,\eta)$.
Then we get the associated group $\bar{\H}$ as in Section \ref{subsec:decomposition}.
We fix invariant symmetric non-degenerate bilinear forms $B_{\mfg_{\eta}}$ on $\mfg_{\eta}$, $B_{\bar{\mfh}}$ on $\bar{\mfh}$, and $B_{\mfh_{y}}$ on $\mfg_{\eta}$ such that the restriction of $B_{\mfg_{\eta}}$ to $\mfz_{\H_{y}}$ is identified with the orthogonal sum of the restrictions of $B_{\mfg_{\eta}}$ to $\mfz_{\G_{\eta}}$ and $B_{\bar{\mfh}}$ to $\mfz_{\bar{\H}}$ through the isomorphism $\mfz_{\H_{y}}\cong\mfz_{\G_{\eta}}\oplus\mfz_{\bar{\H}}$.

We take
\begin{itemize}
\item
a strongly regular semisimple element $Y^{\ast}\in\mfh_{y,0+}$ with decomposition $Y^{\ast}=Y_{\sc}^{\ast}+Y^{\ast}_{Z}\in\mfh_{y,\sc}\oplus\mfz_{\H_{y}}$,
\item
a strongly regular semisimple element $\bar{Y}_{\sc}^{\ast}\in\bar{\mfh}_{\sc,0+}$ whose stable conjugacy class corresponds to that of $Y_{\sc}^{\ast}$, and
\item
a strongly regular semisimple element $X\in\mfg_{\eta,0+}$ with decomposition $X=X_{\sc}+X_{Z}\in\mfg_{\eta,\sc}\oplus\mfz_{\G_{\eta}}$.
\end{itemize}
Let $Y^{\ast}_{Z}=\bar{Y}_{Z}^{\ast}+X^{\ast}_{Z}\in\mfz_{\bar{\H}}\oplus\mfz_{\G_{\eta}}$ be the decomposition of the center part $Y^{\ast}_{Z}\in\mfz_{\H_{y}}$.
We put $\bar{Y}^{\ast}\colonequals \bar{Y}_{\sc}^{\ast}+\bar{Y}_{Z}^{\ast}$.

\begin{prop}\label{prop:Lie-twisted}
With the above notation, we have
\[
\sum_{Y\overset{D}{\leftrightarrow} X/{\sim_{\H_{y}}}}\mr{\Delta}\!^{D}(\bar{Y},X_{\sc})
D^{\H_{y},\st}_{Y^{\ast}}(Y)
=
\sum_{X^{\ast}\overset{D}{\leftrightarrow} Y^{\ast} /{\sim_{G_{\eta}}}}
\mr{\Delta}\!^{D}(\bar{Y}^{\ast},X^{\ast}_{\sc})
D^{\G_{\eta}}_{X^{\ast}}(X),
\]
where
\begin{itemize}
\item
the left sum is over the stable conjugacy classes of strongly regular semisimple elements $Y$ of $\mfh_{y,0+}$ such that $(Y,X)$ is a $D$-norm pair ($\bar{Y}$ is the element associated to $(Y,X)$ as in Definition \ref{defn:norm-pair}), and
\item
the right sum is over the $G_{\eta}$-conjugacy classes of strongly regular semisimple elements $X^{\ast}$ of $\mfg_{\eta,0+}$ such that $(Y^{\ast},X^{\ast})$ is a $D$-norm pair.
\end{itemize}
\end{prop}

\begin{proof}
Let $Y\in\mfh_{y,0+}$ be a strongly regular semisimple element with decomposition $Y_{\sc}+Y_{Z}\in\mfh_{y,\sc}\oplus\mfz_{\H_{y}}$.
Suppose that $Y_{\sc}$ corresponds to the stable conjugacy class of a strongly regular semisimple element $\bar{Y}_{\sc}\in\bar{\mfh}_{\sc}$.
Also suppose that $Y_{Z}$ equals $\bar{Y}_{Z}+X'_{Z}$ under the isomorphism $\mfz_{\H_{y}}\cong\mfz_{\bar{\H}}\oplus\mfz_{\G_{\eta}}$.
Then, by definition, $(Y,X)$ is a $D$-norm pair if and only if $\bar{Y}\colonequals \bar{Y}_{\sc}+\bar{Y}_{Z}$ is a norm of $X_{\sc}$ and $X'_{Z}=X_{Z}$.
Hence, by noting that two strongly regular semisimple elements $Y_{1},Y_{2}\in\mfh_{y}$ are stably conjugate if and only if $Y_{1,\sc}, Y_{2,\sc}\in\mfh_{y,\sc}$ are stably conjugate and $Y_{1,Z}=Y_{2,Z}$, we see that the left-hand side of the desired identity equals
\begin{align}\label{eq:combined-1}
\sum_{\bar{Y}_{Z}\in\mfz_{\bar{\H}}}
\gamma(\mfz_{\H_{y}})\cdot\psi_{F}(B_{\mfh_{y}}(Y^{\ast}_{Z}, Y_{Z}))
\sum_{Y_{\sc}\leftrightarrow X_{\sc}/{\sim_{\H_{y,\sc}}}}\mr{\Delta}\!^{D}(\bar{Y},X_{\sc})
D^{\H_{y,\sc},\st}_{Y_{\sc}^{\ast}}(Y_{\sc})
\end{align}
by Lemma \ref{lem:isog} (transfer for isogeny) for $\H_{y}$.
Here, the second sum is over the stable conjugacy classes of strongly regular semisimple elements of $\mfh_{y,\sc}$ such that $\bar{Y}_{\sc}+\bar{Y}_{Z}$ is a norm of $X_{\sc}$, where $\bar{Y}_{\sc}\in\bar{\mfh}_{\sc}$ is an element whose stable conjugacy class corresponds to $Y_{\sc}$.
Note that the index set $\{\bar{Y}_{Z}\in\mfz_{\bar{\H}}\}$ of the first sum is infinite, but only finite of them have a nontrivial contribution because of the second sum.

By Theorem \ref{thm:tran-nonst} (transfer for non-standard endoscopy), \eqref{eq:combined-1} equals
\begin{align}\label{eq:combined-2}
\sum_{\bar{Y}_{Z}\in\mfz_{\bar{\H}}}
\gamma(\mfz_{\H_{y}})\cdot\psi_{F}(B_{\mfh_{y}}(Y^{\ast}_{Z}, Y_{Z}))
\sum_{\bar{Y}_{\sc}\leftrightarrow X_{\sc}/{\sim_{\bar{\H}_{\sc}}}}
\mr{\Delta}\!^{D}(\bar{Y},X_{\sc})
D^{\bar{\H}_{\sc},\st}_{\bar{Y}_{\sc}^{\ast}}(\bar{Y}_{\sc}),
\end{align}
where the second sum is over the stable conjugacy classes of strongly regular semisimple elements $\bar{Y}_{\sc}\in\bar{\mfh}_{\sc}$ such that $\bar{Y}_{\sc}+\bar{Y}_{Z}$ is a norm of $X_{\sc}\in\mfg_{\eta,\sc}$.
By rearranging the sums, we see that \eqref{eq:combined-2} equals
\begin{align}\label{eq:combined-3}
\sum_{\bar{Y}\leftrightarrow X_{\sc} /{\sim_{\bar{\H}}}}
\gamma(\mfz_{\H_{y}})\cdot\psi_{F}(B_{\mfh_{y}}(Y^{\ast}_{Z}, Y_{Z}))
\cdot\mr{\Delta}\!^{D}(\bar{Y},X_{\sc})
D^{\bar{\H}_{\sc},\st}_{\bar{Y}^{\ast}_{\sc}}(\bar{Y}_{\sc}),
\end{align}
where the sum is over the set of stable conjugacy classes of strongly regular semisimple elements of $\bar{\mfh}$ which are norms of $X_{\sc}\in\mfg_{\eta,\sc}$.
By noting that $\gamma(\mfz_{\H_{y}})=\gamma(\mfz_{\G_{\eta}})\cdot\gamma(\mfz_{\bar{\H}})$ and that $\psi_{F}(B_{\mfh_{y}}(Y^{\ast}_{Z}, Y_{Z}))=\psi_{F}(B_{\bar{\mfh}}(\bar{Y}_{Z}^{\ast}, \bar{Y}_{Z}))\cdot\psi_{F}(B_{\mfg_{\eta}}(X^{\ast}_{Z}, X_{Z}))$, Lemma \ref{lem:isog} (transfer for isogeny) for $\bar{\H}$ implies that \eqref{eq:combined-3} equals
\begin{align}\label{eq:combined-4}
\gamma(\mfz_{\G_{\eta}})\cdot\psi_{F}(B_{\mfg_{\eta}}(X^{\ast}_{Z}, X_{Z}))
\sum_{\bar{Y}\leftrightarrow X_{\sc}/{\sim_{\bar{\H}}}}
\mr{\Delta}\!^{D}(\bar{Y},X_{\sc})
D^{\bar{\H},\st}_{\bar{Y}^{\ast}}(\bar{Y}).
\end{align}

Finally, by Theorem \ref{thm:tran-std} (transfer for standard endoscopy), \eqref{eq:combined-4} equals
\begin{align}\label{eq:combined-5}
\gamma(\mfz_{\G_{\eta}})\cdot\psi_{F}(B_{\mfg_{\eta}}(X^{\ast}_{Z}, X_{Z}))\sum_{X^{\ast}_{\sc}\leftrightarrow \bar{Y}^{\ast} /{\sim_{G_{\eta,\sc}}}}
\mr{\Delta}\!^{D}(\bar{Y}^{\ast},X^{\ast}_{\sc})
D^{\G_{\eta,\sc}}_{X^{\ast}_{\sc}}(X_{\sc}),
\end{align}
where the index set is over the $G_{\eta,\sc}$-conjugacy classes of strongly regular semisimple elements of $\mfg_{\eta,\sc}$ which correspond to $\bar{Y}^{\ast}$.
Then the same argument as in the proof of Lemma \ref{lem:isog} implies that \eqref{eq:combined-5} equals
\begin{align}\label{eq:combined-6}
\sum_{X^{\ast}_{\sc}\leftrightarrow \bar{Y}^{\ast}/{\sim_{G_{\eta}}}}
\mr{\Delta}\!^{D}(\bar{Y}^{\ast},X^{\ast}_{\sc})
D^{\G_{\eta}}_{X^{\ast}}(X).
\end{align}
We recall that $(Y^{\ast},X^{\ast})$ is a $D$-norm pair if and only if $X^{\ast}_{\sc}$ corresponds to $\bar{Y}^{\ast}$ and the center part of $X^{\ast}$ is given by $X^{\ast}_{Z}$, which determined by $Y^{\ast}$.
Thus we see that the index set of the sum in \eqref{eq:combined-6} is nothing but that of the sum on the right-hand side of the desired identity.
\end{proof}

\section{Toral invariants for restricted roots}\label{sec:res-toral-inv}

\subsection{Root systems}\label{subsec:rootsys}

Let $\eta\in\t{G}_{\ss}$ and $y\in\mfH_{\eta}$ (see Section \ref{subsec:descent-lem}) such that $\bfD(y,\eta)$ is not empty.
In the following, we fix a diagram $D=(\B^{\flat},\T^{\flat},\B^{\dia},\T^{\dia})\in\bfD(y,\eta)$ and simply write $\nu$ for $\nu_{D}$ (resp.\ $\mu$ for $\mu_{D}$).

Recall that, with the notation as in Section \ref{subsec:twisted-tori}, we have
\begin{align*}
\Phi(\G_{\nu\theta},\T^{\theta,\circ})
&=\{p^{\ast}(\alpha) \mid \alpha\in \Phi(\G,\T); N(\alpha)(\nu)=1\}
\subset\Phi_{\res}(\G,\T),\\
\Phi^{\vee}(\G_{\nu\theta},\T^{\theta,\circ})
&=\{N(\alpha^{\vee}) \mid \alpha^{\vee}\in \Phi^{\vee}(\G,\T); N(\alpha)(\nu)=1\}
\subset\Phi_{\res}^{\vee}(\G,\T)
\end{align*}
(note that now we assume that $\Phi_{\res}(\G,\T)$ does not contain any restricted root of type $2$ or $3$).
The sets $\Phi(\H_{\mu},\T_{\H})$ and $\Phi^{\vee}(\H_{\mu},\T_{\H})$ are given as follows:
\begin{align*}
\Phi(\H_{\mu},\T_{\H})
&=\{ N(\alpha) \mid \alpha\in \Phi(\G,\T); N(\alpha^{\vee})(s)=1, N(\alpha)(\nu)=1\}
\subset X^{\ast}(\T)^{\theta}\cong X^{\ast}(\T_{\H}),\\
\Phi^{\vee}(\H_{\mu},\T_{\H})
&=\{p_{\ast}(\alpha^{\vee}) \mid \alpha^{\vee}\in \Phi^{\vee}(\G,\T); N(\alpha^{\vee})(s)=1, N(\alpha)(\nu)=1\}
\subset Y_{\ast}(\T)\cong X_{\ast}(\T_{\H}).
\end{align*}
We define an injective map $i^{\vee}\colon \Phi^{\vee}(\H_{\mu},\T_{\H})\rightarrow \Phi^{\vee}(\G_{\nu\theta},\T^{\theta,\circ})$ by $i^{\vee}(p_{\ast}(\alpha^{\vee}))\colonequals N(\alpha^{\vee})$.\symdef{i-check}{$i^{\vee}$}
Then we can regard $\Phi^{\vee}(\H_{\mu},\T_{\H})$ as a subset of $\Phi^{\vee}(\G_{\nu\theta},\T^{\theta,\circ})$ via $i^{\vee}$.

Now we transfer this discussion to $\Phi(\G,\T^{\dia})$ by using the fixed diagram $D$.
Via the map $\xi_{\dia}$, $\Phi_{\res}^{(\vee)}(\G,\T^{\dia})$ is identified with $\Phi_{\res}^{(\vee)}(\G,\T)$.
Moreover, $\Phi^{(\vee)}(\G_{\eta},\T^{\nat})$ is mapped to $\Phi^{(\vee)}(\G_{\nu\theta},\T^{\theta,\circ})$ by this identification.
Similarly, via the map $\xi_{\flat}$, $\Phi^{(\vee)}(\H,\T^{\flat})$ is identified with $\Phi^{(\vee)}(\H,\T_{\H})$, and $\Phi^{(\vee)}(\H_{y},\T^{\flat})$ is mapped to $\Phi^{(\vee)}(\H_{\mu},\T_{\H})$.
By combining these bijective maps with the previous injective map $i^{\vee}$, we may identify $\Phi^{\vee}(\H_{y},\T^{\flat})$ as a subset of $\Phi^{\vee}(\G_{\eta},\T^{\nat})$.
Accordingly, we may also may identify $\Phi(\H_{y},\T^{\flat})$ as a subset of $\Phi(\G_{\eta},\T^{\nat})$.

Let $(Y,X)\in\mfh_{y,0+}\times\mfg_{\eta,0+}$ be a $D$-norm pair.
Let us fix bilinear forms $B_{\mfg_{\eta}}$ on $\mfg_{\eta}$, $B_{\bar{\mfh}}$ on $\bar{\mfh}$, and $B_{\mfh_{y}}$ on $\mfh_{y}$ as in Section \ref{subsec:Lie-tran}.
Then $X\in\mfg_{\eta}$ (resp.\ $Y\in\mfh_{y}$) can be identified with an element $X^{\ast}\in\mfg_{\eta}^{\ast}$ (resp.\ $Y^{\ast}\in\mfh_{y}^{\ast}$).

\begin{lem}\label{lem:appearance-of-l-alpha}
Let $\alpha_{y}^{\vee}$ be an element of $\Phi^{\vee}(\H_{y},\T^{\flat})$ which is identified with an element $\alpha_{\eta}^{\vee}$ of $\Phi^{\vee}(\G_{\eta},\T^{\nat})$.
Let $\alpha^{\vee}\in \Phi^{\vee}(\G,\T^{\dia})$ be the coroot satisfying $\alpha_{y}^{\vee}=p_{\ast}(\alpha^{\vee})\in X_{\ast}(\T^{\flat})$ and $\alpha_{\eta}^{\vee}=N(\alpha^{\vee})\in X_{\ast}(\T^{\nat})$.
Then we have
\[
l_{\alpha}\cdot\langle d\alpha_{y}^{\vee}(1),Y^{\ast}\rangle
=\langle d\alpha_{\eta}^{\vee}(1),X^{\ast}\rangle.
\]
\end{lem}

\begin{proof}
Since $(Y,X)$ is a norm pair with respect to $D=(\B^{\flat},\T^{\flat},\B^{\dia},\T^{\dia})$, we may suppose that $X\in\mft^{\nat}$ and $Y\in\mft^{\flat}$ and $\xi_{D}\colon\mft^{\dia}\twoheadrightarrow\mft^{\flat}$ maps $X$ to $Y$ (cf.\ the argument in the proof of Lemma \ref{lem:diag-H-conj}).
Hence, by our choice of bilinear forms, $X^{\ast}$ is identified with $Y^{\ast}$ under the dual isomorphism $\mft^{\nat\ast}\cong\mft^{\flat\ast}$ to $\xi_{D}\colon\mft^{\nat}\cong\mft^{\flat}$.
Thus we can see that, under the identification $\xi_{D}$, we have $\langle d\alpha_{\eta}^{\vee}(1),X^{\ast}\rangle
=\sum_{i=0}^{l_{\alpha}}\langle d[\eta]^{i}(\alpha)^{\vee}(1),X^{\ast}\rangle
=l_{\alpha}\langle d\alpha^{\vee}(1),X^{\ast}\rangle
=l_{\alpha}\langle d\alpha_{y}^{\vee}(1),Y^{\ast}\rangle$.
\end{proof}

\subsection{Analysis of ramified restricted roots}\label{subsec:ram-rest-roots}
Let us next suppose that we have an $F$-rational twisted maximal torus $\t\T^{\dia}$ of $\t\G$ and topologically semisimple elements $\eta,\eta'\in\t{T}^{\dia}$.
We investigate the relation between the ramified roots of $\Phi(\G_{\eta},\T^{\nat})$ and those of $\Phi(\G_{\eta'},\T^{\nat})$.
Note that both sets are regarded as subsets of the restricted root system $\Phi_{\res}(\G,\T^{\dia})$ as explained in Section \ref{subsec:twisted-tori}.
For convenience, let us introduce the following notation:
\[
\Phi_{\res}(\G,\T^{\dia})_{\bullet}^{(\star)}
\colonequals \{\alpha_{\res}\in\Phi_{\res}(\G,\T^{\dia})_{\bullet}\mid \alpha\in\Phi(\G,\T^{\dia})_{\star}\},
\]\symdef{Phi-res-G-S-star-bullet}{$\Phi_{\res}(\G,\bfS)_{\bullet}^{(\star)}$}
where $\bullet,\star\in\{\asym,\ur,\ram\}$.
By fixing a Borel subgroup $\bfB^{\dia}$ containing $\T^{\dia}$ and stabilized by the action of $\t\T^{\dia}$, we take an element $g^{\dia}\in\G$ satisfying $[g^{\dia}](\B^{\dia},\T^{\dia})=(\B,\T)$.
Let $\nu\theta\colonequals [g^{\dia}](\eta)$ and $\nu'\theta\colonequals [g^{\dia}](\eta')$.
We simply write $\theta_{\dia}$ for the involution $\theta_{\T^{\dia}}$ of $\T^{\dia}$ determined by its twisted structure.

\begin{lem}\label{lem:ram-res-from-asym}
For any $\alpha_{\res}\in\Phi_{\res}(\G,\T^{\dia})_{\ram}^{(\asym)}$, we have $\alpha_{\res}\in\Phi(\G_{\eta},\T^{\nat})$.
\end{lem}

\begin{proof}
For any $\alpha_{\res}\in\Phi_{\res}(\G,\T^{\dia})_{\ram}^{(\asym)}$, the following hold:
\begin{itemize}
\item
$F_{\alpha}=F_{\pm\alpha}$ and $F_{\alpha}=F_{\alpha_{\res}}$, 
\item
$F_{\alpha_{\res}}/F_{\pm\alpha_{\res}}$ and $F_{\pm\alpha}/F_{\pm\alpha_{\res}}$ are quadratic ramified, and
\item
$\theta_{\dia}(\alpha)\in-\dot{\alpha}$; let $\tau_{\alpha}\in\Gamma$ be an element satisfying $\tau_{\alpha}(\alpha)=-\theta_{\dia}(\alpha)$.
\end{itemize}
\[
\xymatrix{
F_{\alpha_{\res}}\ar@{=}[r]&F_{\alpha}\\
F_{\pm\alpha_{\res}}\ar^-{\text{quad}}_-{\ram}@{-}[r]\ar^-{\text{quad}}_-{\ram}@{-}[u]&F_{\pm\alpha}\ar@{=}[u]
}
\]
By the description explained in Section \ref{subsec:rootsys}, it suffices to show that $N(\alpha)(\nu)=N(\alpha)(\nu')$, which is equivalent to $\alpha(\eta^{2})=\alpha(\eta^{\prime2})$.
(Here, in the first equality, we regard $\alpha$ as a root of $\T$ and again write $\alpha$ for it.)
Since both $\alpha(\eta^{2})$ and $\alpha(\eta^{\prime2})$ are of finite prime-to-$p$ order, it is enough to show that $\alpha(\eta^{2})\equiv\alpha(\eta^{\prime2})\pmod{\mfp_{\ol{F}}}$.

Let $t\in T^{\dia}$ be an element satisfying $\eta'=t\eta$.
Then we have $\eta^{\prime2}=(t\eta)^{2}=t\cdot\theta_{\dia}(t)\cdot\eta^{2}$.
Since $t$ is $F$-rational and $\tau_{\alpha}(\alpha)=-\theta_{\dia}(\alpha)$, we have
\[
\alpha(t\cdot\theta_{\dia}(t))
=\alpha(t)\cdot\theta_{\dia}(\alpha)(t)
=\alpha(t)\cdot \tau_{\alpha}(\alpha(t))^{-1}.
\]
By noting that $F_{\alpha_{\res}}/F_{\pm\alpha_{\res}}$ is ramified, we have $\alpha(t)\cdot \tau_{\alpha}(\alpha(t))^{-1}\equiv1\pmod{\mfp_{\ol{F}}}$.
\end{proof}

\begin{lem}\label{lem:ram-res-from-ur}
For any $\alpha_{\res}\in\Phi_{\res}(\G,\T^{\dia})_{\ram}^{(\ur)}$, we have $\alpha_{\res}\in\Phi(\G_{\eta},\T^{\nat})$ if and only if $\alpha_{\res}\in\Phi(\G_{\eta'},\T^{\nat})$.
\end{lem}

\begin{proof}
For any $\alpha_{\res}\in\Phi_{\res}(\G,\T^{\dia})_{\ram}^{(\ur)}$, the following hold:
\begin{itemize}
\item
$F_{\alpha}/F_{\pm\alpha}$ and $F_{\alpha}/F_{\alpha_{\res}}$ are quadratic unramified, and
\item
$F_{\alpha_{\res}}/F_{\pm\alpha_{\res}}$ and $F_{\pm\alpha}/F_{\pm\alpha_{\res}}$ are quadratic ramified.
\end{itemize}
\[
\xymatrix{
F_{\alpha_{\res}}\ar^-{\text{quad}}_-{\ur}@{-}[r]&F_{\alpha}\\
F_{\pm\alpha_{\res}}\ar^-{\text{quad}}_-{\ram}@{-}[r]\ar^-{\text{quad}}_-{\ram}@{-}[u]&F_{\pm\alpha}\ar^-{\text{quad}}_-{\ur}@{-}[u]
}
\]
We let $\sigma_{\alpha},\tau_{\alpha}\in\Gamma$ be elements satisfying $\sigma_{\alpha}(\alpha)=\theta_{\dia}(\alpha)$ and $\tau_{\alpha}(\alpha)=-\alpha$, respectively.
With the same notation and arguments as in the proof of Lemma \ref{lem:ram-res-from-asym}, it suffices to show that $\alpha(t)\cdot\theta_{\dia}(\alpha)(t)\equiv1\pmod{\mfp_{\ol{F}}}$ for any $t\in T^{\dia}$.
Since $t$ is $F$-rational and $\tau_{\alpha}(\alpha)=-\alpha$, we have
\[
\Nr_{F_{\alpha}/F_{\pm\alpha}}(\alpha(t))
=\alpha(t)\cdot \tau_{\alpha}(\alpha(t))
=\alpha(t)\cdot \alpha(t)^{-1}
=1.
\]
Similarly, we have
\[
\Nr_{F_{\alpha}/F_{\alpha_{\res}}}(\alpha(t))
=\alpha(t)\cdot \sigma_{\alpha}(\alpha(t))
=\alpha(t)\cdot \theta_{\dia}(\alpha)(t).
\]
Since both $F_{\alpha}/F_{\alpha_{\res}}$ and $F_{\alpha}/F_{\pm\alpha}$ are unramified quadratic extensions, we get
\[
\alpha(t)\cdot \theta_{\dia}(\alpha)(t)
=\Nr_{F_{\alpha}/F_{\alpha_{\res}}}(\alpha(t))
\equiv\Nr_{F_{\alpha}/F_{\pm\alpha}}(\alpha(t))
=1\pmod{\mfp_{\ol{F}}}.
\]
This completes the proof.
\end{proof}

\begin{lem}\label{lem:ram-res-from-ram}
For any $\alpha_{\res}\in\Phi_{\res}(\G,\T^{\dia})_{\ram}^{(\ram)}$ with $l_{\alpha}=2$, we have $\alpha_{\res}\in\Phi(\G_{\eta},\T^{\nat})$ if and only if $\alpha_{\res}\in\Phi(\G_{\eta'},\T^{\nat})$.
\end{lem}

\begin{proof}
Suppose that $\alpha\in\Phi(\G,\T^{\dia})$ is a symmetric ramified root with $l_{\alpha}=2$.
In this case, the following hold:
\begin{itemize}
\item
$F_{\alpha}/F_{\pm\alpha}$ and $F_{\alpha_{\res}}/F_{\pm\alpha_{\res}}$ are quadratic ramified, and
\item
$F_{\alpha}/F_{\alpha_{\res}}$ and $F_{\alpha_{\res}}/F_{\pm\alpha_{\res}}$ are quadratic unramified.
\end{itemize}
\[
\xymatrix{
F_{\alpha_{\res}}\ar^-{\text{quad}}_-{\ur}@{-}[r]&F_{\alpha}\\
F_{\pm\alpha_{\res}}\ar_-{\text{quad}}^-{\ur}@{-}[r]\ar^-{\text{quad}}_{\ram}@{-}[u]&F_{\pm\alpha}\ar_-{\text{quad}}^{\ram}@{-}[u]
}
\]
(Indeed, since $\theta_{\dia}(\alpha)\neq-\alpha$, we cannot have $F_{\alpha_{\res}}=F_{\pm\alpha}$. If $F_{\pm\alpha}/F_{\pm\alpha_{\res}}$ is a quadratic ramified extension different from $F_{\alpha_{\res}}$, then $F_{\alpha}$ must contain a quadratic unramified extension of $F_{\pm\alpha_{\res}}$, hence we get a contradiction. Thus $F_{\pm\alpha}/F_{\pm\alpha_{\res}}$ must be quadratic unramified.)
Let $\tau_{\alpha}\in\Gamma$ be an element satisfying $\tau_{\alpha}(\alpha)=-\alpha$.
Then $\tau_{\alpha}$ restricts to the nontrivial element of $\Gal(F_{\alpha_{\res}}/F_{\pm\alpha_{\res}})$.
Let $\sigma_{\alpha}\in\Gamma$ be an element satisfying $\sigma_{\alpha}(\alpha)=\theta_{\dia}(\alpha)$.

With the same notation and arguments as in the proof of  Lemma \ref{lem:ram-res-from-asym}, it suffices to show that $\alpha(t)\cdot\theta_{\dia}(\alpha)(t)\equiv1\pmod{\mfp_{\ol{F}}}$.
Since $t$ is $F$-rational and $\tau_{\alpha}(\alpha)=-\alpha$, 
\[
\Nr_{F_{\alpha}/F_{\pm\alpha}}(\alpha(t))
=\alpha(t)\cdot \tau_{\alpha}(\alpha(t))
=\alpha(t)\cdot \alpha(t)^{-1}
=1.
\]
As $F_{\alpha}/F_{\pm\alpha}$ is ramified, this implies that $\alpha(t)\equiv\pm1\pmod{\mfp_{\ol{F}}}$.
Thus we get
\[
\alpha(t)\cdot \theta_{\dia}(\alpha)(t)
=\alpha(t)\cdot \sigma_{\alpha}(\alpha(t))
=\Nr_{F_{\alpha}/F_{\alpha_{\res}}}(\alpha(t))
\equiv\Nr_{F_{\alpha}/F_{\alpha_{\res}}}(\pm1)
=1\pmod{\mfp_{\ol{F}}}.
\]
This completes the proof.
\end{proof}

\subsection{Descent of toral invariants}\label{subsec:toral-descent}
Let us keep the notation as in Section \ref{subsec:ram-rest-roots}.
We next investigate the relation between the toral invariants for $(\G_{\eta},\T^{\nat})$ and those for $(\G_{\eta'},\T^{\nat})$.
Before we start our discussion, we note that the roots in the $\Theta$-orbits $\Theta\alpha$ of $\alpha\in\Phi(\G,\T)$ are orthogonal to each other and that, for any root vector $X_{\alpha}\in\bmfg_{\alpha}$, we have $\theta^{l_{\alpha}}(X_{\alpha})=X_{\alpha}$ (these are true since we are assuming that there is no restricted root of type $2$ or $3$; see \cite[(1.3.5-1.3.7)]{KS99}).

Let $t\in T^{\dia}$ be the element satisfying $\eta'=t\eta$.

\begin{prop}\label{prop:toral-descent-asym-ram}
Let $\alpha_{\res}\in\Phi_{\res}(\G,\T^{\dia})^{(\asym)}_{\ram}$.
Suppose that $\alpha_{\res}\in\Phi(\G_{\eta},\T^{\nat})$, which is equivalent to $\alpha_{\res}\in\Phi(\G_{\eta'},\T^{\nat})$ by Lemma \ref{lem:ram-res-from-asym}.
Then we have
\[
f_{(\G_{\eta},\T^{\nat})}(\alpha_{\res})
=
f_{(\G_{\eta'},\T^{\nat})}(\alpha_{\res})\cdot\epsilon_{\alpha}(t).
\]
\end{prop}

\begin{proof}
We use the notation as in the proof of Lemma \ref{lem:ram-res-from-asym}.
We take an element $X_{\alpha}$ of $\bmfg_{\alpha}(F_{\alpha})$.
Then $X_{\alpha_{\res},\eta}\colonequals X_{\alpha}+[\eta](X_{\alpha})$ belongs to $\bmfg_{\eta,\alpha_{\res}}(F_{\alpha})$.
As $F_{\alpha_{\res}}=F_{\alpha}$, we get $X_{\alpha_{\res},\eta}\in\bmfg_{\eta,\alpha_{\res}}(F_{\alpha_{\res}})$.
Thus, by the definition of the toral invariant (see Section \ref{subsec:invariants-for-LLC}), 
\begin{align*}
f_{(\G_{\eta},\T^{\nat})}(\alpha_{\res})
&=
\kappa_{\alpha_{\res}}\biggl(\frac{[X_{\alpha_{\res},\eta},\tau_{\alpha}(X_{\alpha_{\res},\eta})]}{H_{\alpha_{\res}}}\biggr)\\
&=
\kappa_{\alpha_{\res}}\biggl(\frac{[X_{\alpha},[\eta](\tau_{\alpha}(X_{\alpha}))]+[[\eta](X_{\alpha}),\tau_{\alpha}(X_{\alpha})]}{H_{\alpha_{\res}}}\biggr)
\end{align*}
(note that $\tau_{\alpha}(\alpha)=-\theta_{\dia}(\alpha)$).
Note that $H_{\alpha_{\res}}=H_{\alpha}+[\eta](H_{\alpha})$; here, $[\eta](H_{\alpha})$ is independent of the choice of $\eta$ and linearly independent with $H_{\alpha}$.
Since $[X_{\alpha},[\eta](\tau_{\alpha}(X_{\alpha}))]$ and $[[\eta](X_{\alpha}),\tau_{\alpha}(X_{\alpha})]$ are proportional to $H_{\alpha}$ and $[\eta](H_{\alpha})$, respectively, we have
\[
f_{(\G_{\eta}, \T^{\nat})}(\alpha_\res)
=
\kappa_{\alpha_{\res}}\biggl(\frac{[X_{\alpha},[\eta](\tau_{\alpha}(X_{\alpha}))]}{H_{\alpha}}\biggr).
\]
For the same reasoning, we get
\[
f_{(\G_{\eta'},\T^{\nat})}(\alpha_{\res})
=
\kappa_{\alpha_{\res}}\biggl(\frac{[X_{\alpha},[\eta'](\tau_{\alpha}(X_{\alpha}))]}{H_{\alpha}}\biggr).
\]
Hence, we get
\[
f_{(\G_{\eta'},\T^{\nat})}(\alpha_{\res})
=
f_{(\G_{\eta},\T^{\nat})}(\alpha_{\res})\cdot\kappa_{\alpha_{\res}}(\tau_{\alpha}(\alpha(t))).
\]
Since $\kappa_{\alpha_{\res}}$ is the quadratic character of $F_{\pm\alpha_{\res}}$ corresponding to the extension $F_{\alpha_{\res}}/F_{\pm\alpha_{\res}}$, by noting that $\alpha(t)$ belongs to $\mcO_{F_{\pm\alpha_{\res}}}^{\times}$, we conclude that
\[
\kappa_{\alpha_{\res}}(\tau_{\alpha}(\alpha(t)))
=\sgn_{k_{\alpha}^{\times}}(\ol{\alpha(t)})
=\epsilon_{\alpha}(t).
\]
This completes the proof.
\end{proof}

\begin{prop}\label{prop:toral-descent-ur-ram}
Let $\alpha_{\res}\in\Phi_{\res}(\G,\T^{\dia})^{(\ur)}_{\ram}$.
Suppose that $\alpha_{\res}\in\Phi(\G_{\eta},\T^{\nat})$, which is equivalent to $\alpha_{\res}\in\Phi(\G_{\eta'},\T^{\nat})$ by Lemma \ref{lem:ram-res-from-ur}.
Then we have
\[
f_{(\G_{\eta},\T^{\nat})}(\alpha_{\res})
=
f_{(\G_{\eta'},\T^{\nat})}(\alpha_{\res})\cdot\epsilon_{\alpha}(t).
\]
\end{prop}

\begin{proof}
We use the notation as in the proof of Lemma \ref{lem:ram-res-from-ur}.
We take an element $X_{\alpha}$ of $\bmfg_{\alpha}(F_{\alpha})$.
Then $X_{\alpha_{\res},\eta}\colonequals X_{\alpha}+[\eta](X_{\alpha})$ belongs to $\bmfg_{\eta,\alpha_{\res}}(F_{\alpha})$.
To compute the toral invariant $f_{(\G_{\eta},\T^{\nat})}(\alpha_{\res})$, we scale $X_{\alpha_{\res},\eta}$ so that it is $F_{\alpha_{\res}}$-rational.
Let $C_{\eta}\in F_{\alpha}^{\times}$ be the constant determined by $\sigma_{\alpha}(X_{\alpha})=C_{\eta}\cdot [\eta](X_{\alpha})$.
Note that then $\sigma_{\alpha}([\eta](X_{\alpha}))=C_{\eta}\cdot X_{\alpha}$.
Indeed, since $\eta$ is $F$-rational, we have
\[
\sigma_{\alpha}([\eta](X_{\alpha}))
=[\eta](\sigma_{\alpha}(X_{\alpha}))
=[\eta](C_{\eta}\cdot[\eta](X_{\alpha}))
=C_{\eta}\cdot\alpha(\eta^{2})\cdot X_{\alpha}
=C_{\eta}\cdot X_{\alpha},
\]
where $\alpha(\eta^{2})=1$ since $\alpha_\res\in\Phi(\G_{\eta},\T^{\nat})$.
Thus we have
\[
\sigma_{\alpha}^{2}(X_{\alpha})
=\sigma_{\alpha}(C_{\eta}\cdot [\eta](X_{\alpha}))
=\sigma_{\alpha}(C_{\eta})\cdot C_{\eta}\cdot X_{\alpha}
=\Nr_{F_{\alpha}/F_{\alpha_{\res}}}(C_{\eta})\cdot X_{\alpha}.
\]
On the other hand, since $\sigma_{\alpha}^{2}$ belongs to $\Gamma_{\alpha}$, $\sigma_{\alpha}^{2}$ fixes $X_{\alpha}$.
This implies that $\Nr_{F_{\alpha}/F_{\alpha_{\res}}}(C_{\eta})=1$.
Hence, by the Hilbert 90th theorem, we can find an element $x_{\eta}\in F_{\alpha}^{\times}$ satisfying $C_{\eta}=x_{\eta}/\sigma_{\alpha}(x_{\eta})$.
By putting $\t{X}_{\alpha_{\res},\eta}\colonequals x_{\eta}\cdot X_{\alpha_{\res},\eta}$, we get an $F_{\alpha_{\res}}$-rational root vector $\t{X}_{\alpha_{\res},\eta}\in\bmfg_{\eta,\alpha_{\res}}(F_{\alpha_{\res}})$.

Now let us compute the toral invariant using $\t{X}_{\alpha_{\res},\eta}$:
\begin{align*}
f_{(\G_{\eta},\T^{\nat})}(\alpha_{\res})
&=
\kappa_{\alpha_{\res}}\biggl(\frac{[\t{X}_{\alpha_{\res},\eta},\tau_{\alpha}(\t{X}_{\alpha_{\res},\eta})]}{H_{\alpha_{\res}}}\biggr)\\
&=
\kappa_{\alpha_{\res}}\biggl(\frac{[x_{\eta}X_{\alpha},\tau_{\alpha}(x_{\eta}X_{\alpha})]+[x_{\eta}[\eta](X_{\alpha}),\tau_{\alpha}(x_{\eta}[\eta](X_{\alpha}))]}{H_{\alpha_{\res}}}\biggr).
\end{align*}
By the same argument as in the proof of Proposition \ref{prop:toral-descent-asym-ram}, this equals
\[
\kappa_{\alpha_{\res}}\biggl(\frac{[x_{\eta}X_{\alpha},\tau_{\alpha}(x_{\eta}X_{\alpha})]}{H_{\alpha}}\biggr).
\]
Similarly, putting $C_{\eta'}\in F_{\alpha}^{\times}$ and $x_{\eta'}\in F_{\alpha}^{\times}$ in the same manner, 
\[
f_{(\G_{\eta'},\T^{\nat})}(\alpha_{\res})
=
\kappa_{\alpha_{\res}}\biggl(\frac{[x_{\eta'}X_{\alpha},\tau_{\alpha}(x_{\eta'}X_{\alpha})]}{H_{\alpha}}\biggr).
\]

Thus we get
\[
f_{(\G_{\eta'},\T^{\nat})}(\alpha_{\res})
=
f_{(\G_{\eta},\T^{\nat})}(\alpha_{\res})
\cdot
\kappa_{\alpha_{\res}}\bigl((x_{\eta'}x_{\eta}^{-1})\cdot\tau_{\alpha}(x_{\eta'}x_{\eta}^{-1})\bigr).
\]
We put $y\colonequals x_{\eta'}x_{\eta}^{-1}\in F_{\alpha}^{\times}$, hence $\kappa_{\alpha_{\res}}((x_{\eta'}x_{\eta}^{-1})\cdot\tau_{\alpha}(x_{\eta'}x_{\eta}^{-1}))=\kappa_{\alpha_{\res}}(y\cdot\tau_{\alpha}(y))$.
As $C_{\eta}\cdot [\eta](X_{\alpha})=\sigma_{\alpha}(X_{\alpha})=C_{\eta'}\cdot [\eta'](X_{\alpha})$, we have $C_{\eta'}=\alpha(t)^{-1}C_{\eta}$.
Hence we have $y/\sigma_{\alpha}(y)=\alpha(t)^{-1}$.
Since $F_{\alpha}/F_{\alpha_{\res}}$ is unramified, we can choose $x_{\eta}$ and $x_{\eta'}$ to be elements of $\mcO_{F_{\alpha}}^{\times}$, hence $y\in\mcO_{F_{\alpha}}^{\times}$.
We note that the composition
\[
k_{\alpha}^{1}\xleftarrow{\cong}k_{\alpha}^{\times}/k_{\pm\alpha}^{\times}\xrightarrow{\Nr_{k_{\alpha}/k_{\pm\alpha}}} k_{\pm\alpha}^{\times}/k_{\pm\alpha}^{\times2}\cong\mu_{2}\colon\quad
y/\sigma_{\alpha}(y)\mapsto y\mapsto y\cdot\tau_{\alpha}(y)
\]
defines the unique nontrivial quadratic character of $k_{\alpha}^{1}$.
Hence we get
\[
\kappa_{\alpha_{\res}}(y\cdot\tau_{\alpha}(y))
=\sgn_{k_{\alpha}^{1}}(\ol{\alpha(t)})^{-1}
=\epsilon_{\alpha}(t).
\]
This completes the proof.
\end{proof}

\begin{prop}\label{prop:toral-descent-ram-ram}
Let $\alpha_{\res}\in\Phi_{\res}(\G,\T^{\dia})^{(\ram)}_{\ram}$.
Suppose that $\alpha_{\res}\in\Phi(\G_{\eta},\T^{\nat})$ and $\alpha_{\res}\in\Phi(\G_{\eta'},\T^{\nat})$.
Then we have
\[
f_{(\G_{\eta},\T^{\nat})}(\alpha_{\res})
=
\begin{cases}
f_{(\G_{\eta'},\T^{\nat})}(\alpha_{\res}) & \text{if $F_{\alpha}=F_{\alpha_{\res}}$,}\\
f_{(\G_{\eta'},\T^{\nat})}(\alpha_{\res})\cdot\alpha(t)& \text{if $F_{\alpha}\neq F_{\alpha_{\res}}$.}
\end{cases}
\]
Here, in the latter case, we have $\alpha(t)=\pm1$, hence regard $\alpha(t)\in\{\pm1\}\subset\C^{\times}$.
\end{prop}

\begin{proof}
We first consider the case where $F_{\alpha}=F_{\alpha_{\res}}$.
In this case, 
\begin{itemize}
\item
$F_{\alpha}/F_{\pm\alpha}$ and $F_{\alpha_{\res}}/F_{\pm\alpha_{\res}}$ are quadratic ramified, and
\item
$F_{\pm\alpha}=F_{\pm\alpha_{\res}}$.
\end{itemize}
\[
\xymatrix{
F_{\alpha_{\res}}\ar@{=}[r]&F_{\alpha}\\
F_{\pm\alpha_{\res}}\ar@{=}[r]\ar^-{\text{quad}}_-{\ram}@{-}[u]&F_{\pm\alpha}\ar_-{\text{quad}}^-{\ram}@{-}[u]
}
\]
Let $\tau_{\alpha}\in\Gamma$ be an element satisfying $\tau_{\alpha}(\alpha)=-\alpha$.
Then $\tau_{\alpha}$ restricts to the nontrivial element of $\Gal(F_{\alpha_{\res}}/F_{\pm\alpha_{\res}})$.

When $\theta_{\dia}(\alpha)=\alpha$, we get $f_{(\G_{\eta},\T^{\nat})}(\alpha_{\res})=f_{(\G,\T^{\dia})}(\alpha)$ simply because we can use the same root vector $X_{\alpha}$ both for computing $f_{(\G_{\eta},\T^{\nat})}(\alpha_{\res})$ and $f_{(\G,\T^{\dia})}(\alpha)$.
When $\theta_{\dia}(\alpha)\neq\alpha$, by taking $X_{\alpha}$ and $X_{\alpha_{\res},\eta}$ in the same way and applying the same argument as in the case where $\alpha_{\res}\in\Phi_{\res}(\G,\bfS)^{(\ur)}_{\ram}$, we have
\begin{align*}
f_{(\G_{\eta},\T^{\nat})}(\alpha_{\res})
&=
\kappa_{\alpha_{\res}}\biggl(\frac{[X_{\alpha_{\res},\eta},\tau_{\alpha}(X_{\alpha_{\res},\eta})]}{H_{\alpha_{\res}}}\biggr)\\
&=
\kappa_{\alpha_{\res}}\biggl(\frac{[X_{\alpha},\tau_{\alpha}(X_{\alpha})]+[[\eta](X_{\alpha}),\tau_{\alpha}([\eta](X_{\alpha}))]}{H_{\alpha_{\res}}}\biggr)\\
&=
\kappa_{\alpha}\biggl(\frac{[X_{\alpha},\tau_{\alpha}(X_{\alpha})]}{H_{\alpha}}\biggr)
=f_{(\G,\T^{\dia})}(\alpha).
\end{align*}
Hence, eventually, we get $f_{(\G_{\eta},\T^{\nat})}(\alpha_{\res})=f_{(\G,\T^{\dia})}(\alpha)$ regardless of whether $\theta_{\dia}(\alpha)=\alpha$ or not.
Since $f_{(\G_{\eta'},\T^{\nat})}(\alpha_{\res})=f_{(\G,\T^{\dia})}(\alpha)$ for the same reason, we get $f_{(\G_{\eta},\T^{\nat})}(\alpha_{\res})=f_{(\G_{\eta'},\T^{\nat})}(\alpha_{\res})$.

We next consider the case where $F_{\alpha}/F_{\alpha_{\res}}$ is quadratic.
We use the same notation as in the proof of Lemma \ref{lem:ram-res-from-ram}.
By the same usage of notation and arguments as in the case where $\alpha\in\Phi_{\res}(\G,\bfS)^{(\ur)}_{\ram}$, we get
\[
f_{(\G_{\eta'},\T^{\nat})}(\alpha_{\res})
=
f_{(\G_{\eta},\T^{\nat})}(\alpha_{\res})
\cdot
\kappa_{\alpha_{\res}}(y\cdot\tau_{\alpha}(y)).
\]
Recall that $y\in\mcO_{F_{\alpha}^{\times}}$ is an element such that $y/\sigma_{\alpha}(y)=\alpha(t)^{-1}$.
In the present case (where $F_{\alpha}/F_{\pm\alpha}$ is ramified), since $\alpha(t)\in\Ker(\Nr_{F_{\alpha}/F_{\pm\alpha}})$, we have $\alpha(t)\equiv\pm1\pmod{\mfp_{\ol{F}}}$.
Furthermore, as $\alpha(t)^2=\alpha(\eta^{\prime2})\cdot \alpha(\eta^{-2})$, $\alpha(t)$ is a root of unity of order prime to $p$.
Thus we necessarily have $\alpha(t)=\pm1$. 
Then we can check that $\kappa_{\alpha_{\res}}(y\cdot\tau_{\alpha}(y))$ equals $+1$ or $-1$ according to $\alpha(t)=+1$ or $\alpha(t)=-1$.
Hence we get
$
f_{(\G_{\eta'},\T^{\nat})}(\alpha_{\res})
=
f_{(\G_{\eta},\T^{\nat})}(\alpha_{\res})
\cdot
\alpha(t).
$
\end{proof}

\section{Some computation of transfer factors}\label{sec:transfer-factor}

In this section, we establish some formulas on transfer factors which will be needed in our proof of the twisted endoscopic character relation.

Let $(\gamma,\delta)\in\mcD$ and we fix $D=(\B^{\flat},\T^{\flat},\B^{\dia},\T^{\dia})\in\bfD(\gamma,\delta)$, which is unique up to equivalence by Lemma \ref{lem:unique-diag}.
Note that, in particular, we have $\gamma\in T^{\flat}$ and $\delta\in\t{T}^{\dia}$.
We also fix a set $a=\{a_{\alpha_{\res}}\}_{\alpha_{\res}}$ of $a$-data and a set $\chi=\{\chi_{\alpha_{\res}}\}_{\alpha_{\res}}$ of minimally ramified $\chi$-data for $\Phi_{\res}(\G,\T^{\dia})$ (in the sense of Definition \ref{defn:min-ram}).

Let us suppose that $\delta\in\t{T}^{\dia}$ is elliptic strongly regular semisimple with a normal $r$-approximation $\delta=\delta_{0}\delta^{+}_{<r}\delta_{\geq r}$ (recall Definition \ref{defn:twisted-approx}).
Then, by using the maps $\t{\xi}_{D}$ and $\xi_{D}$, we can associate a decomposition $\gamma=\gamma_{0}\gamma^{+}_{<r}\gamma_{\geq r}$ to $\gamma$ by transferring the decomposition $\delta=\delta_{0}\delta^{+}_{<r}\delta_{\geq r}$.

\begin{lem}\label{lem:normal-approx-descent}
The decomposition $\gamma=\gamma_{0}\gamma^{+}_{<r}\gamma_{\geq r}$ gives a normal $r$-approximation.
\end{lem}

\begin{proof}
This follows from that $\Phi(\H,\T^{\flat})$ is identified with a subset of $\Phi_{\res}(\G,\T^{\dia})$ (Section \ref{subsec:rootsys}) and that $p\neq2$.
\end{proof}

We put $\nu_{0}\colonequals \t\xi_{\dia}(\delta_{0})\cdot \theta^{-1}$ and $\nu_{+}\colonequals \xi_{\dia}(\delta_{+})$.
We also put $\nu^{+}_{<r}\colonequals \xi_{\dia}(\delta^{+}_{<r})$ and $\nu_{\geq r}\colonequals \xi_{\dia}(\delta_{\geq r})$.
Thus we have $\nu\theta=(\nu_{0}\theta)\cdot\nu^{+}_{<r}\cdot\nu_{\geq r}$.

\subsection{First factor $\Delta_{\I}$}

\begin{lem}\label{lem:tran-I}
For any $(\b\gamma,\b\delta)\in\mcD$ satisfying $\b\gamma\in T^{\flat}$ and $\b\delta\in\t{T}^{\dia}$, we have
\[
\Delta_{\I}[a,\chi](\gamma,\delta)
=
\Delta_{\I}[a,\chi](\b\gamma,\b\delta).
\]
\end{lem}

\begin{proof}
By definition (see \cite[Section 4.2]{KS99}), the first factor $\Delta_{\I}(\bar\gamma,\bar\delta)$ depends only on the $F$-rational (twisted) maximal tori of $\H$ and $\t\G$ containing $\bar\gamma$ and $\bar\delta$, respectively.
Thus we get the assertion.
\end{proof}

\subsection{Second factor $\Delta_{\II}$}
For any element $\delta'\in\t{T}^{\dia}$, we put
\[
\Delta_{\II}^{\t\G}[a,\chi](\delta')
\colonequals \prod_{\begin{subarray}{c}\alpha_{\res}\in \dot{\Phi}_{\res}(\G,\T^{\dia})\\ N(\alpha)(\nu')\neq1\end{subarray}}\chi_{\alpha_{\res}}\biggl(\frac{N(\alpha)(\nu')-1}{a_{\alpha_{\res}}}\biggr),
\]
where $\nu'\in\T$ is the element such that $\t{\xi}_{\dia}(\delta')=\nu'\theta$.

We put 
\[
\chi(\delta_{0})\colonequals \prod_{\alpha_{\res}\in\dot{\Phi}(\G_{\delta_{0}},\T^{\nat})}
\chi_{\alpha_{\res}}(l_{\alpha}).
\]\symdef{chi-delta-zero}{$\chi(\delta_0)$}
The following is the twisted version of \cite[Lemma 4.6.7]{Kal19}.

\begin{lem}\label{lem:tran-II-descent}
We have
\[
\Delta_{\II}^{\t\G}[a,\chi](\delta)
=\Delta_{\II}^{\t\G}[a,\chi](\delta_{0})\cdot\Delta_{\II}^{\G_{\delta_{0}}}[a,\chi](\delta_{+})\cdot\chi(\delta_{0}).
\]
\end{lem}

\begin{proof}
By definition, we have
\[
\Delta_{\II}^{\t\G}[a,\chi](\delta_{0})
\colonequals \prod_{\begin{subarray}{c}\alpha_{\res}\in \dot{\Phi}_{\res}(\G,\T^{\dia}) \\ N(\alpha)(\nu_{0})\neq1\end{subarray}}\chi_{\alpha_{\res}}\biggl(\frac{N(\alpha)(\nu_{0})-1}{a_{\alpha_{\res}}}\biggr),
\]
and
\[
\Delta_{\II}^{\G_{\delta_{0}}}[a,\chi](\delta_{+})
=\prod_{\begin{subarray}{c}\alpha_{\res}\in \dot{\Phi}(\G_{\delta_{0}},\T^{\nat}) \\ \alpha_{\res}(\nu_{+})\neq1\end{subarray}}\chi_{\alpha_{\res}}\biggl(\frac{\alpha_{\res}(\nu_{+})-1}{a_{\alpha_{\res}}}\biggr).
\]

Let $\alpha_{\res}\in \Phi_{\res}(\G,\T^{\dia})$.
Note that we have
\[
N(\alpha)(\nu_{0})^{\frac{2}{l_{\alpha}}}
=\biggl(\prod_{i=0}^{l_{\alpha}-1}\theta^{i}(\alpha)(\nu_{0})\biggr)^{\frac{2}{l_{\alpha}}}
=\prod_{i=0}^{1}\theta^{i}(\alpha)(\nu_{0})
=\alpha\bigl((\nu_{0}\theta)^{2}\bigr).
\]
As $\nu_{0}\theta$ is topologically semisimple and $p\neq2$, we see that $N(\alpha)(\nu_{0})$ is a root of unity of prime-to-$p$-order in $\ol{F}$.
Since we have
\[
N(\alpha)(\nu)
=N(\alpha)(\nu_{0})\cdot N(\alpha)(\nu_{+})
=N(\alpha)(\nu_{0})\cdot \alpha(\nu_{+})^{l_{\alpha}},
\]
and $\nu_{+}$ is pro-unipotent, $N(\alpha)(\nu)\neq1$ holds if and only if exactly one of the following holds:
\begin{itemize}
\item
$N(\alpha)(\nu_{0})\neq1$, or
\item
$N(\alpha)(\nu_{0})=1$ and $\alpha(\nu_{+})^{l_{\alpha}}\neq1$ (the latter condition is furthermore equivalent to $\alpha(\nu_{+})\neq1$ as $l_{\alpha}$ is prime to $p$).
\end{itemize}

When $N(\alpha)(\nu_{0})\neq1$, by noting that $\chi_{\alpha_{\res}}$ is tamely ramified, we get
\[
\chi_{\alpha_{\res}}\biggl(\frac{N(\alpha)(\nu)-1}{a_{\alpha_{\res}}}\biggr)
=
\chi_{\alpha_{\res}}\biggl(\frac{N(\alpha)(\nu_{0})-1}{a_{\alpha_{\res}}}\biggr).
\]
When $N(\alpha)(\nu_{0})=1$ and $\alpha(\nu_{+})\neq1$, we have
\[
\chi_{\alpha_{\res}}\biggl(\frac{N(\alpha)(\nu)-1}{a_{\alpha_{\res}}}\biggr)
=
\chi_{\alpha_{\res}}\biggl(\frac{N(\alpha)(\nu_{+})-1}{a_{\alpha_{\res}}}\biggr).
\]
As we have $N(\alpha)(\nu_{+})=\alpha(\nu_{+})^{l_{\alpha}}$ and $l_{\alpha}$ is prime to $p$, we have
\begin{align*}
N(\alpha)(\nu_{+})-1
=\bigl(\alpha(\nu_{+})-1\bigr)\bigl(\alpha(\nu_{+})^{l_{\alpha}-1}+\cdots+1\bigr)
\in \bigl(\alpha(\nu_{+})-1\bigr)\cdot l_{\alpha}\cdot (1+\mfp_{F_{\alpha}}).
\end{align*}
Hence again the tamely-ramifiedness of $\chi_{\alpha_{\res}}$ implies that
\[
\chi_{\alpha_{\res}}\biggl(\frac{N(\alpha)(\nu_{+})-1}{a_{\alpha_{\res}}}\biggr)
=
\chi_{\alpha_{\res}}\biggl(\frac{\alpha(\nu_{+})-1}{a_{\alpha_{\res}}}\biggr)\cdot\chi_{\alpha_{\res}}(l_{\alpha}).
\]
Recall that $\Phi(\G_{\delta_{0}},\T^{\nat})$ is identified with the subset of $\Phi_{\res}(\G,\T^{\dia})$ consisting of restricted roots $\alpha_{\res}$ satisfying $N(\alpha)(\nu_{0})=1$ (see Section \ref{subsec:twisted-tori}).
Thus, by noting that $\delta$ is strongly regular semisimple, hence no root $\alpha_{\res}\in \Phi(\G_{\delta_{0}},\T^{\nat})$ satisfies $\alpha_{\res}(\nu_{+})=1$, we get the assertion.
\end{proof}

The following will be a crucially important ingredient in our proof of the twisted endoscopic character relation.

\begin{lem}\label{lem:chi-constant}
The constant $\chi(\delta_{0})$ does not depend on $\delta_{0}\in\t{T}^{\dia}$.
\end{lem}

\begin{proof}
Recall that $\chi(\delta_{0})$ is defined to be the product of $\chi_{\alpha_{\res}}(l_{\alpha})$ over the set $\alpha_{\res}\in\dot{\Phi}(\G_{\delta_{0}},\T^{\nat})$.
Since $\chi$ is minimally ramified, $\chi_{\alpha_{\res}}(l_{\alpha})$ can be nontrivial only when $\alpha_{\res}$ is ramified and $l_{\alpha}\neq1$.
However, the set of such restricted roots does not depend on $\delta_{0}$ by Lemmas \ref{lem:ram-res-from-asym}, \ref{lem:ram-res-from-ur}, and \ref{lem:ram-res-from-ram}.
\end{proof}

\begin{lem}\label{lem:tran-II}
For any sufficiently large positive integer $m\in\Z_{>0}$, we have
\[
\Delta_{\II}[a,\chi](\gamma_{<r}\cdot\gamma^{p^{2m}}_{\geq r}, \delta_{<r}\cdot\delta^{p^{2m}}_{\geq r})
=
\Delta_{\II}[a,\chi](\gamma_{<r}\cdot\gamma_{\geq r}, \delta_{<r}\cdot\delta_{\geq r}).
\]
\end{lem}

\begin{proof}

By Lemma \ref{lem:tran-II-descent}, we have
\begin{align*}
\Delta_{\II}^{\t\G}[a,\chi](\delta_{<r}\cdot\delta_{\geq r})
&=
\Delta_{\II}^{\t\G}[a,\chi](\delta_{0})\cdot\Delta_{\II}^{\G_{\delta_{0}}}[a,\chi](\delta^{+}_{<r}\cdot\delta_{\geq r})\cdot\chi(\delta_{0}),\\
\Delta_{\II}^{\t\G}[a,\chi](\delta_{<r}\cdot\delta_{\geq r}^{p^{2m}})
&=
\Delta_{\II}^{\t\G}[a,\chi](\delta_{0})\cdot\Delta_{\II}^{\G_{\delta_{0}}}[a,\chi](\delta^{+}_{<r}\cdot\delta_{\geq r}^{p^{2m}})\cdot\chi(\delta_{0}).
\end{align*}
According to the proof of \cite[Lemma 6.3.3]{Kal19}, we have
\[
\Delta_{\II}^{\G_{\delta_{0}}}[a,\chi](\delta^{+}_{<r}\cdot\delta_{\geq r})
=
\Delta_{\II}^{\G_{\delta_{0}}}[a,\chi](\delta^{+}_{<r}\cdot\delta_{\geq r}^{p^{2m}})
\]
for any sufficiently large positive integer $m\in\Z_{>0}$, hence get $\Delta_{\II}^{\t\G}[a,\chi](\delta_{<r}\cdot\delta_{\geq r})=\Delta_{\II}^{\t\G}[a,\chi](\delta_{<r}\cdot\delta_{\geq r}^{p^{2m}})$.
Similarly, we have $\Delta_{\II}^{\H}[a,\chi](\gamma_{<r}\cdot\gamma_{\geq r})=
\Delta_{\II}^{\H}[a,\chi](\gamma_{<r}\cdot\gamma_{\geq r}^{p^{2m}})$.
Since the second factor $\Delta_{\II}[a,\chi](\gamma,\delta)$ is defined to be the ratio of $\Delta_{\II}^{\t\G}[a,\chi](\delta)$ to $\Delta_{\II}^{\H}[a,\chi](\gamma)$ (see \cite[Section 4.3]{KS99}), we get the assertion.
\end{proof}

\subsection{Third factor $\Delta_{\III}$}\label{subsec:Delta-III}

Since we assume that $\G$ is quasi-split and also fix a $\theta$-stable $F$-splitting of $\G$, we have the absolute third factor $\Delta_{\III}[a,\chi](\gamma,\delta)$ (\cite[Section 5.3]{KS99}), which satisfies
\[
\Delta_{\III}[a,\chi](\gamma,\delta; \bar{\gamma},\bar{\delta})
=
\Delta_{\III}[a,\chi](\gamma,\delta)/\Delta_{\III}[a,\chi](\bar{\gamma},\bar{\delta})
\]
for any $(\bar{\gamma},\bar{\delta})\in\mcD$.
We review the construction of $\Delta_{\III}[a,\chi](\gamma,\delta; \bar{\gamma},\bar{\delta})$ following \cite[Section 4.4]{KS99}.
We fix a diagram $\bar{D}=(\bar{\B}^{\flat},\bar{\T}^{\flat},\bar{\B}^{\dia},\bar{\T}^{\dia})\in\bfD(\bar{\gamma},\bar{\delta})$.

The relative third factor is given by using the following Take--Nakayama pairing for hyper-cohomology of $F$-rational tori $\bfU_{0}$ and $\bfS_{0}$ (see \cite[Appendix A]{KS99}):
\begin{align}\label{eq:TN}
\langle-,-\rangle_{\TN}\colon H^{1}(F,\bfU_{0}\xrightarrow{1-\theta}\bfS_{0})\times H^{1}(W_{F},\h\bfS_{0}\xrightarrow{1-\h\theta}\h\bfU_{0})\rightarrow\C^{\times}.
\end{align}
Let us recall the definitions of the tori $\bfU_{0}$ and $\bfS_{0}$.

We begin with the following lemma, which is a reformulation of \cite[Lemma 3.3.B]{KS99}:
\begin{lem}\label{lem:T-0}
There exists a $\theta$-stable $F$-rational maximal torus $\T_{0}$ and a $\theta$-stable Borel subgroup $\bfB_{0}$ containing $\T_{0}$ such that the isomorphism $\T^{\dia}\rightarrow\T_{0}$ determined by the Borel pairs $(\B^{\dia},\T^{\dia})$ and $(\bfB_{0},\T_{0})$ is $F$-rational.
\end{lem}

In the following, we fix $(\bfB_{0},\T_{0})$ and $(\b\bfB_{0},\b\T_{0})$ as in this lemma for $(\bfB^{\dia},\T^{\dia})$ and $(\b{\bfB}^{\dia},\b{\T}^{\dia})$.
We then get canonical isomorphisms $\T_{0}\cong\T$ and $\bar{\T}_0\cong\T$ (induced by the Borel pairs), which induce $\h\T_{0}\cong\h\T$ and $\h{\b{\T}}_{0}\cong\h\T$ (not necessarily $\Gamma$-equivariant).
We identify $\h{\T}_{0}$ and $\h{\b{\T}}_{0}$ with $\h\T$ via these isomorphisms but keep using the symbols $\h{\T}_{0}$ and $\h{\b{\T}}_{0}$ in order to emphasize that their Galois actions are not the one coming from the $\Gamma$-action on $\h\T\subset\h\G$.
We take $g_{1}\in\G_{\sc}$ such that $[g_{1}](\bfB^{\dia},\T^{\dia})=(\bfB_{0},\T_{0})$, i.e., $[g_{1}]\colon \T^{\dia}\rightarrow\T_{0}$ is the $F$-rational isomorphism as in Lemma \ref{lem:T-0}.
Similarly, also for $(\bar{\bfB}^{\dia},\bar{\T}^{\dia})$, we take $\bar{g}_{1}\in\G_{\sc}$ realizing the $F$-rational isomorphism $[\bar{g}_{1}]\colon \bar{\T}^{\dia}\rightarrow\bar{\T}_{0}$.

In the following, the subscript ``$\sc$'' denotes the preimage in the simply-connected cover $\G_{\sc}$ of the derived group of $\G$ and the subscript ``$\ad$'' denotes the image in the adjoint group $\G_{\ad}$ of $\G$.
We use a similar notation also for $\h\G$.
We define $F$-rational tori $\bfS_{0}$ and $\bfU_{0}$ by
\[
\bfS_{0}:=(\T_{0}\times\b\T_{0})/\Delta_{-}\bfZ_{\G},
\]
\[
\bfU_{0}:=(\T_{0,\sc}\times\b\T_{0,\sc})/\Delta_{-}\bfZ_{\G_{\sc}}.
\]
Here, $\Delta_{-}\bfZ_{\G_{(\sc)}}\colonequals \{(z,z^{-1})\mid z\in\bfZ_{\G_{(\sc)}}\}$.
Thus the dual tori are given by
\[
\h\bfS_{0}\cong \h\T_{0}\times\h{\b{\T}}_{0,\sc}
\quad\text{and}\quad
\h\bfU_{0}\cong \h\T_{0,\ad}\times\h{\b{\T}}_{0,\sc}.
\]

We consider the homomorphism $1-\theta\colon\T_{0}\rightarrow\T_{0}\colon t\mapsto t/\theta(t)$ and also its lift to $\T_{0,\sc}$.
We define homomorphisms of $\bfS_{0}$ and $\bfU_{0}$, for which we again write $1-\theta$, to be the one induced by $(t,\b{t})\mapsto ((1-\theta)(t),(1-\theta)(\b{t}))$.
Then we get homomorphisms $1-\h\theta$ on $\h\bfS_{0}$ and $\h\bfU_{0}$.
Accordingly, the hyper-cohomology and the Tate--Nakayama pairing as in \eqref{eq:TN} makes sense.
The relative third factor is defined by
\[
\Delta_{\III}[a,\chi](\gamma,\delta; \bar{\gamma},\bar{\delta})
\colonequals 
\langle \inv(\gamma,\delta; \bar{\gamma},\bar{\delta}), \bfA\rangle_{\TN}^{-1}.
\]
(Note that the right-hand side is inverted according to \cite{KS12}.)
Thus let us next explain the constructions of $\inv(\gamma,\delta; \bar{\gamma},\bar{\delta})$ and $\bfA$.

We first consider $\inv(\gamma,\delta; \bar{\gamma},\bar{\delta})$.
By putting $v_{\sigma}\colonequals g_{1}\cdot\sigma(g_{1})^{-1}$ and $\b{v}_{\sigma}\colonequals \b{g}_{1}\cdot\sigma(\b{g}_{1})^{-1}$, we get a $1$-cocycle $V\colon \Gamma\rightarrow \bfS_{0}$ which maps $\sigma$ to the image of $(v_{\sigma}^{-1},\b{v}_{\sigma})\in\T_{0}\times\b\T_{0}$ in $\bfS_{0}$.
On the other hand, we put $\delta_{0}\rtimes\theta\colonequals [g_{1}](\delta)$ and $\b\delta_{0}\rtimes\theta\colonequals [g_{1}](\b\delta)$ (thus $\delta_{0},\b\delta_{0}\in\T_{0}$).
We define an element $D\in\bfS_{0}$ to be the image of $(\delta_{0},\b{\delta}_{0}^{-1})\in\T_{0}\times\b\T_{0}$ in $\bfS_{0}$.
Then $(V,D)$ forms a $1$-hyper-cocycle.
We let $\inv(\gamma,\delta; \bar{\gamma},\bar{\delta})$ be the hyper-cohomology class of $(V,D)$.

We next consider $\bfA$.
We introduce two kinds of $L$-embeddings $\Lj_{\chi}^{1}$\symdef{L-j-1-chi}{$\Lj_{\chi}^{1}$} and $\Lj_{\chi}^{\H}$\symdef{L-j-H-chi}{$\Lj_{\chi}^{\H}$}.
\begin{enumerate}
\item
Let $\L\G^{1}\colonequals \h\G^{1}\rtimes W_{F}$, where $\h\G^{1}\colonequals \h\G^{\h\theta,\circ}$.
If we put $\h\T^{1}\colonequals \h\T^{\h\theta,\circ}$, then $\h\G^{1}$ is a connected reductive group whose root system $\Phi(\h\G^{1},\h\T^{1})$ is regarded as a subset of $\Phi_{\res}(\h\G,\h\T)$ (see Section \ref{subsec:twisted-tori}).
Since we fixed sets of $a$-data and $\chi$-data for $\Phi_{\res}(\G,\T^{\dia})$, we also have sets of $a$-data and $\chi$-data for $\Phi_{\res}(\h\G,\h\T)$ which is equipped with a $\Gamma$-action derived from that of $\Phi_{\res}(\G,\T^{\dia})$.
Hence, by the Langlands--Shelstad construction \cite[Section 2.6]{LS87}, we obtain an $L$-embedding $\Lj_{\chi}^{1}\colon \h\T^{1}\rtimes W_{F}\hookrightarrow\L\G^{1}$.
Here, we emphasize that the $\Gamma$-action on $\h\T^{1}$ is imported from that on $\h\T_{0}$ through the isomorphism $\h\T_{0}\cong\h\T$.
Thus $\h\T^{1}\rtimes W_{F}$ is nothing but the $L$-group of the $\theta$-coinvariant $\T_{0,\theta}$ of $\T_{0}$.
\item
On the other hand, as $\Phi(\H,\T^{\flat})$ is regarded as a subset of $\Phi_{\res}(\G,\T^{\dia})$ (see Section \ref{subsec:rootsys}), the fixed sets of $a$-data and $\chi$-data also induce those for $\Phi(\H,\T^{\flat})$.
Hence, by the Langlands--Shelstad construction \cite[Section 2.6]{LS87}, we obtain an $L$-embedding $\Lj^{\H}_{\chi}\colon \L\T^{\flat}\hookrightarrow\L\H$.
\end{enumerate}
Now we note that the homomorphism $\xi_{D}\circ[g_{1}]^{-1}\colon \T_{0}\rightarrow\T^{\flat}$ is $F$-rational and induces an isomorphism $\T_{0,\theta}\cong\T^{\flat}$.
Thus we can compare two $L$-embeddings $\Lj_{\chi}^{1}$ and $\Lj_{\chi}^{\H}$ through this identification $\T_{0,\theta}\cong\T^{\flat}$ and $\h\xi\colon\L\H\hookrightarrow\L\G$.
We define a map $a_{\T_{0}}\colon\Gamma\rightarrow\h{\T}^{1}\colon \sigma\mapsto a_{\T_{0},\sigma}$ by
\[
\h\xi\circ\Lj_{\chi}^{\H}(1\rtimes\sigma)
=
a_{\T_{0},\sigma}\cdot\Lj_{\chi}^{1}(1\rtimes\sigma),
\]\symdef{a-T-0}{$a_{\T_{0}}$ ($a_{\T_{0},\sigma}$)}
then $a_{\T_{0}}$ is regarded as a $1$-cocycle $\Gamma\rightarrow \h\T_{0}$ under the identification $\h\T_{0}\cong\h\T$.
We define a $1$-cocycle $a_{\b\T_{0}}\colon\Gamma\rightarrow \h{\b\T}_{0}$ in the same manner and put $A\colonequals (a_{\T_{0}},a_{\b\T_{0}}/a_{\T_{0}})\colon \Gamma\rightarrow\h\bfS_{0}$.
On the other hand, we take an element $s_{\sc}\in\h\T_{\sc}$ having the same image in $\h\T_{\ad}$ as $s\in\h\T$ and write $s_{\T_{0}}$ and $s_{\b\T_{0}}$ for its images in $\h{\T}_{0,\sc}$ and $\h{\b{\T}}_{0,\sc}$, respectively.
We put $s_{\bfS_{0}}\colonequals (s_{\T_{0},\ad}, s_{\b\T_{0}}/s_{\T_{0}})\in\h\bfU_{0}$.
Then $(A^{-1},s_{\bfS_{0}})$ forms a $1$-hyper-cocycle.
We let $\bfA$ be the hyper-cohomology class of $(A^{-1},s_{\bfS_{0}})$.

The following proposition and its proof are inspired by \cite[Lemma 17]{Mez13}.

\begin{prop}\label{prop:Delta-III-same-torus}
Let $(\b\gamma,\b\delta)\in\mcD$ be another pair such that $D=(\B^{\flat},\T^{\flat},\B^{\dia},\T^{\dia})\in\bfD(\b\gamma,\b\delta)$.
Then we have
\[
\Delta_{\III}[a,\chi](\gamma,\delta;\b\gamma,\b\delta)
=
\langle \delta/\b\delta,a_{\T_{\dia}}\rangle_{\TN}.
\]
Here, the pairing on the right-hand side is the Tate--Nakayama pairing for $\T^{\dia}$ and $a_{\T_{\dia}}$ is the $1$-cocycle transported from $a_{\T_{0}}$ via the identification $[g_{1}]\colon\T_{\dia}\rightarrow\T_{0}$.
\end{prop}

\begin{proof}
We examine the construction of $\Delta_{\III}[a,\chi](\gamma,\delta;\b\gamma,\b\delta)$ explained above by assuming that $\bar{D}=D$.
We note that, when $\bar{\T}^\dia=\T^\dia$, we have the identifications
\[
\bfS_{0}=(\T_{0}\times\b\T_{0})/\Delta_{-}\bfZ_{\G}\cong \T_{0}\times\b\T_{0,\ad}\colon (t,\b{t})\mapsto (t\b{t}, \b{t}_{\ad})
\]
\[
\bfU_{0}=(\T_{0,\sc}\times\b\T_{0,\sc})/\Delta_{-}\bfZ_{\G_{\sc}}\cong \T_{0,\sc}\times\b\T_{0,\ad}\colon (t,\b{t})\mapsto (t\b{t}, \b{t}_{\ad}).
\]
Under these identifications, we see that 
\[
V\colon\sigma\mapsto (\b{v}_{\sigma}/v_{\sigma},\b{v}_{\sigma,\ad})=(1,v_{\sigma,\ad})
\quad
\text{and}
\quad
D=(\delta_{0}/\b{\delta}_{0},\b{\delta}_{0,\ad}^{-1}).
\]
On the other hand, $A$ and $s_{\bfS_{0}}$ are given by
\[
A\colon\sigma\mapsto (a_{\T_{0}},1)
\quad
\text{and}
\quad
s_{\bfS_{0}}=(s_{\T_{0},\ad},1).
\]

We have the following commutative diagrams which are dual to each other:
\[
\xymatrix{
\T_{0,\sc} \ar^-{1-\theta}[rr]&& \T_{0}\\
\bfU_{0}\cong\T_{0,\sc}\times\T_{0,\ad} \ar^-{\pr_{1}}[u]\ar^-{1-\theta}[rr]&& \bfS_{0}\cong\T_{0}\times\T_{0,\ad}\ar^-{\pr_{1}}[u]
}
\]
\[
\xymatrix{
\h\T_{0,\ad} \ar_-{i_{1}}[d]&& \h\T_{0} \ar_-{i_{1}}[d] \ar_-{1-\h\theta}[ll]\\
\h\bfU_{0}\cong\h\T_{0,\ad}\times\h\T_{0,\sc} && \h\bfS_{0}\cong\h\T_{0}\times\h\T_{0,\sc} \ar_-{1-\h\theta}[ll]
}
\]
Here, $\pr_{1}$ and $i_{1}$ denote the first projection and the injection to the first entry, respectively.
We note that $(A^{-1},s_{\bfS_{0}})$ is the push-out of the $1$-hyper-cocycle $(a_{\T_{0}}^{-1},s_{\T_{0},\ad})$ along the map $i_{1}$.
Thus, since $(\pr_{1}(V),\pr_{1}(D))=(1,\delta_{0}/\b\delta_{0})$, the functoriality of the Tate--Nakayama pairing (see \cite[Section 6.3]{Wal08}) implies that
\begin{align*}
\langle \inv(\gamma,\delta; \bar{\gamma},\bar{\delta}), \bfA\rangle_{\TN}
=\langle (V,D), (A^{-1},s_{\bfS_{0}})\rangle_{\TN}
=\langle (1,\delta_{0}/\b\delta_{0}), (a_{\T_{0}}^{-1},s_{\T_{0},\ad})\rangle_{\TN},
\end{align*}
where the last pairing is the one for $(\T_{0,\sc}\xrightarrow{1-\theta},\T_{0}, \h\T_{0}\xrightarrow{1-\h\theta}\h\T_{0,\ad})$.

We next note the following commutative diagrams which are dual to each other:
\[
\xymatrix{
\{1\} \ar[r] \ar[d]& \T_{0}\ar_-{\id}[d]&&\{1\}& \h\T_{0}\ar[l]\\
\T_{0,\sc}\ar^-{1-\theta}[r]&\T_{0}&&\h\T_{0,\ad}\ar[u]&\h\T_{0}\ar_-{\id}[u]\ar_-{1-\h\theta}[l]\\
}
\]
Since $(1,\delta_{0}/\b\delta_{0})$ is obtained by the push-out from ${1}\rightarrow\T_{0}$ and the image of $(a_{\T_{0}}^{-1},s_{\T_{0},\ad})$ in $\h\T_{0}\rightarrow\{1\}$ is $(a_{\T_{0}}^{-1},1)$, again the functoriality of the Tate--Nakayama pairing implies
\begin{align*}
\langle (1,\delta_{0}/\b\delta_{0}), (a_{\T_{0}}^{-1},s_{\T_{0},\ad})\rangle_{\TN}
=\langle \delta_{0}/\b\delta_{0}, a_{\T_{0}}^{-1}\rangle_{\TN},
\end{align*}
where the pairing on the right-hand side is the one for $(\{1\}\rightarrow\T_{0},\h\T_{0}\rightarrow\{1\})$, which is nothing but the usual Tate--Nakayama pairing for $\T_{0}$ (see \cite[A.3.13]{KS99} and also \cite[Section 4.3]{KS12}).
Finally, by noting that $\delta_{0}/\bar{\delta}_{0}=[g_{1}](\delta/\b\delta)$, we get the assertion.
\end{proof}

\begin{lem}\label{lem:tran-III}
For any positive integer $m\in\Z_{>0}$, we have
\[
\Delta_{\III}[a,\chi](\gamma_{<r}\cdot\gamma^{m}_{\geq r}, \delta_{<r}\cdot\delta^{m}_{\geq r})
=
\Delta_{\III}[a,\chi](\gamma,\delta).
\]
\end{lem}

\begin{proof}
It suffices to show that the relative factor $\Delta_{\III}[a,\chi](\gamma_{<r}\cdot\gamma^{m}_{\geq r}, \delta_{<r}\cdot\delta^{m}_{\geq r}; \gamma,\delta)$ is trivial.
By Proposition \ref{prop:Delta-III-same-torus}, this relative factor equals $\langle \delta_{\geq r}^{m-1},a_{\T^{\dia}}\rangle_{\TN}$.
Since we assume that $\chi$ is minimally ramified, the character $\langle -,a_{\T^{\dia}}\rangle_{\TN}$ of $T^{\dia}$ is tamely ramified.
Thus we get $\langle \delta_{\geq r}^{m-1},a_{\T^{\dia}}\rangle_{\TN}=1$.
\end{proof}

\subsection{Fourth factor $\Delta_{\IV}$}\label{subsec:tran-IV}
Recall from \cite[Section 4.5]{KS99} that the fourth factor $\Delta_{\IV}(\gamma,\delta)$ is defined by 
\[
\Delta_{\IV}(\gamma,\delta)
\colonequals 
\Delta_{\IV}^{\t\G}(\delta)/\Delta_{\IV}^{\H}(\gamma),
\]
where
\[
\Delta_{\IV}^{\t\G}(\delta)
\colonequals |\det([\delta]-1 \mid \mfg/\mft^{\dia})|_{\ol{F}}^{\frac{1}{2}}
\quad
\text{and}
\quad
\Delta_{\IV}^{\H}(\gamma)
\colonequals |\det([\gamma]-1 \mid \mfh/\mft^{\flat})|_{\ol{F}}^{\frac{1}{2}}.
\]
Recall that, since we are assuming that $\Phi_{\res}(\G,\T^{\dia})$ does not contain a restricted root of type $2$ or $3$, we have
\[
\Delta_{\IV}^{\t\G}(\delta)
=
\prod_{\alpha\in\Phi_{\res}(\G,\T^{\dia})}|N(\alpha)(\nu)-1|^{\frac{1}{2}}_{\ol{F}}
\]
(see \cite[Section 4.5]{KS99}).
By noting this, we extend the definition of $\Delta_{\IV}^{\t\G}$ also for any semisimple element $\delta'\in\t{T}^{\dia}$ by
\[
\Delta_{\IV}^{\t\G}(\delta')
=
\prod_{\begin{subarray}{c}\alpha\in \Phi_{\res}(\G,\T^{\dia}) \\ N(\alpha)(\nu')\neq1\end{subarray}}|N(\alpha)(\nu')-1|^{\frac{1}{2}}_{\ol{F}},
\]
where $\nu'\in\T$ is the element such that $\t{\xi}_{\dia}(\delta')=\nu'\theta$.
We define $\Delta_{\IV}^{\H}(\gamma')$ for any semisimple $\gamma'\in T^{\flat}$ in a similar way.

\begin{lem}\label{lem:tran-IV-descent}
We have
\begin{align*}
\Delta_{\IV}^{\t{\G}}(\delta)
&=
\Delta_{\IV}^{\t\G}(\delta_{<r})\cdot\Delta_{\IV}^{\G_{\delta_{<r}}}(\delta_{\geq r})
=
|D^{\red}_{G_{\eta_{0}}}(\eta_{+})|^{\frac{1}{2}}\cdot|D^{\red}_{G_{\eta}}(\log(\delta_{\geq r}))|^{\frac{1}{2}},\\
\Delta_{\IV}^{\H}(\gamma)
&=\Delta_{\IV}^{\H}(\gamma_{<r})\cdot\Delta_{\IV}^{\H_{\gamma_{<r}}}(\gamma_{\geq r}).
\end{align*}
\end{lem}

\begin{proof}
We consider only $\Delta_{\IV}^{\t{\G}}(\delta)$ since the formula for $\Delta_{\IV}^{\H}(\gamma)$ follows from a simpler argument.
By noting that the valuation of $N(\alpha)(\nu_{<r})-1$ is smaller than $r$ when $N(\alpha)(\nu_{<r})\neq1$, we have
\[
|N(\alpha)(\nu)-1|^{\frac{1}{2}}_{\ol{F}}
=
\begin{cases}
|N(\alpha)(\nu_{<r})-1|^{\frac{1}{2}}_{\ol{F}} & \text{if $N(\alpha)(\nu_{<r})\neq1$,}\\
|N(\alpha)(\nu_{\geq r})-1|^{\frac{1}{2}}_{\ol{F}} & \text{if $N(\alpha)(\nu_{<r})=1$.}
\end{cases}
\]
Since $\{\alpha\in \Phi_{\res}(\G,\T^{\dia}) \mid N(\alpha)(\nu_{<r})=1\}$ is identified with the set $\Phi(\G_{\delta_{<r}},\T^{\nat})$ (Section \ref{subsec:twisted-tori}) and 
\[
|N(\alpha)(\nu_{\geq r})-1|^{\frac{1}{2}}_{\ol{F}}
=|\alpha(\nu_{\geq r})^{l_{\alpha}}-1|^{\frac{1}{2}}_{\ol{F}}
=|\alpha(\nu_{\geq r})-1|^{\frac{1}{2}}_{\ol{F}}
\]
(use that $l_{\alpha}=1,2$ and $p\neq2$), we get $\Delta_{\IV}^{\t\G}(\delta)=\Delta_{\IV}^{\t\G}(\delta_{<r})\cdot\Delta_{\IV}^{\G_{\delta_{<r}}}(\delta_{\geq r})$.

By applying the same argument to $\Delta_{\IV}^{\t\G}(\delta_{<r})$, we also have a decomposition $\Delta_{\IV}^{\t\G}(\delta_{<r})=\Delta_{\IV}^{\t\G}(\delta_{0})\cdot \Delta_{\IV}^{\G_{\delta_{0}}}(\delta_{<r}^{+})$.
However, we have $|N(\alpha)(\nu_{0})-1|_{\ol{F}}=1$ whenever $N(\alpha)(\nu_{0})\neq1$ since $N(\alpha)(\nu_{0})$ is of prime-to-$p$ order.
Hence we get $\Delta_{\IV}^{\t\G}(\delta)=\Delta_{\IV}^{\G_{\delta_{0}}}(\delta_{<r}^{+})\cdot\Delta_{\IV}^{\G_{\delta_{<r}}}(\delta_{\geq r})$.
This can be rewritten as $\Delta_{\IV}^{\t{\G}}(\delta)=|D^{\red}_{G_{\eta_{0}}}(\eta_{+})|^{\frac{1}{2}}\cdot|D^{\red}_{G_{\eta}}(\delta_{\geq r})|^{\frac{1}{2}}$ by \cite[Remark 2.12]{DS18}.
By also noting that $|D^{\red}_{G_{\eta}}(\delta_{\geq r})|=|D^{\red}_{G_{\eta}}(\log(\delta_{\geq r}))|$, we obtain the assertion.
\end{proof}

\begin{lem}\label{lem:tran-IV}
There exists a constant $d\in\Z_{\geq0}$\symdef{d}{$d$} determined by $\delta_{<r}$ such that, for any positive integer $m\in\Z_{>0}$, we have
\[
\Delta_{\IV}(\gamma_{<r}\cdot\gamma^{p^{m}}_{\geq r}, \delta_{<r}\cdot\delta^{p^{m}}_{\geq r})
=
|p|^{md}_{\ol{F}}\cdot\Delta_{\IV}(\gamma_{<r}\cdot\gamma_{\geq r}, \delta_{<r}\cdot\delta_{\geq r}).
\]
\end{lem}

\begin{proof}
By Lemma \ref{lem:tran-IV-descent}, we have $\Delta_{\IV}^{\t{\G}}(\delta)=\Delta_{\IV}^{\t\G}(\delta_{<r})\cdot\Delta_{\IV}^{\G_{\delta_{<r}}}(\delta_{\geq r})$ and $\Delta_{\IV}^{\t\G}(\delta_{<r}\cdot\delta_{\geq r}^{p^{m}})=\Delta_{\IV}^{\t\G}(\delta_{<r})\cdot\Delta_{\IV}^{\G_{\delta_{<r}}}(\delta_{\geq r}^{p^{m}})$.
On the other hand, by \cite[Lemma 3.1]{Hal93}, we have $\Delta_{\IV}^{\G_{\delta_{<r}}}(\delta_{\geq r}^{p^{m}})=|p|^{m|\Phi(\G_{\delta_{<r}},\T^{\nat})|}_{\ol{F}}\cdot\Delta_{\IV}^{\G_{\delta_{<r}}}(\delta_{\geq r})$ when $p>e_{F}+1$, which is assumed to hold (see the beginning of Section \ref{sec:Waldspurger}).
Similarly, we have $\Delta_{\IV}^{\H}(\gamma_{<r}\cdot\gamma_{\geq r})
=\Delta_{\IV}^{\H}(\gamma_{<r})\cdot\Delta_{\IV}^{\H}(\gamma_{\geq r})$, $\Delta_{\IV}^{\H}(\gamma_{<r}\cdot\gamma_{\geq r}^{p^{m}})
=\Delta_{\IV}^{\H}(\gamma_{<r})\cdot\Delta_{\IV}^{\H}(\gamma_{\geq r}^{p^{m}})$, and $\Delta_{\IV}^{\H_{\gamma_{<r}}}(\gamma_{\geq r}^{p^{m}})=|p|^{m|\Phi(\H_{\gamma_{<r}},\T^{\flat})|}_{\ol{F}}\cdot\Delta_{\IV}^{\H_{\gamma_{<r}}}(\gamma_{\geq r})$ when $p>e_{F}+1$.
By putting all of these into together, we get the assertion.
\end{proof}

\subsection{Tail-scaling lemma on the full transfer factor}

By combining Lemmas \ref{lem:tran-I}, \ref{lem:tran-II}, \ref{lem:tran-III}, and \ref{lem:tran-IV}, we get the following proposition, which is the twisted version of \cite[Lemma 6.3.3]{Kal19}.

\begin{lem}\label{lem:tran-scaling}
For any sufficiently large positive integer $m\in\Z_{>0}$, we have
\[
\mr{\Delta}(\gamma_{<r}\cdot\gamma^{p^{2m}}_{\geq r}, \delta_{<r}\cdot\delta^{p^{2m}}_{\geq r})
=
\mr{\Delta}(\gamma_{<r}\cdot\gamma_{\geq r}, \delta_{<r}\cdot\delta_{\geq r})
\quad\text{and}
\]
\[
\Delta(\gamma_{<r}\cdot\gamma^{p^{2m}}_{\geq r}, \delta_{<r}\cdot\delta^{p^{2m}}_{\geq r})
=
|p|_{\ol{F}}^{2md}\cdot\Delta(\gamma_{<r}\cdot\gamma_{\geq r}, \delta_{<r}\cdot\delta_{\geq r})
\]
with constant $d$ as in Lemma \ref{lem:tran-IV}.
\end{lem}

An important consequence of this lemma is the following.

\begin{prop}\label{prop:NST}
With the notation as in Theorem \ref{thm:NST} and Corollary \ref{cor:NST}, for any $D$-norm pair $(Y,X)\in\mfh_{y,0+}\times\mfg_{\eta,0+}$, we have
\[
\Delta\bigl(y\exp(Y),\eta\exp(X)\bigr)
=
\Delta\!^{D}(\bar{Y},X_{\sc})
\quad\text{and}
\]
\[
\mr{\Delta}\bigl(y\exp(Y),\eta\exp(X)\bigr)\cdot\Delta_{\IV}(y,\eta)
=
\mr{\Delta}\!^{D}(\bar{Y},X_{\sc})
\]
\end{prop}

\begin{proof}
By Lemma \ref{lem:tran-scaling}, we have
\[
\Delta\bigl(y\exp(Y),\eta\exp(X)\bigr)
=|p|^{-2md}_{\ol{F}}\cdot\Delta(y\exp(p^{2m}Y),\eta\exp(p^{2m}X))
\]
for any sufficiently large $m\in\Z_{>0}$.
By taking $m$ to be a sufficiently large integer so that $p^{2m}Y$ belongs to the set $\mfV$ as in Corollary \ref{cor:NST}, we get
\[
\Delta(y\exp(p^{2m}Y),\eta\exp(p^{2m}X))
=\Delta\!^{D}(p^{2m}\bar{Y},p^{2m}X_{\sc})
\]
by Corollary \ref{cor:NST}.
By the homogeneity of the Lie algebra transfer factor, we have $\Delta\!^{D}(p^{2m}\bar{Y},p^{2m}X_{\sc})=|p|^{2md}_{\ol{F}}\cdot \Delta\!^{D}(\bar{Y},X_{\sc})$ (see \cite[Section 2.3]{Wal97} and \cite[Section 10]{Hal93}).
Thus we get the first equality.

By Lemma \ref{lem:tran-IV-descent}, we have $\Delta_{\IV}(\gamma,\delta)=\Delta_{\IV}(y,\eta)\cdot\Delta_{\IV}^{\G_{\eta}}(\delta_{\geq r})\cdot\Delta_{\IV}^{\H_{y}}(\gamma_{\geq r})^{-1}$.
Hence, by noting $\Delta_{\IV}^{\G_{\eta}}(\delta_{\geq r})\cdot\Delta_{\IV}^{\H_{y}}(\gamma_{\geq r})^{-1}=\Delta^{D}_{\IV}(\bar{Y},X_{\sc})$, we get the second equality.
\end{proof}

\section{Twisted endoscopic character relation}\label{sec:TECR}

\subsection{Twisted endoscopic character relation}

We assume that $(\bfS,\hat{\jmath},\chi,\vartheta)$ in a toral supercuspidal $L$-packet datum of $\G$ whose $L$-parameter $\phi$ factors though the $L$-embedding $\h\xi$ for an endoscopic data $(\H,\L\H,s,\hat{\xi})$ (i.e., we are in the situation as in Section \ref{subsec:str-theta-L}).
Here, by replacing $(\bfS,\hat{\jmath},\chi,\vartheta)$ with its isomorphic data if necessary, we may assume that $\chi=\chi_{\vartheta_{\h\j}}$ (see Section \ref{subsec:const-packet}).
As in the manner of Section \ref{subsec:L-packet-descent}, we get a toral supercuspidal $L$-packet datum $(\bfS_{\H},\hat{\jmath}_{\H},\chi_{\H},\vartheta_{\H})$.
Similarly, we may assume that $\chi_{\H}=\chi_{\vartheta_{\h\j_{\H}}}$.
Let $\Pi_{\phi}^{\G}$ (resp.\ $\Pi_{\phi_{\H}}^{\H}$) denote the associated toral supercuspidal $L$-packet of $\G$ (resp.\ $\H$).

The aim of this section is to establish the following in some special cases:

\begin{expect}\label{expect:TECR}
For each $\pi\in\Pi_{\phi}^{\G}$ there exists a constant $\Delta^{\spec}_{\phi,\pi}\in\C$ such that the following identity holds for any elliptic strongly regular semisimple $\delta\in\t{G}$:
\[
\sum_{\pi\in\Pi_{\phi}^{\G}}
\Delta^{\spec}_{\phi,\pi}\Theta_{\t{\pi}}(\delta)
=
\sum_{\gamma\in H/{\st}} \frac{\Delta_{\IV}^{\H}(\gamma)^{2}}{\Delta_{\IV}^{\t\G}(\delta)^{2}} \Delta(\gamma,\delta)
\sum_{\pi_{\H}\in\Pi_{\phi_{\H}}^{\H}}\Theta_{\pi_{\H}}(\gamma),
\]
or equivalently, 
\begin{align}\label{eq:TECR}
\sum_{\pi\in\Pi_{\phi}^{\G}}
\Delta^{\spec}_{\phi,\pi}\Phi_{\t{\pi}}(\delta)
=
\sum_{\gamma\in H/{\st}} \mr{\Delta}(\gamma,\delta)
\sum_{\pi_{\H}\in\Pi_{\phi_{\H}}^{\H}}\Phi_{\pi_{\H}}(\gamma),
\end{align}
where the first sum on the right-hand sides is over the stable conjugacy classes of strongly $\G$-regular semisimple elements of $H$ and we put $\Phi_{\t{\pi}}(\delta)\colonequals \Delta_{\IV}^{\t\G}(\delta)\cdot\Theta_{\t{\pi}}(\delta)$ and $\Phi_{\pi_{\H}}(\gamma)\colonequals \Delta_{\IV}^{\H}(\gamma)\cdot\Theta_{\pi_{\H}}(\gamma)$.
\end{expect}

\subsection{Several preliminary considerations}

\subsubsection{Initial observation on the index sets}\label{subsubsec:initial}

In the following, we fix an elliptic strongly regular semisimple element $\delta\in\t{G}$ and also fix a normal $r$-approximation $\delta=\delta_{0}\delta^{+}_{<r}\delta_{\geq r}$ in the sense of Definition \ref{defn:twisted-approx} (recall that we can always find a normal $r$-approximation by Proposition \ref{prop:twisted-Jordan}).
We let $\eta$ denote $\delta_{<r}\in\t{G}_{\ss}$.
We take a set $\mfH_{\eta}\subset H_{\ss}$ as in Section \ref{subsec:descent-lem}, i.e., $\mfH_{\eta}$ is a set of representatives for the stable conjugacy classes of semisimple elements of $H$ such that $y$ corresponds to $\eta$ and $\H_{y}$ is quasi-split for any $y\in \mfH_{\eta}$.

Recall that the $\theta$-stable members of $\Pi^{\G}_{\phi}$ are parametrized by 
\[
\t\mcJ_{G}^{\G}
\colonequals 
\{j\colon \t\bfS\hookrightarrow\t\G \mid \text{$j$ is $F$-rational and $j\sim_{\G}\j^{-1}$}\}
/{\sim_{G}}
\]
(see Section \ref{subsec:twisted-parametrization})
For each $j\in\t{\mcJ}^{\G}_{G}$, the corresponding member is given to be the toral supercuspidal representation (let us write $\pi_{j}$) arising from the tame elliptic regular pair $(\bfS_{j},\vartheta'_{j})$ of $\G$ (see Section \ref{subsec:const-packet}).
According to the twisted character formula (Theorem \ref{thm:TCF-final}), the twisted character $\Phi_{\t{\pi}_{j}}(\delta)$ is expressed by a sum over the set $\{g\in S_{j}\backslash G/G_{\eta} \mid {}^{g}\eta\in \t{S}_{j}\}$.
For each $j\in\t{\mcJ}^{\G}_{G}$, we put
\[
\t\mcJ^{G}_{G_{\eta}}(j)\colonequals 
\{k\colon\t\bfS\hookrightarrow\t\G \mid \text{$k$ is $F$-rational, $k\sim_{G}j$, and $\eta\in\t{S}_{k}$}\} / {\sim_{G_{\eta}}}.
\]\symdef{J-tilde-G-G-eta-j}{$\tilde{\mathcal{J}}_{G_{\eta}}^{G}(j)$}

\begin{lem}\label{lem:index-bij}
We have a bijection
\[
\{g\in S_{j}\backslash G/G_{\eta} \mid {}^{g}\eta\in \t{S}_{j}\}
\xleftrightarrow{1:1}
\t\mcJ^{G}_{G_{\eta}}(j)
\colon
g\mapsto [g^{-1}]\circ j.
\]
\end{lem}

\begin{proof}
It suffices to check that the surjective map from $\{g\in S_{j}\backslash G\mid {}^{g}\eta\in \t{S}_{j}\}$ to $\{k\colon\t\bfS\hookrightarrow\t\G \mid \text{$k$ is $F$-rational, $k\sim_{G}j$, and $\eta\in\t{S}_{k}$}\}$ given by $g\mapsto [g]^{-1}\circ j$
is in fact injective.
(Then we get the assertion by taking the quotient by $G_{\eta}$.) 
Let us suppose that two elements $g$ and $g'$ of $G$ map to the same element, i.e., we have $[g]^{-1}\circ j=[g']^{-1}\circ j$.
Then $g'g^{-1}$ belongs to $S_{j}$, hence $g$ and $g'$ belong to the same double coset.
Hence the map in the assertion is injective.
\end{proof}

By this lemma, the $\G$-side $\sum_{\pi\in\Pi_{\phi}^{\G}}\Delta^{\spec}_{\phi,\pi}\Theta_{\t{\pi}}(\delta)$ of the twisted endoscopic character relation \eqref{eq:TECR} can be written as a double sum over the sets $\t{\mcJ}^{\G}_{G}$ and $\t\mcJ^{G}_{G_{\eta}}(j)$ (for $j\in\t{\mcJ}^{\G}_{G}$).
We rearrange this double as follows.
We first combine $\t{\mcJ}^{\G}_{G}$ and $\t\mcJ^{G}_{G_{\eta}}(j)$ (for $j\in\t{\mcJ}^{\G}_{G}$) into the following single set:
\[
\t\mcJ^{\G}_{G_{\eta}}\colonequals 
\{k\colon\t\bfS\hookrightarrow\t\G \mid \text{$k$ is $F$-rational, $k\sim_{\G}\jmath^{-1}$, and $\eta\in\t{S}_{k}$}\} / {\sim_{G_{\eta}}}.
\]\symdef{J-tilde-G-G-eta-rat}{$\tilde{\mathcal{J}}_{G_{\eta}}^{\G}$}
Then we again divide $\mcJ^{\G}_{G_{\eta}}$ into the sets $\t\mcJ_{\G_{\eta}}^{\G}$ and $\t\mcJ^{\G_{\eta}}_{G_{\eta}}(j)$ (for $j\in\t{\mcJ}^{\G}_{\G_{\eta}}$), where 
\begin{itemize}
\item
$\t\mcJ_{\G_{\eta}}^{\G}
\colonequals 
\{j\colon\t\bfS\hookrightarrow\t\G \mid \text{$j$ is $F$-rational, $j\sim_{\G}\jmath^{-1}$, and $\eta\in\t{S}_{j}$}\} / {\sim_{\G_{\eta}}}$\symdef{J-tilde-G-G-eta}{$\tilde{\mathcal{J}}_{\G_{\eta}}^{\G}$},
\item
$\t\mcJ^{\G_{\eta}}_{G_{\eta}}(j)\colonequals 
\{k\colon\t\bfS\hookrightarrow\t\G \mid \text{$k$ is $F$-rational, $k\sim_{\G_{\eta}}j$, and $\eta\in\t{S}_{k}$}\} / {\sim_{G_{\eta}}}$.\symdef{J-tilde-G-eta-G-eta-j}{$\tilde{\mathcal{J}}_{G_{\eta}}^{\G_\eta}(j)$}
\end{itemize}
In the following arguments, we fix representatives of these sets and loosely identify these sets with the fixed sets of representatives.
\[
\xymatrix@R=0pt{
&\G\ar@{-}_-{\t{\mcJ}^{\G}_{G}}[ld]\ar@{-}^-{\t{\mcJ}^{\G}_{\G_{\eta}}}[rd]&\\
G\ar@{-}_-{\t\mcJ^{G}_{G_{\eta}}}[rd]&&\G_{\eta}\ar@{-}^-{\t\mcJ^{\G_{\eta}}_{G_{\eta}}}[ld]\\
&G_{\eta}&
}
\]

Keeping this observation in mind, let us first consider the case where
\[
\text{$\bfD(y,\eta)=\varnothing$ for any $y\in\mfH_{\eta}$.}
\]
In this case, by Proposition \ref{prop:descent-lemma} (and Remark \ref{rem:descent-lemma}), $\t{\mcJ}^{\G}_{\G_{\eta}}$ is necessarily empty.
This implies that the $\G$-side of the twisted endoscopic character relation contains a sum over the empty set, hence equals $0$.
On the other hand, we see that also the $\H$-side equals $0$ by the following lemma:

\begin{lem}\label{lem:no-norm}
If $\bfD(y,\eta)=\varnothing$ for any $y\in\mfH_{\eta}$, then there is no norm of $\delta$ in $H$.
\end{lem}

\begin{proof}
For the sake of contradiction, let us suppose that there exists a norm $\gamma\in H$ of $\delta$.
Then we have a diagram $D\in\bfD(\gamma,\delta)$ associated to $(\gamma,\delta)$ by Lemma \ref{lem:unique-diag}.
If we put $\epsilon\colonequals \t\xi_{D}(\eta)$, $D$ is also a diagram associated to $(\epsilon,\eta)$.
Then, by the same argument as in the proof of the surjectivity part of Proposition \ref{prop:descent-lemma}, we can construct a unique $y\in\mfH_{\eta}$ and a diagram $D'$ associated to $(y,\eta)$ by modifying $D$ via $\H$-conjugacy.
Thus we get a contradiction.
\end{proof}

Therefore, in the following, we focus on an(y) elliptic strongly regular semisimple element $\delta\in\G$ such that
\[
\text{$\bfD(y,\eta)\neq\varnothing$ for some $y\in\mfH_{\eta}$.}
\]

\begin{rem}\label{rem:no-norm}
This argument shows that if there does not exist an elliptic strongly regular semisimple element $\delta\in\G$ such that $\bfD(y,\eta)\neq\varnothing$ for some $y\in\mfH_{\eta}$, then there is nothing to prove anymore.
In this case, we simply put $\Delta^{\spec}_{\phi,\pi}\colonequals 0$  for any $\pi\in\Pi_{\phi}^{\G}$.
\end{rem}

\subsubsection{Head-tail stratification of the endoscopic index set}\label{subsubsec:head-tail}

When $\gamma\in H$ is a norm of $\delta\in\t{G}$, we take a diagram $D\in\bfD(\gamma,\delta)$ and consider the associated map $\t{\xi}_{D}$.
According to Lemma \ref{lem:unique-diag}, such a diagram always exists uniquely up to equivalence and the map $\t{\xi}_{D}$ is independent of the choice of $D$.
By noting this, we put $\gamma_{i}\colonequals \t{\xi}_{D}(\delta_{i})$ for $i\in\R_{\geq0}$.
Then we get a normal $r$-approximation $\gamma=\gamma_{<r}\cdot\gamma_{\geq r}$ (see Lemma \ref{lem:normal-approx-descent}).
Note that the $r$-approximation to $\gamma$ induced from that to $\delta$ in this way is consistent with respect to the stable $\H$-conjugacy.
More precisely, for any norms $\gamma\in\H$ and $\bar{\gamma}\in\H$ of $\delta$ which are stably conjugate by $h\in\H$ (say $\bar{\gamma}={}^{h}\gamma$), we have $\bar{\gamma}_{i}={}^{h}\gamma_{i}$.

For each $y\in\mfH_{\eta}$, we put
\[
H_{y}[\delta]_{r}
\colonequals 
\bigg\{
z\in H_{y,\srs}
\,\bigg\vert\,
\begin{array}{l}
\text{$yz\in H$ is a norm of $\delta$,} \\
\text{$y\cdot z$ is the fixed normal $r$-approximation to $yz$}
\end{array}
\bigg\}.
\]\symdef{H-y-delta-r}{$H_{y}[\delta]_{r}$}

The following is the twisted version of \cite[Lemma 6.4]{Kal15}.

\begin{lem}\label{lem:index-set}
The map
\[
\bigsqcup_{y\in\mfH_{\eta}} H_{y}[\delta]_{r}/{\sim_{\H_{y}}}\rightarrow\{\gamma\in H_{\srs}\mid\text{$\gamma$ is a norm of $\delta$}\}/{\sim_{\H}}
\colon z\mapsto yz
\]
is a $\pi_{0}(\H^{y})(F)$-torsor on each disjoint summand $H_{y}[\delta]_{r}/{\sim_{\H_{y}}}$ (onto its image).
In other words, the induced map 
\[
\bigsqcup_{y\in\mfH_{\eta}} (H_{y}[\delta]_{r}/{\sim_{\H_{y}}})/\pi_{0}(\H^{y})(F)\rightarrow\{\gamma\in H_{\srs}\mid\text{$\gamma$ is a norm of $\delta$}\}/{\sim_{\H}}
\colon z\mapsto yz
\]
is bijective.
Here, ${\sim_{\H_{y}}}$ on the left-hand side (resp.\ ${\sim_{\H}}$ on the right-hand side) denotes the stable conjugacy in $\H_{y}$ (resp.\ $\H$).
\end{lem}

\begin{proof}
The well-definedness of the map is obvious.

We first show the surjectivity of the map.
Let $\gamma\in H_{\srs}$ be a norm of $\delta$.
Then, according to Lemma \ref{lem:unique-diag}, there exists a diagram $D\in\bfD(\gamma,\delta)$ unique up to equivalence.
We put $\epsilon\colonequals \t{\xi}_{D}(\eta)$.
Then, by the definition of $\mfH_{\eta}$, there uniquely exists a $y\in\mfH_{\eta}$ which is stably $\H$-conjugate to $\epsilon$.
Let us take an element $h\in\H$ giving this stable conjugacy, that is, $[h](\epsilon)=y$ and $\sigma(h)^{-1}h\in\H_{\epsilon}$ for any $\sigma\in\Gamma$.
Then the map $[h]$ gives an inner twist between $\H_{\epsilon}$ and the quasi-split connected reductive group $\H_{y}$.
Since any $F$-rational maximal torus transfers to the quasi-split inner form, we may suppose that the map $[h]\colon\H_{\epsilon}\rightarrow\H_{y}$ induces an $F$-rational isomorphism from $\H_{\gamma}$ (this is a maximal torus of $\H_{\epsilon}$) to a maximal torus of $\H_{y}$.
Then $z\colonequals [h](\xi_{D}(\delta_{\geq r}))$ is an element of $\H_{y}$ such that $yz=[h](\gamma)$.
This means that the map in the assertion is surjective.

We next investigate the fibers of the map.
Suppose that we have $z\in H_{y}[\delta]_{r}$ and $\bar{z}\in H_{\bar{y}}[\delta]_{r}$ for $y,\bar{y}\in\mfH_{\eta}$ such that $yz$ and $\bar{y}\bar{z}$ are stably $\H$-conjugate.
Let $h\in\H$ be an element giving the stable conjugacy, i.e., $[h](yz)=\bar{y}\bar{z}$.
As we took the normal $r$-approximations to be consistent with the stable conjugacy, this implies that $[h](y)=\bar{y}$.
Then, by the definition of $\mfH_{\eta}$, we get $y=\bar{y}$, hence $h\in\H^{y}$.
Since we also have $[h](z)=\bar{z}$, we know that $z$ and $\bar{z}$ are conjugate under the action of $\pi_{0}(\H^{y})(F)$ by \cite[Lemma 6.3]{Kal15}.
\end{proof}

\begin{lem}\label{lem:norm-descent}
We put $X\colonequals \log(\delta_{\geq r})\in\mfg_{\eta}$.
Let $z\in H_{y}[\delta]_{r}$ and we put $Y\colonequals \log(z)\in\mfh_{y}$.
There exists a diagram $D\in\bfD(y,\eta)$ uniquely up to equivalence such that $(Y,X)$ is a $D$-norm pair.
\end{lem}

\begin{proof}
By the definition of the set $H_{y}[\delta]_{r}$, $yz$ is a norm of $\delta$ and $y\cdot z$ is the fixed $r$-approximation to $yz$.
Hence, according to our choice of normal $r$-approximations, there exists a diagram $D=(\B^{\flat},\T^{\flat},\B^{\dia},\T^{\dia})\in\bfD(y,\eta)$ satisfying $\t{\xi}_{D}(\eta)=y$ and $\xi_{D}(\delta_{\geq r})=z$.
This implies that $(Y,X)$ is a $D$-norm pair.

To check the uniqueness of $D$, let us suppose that $(Y,X)$ is a $\bar{D}$-norm pair for another diagram $\bar{D}=(\b\B^{\flat},\b\T^{\flat},\b\B^{\dia},\b\T^{\dia})\in\bfD(y,\eta)$.
Then, by replacing $\bar{D}$ with its equivalent diagram appropriately, we may assume that $\b\T^{\flat}=\T^{\flat}$ and $\b\T^{\dia}=\T^{\dia}$ and that $\xi_{D}(\exp(X))=\exp(Y)=\xi_{\bar{D}}(\exp(X))$ (cf.\ the argument in the proof of Lemma \ref{lem:diag-H-conj}).
This implies that both $D$ and $\b{D}$ belong to $\bfD(y\exp(Y),\eta\exp(X))$.
Thus, by Lemma \ref{lem:unique-diag}, $D$ and $\b{D}$ are equivalent in $\bfD(y\exp(Y),\eta\exp(X))$, hence also in $\bfD(y,\eta)$.
\end{proof}

By the invariance of the logarithm map, Lemma \ref{lem:norm-descent} implies the following.

\begin{lem}\label{lem:norm-tran}
The association $z\mapsto\log(z)$ induces a bijection
\[
H_{y}[\delta]_{r}/{\sim_{\H_{y}}}
\xrightarrow{1:1}
\bigsqcup_{D\in\bbD(y,\eta)}
\{Y\xleftrightarrow{D} X\}/{\sim_{\H_{y}}},
\]
where the left-hand side denotes the set of $\H_{y}$-conjugacy classes of elements of $H_{y}[\delta]_{r}$ and the right-hand side denotes the set of $\H_{y}$-conjugacy classes of elements $Y\in\mfh_{y}$ which constitute a $D$-norm pair with $X\colonequals \log(\delta_{\geq r})$ (over $D\in\bbD(y,\eta)$).
\end{lem}

\subsubsection{Lie algebra transfer: revisited}\label{subsubsec:Lie-trans}

We introduce the sets $\mcJ^{\H}_{H}$, $\mcJ^{H}_{H_{y}}$, $\mcJ^{\H}_{\H_{y}}$ and $\mcJ^{\H_{y}}_{H_{y}}$ to rearrange the index sets on the $\H$-side of \eqref{eq:TECR} in a similar manner to Section \ref{subsubsec:initial}.
\[
\xymatrix@R=0pt{
&\G\ar@{-}_-{\t\mcJ^{\G}_{G}}[ld]\ar@{-}^-{\t\mcJ^{\G}_{\G_{\eta}}}[rd]&\\
G\ar@{-}_-{\t\mcJ^{G}_{G_{\eta}}}[rd]&&\G_{\eta}\ar@{-}^-{\t\mcJ^{\G_{\eta}}_{G_{\eta}}}[ld]\\
&G_{\eta}&
}
\quad\quad
\xymatrix@R=0pt{
&\H\ar@{-}_-{\mcJ^{\H}_{H}}[ld]\ar@{-}^-{\mcJ^{\H}_{\H_{y}}}[rd]&\\
H\ar@{-}_-{\mcJ^{H}_{H_{y}}}[rd]&&\H_{y}\ar@{-}^-{\mcJ^{\H_{y}}_{H_{y}}}[ld]\\
&H_{y}&
}
\]
Recall that, by Proposition \ref{prop:descent-lemma}, we have a bijective map
\[
\tran\colon\bigsqcup_{y\in \mfH_{\eta}} \bigl(\bbD(y,\eta)\times\mcJ_{\H_{y}}^{\H}\bigr)/\pi_{0}(\H^{y})(F) \rightarrow\t\mcJ_{\G_{\eta}}^{\G}.
\]
Suppose that $D\in \bfD(y,\eta)$, $j_{\H}\in\mcJ^{\H}_{\H_{y}}$, and $j\in\t\mcJ^{\G}_{\G_{\eta}}$ satisfy $\tran(D,j_{\H})=j$.

We also recall that, in Section \ref{sec:theta-stable}, we fixed $X^{\ast}\in\mfs^{\ast}$ and $Y^{\ast}\in\mfs_{\H}^{\ast}$, which are the elements realizing the toral characters $\vartheta$ and $\vartheta_{\H}$, respectively (see Corollary \ref{cor:vartheta}).
In the following, for $k\in\mcJ^{\G}$, we put
\[
X^{\ast}_{k}\colonequals (dk^{\ast})^{-1}(X^{\ast})\in \mfs_{k}^{\ast}\hookrightarrow\mfg^{\ast}
\]\symdef{X-ast-k}{$X^{\ast}_{k}$}
(note that $dk\colon\mfs\cong\mfs_{k}$, hence $dk^{\ast}\colon\mfs_{k}^{\ast}\cong\mfs^{\ast}$).
This element $X^{\ast}_{k}\in\mfs_{k}^{\ast}$ represents the character $\vartheta_{k}|_{S_{k,r}}=\vartheta'_{k}|_{S_{k,r}}$.
Note that when $k$ belongs to $\t\mcJ^{\G}_{G_{\eta}}$, we also have an element $X^{\ast}_{k}\in\mfs_{k}^{\nat\ast}\hookrightarrow\mfg_{\eta}^{\ast}$, which can be thought of as the image of the above $X_{k}^{\ast}$  under the natural map $\mfs_{k}^{\ast}\twoheadrightarrow\mfs_{k}^{\nat\ast}$.
Similarly, for $k_{\H}\in\mcJ^{\H}$, we put
\[
Y^{\ast}_{k_{\H}}\colonequals (dk_{\H}^{\ast})^{-1}(Y^{\ast})\in \mfs_{k_{\H}}^{\ast}\hookrightarrow\mfh^{\ast}.
\]\symdef{Y-ast-k-H}{$Y^{\ast}_{k_{\H}}$}

By the construction of the map $\tran$, the maps $\xi_{\bfS}$ and $\xi_{D}$ coincide under the embeddings $j$ and $j_{\H}$, i.e., $\xi_{D}\circ j= j_{\H}\circ\xi_{\bfS}$.
This implies that $d\xi_{D}^{\ast}(Y^{\ast}_{j_{\H}})=X^{\ast}_{j}\in(\mft^{\dia\ast})^{\theta_{\dia}}$.

\begin{lem}\label{lem:norm-pairs-TECR}
\begin{enumerate}
\item
The set $\{Y^{\ast}_{k_{\H}}\mid k_{\H}\in\mcJ^{\H_{y}}_{H_{y}}(j_{\H})\}$ represents the $H_{y}$-conjugacy classes within a stable $\H_{y}$-conjugacy class.
\item
The set $\{X^{\ast}_{k}\mid k\in\t\mcJ^{\G_{\eta}}_{G_{\eta}}(j)\}$ represents the $G_{\eta}$-conjugacy classes of elements of $\mfg_{\eta,0+}$ constituting a $D$-norm pair with $Y^{\ast}_{k_{\H}}$ for a(ny) $k_{\H}\in\mcJ^{\H_{y}}_{H_{y}}(j_{\H})$.
\end{enumerate}
\end{lem}

\begin{proof}
The assertion (1) is obvious by the definitions of $Y^{\ast}_{k_{\H}}$ and $\mcJ^{\H_{y}}_{H_{y}}(j_{\H})$.
Since we have $d\xi_{D}^{\ast}(Y^{\ast}_{j_{\H}})=X^{\ast}_{j}$, $(Y^{\ast}_{j_{\H}}, X^{\ast}_{j})$ is a $D$-norm pair.
Noting that all elements of $\mfg_{\eta,0+}$ constituting a norm pair with $Y^{\ast}_{k_{\H}}$ are $\G_{\eta}$-conjugate, the assertion (2) follows.
\end{proof}

Lemma \ref{lem:norm-pairs-TECR} enables us to rewrite Proposition \ref{prop:Lie-twisted} as follows:

\begin{prop}\label{prop:Lie-twisted-TECR}
We have
\[
\sum_{Y\overset{D}{\leftrightarrow} X /{\sim_{\H_{y}}}}
\mr{\Delta}\!^{D}(\bar{Y},X_{\sc})
\sum_{k_{\H}\in \mcJ^{\H_{y}}_{H_{y}}(j_{\H})} D^{\H_{y}}_{Y_{k_{\H}}^{\ast}}(Y)
=
\sum_{k\in \t\mcJ^{\G_{\eta}}_{G_{\eta}}(j)}
\mr{\Delta}\!^{D}(\bar{Y}_{j_{\H}}^{\ast},X^{\ast}_{k,\sc})
D^{\G_{\eta}}_{X^{\ast}_{k}}(X).
\]
\end{prop}

\subsubsection{$a$-data and $\chi$-data for restricted roots}\label{subsubsec:a-chi-res}
In our computation of the transfer factor carried out later, we need to fix sets of $a$-data and $\chi$-data for the restricted roots.
We explain our choice in the following.

We first discuss the $\bfG$-side.
Suppose that $j\in\t\mcJ^{\G}_{\G_{\eta}}$.
For any $k\in\t\mcJ^{\G_{\eta}}_{G_{\eta}}(j)$, we get an $\eta$-stable (hence also $\eta_{0}$-stable) tame elliptic toral pair $(\bfS_{k},\vartheta'_{k})$ of $\G$.
Then we have the set $\Phi_{\res}(\G,\bfS_{k})$ of restricted roots.
We define a set $a^{\res}_{k}=\{a^{\res}_{k,\alpha_{\res}}\}_{\alpha_{\res}\in \Phi_{\res}(\G,\bfS_{k})}$ of $a$-data for $\Phi_{\res}(\G,\bfS_{k})$\symdef{a-res-k}{$a^{\res}_{k}$} by
\[
a^{\res}_{k,\alpha_{\res}}=\langle H_{\alpha_{\res}},X_{k}^{\ast}\rangle,
\]\symdef{a-res-k}{$a^{\res}_{k}$}
where
\begin{itemize}
\item
$H_{\alpha_{\res}}\colonequals  d\alpha_{\res}^{\vee}(1)\in \bmfs_{k}^{\nat}(F_{\alpha_{\res}})$, and
\item
$X_{k}^{\ast}\in\mfs^{\nat\ast}_{k,-r}$ is an element associated to $\vartheta'_{k}$ as in Section \ref{subsubsec:Lie-trans}.
\end{itemize}

We define a set $\chi^{\res}_{k}=\{\chi^{\res}_{k,\alpha_{\res}}\}_{\alpha_{\res}\in \Phi_{\res}(\G,\bfS_{k})}$ of $\chi$-data for $\Phi_{\res}(\G,\bfS_{k})$\symdef{chi-res-k}{$\chi^{\res}_{k}$} as follows:\symdef{chi-res-k}{$\chi^{\res}_{k}$}
\begin{itemize}
\item
For $\alpha_{\res}\in\Phi_{\res}(\G,\bfS_{k})_{\asym}$, let $\chi^{\res}_{k,\alpha_{\res}}$ be the trivial character of $F_{\alpha_{\res}}^{\times}$.
\item
For $\alpha_{\res}\in\Phi_{\res}(\G,\bfS_{k})_{\ur}$, let $\chi^{\res}_{k,\alpha_{\res}}$ be the unique unramified nontrivial quadratic character of $F_{\alpha_{\res}}^{\times}$.
\item
For $\alpha_{\res}\in\Phi_{\res}(\G,\bfS_{k})_{\ram}$, let $\chi^{\res}_{k,\alpha_{\res}}$ be the unique tamely ramified character of $F_{\alpha_{\res}}^{\times}$ characterized by the following properties:
\[
\chi^{\res}_{k,\alpha_{\res}}|_{F_{\pm\alpha_{\res}}^{\times}}=\kappa_{\alpha_{\res}}
\quad\text{and}\quad
\chi^{\res}_{j,\alpha_{\res}}(2a^{\res}_{k,\alpha_{\res}})=\lambda_{\alpha_{\res}}.
\]
\end{itemize}

\begin{rem}
We can check that the above conditions uniquely specify the tamely ramified quadratic character $\chi_{k,\alpha_{\res}}^{\res}$ for $\alpha_{\res}\in\Phi_{\res}(\G,\bfS_{k})_{\ram}$ in the same manner as in \cite[Section 4.7]{Kal19}.
Indeed, if we let $\sigma_{\alpha_{\res}}\in\Gal(F_{\alpha_{\res}}/F_{\pm\alpha_{\res}})$ be the unique nontrivial element, then we have $\sigma_{\alpha_{\res}}(H_{\alpha_{\res}})=\sigma_{\alpha_{\res}}(d\alpha^{\vee}_{\res}(1))=d\sigma_{\alpha_{\res}}(\alpha^{\vee}_{\res})(1)=-d\alpha^{\vee}_{\res}(1)=-H_{\alpha_{\res}}$ and $\sigma_{\alpha_{\res}}(X_{k}^{\ast})=X_{k}^{\ast}$.
Hence $\sigma_{\alpha_{\res}}(a_{k,\alpha_{\res}}^{\res})=-a_{k,\alpha_{\res}}^{\res}$.
This implies that the valuation (normalized with respect to $F_{\alpha_{\res}}$) of $a_{k,\alpha_{\res}}^{\res}$ is odd.
\end{rem}

We note that, by restriction, $(\bfS_{k},\vartheta'_{k})$ induces a tame elliptic toral pair $(\bfS_{k}^{\nat},\vartheta^{\prime\nat}_{k})$ of $\G_{\eta_{0}}$, where $\vartheta^{\prime\nat}_{k}=\vartheta'_{k}|_{S_{k}^{\nat}}$ (\cite[Lemma 5.5]{Oi25}).
Then, by the construction of Kaletha (see Section \ref{subsec:invariants-for-LLC}), we have a set $a_{\vartheta^{\prime\nat}_{k}}=\{a_{\vartheta^{\prime\nat}_{k},\alpha}\}_{\alpha\in\Phi(\G_{\eta_{0}},\bfS_{k}^{\nat})}$ of $a$-data and a set $\chi_{\vartheta^{\prime\nat}_{k}}=\{\chi_{\vartheta^{\prime\nat}_{k},\alpha}\}_{\alpha\in\Phi(\G_{\eta_{0}},\bfS_{k}^{\nat})}$ of $\chi$-data for $\Phi(\G_{\eta_{0}},\bfS_{k}^{\nat})$.
We shortly write $(a_{k}^{\nat},\chi_{k}^{\nat})$ for $(a_{\vartheta^{\prime\nat}_{k}},\chi_{\vartheta^{\prime\nat}_{k}})$.\symdef{a-nat-k}{$a^{\nat}_{k}$}\symdef{chi-nat-k}{$\chi^{\nat}_{k}$}
As explained in Section \ref{subsec:rootsys}, the set $\Phi(\G_{\eta_{0}},\bfS_{k}^{\nat})$ can be regarded as a subset (root subsystem) of $\Phi_{\res}(\G,\bfS_{k})$.
By construction, we have the following:

\begin{lem}\label{lem:a-res-eta}
The sets of $a$-data and $\chi$-data $(a_{k}^{\nat},\chi_{k}^{\nat})$ are restrictions of $(a_{k}^{\res},\chi_{k}^{\res})$.
\end{lem}

We next discuss the $\bfH$-side.
Suppose that $D\in \bfD(y,\eta)$, $j_{\H}\in\mcJ^{\H}_{\H_{y}}$, and $j\in\t\mcJ^{\G}_{\G_{\eta}}$ satisfy $\tran(D,j_{\H})=j$.
For any $k_{\H}\in\mcJ^{\H_{y}}_{H_{y}}$, we get a tame elliptic toral pair $(\bfS_{k_{\H}},\vartheta'_{k_{\H}})\colonequals (\bfS_{\H,k_{\H}},\vartheta'_{\H,k_{\H}})$ of $\H$.
By applying Kaletha's construction (Section \ref{subsec:invariants-for-LLC}) to $(\bfS_{k_{\H}},\vartheta'_{k_{\H}})$, we get the sets $a_{k_{\H}}\colonequals a_{\vartheta'_{k_{\H}}}$ of $a$-data and $\chi_{k_{\H}}\colonequals \chi_{\vartheta'_{k_{\H}}}$ of $\chi$-data with respect to $\Phi(\H,\bfS_{k_{\H}})$.
As explained in Section \ref{subsec:rootsys}, the set $\Phi(\H_{y},\bfS_{k_{\H}})$ can be regarded as a subset (root subsystem) of $\Phi(\G_{\eta},\bfS_{k}^{\nat})$, which is a subset of $\Phi(\G_{\eta_{0}},\bfS_{k}^{\nat})$.

\begin{lem}\label{lem:a-res-H}
Suppose that $\alpha_{y}\in\Phi(\H_{y},\bfS_{k_{\H}})$ is identified with $\alpha_{\eta}\in\Phi(\G_{\eta},\bfS_{k}^{\nat})$.
Let $\alpha^{\vee}\in \Phi^{\vee}(\G,\bfS_{k})$ be the coroot satisfying $\alpha_{y}^{\vee}=p_{\ast}(\alpha^{\vee})$ and $\alpha_{\eta}^{\vee}=N(\alpha^{\vee})$.
Then, we have $l_{\alpha}\cdot a_{k_{\H},\alpha_{y}}=a_{k,\alpha_{\eta}}^{\res}$.
\end{lem}

\begin{proof}
By definition, we have $a_{k_{\H},\alpha_{y}}=\langle H_{\alpha_{y}},Y_{k_{\bfH}}^{\ast}\rangle$ and $a^{\res}_{k,\alpha_{\eta}}=\langle H_{\alpha_{\eta}},X_{k}^{\ast}\rangle$.
Since $(Y_{k_{\bfH}}^{\ast},X_{k}^{\ast})$ is a $D$-norm pair (Lemma \ref{lem:norm-pairs-TECR}) and we have $H_{\alpha_{y}}=d\alpha_{y}^{\vee}(1)$, $H_{\alpha_{\eta}}=d\alpha_{\eta}^{\vee}(1)$, we get the equality $l_{\alpha}\cdot a_{k_{\H},\alpha_{y}}=a_{k,\alpha_{\eta}}^{\res}$ by Lemma \ref{lem:appearance-of-l-alpha}.
\end{proof}

\subsubsection{Twisted character formula of a normalized form}

In the following, for each $j\in\t{\mcJ}^{\G}_{G}$, we fix a set of elements $\{g_{k}\in G\mid k\in\t\mcJ^{G}_{G_{\eta}}(j)\}$ such that $\{[g_{k}]^{-1}\circ j\}$ is a (fixed) set of representatives of $\t\mcJ^{G}_{G_{\eta}}(j)$.
(Note that this set also represents $\{g\in S_{j}\backslash G/G_{\eta} \mid {}^{g}\eta\in \t{S}_{j}\}$ by Lemma \ref{lem:index-bij}.)
Moreover, we fix a topologically semisimple element $\ul{\eta}_{j}$ of the twisted space $\t{S}_{j}$.\symdef{eta-ul-j}{$\ul{\eta}_{j}$}
For each $k\in\t{\mcJ}^{G}_{G_{\eta}}(j)$, we fix a base point $\ul{\eta}_{k}$ of the twisted space $\t{S}_{k}$ by $\ul{\eta}_{k}\colonequals [g_{k}]^{-1}(\ul{\eta}_{j})$.\symdef{eta-ul-k}{$\ul{\eta}_{k}$}

For any $k\in\t\mcJ^{\G}_{G_{\eta}}$, we define a character $\epsilon_{\vartheta_{k}}^{\star}$ of $S_{k}$ by
\[
\epsilon_{\vartheta_{k}}^{\star}(s)
\colonequals 
\prod_{\begin{subarray}{c} \ddot{\alpha}\in\ddot{\Xi}(\G,\bfS_{k})\\ \alpha_{\res}: \, \ram\end{subarray}}\epsilon_{\alpha}(s).
\]\symdef{epsilon-star-vartheta-k}{$\epsilon_{\vartheta_{k}}^{\star}$}
(Recall that $\ddot{\Xi}(\G,\bfS_{k})$ is the set of $\Sigma$-orbits of roots appearing in the Heisenberg quotient determined by the character $\vartheta_{k}$. Also recall that any such root must be asymmetric or unramified, hence the character $\epsilon_{\alpha}$ is defined; see Section \ref{subsec:Heisen-structure}.)

\begin{prop}\label{prop:TCF-final}
Let $j\in\t\mcJ^{\G}_{G}$.
For each $k\in\t\mcJ^{G}_{G_{\eta}}(j)$, we write $\eta=s_{k}\cdot\ul{\eta}_{k}\in\t{S}_{k}$.
Then we have
\begin{multline*}
\Phi_{\t{\pi}_{j}}(\delta)
=
C_{\ul{\eta}_{j}}
\cdot(-1)^{|\ddot{\Xi}_{\eta_{0},\ur}|}
\cdot e(\G_{\eta_{0}})\cdot e(\G_{\eta})
\cdot \varepsilon(\T_{\G_{\eta_{0}}^{\ast}})\cdot\varepsilon(\T_{\G_{\eta}^{\ast}})^{-1}\\
\cdot\sum_{k\in\t\mcJ^{G}_{G_{\eta}}(j)}
\vartheta_{k}(s_{k})
\cdot\epsilon_{\bfS_{k},\ram}(s_{k})
\cdot\epsilon_{\vartheta_{k}}^{\star}(s_{k})
\cdot\Delta_{\II}^{\G_{\eta_{0}}}[a_{k}^{\res},\chi_{k}^{\res}](\eta_{+})
\cdot\hat{\iota}^{\G_{\eta}}_{X_{k}^{\ast}} \bigl(\log(\delta_{\geq r})\bigr),
\end{multline*}
where $a_{k}^{\res}$ and $\chi_{k}^{\res}$ are the sets of $a$-data and $\chi$-data for $\Phi(\bfG_{\eta_{0}},\bfS_{k}^{\nat})$ as in Section \ref{subsubsec:a-chi-res}.
\end{prop}

\begin{proof}
By Proposition \ref{prop:TCF-normal}, we have
\begin{multline*}
\Phi_{\t{\pi}_{j}}(\delta)
=
C_{\ul{\eta}_{j}}
\cdot(-1)^{|\ddot{\Xi}_{\eta_{0},\ur}|}
\cdot e(\G_{\eta_{0}})\cdot e(\G_{\eta})
\cdot \varepsilon(\T_{\G_{\eta_{0}}^{\ast}})\cdot\varepsilon(\T_{\G_{\eta}^{\ast}})^{-1}\\
\sum_{\begin{subarray}{c} g\in S_{j}\backslash G/G_{\eta} \\ {}^{g}\eta\in\t{S}_{j}\end{subarray}}
\t\epsilon_{\Xi}(s_{j,g})\cdot\vartheta'_{j}(s_{j,g})
\cdot\Delta_{\II}^{\G_{{}^{g}\eta_{0}}}[a_{\vartheta'_{j}}^{\nat},\chi_{\vartheta'_{j}}^{\nat}]({}^{g}\eta_{+})
\cdot\hat{\iota}^{\G_{{}^{g}\eta}}_{X_{j}^{\ast}}(\log({}^{g}\delta_{\geq r})),
\end{multline*}
where $s_{j,g}\in S_{j}$ is the element satisfying ${}^{g}\eta=s_{j,g}\ul{\eta}_{j}$.
Recall that $\vartheta'_{j}=\epsilon_{\vartheta_{j}}\cdot\vartheta_{j}$ (Section \ref{subsec:const-packet}).
Here, we caution that we took the initial regular supercuspidal packet datum such that its $\chi$-data is equal to $\chi_{\vartheta_{\h\j}}$, hence the zeta character contained in $\vartheta'_{j}$ is trivial.
As we have $\epsilon_{\vartheta_{j}}=\epsilon_{\vartheta_{j},\asym}\cdot\epsilon_{\vartheta_{j},\ur}\cdot\epsilon_{\bfS_{j},\ram}$ and $\t\epsilon_{\Xi}$ is the product of $\epsilon_{\alpha}$'s for $\alpha\in\ddot{\Xi}$ such that $\alpha_{\res}$ is asymmetric or unramified (see Section \ref{subsec:TCF}), we get
\[
\t\epsilon_{\Xi}(s_{j,g})\cdot\vartheta'_{j}(s_{j,g})
=\epsilon_{\bfS_{j},\ram}(s_{j,g})\cdot\epsilon_{\vartheta_{j}}^{\star}(s_{j,g})\cdot\vartheta_{j}(s_{j,g}).
\]

By our choice of base points, when $g=g_{k}$, we have $\ul{\eta}_{j}={}^{g}\ul{\eta}_{k}$, hence $s_{j,g}={}^{g}s_{k}$.
Thus we get $\vartheta_{j}(s_{j,g})=\vartheta(j^{-1}(s_{j,g}))=\vartheta(j^{-1}\circ[g](s_{k}))=\vartheta(k^{-1}(s_{k}))=\vartheta_{k}(s_{k})$.
Similarly, we have $\epsilon_{\bfS_{j},\ram}(s_{j,g})=\epsilon_{\bfS_{j},\ram}({}^{g}s_{k})=\epsilon_{\bfS_{k},\ram}(s_{k})$,  $\epsilon_{\vartheta_{j}}^{\star}(s_{j,g})=\epsilon_{\vartheta_{j}}^{\star}({}^{g}s_{k})=\epsilon_{\vartheta_{k}}^{\star}(s_{k})$, and $\hat{\iota}^{\G_{{}^{g}\eta}}_{X_{j}^{\ast}}(\log({}^{g}\delta_{\geq r}))=\hat{\iota}^{\G_{\eta}}_{{}^{g^{-1}}X_{j}^{\ast}}(\log(\delta_{\geq r}))=\hat{\iota}^{\G_{\eta}}_{X_{k}^{\ast}}(\log(\delta_{\geq r}))$.
Moreover, by noting that $\vartheta^{\prime\nat}_{j}$ and $\vartheta^{\nat}_{j}$ give rise to the same $a$-data and $\chi$-data (Section \ref{subsec:invariants-for-LLC}) and using Lemma \ref{lem:a-res-eta}, we have $\Delta_{\II}^{\G_{{}^{g}\eta_{0}}}[a_{\vartheta'_{j}}^{\nat},\chi_{\vartheta'_{j}}^{\nat}]({}^{g}\eta_{+})=\Delta_{\II}^{\G_{\eta_{0}}}[a_{k}^{\res},\chi_{k}^{\res}](\eta_{+})$.
Thus we arrive at the claimed formula.
\end{proof}

\subsubsection{Third factor $\Delta_{\III}$: revisited}\label{subsubsec:tran-III-revisited}

We next rewrite Proposition \ref{prop:Delta-III-same-torus} in a form suitable for our purpose.
Suppose that $D=(\B^{\flat},\T^{\flat},\B^{\dia},\T^{\dia})\in \bfD(y,\eta)$, $j_{\H}\in\mcJ^{\H}_{\H_{y}}$, and $j\in\t\mcJ^{\G}_{\G_{\eta}}$ satisfy $\tran(D,j_{\H})=j$.
Hence, we may and do assume that $\T^{\flat}=\bfS_{j_{\H}}$ and $\T^{\dia}=\bfS_{j}$.
Let $k\in\t\mcJ^{\G_{\eta}}_{G_{\eta}}(j)$.
Note that then we also have $\bfS_{k}\cong\bfS_{j}$.
Combining it with the $F$-rational homomorphism $\xi_D\colon \bfT^\dia \rightarrow\bfT^\flat$, we have an $F$-rational isomorphism $\bfS_{k}\cong \bfS_{j}\rightarrow\bfS_{j_\bfH}$.

We introduce the following character according to \cite[Proposition 5.25]{Kal19-cover}:
\begin{defn}
Let $\zeta_{\desc}\colon S_{k}\rightarrow\C^{\times}$ be a character given by
\[
\zeta_{\desc}(s)
\colonequals 
\prod_{\begin{subarray}{c}\ddot{\alpha}\in\ddot{\Phi}(\G,\bfS_{k})_{\asym}\\ \text{$\alpha_{\res}$: ramified} \end{subarray}} \epsilon_{\alpha}(s)
\prod_{\begin{subarray}{c}\ddot{\alpha}\in\dot{\Phi}(\G,\bfS_{k})_{\ur}\\ \text{$\alpha_{\res}$: ramified} \end{subarray}} \epsilon_{\alpha}(s).
\]\symdef{zeta-desc}{$\zeta_{\desc}$}
\end{defn}

Our aim here is to show the following:
\begin{prop}\label{prop:delta-III-ratio}
Suppose that $(\b\gamma,\b\delta),(\b\gamma,'\b\delta')\in\mcD$ are such that $D\in\bfD(\b\gamma,\b\delta)$ and $D\in\bfD(\b\gamma',\b\delta')$.
Then we have
\[
\Delta_{\III}[a_{k}^{\res},\chi_{k}^{\res}](\b\gamma,\b\delta;\b\gamma',\b\delta')
=
\frac{\vartheta_{k}(\b\delta/\b\delta')}{\vartheta_{j_{\H}}(\b\gamma/\b\gamma')}
\cdot\zeta_{\desc}(\b\delta/\b\delta')
\cdot \zeta_{\chi^{\res}_{j}/\chi_{j_{\H}},S_{j_{\H}}}(\b\gamma/\b\gamma').
\]
\end{prop}

Recall that we introduced a $1$-cocycle $a_{\bfS_{k}}$ which measures the difference between $\Lj^{1}_{\chi_{k}^{\res}}$ and $\h\xi\circ\Lj^{\H}_{\chi_{k}^{\res}}$ in Section \ref{subsec:Delta-III}.
Let us write $a[\Lj^{\H}_{\chi_{k}^{\res}}/\Lj^{1}_{\chi_{k}^{\res}}]$ for $a_{\bfS_{k}}$.

On the other hand, we also have $L$-embeddings $\Lj_{\chi_{k}}\colon\L\bfS_{k}\hookrightarrow\L\G$ and $\Lj_{\chi_{\H}}\colon\L\bfS_{j_{\H}}\hookrightarrow\L\H$ obtained by applying the Langlands--Shelstad construction to the $\chi$-data $\chi_{k}$ and $\chi_{j_{\H}}$ as in Section \ref{subsubsec:a-chi-res}.
We define a $1$-cocycle 
\[
  a[\Lj_{\chi_{j_{\H}}}/\Lj_{\chi_{k}}] \colon W_F \rightarrow \hat{\T}
\]
by $a[\Lj_{\chi_{j_{\H}}}/\Lj_{\chi_{k}}]\cdot \Lj_{\chi_{k}}=\h\xi\circ\Lj_{\chi_{j_{\H}}}$.
Here, the $\Gamma$-action on $\hat{\T}$ on the target is the one coming from the embedding $\hat{\bfS}_k$ through the isomorphism $\hat{\bfS}_k\cong\hat{\bfT}$.

We furthermore introduce one more $L$-embedding $\Lj_{\inf(\chi_{k}^{\res})}\colon\L\bfS_{k}\hookrightarrow\L\G$.
For this, we first define a set of $\chi$-data $\inf(\chi_{k}^{\res})$ of $\Phi(\G,\bfS_{k})$ by inflating the set of $\chi$-data $\chi_{k}^{\res}$ of $\Phi_{\res}(\G,\bfS_{k})$ along the natural restriction map $\Phi(\G,\bfS_{k})\twoheadrightarrow\Phi_{\res}(\G,\bfS_{k})$ (see \cite[Definition 5.14]{Kal19-cover}) and then apply the Langlands--Shelstad construction.

We define $1$-cocycles
\begin{itemize}
  \item
   $a[\Lj_{\inf(\chi_{k}^{\res})}/\Lj_{\chi_{k}}]\colon W_F \rightarrow \hat{\T}$, 
   \item 
   $a[\Lj_{\chi_{k}^{\res}}^{1}/\Lj_{\inf(\chi_{k}^{\res})}]\colon W_F \rightarrow \hat{\T}$, and 
   \item 
   $a[\Lj_{\chi_{j_{\H}}}/\Lj^{\H}_{\chi_{k}^{\res}}]\colon W_F \rightarrow \hat{\T}_\bfH$
\end{itemize} 
in a similar way to $a[\Lj_{\chi_{j_{\H}}}/\Lj_{\chi_{k}}] \colon W_F \rightarrow \hat{\T}$.

\[
\xymatrix{
\L\H\ar@/^10pt/^-{\h\xi}[rr]&\L\G^{1}\ar@{^{(}->}[r]&\L\G\\
\L\bfS_{j_{\H}}\ar^-{\cong}[r] \ar@{^{(}->}_{\Lj_{\chi^{\res}_{k}}^{\H}}^{\Lj_{\chi_{j_{\H}}}}[u] &\L(\bfS_{k,\theta_{\bfS_{k}}})\ar@{^{(}->}^-{\Lj_{\chi_{k}^{\res}}^{1}}[u] \ar@{^{(}->}[r]& \L\bfS_{k}\ar@{^{(}->}_-{\Lj_{\chi_{k}}}^-{\Lj_{\inf(\chi_{k}^{\res})}}[u]
}
\]

\begin{lem}\label{lem:Delta-III-diff}
For any $t\in S_{k}$, we have
\begin{enumerate}
\item
$\langle t, a[\Lj_{\chi_{j_{\H}}}/\Lj_{\chi_{k}}]\rangle_{\TN}=\vartheta_{k}/\vartheta_{j_{\H}}(s)$,
\item
$\langle t, a[\Lj_{\inf(\chi_{k}^{\res})}/\Lj_{\chi_{k}}]\rangle_{\TN}=\zeta_{\desc}(s)^{-1}$,
\item
$\langle t, a[\Lj_{\chi_{k}^{\res}}^{1}/\Lj_{\inf(\chi_{k}^{\res})}]\rangle_{\TN}=1$,
\item
$\langle t, a[\Lj_{\chi_{j_{\H}}}/\Lj^{\H}_{\chi_{k}^{\res}}]\rangle_{\TN}=\zeta_{\chi_{j_{\H}}/\chi_{k}^{\res},S_{j_\bfH}}(s)$.
\end{enumerate}
Here, in (1) and (4), $t$ is also regarded as an element of $S_{j_{\H}}$ via $S_{k}\rightarrow S_{j_{\H}}$.
\end{lem}

\begin{proof}
We first consider (1).
Recall that we have $\Lj_{\chi}\circ\phi_{\vartheta}=\h\xi\circ\Lj_{\chi_{\H}}\circ\phi_{\vartheta_{\H}}$ (see Section \ref{subsec:L-packet-descent}, especially, \eqref{diag:L-parameters}).
Since we assumed that $\chi=\chi_{\vartheta_{\h\j}}$ and $\chi_{\H}=\chi_{\vartheta_{\h\j_{\H}}}$ (see the beginning of Section \ref{subsubsec:initial}), this identity can be rewritten as $\Lj_{\chi_{k}}\circ\phi_{\vartheta_{k}}=\h\xi\circ\Lj_{\chi_{j_{\H}}}\circ\phi_{\vartheta_{j_{\H}}}$.
This implies that $\phi_{\vartheta_{k}}=a[\Lj_{\chi_{j_{\H}}}/\Lj_{\chi_{k}}]\cdot\phi_{\vartheta_{j_{\H}}}$.
Hence we get the identity (1).

We next consider (2).
We note that the set $\chi_{k}$ of $\chi$-data is minimally ramified, which is obtained by the ``minimalization'' of $\inf(\chi_{k}^{\res})$ (\cite[Definition 5.24]{Kal19-cover}).
Therefore the claimed identity is a direct consequence of \cite[Proposition 5.25]{Kal19-cover}.
(Just note that the integers ``$e(\alpha/\alpha_{\res})$'' in \cite[Proposition 5.25]{Kal19-cover} are all equal to $1$, which can be checked by looking at the proofs of Lemma \ref{lem:ram-res-from-asym} and \ref{lem:ram-res-from-ur}).

It is a routine work to check the assertion (3) by going back to the Langlands--Shelstad construction.

The last assertion (4) is also a direct consequence of the definition of the character $\zeta_{\chi_{j_{\H}}/\chi_{k}^{\res},S_{j_\bfH}}$ (cf.\ the second paragraph of \cite[Proof of Proposition 5.2.7]{Kal19}).
\end{proof}

\begin{proof}[Proof of Proposition \ref{prop:delta-III-ratio}]
By Proposition \ref{prop:Delta-III-same-torus}, we have
\[
\Delta_{\III}[a_{k}^{\res},\chi_{k}^{\res}](\b\gamma,\b\delta;\b\gamma',\b\delta')
=
\langle \b\delta/\b\delta',a[\Lj^{\H}_{\chi_{k}^{\res}}/\Lj^{1}_{\chi_{k}^{\res}}]\rangle_{\TN}.
\]
We note that $a[\Lj_{\chi_{j_{\H}}}/\Lj_{\chi_{k}}]$ is equal to
\[
a[\Lj_{\chi_{j_{\H}}}/\Lj^{\H}_{\chi_{k}^{\res}}]
\cdot a[\Lj^{\H}_{\chi_{k}^{\res}}/\Lj^{1}_{\chi_{k}^{\res}}]
\cdot a[\Lj_{\chi_{k}^{\res}}^{1}/\Lj_{\inf(\chi_{k}^{\res})}]
\cdot a[\Lj_{\inf(\chi_{k}^{\res})}/\Lj_{\chi_{k}}].
\]
Hence, Lemma \ref{lem:Delta-III-diff} implies that
\[
\vartheta_{k}/\vartheta_{j_{\H}}(\b\delta/\b\delta')
=
\zeta_{\chi_{j_{\H}}/\chi_{k}^{\res},S_{j_{\H}}}(\b\delta/\b\delta')
\cdot\Delta_{\III}[a_{k}^{\res},\chi_{k}^{\res}](\b\gamma,\b\delta;\b\gamma',\b\delta')
\cdot\zeta_{\desc}(\b\delta/\b\delta')^{-1}.
\]
Since the image of $\b\delta/\b\delta'$ under the map $S_{k}\rightarrow S_{j_{\H}}$ is $\b\gamma/\b\gamma'$, we get the assertion.
(Also note that $\zeta_{\chi^{\res}_{j}/\chi_{j_{\H}},S_{j_{\H}}}=\zeta_{\chi_{j_{\H}}/\chi_{k}^{\res},S_{j_{\H}}}^{-1}$.)
\end{proof}

\subsection{Appearance of the spectral transfer factor}
We start with rewriting the endoscopic side of \eqref{eq:TECR}.
We put $\Phi_{\phi}^{\H,\st}\colonequals\sum_{\pi_{\H}\in\Pi_{\phi_{\H}}^{\H}}\Phi_{\pi_{\H}}$.
By Lemma \ref{lem:index-set}, we have
\begin{align}\label{H-side}
\sum_{\gamma\in H_{\srs}/{\sim_{\H}}} \mr{\Delta}(\gamma,\delta)\Phi_{\phi}^{\H,\st}(\gamma)
=
\sum_{y\in \mfH_{\eta}} \frac{1}{|\pi_{0}(\H^{y})(F)|}\sum_{z\in H_{y}[\delta]_{r}/{\sim_{\H_{y}}}}
\!\!\!\!\!\!
\mr{\Delta}(yz,\delta)\Phi_{\phi}^{\H,\st}(yz).
\end{align}
In the following, we put $X\colonequals \log(\delta_{\geq r})\in\mfg_{\eta}$.
Let $z\in H_{y}[\delta]_{r}$ and we put $Y\colonequals \log(z)\in\mfh_{y}$.
By Lemma \ref{lem:norm-descent}, there exists a unique $D\in\bbD(y,\eta)$ such that $(Y,X)$ is a $D$-norm pair.
Therefore, by Proposition \ref{prop:NST}, we have
\[
\mr{\Delta}\!^{D}(\bar{Y},X_{\sc})
=
\begin{cases}
\mr{\Delta}(yz,\delta)\cdot\Delta_{\IV}(y,\eta) & \text{for a unique $D\in\bbD(y,\eta)$,}\\
0 & \text{otherwise.}
\end{cases}
\]
Thus, by also using Lemma \ref{lem:norm-tran}, we see that the right-hand side of \eqref{H-side} equals
\begin{multline}\label{H-side-2}
\sum_{y\in\mfH_{\eta}} \frac{1}{|\pi_{0}(\H^{y})(F)|}\sum_{D\in\bbD(y,\eta)}\sum_{Y\overset{D}{\leftrightarrow} X/{\sim_{\H_{y}}}}
\mr{\Delta}\!^{D}(\bar{Y},X_{\sc})\Delta_{\IV}(y,\eta)^{-1}
 \Phi_{\phi}^{\H,\st}(yz).
\end{multline}

Now we utilize Kaletha's normalized and stable version of the character formula for toral supercuspidal representations (\cite[Lemma 6.3.1]{Kal19}):
\[
 \Phi_{\phi}^{\H,\st}(yz)
=
\varepsilon(\T_{\H})\varepsilon(\T_{\H_{y}})^{-1}
\sum_{j_{\H}\in\mcJ_{\H_{y}}^{\H}}\Delta_{\II}^{\H}[a_{j_{\H}},\chi_{j_{\H}}](y)\vartheta_{j_{\H}}(y)
\sum_{k_{\H}\in \mcJ^{\H_{y}}_{H_{y}}(j_{\H})} \hat{\iota}^{\H_{y}}_{Y_{k_{\H}}^{\ast}}(Y),
\]
where $a_{j_{\H}}\colonequals a_{\vartheta_{j_{\H}}}$, $\chi_{j_{\H}}\colonequals \chi_{\vartheta_{j_{\H}}}$ (see Section \ref{subsubsec:a-chi-res}).
Note that the above formula is simplified compared to \cite[Lemma 6.3.1]{Kal19} because now $\H_{y}$ is quasi-split and $\langle\inv(j_{\H,\mfw_{\H}},k_{\H}),1\rangle=1$ since the groups $S_{\phi_{\H}}^{+}$ is abelian. 
Thus \eqref{H-side-2} equals
\begin{multline}\label{H-side-3}
\sum_{y\in\mfH_{\eta}} \frac{\varepsilon(\T_{\H})\varepsilon(\T_{\H_{y}})^{-1}}{|\pi_{0}(\H^{y})(F)|}\sum_{D\in\bbD(y,\eta)}\sum_{j_{\H}\in\mcJ^{\H}_{\H_{y}}}\Delta_{\IV}(y,\eta)^{-1}\\
\cdot\Delta_{\II}^{\H}[a_{j_{\H}},\chi_{j_{\H}}](y)\vartheta_{j_{\H}}(y)
\sum_{Y\overset{D}{\leftrightarrow} X/{\sim_{\H_{y}}}}
\mr{\Delta}\!^{D}(\bar{Y},X_{\sc})
\sum_{k_{\H}\in \mcJ^{\H_{y}}_{H_{y}}(j_{\H})} \hat{\iota}^{\H_{y}}_{Y_{k_{\H}}^{\ast}}(Y).
\end{multline}

By Proposition \ref{prop:descent-lemma}, the first three index sets with $1/|\pi_{0}(\H^{y})(F)|$ are combined into one index set $\t{\mcJ}^{\G}_{\G_{\eta}}$.
Hence we can rewrite the above sum as
\begin{multline}\label{H-side-4}
\sum_{j\in\t{\mcJ}^{\G}_{\G_{\eta}}} \varepsilon(\T_{\H})\varepsilon(\T_{\H_{y}})^{-1}\Delta_{\IV}(y,\eta)^{-1}\Delta_{\II}^{\H}[a_{j_{\H}},\chi_{j_{\H}}](y)\vartheta_{j_{\H}}(y)\\
\cdot\sum_{Y\overset{D}{\leftrightarrow}X/{\sim_{\H_{y}}}}
\mr{\Delta}\!^{D}(\bar{Y},X_{\sc})
\sum_{k_{\H}\in \mcJ^{\H_{y}}_{H_{y}}(j_{\H})} \hat{\iota}^{\H_{y}}_{Y_{k_{\H}}^{\ast}}(Y).
\end{multline}
Here, for each $j\in\t{\mcJ}^{\G}_{\G_{\eta}}$, we let $y\in\mfH_{\eta}$ and $j_{\H}\in\mcJ^{\H}_{\H_{y}}$ denote the unique (up to $\pi_{0}(\H^{y})(F)$-action) elements determined by Proposition \ref{prop:descent-lemma}.
By applying the Lie algebra transfer for twisted endoscopy (Proposition \ref{prop:Lie-twisted-TECR})
\[
\sum_{Y\overset{D}{\leftrightarrow}X /{\sim_{\H_{y}}}}
\mr{\Delta}\!^{D}(\bar{Y},X_{\sc})
\sum_{k_{\H}\in \mcJ^{\H_{y}}_{H_{y}}(j_{\H})} D^{\H_{y}}_{Y_{k_{\H}}^{\ast}}(Y)
=
\sum_{k\in \t\mcJ^{\G_{\eta}}_{G_{\eta}}(j)}
\mr{\Delta}\!^{D}(\bar{Y}_{j_{\H}}^{\ast},X^{\ast}_{k,\sc})
D^{\G_{\eta}}_{X^{\ast}_{k}}(X)
\]
to the last double sum of \eqref{H-side-4}, we see that \eqref{H-side-4} is equal to
\begin{multline}\label{H-side-5}
\sum_{j\in\t{\mcJ}^{\G}_{\G_{\eta}}}
\varepsilon(\T_{\H})\varepsilon(\T_{\H_{y}})^{-1}\Delta_{\IV}(y,\eta)^{-1}\Delta_{\II}^{\H}[a_{j_{\H}},\chi_{j_{\H}}](y)\vartheta_{j_{\H}}(y)\\
 \cdot\gamma(\mfg_{\eta})\gamma(\mfh_{y})^{-1}
\sum_{k\in \mcJ^{\G_{\eta}}_{G_{\eta}}(j)}
\mr{\Delta}\!^{D}(\bar{Y}_{j_{\H}}^{\ast},X^{\ast}_{k,\sc})
\hat{\iota}^{\G_{\eta}}_{X^{\ast}_{k}}(X)
\end{multline}
(recall that $D^{\H_{y}}_{Y_{k_{\H}}^{\ast}}(Y)=\gamma(\mfh_{y})\hat{\iota}^{\H_{y}}_{Y_{k_{\H}}^{\ast}}(Y)$ and $D^{\G_{\eta}}_{X^{\ast}_{k}}(X)=\gamma(\mfg_{\eta})\hat{\iota}^{\G_{\eta}}_{X^{\ast}_{k}}(X)$).
Again using Proposition \ref{prop:NST}, this equals
\begin{multline}\label{H-side-6}
\sum_{j\in\t{\mcJ}^{\G}_{\G_{\eta}}}
\varepsilon(\T_{\H})\varepsilon(\T_{\H_{y}})^{-1}\Delta_{\II}^{\H}[a_{j_{\H}},\chi_{j_{\H}}](y)\vartheta_{j_{\H}}(y)\\
 \cdot\gamma(\mfg_{\eta})\gamma(\mfh_{y})^{-1}\sum_{k\in \mcJ^{\G_{\eta}}_{G_{\eta}}(j)} 
\mr{\Delta}\bigl(y\exp(Y_{j_{\H}}^{\ast}), \eta\exp(X_{k}^{\ast})\bigr)
\hat{\iota}^{\G_{\eta}}_{X_{k}^{\ast}}(X).
\end{multline}

By definition, $\mr{\Delta}$ is given by the product of $\varepsilon(\T_{\G_{\theta}})\varepsilon(\T_{\H})^{-1}$, $\Delta_{\I}$, $\Delta_{\II}$, and $\Delta_{\III}$.
In the following, we choose the $a$-data $a_{k}^{\res}$ and $\chi$-data $\chi_{k}^{\res}$ as in Section \ref{subsubsec:a-chi-res} to compute these factors.
Then, in summary, the $\H$-side \eqref{H-side-6} equals
\begin{multline}\label{H-side-7}
\sum_{j\in\t{\mcJ}^{\G}_{\G_{\eta}}}
\varepsilon(\T_{\G_{\theta}})\varepsilon(\T_{\H_{y}})^{-1}
\cdot\Delta_{\II}^{\H}[a_{j_{\H}},\chi_{j_{\H}}](y)\vartheta_{j_{\H}}(y)
\gamma(\mfg_{\eta})\gamma(\mfh_{y})^{-1}\\
\cdot\sum_{k\in \mcJ^{\G_{\eta}}_{G_{\eta}}(j)} 
\Delta_{\I,\II,\III}[a_{k}^{\res},\chi_{k}^{\res}]\bigl(y\exp(Y_{j_{\H}}^{\ast}), \eta\exp(X_{k}^{\ast})\bigr)
\hat{\iota}^{\G_{\eta}}_{X_{k}^{\ast}}(X).
\end{multline}

Now we are reduced to comparing the above to the $\G$-side of \eqref{eq:TECR}.
Let $\{\Delta^{\spec}_{\phi,j}\}_{j\in\mcJ^{\G}_{G}}$ be any family of constants such that $\Delta^{\spec}_{\phi,j}=0$ for any $j\in\mcJ^{\G}_{G}\smallsetminus\t\mcJ^{\G}_{G}$.\symdef{Delta-spec}{$\Delta^{\spec}_{\phi,j}$}
By Proposition \ref{prop:TCF-final}, $\sum_{j\in\t{\mcJ}^{\G}_{G}}\Delta^{\spec}_{\phi,j}\Phi_{\t{\pi}_{j}}(\delta)$ equals 
\begin{multline}\label{G-side-2-pre}
\sum_{j\in\t{\mcJ}^{\G}_{G}}\Delta^{\spec}_{\phi,j}
 C_{\ul{\eta}_{j}}
(-1)^{|\ddot{\Xi}_{\eta_{0},\ur}|}
 e(\G_{\eta_{0}}) e(\G_{\eta})
\varepsilon(\T_{\G_{\eta_{0}}^{\ast}})\varepsilon(\T_{\G_{\eta}^{\ast}})^{-1}\\
\cdot\sum_{k\in\t\mcJ^{G}_{G_{\eta}}(j)}
\vartheta_{k}(s_{k})
\cdot\epsilon_{\bfS_{k},\ram}(s_{k})
\cdot\epsilon_{\vartheta_{k}}^{\star}(s_{k})
\cdot\Delta_{\II}^{\G_{\eta_{0}}}[a_{k}^{\res},\chi_{k}^{\res}](\eta_{+}) \cdot\hat{\iota}^{\G_{\eta}}_{X_{k}^{\ast}}(X).
\end{multline}
By putting $\b{\Delta}^{\spec}_{\phi,k}\colonequals \Delta^{\spec}_{\phi,j}C_{\ul{\eta}_{j}}$\symdef{Delta-spec-bar}{$\bar{\Delta}^{\spec}_{\phi,j}$} for any $k\in\t\mcJ^{G}_{G_{\eta}}(j)$, \eqref{G-side-2-pre} equals
\begin{multline}\label{G-side-2}
\sum_{j\in\t{\mcJ}^{\G}_{G}}
\sum_{k\in\t\mcJ^{G}_{G_{\eta}}(j)}
\b\Delta^{\spec}_{\phi,k}
\cdot(-1)^{|\ddot{\Xi}_{\eta_{0},\ur}|}
\cdot e(\G_{\eta_{0}}) e(\G_{\eta})
\cdot\varepsilon(\T_{\G_{\eta_{0}}^{\ast}})\varepsilon(\T_{\G_{\eta}^{\ast}})^{-1}\\
\cdot\vartheta_{k}(s_{k})
\cdot\epsilon_{\bfS_{k},\ram}(s_{k})
\cdot\epsilon_{\vartheta_{k}}^{\star}(s_{k})
\cdot\Delta_{\II}^{\G_{\eta_{0}}}[a_{k}^{\res},\chi_{k}^{\res}](\eta_{+}) \cdot\hat{\iota}^{\G_{\eta}}_{X_{k}^{\ast}}(X).
\end{multline}
Therefore it suffices to prove that, for every $j\in\t\mcJ^{\G}_{\G_{\eta}}$, the contribution of each $k\in\t\mcJ^{\G_{\eta}}_{G_{\eta}}(j)$ to the $\H$-side \eqref{H-side-7}
\begin{multline}\label{H-side-summand}
\varepsilon(\T_{\G_{\theta}})\varepsilon(\T_{\H_{y}})^{-1}
\cdot\Delta_{\II}^{\H}[a_{j_{\H}},\chi_{j_{\H}}](y)\vartheta_{j_{\H}}(y)
\cdot\gamma(\mfg_{\eta})\gamma(\mfh_{y})^{-1}\\
\cdot\Delta_{\I,\II,\III}[a_{k}^{\res},\chi_{k}^{\res}]\bigl(y\exp(Y_{j_{\H}}^{\ast}), \eta\exp(X_{k}^{\ast})\bigr)
\end{multline}
(other than the Fourier transform of the orbital integral $\hat{\iota}^{\G_{\eta}}_{X_{k}^{\ast}}(X)$) is equal to that to the $\G$-side \eqref{G-side-2}
\begin{multline}\label{G-side-summand}
\b\Delta^{\spec}_{\phi,k}
\cdot(-1)^{|\dot{\Xi}_{\eta_{0},\ur}|}
\cdot
e(\G_{\eta_{0}})e(\G_{\eta})
\cdot \varepsilon(\T_{\G_{\eta_{0}}^{\ast}})\varepsilon(\T_{\G_{\eta}^{\ast}})^{-1}\\
\cdot\vartheta_{k}(s_{k})
\cdot\epsilon_{\bfS_{k},\ram}(s_{k})
\cdot\epsilon_{\vartheta_{k}}^{\star}(s_{k})
\cdot\Delta_{\II}^{\G_{\eta_{0}}}[a_{k}^{\res},\chi_{k}^{\res}](\eta_{+}).
\end{multline}
Hence let us just define $\Delta^{\spec}_{\phi,k}$, or equivalently $\b\Delta^{\spec}_{\phi,k}$, so that \eqref{H-side-summand} equals \eqref{G-side-summand}:
\begin{multline}\label{spectral}
\b\Delta^{\spec}_{\phi,k}
\colonequals 
\frac{\varepsilon(\T_{\G_{\eta}^{\ast}})\cdot\varepsilon(\T_{\H_{y}})^{-1}\cdot\gamma(\mfg_{\eta})\gamma(\mfh_{y})^{-1}}{e(\G_{\eta})}
\cdot
\frac{\Delta_{\II}^{\H}[a_{j_{\H}},\chi_{j_{\H}}](y)}{\Delta_{\II}^{\G_{\eta_{0}}}[a_{k}^{\res},\chi_{k}^{\res}](\eta_{+})}\\
\cdot\frac{\varepsilon(\T_{\G_{\theta}})\cdot\varepsilon(\T_{\G_{\eta_{0}}^{\ast}})^{-1}}{e(\G_{\eta_{0}})\cdot(-1)^{|\dot{\Xi}_{\eta_{0},\ur}|}\cdot \epsilon_{\bfS_{k},\ram}(s_{k})\cdot\epsilon_{\vartheta_{k}}^{\star}(s_{k})}\\
\cdot\frac{\vartheta_{j_{\H}}(y)}{\vartheta_{k}(s_{k})}\cdot\Delta_{\I,\II,\III}[a_{k}^{\res},\chi_{k}^{\res}]\bigl(y\exp(Y_{j_{\H}}^{\ast}), \eta\exp(X_{k}^{\ast})\bigr).
\end{multline}\symdef{Delta-spec-bar}{$\bar{\Delta}^{\spec}_{\phi,j}$}
Then the problem is that this quantity a priori heavily depends on $k\in \t\mcJ^{\G_{\eta}}_{G_{\eta}}(j)$, $\eta$, and $y$.
What we have to do now is to check the well-definedness of $\Delta^{\spec}_{\phi,k}$; in other words, 
\begin{enumerate}
\item
$\b\Delta^{\spec}_{\phi,k}$ is constant for $k\in\t\mcJ^{G}_{G_{\eta}}(j)$, and
\item
$\b\Delta^{\spec}_{\phi,k}$ is independent of $\eta$ and $y$.
\end{enumerate}

We first recall the following formula of Kaletha--Kottwitz:

\begin{prop}[{\cite[Lemma 4.8, Theorem 4.10]{Kal15}}]\label{prop:Kaletha--Kottwitz}
Let $\J$ be a connected reductive group over $F$ and $\bfS_{\J}$ an $F$-rational maximal torus of $\J$.
We fix a $J$-invariant symmetric non-degenerate bilinear form $B_{\mfj}$ on $\mfj$.
Then we have
\[
\varepsilon(\bfS_{\J})\varepsilon(\T_{\J^{\ast}})^{-1}
=
e(\J)\gamma(\mfj)\prod_{\alpha\in\dot{\Phi}(\J,\bfS_{\J})_{\sym}}\kappa_{\alpha}(B_{\mfj,\alpha})^{-1},
\]
where $\T_{\J^{\ast}}$ denotes a minimal Levi subgroup of the quasi-split inner form of $\J$ and $B_{\mfj,\alpha}\colonequals B_{\mfj}(X_{\alpha},Y_{\alpha})\in F_{\pm\alpha}^{\times}$ for any elements $X_{\alpha}\in\bmfj_{\alpha}(F_{\alpha})$ and $Y_{\alpha}\in\bmfj_{-\alpha}(F_{\alpha})$ satisfying $[X_{\alpha},Y_{\alpha}]=H_{\alpha} (\colonequals d\alpha^{\vee}(1))$.
\end{prop}

\begin{lem}\label{lem:spec-1}
We have
\[
\frac{\varepsilon(\T_{\G_{\eta}^{\ast}})\cdot\varepsilon(\T_{\H_{y}})^{-1}\cdot\gamma(\mfg_{\eta})\gamma(\mfh_{y})^{-1}}{e(\G_{\eta})}
=
\frac{\Delta_{\II}^{\H_{y}}[a_{j_{\H}},\chi_{j_{\H}}](\exp(Y_{j_{\H}}^{\ast}))}{\Delta_{\II}^{\G_{\eta}}[a_{k}^{\res},\chi_{k}^{\res}](\exp(X_{k}^{\ast}))}.
\]
\end{lem}

\begin{proof}
By Proposition \ref{prop:Kaletha--Kottwitz}, we have
\begin{align*}
\varepsilon(\bfS_{k}^{\nat})\cdot\varepsilon(\T_{\G_{\eta}^{\ast}})^{-1}
&=
e(\G_{\eta})\gamma(\mfg_{\eta})\prod_{\alpha_{\eta}\in\dot{\Phi}(\G_{\eta},\bfS_{k}^{\nat})_{\sym}}\kappa_{\alpha_{\eta}}(B_{\mfg_{\eta},\alpha_{\eta}})^{-1},\\
\varepsilon(\bfS_{j_{\H}})\cdot\varepsilon(\T_{\H_{y}})^{-1}
&=
\gamma(\mfh_{y})\prod_{\alpha_{y}\in\dot{\Phi}(\H_{y},\bfS_{j_{\H}})_{\sym}}\kappa_{\alpha_{y}}(B_{\mfh_{y},\alpha_{y}})^{-1}.
\end{align*}
(note that $e(\H_{y})=1$ since $\H_{y}$ is quasi-split).
Hence, by noting that $X^{\ast}(\bfS_{j_{\H}})_{\C}\cong X^{\ast}(\bfS_{k}^{\nat})_{\C}$, the left-hand side of the assertion is equal to
\begin{align}\label{kappa}
\prod_{\alpha_{\eta}\in\dot{\Phi}(\G_{\eta},\bfS_{k}^{\nat})_{\sym}}\kappa_{\alpha_{\eta}}(B_{\mfg_{\eta},\alpha_{\eta}})
\prod_{\alpha_{y}\in\dot{\Phi}(\H_{y},\bfS_{j_{\H}})_{\sym}}\kappa_{\alpha_{y}}(B_{\mfh_{y},\alpha_{y}})^{-1}.
\end{align}

This can be computed by the same argument as in the final paragraph of the proof of \cite[Theorem 6.3.4]{Kal19} as follows.
For any $\alpha_{\eta}\in\Phi(\G_{\eta},\bfS_{k}^{\nat})_{\sym}$, we have
\[
\lim_{m\rightarrow\infty}\frac{\alpha_{\eta}(\exp(p^{2m}X^{\ast}_{k}))-1}{p^{2m}}
=d\alpha_{\eta}(X^{\ast}_{k}),
\]
where $X^{\ast}_{k}\in\mfg_{\eta}^{\ast}$ is regarded as an element of $\mfg_{\eta}$ via non-degenerate bilinear form $B_{\mfg_{\eta}}$ on $\mfg_{\eta}$.
By noting that $B_{\mfg_{\eta}}$ is an invariant bilinear form, we have
\[
\langle H_{\alpha_{\eta}},X^{\ast}_{k}\rangle
=B_{\mfg_{\eta}}(X^{\ast}_{k},H_{\alpha_{\eta}})
=B_{\mfg_{\eta}}(X^{\ast}_{k},[X_{\alpha_{\eta}},Y_{\alpha_{\eta}}])
=B_{\mfg_{\eta}}([X^{\ast}_{k},X_{\alpha_{\eta}}],Y_{\alpha_{\eta}}).
\]
Since we have $[X^{\ast}_{k},X_{\alpha_{\eta}}]=d\alpha_{\eta}(X^{\ast}_{k})X_{\alpha_{\eta}}$, we get $\langle H_{\alpha_{\eta}},X^{\ast}_{k}\rangle=d\alpha_{\eta}(X^{\ast}_{k})\cdot B_{\mfg_{\eta},\alpha_{\eta}}$.
Hence, as we have $a^{\res}_{k,\alpha_{\eta}}=\langle H_{\alpha_{\eta}},X^{\ast}_{k}\rangle$, we get
\[
\chi_{k,\alpha_{\eta}}^{\res}\biggl(\frac{\alpha_{\eta}(\exp(X^{\ast}_{k}))-1}{a_{k,\alpha_{\eta}}^{\res}}\biggr)
=
\chi_{k,\alpha_{\eta}}^{\res}\biggl(\frac{\alpha_{\eta}(\exp(p^{2m}X^{\ast}_{k}))-1}{a_{k,\alpha_{\eta}}^{\res}}\biggr)
=\kappa_{\alpha_{\eta}}(B_{\mfg_{\eta},\alpha_{\eta}})^{-1},
\]
where we used that $\chi_{k}^{\res}$ is minimally ramified in the first equality.

Similarly, for any $\alpha_{y}\in\Phi(\bfS_{\H},\H_{y})_{\sym}$, we have
\[
\lim_{m\rightarrow\infty}\frac{\alpha_{y}(\exp(p^{2m}Y^{\ast}_{j_{\H}}))-1}{p^{2m}}
=d\alpha_{y}(Y^{\ast}_{j_{\H}})
=\langle H_{\alpha_{y}},Y^{\ast}_{j_{\H}}\rangle\cdot B_{\mfh_{y},\alpha_{y}}^{-1}.
\]
Since we have $a_{j_{\H},\alpha_{y}}=\langle H_{\alpha_{y}},Y^{\ast}_{j_{\H}}\rangle$, we get
\[
\chi_{j_{\H},\alpha_{y}}\biggl(\frac{\alpha_{y}(\exp(Y^{\ast}_{j_{\H}}))-1}{a_{j_{\H},\alpha_{y}}}\biggr)
=\chi_{j_{\H},\alpha_{y}}\biggl(\frac{\alpha_{y}(\exp(p^{2m}Y^{\ast}_{j_{\H}}))-1}{a_{j_{\H},\alpha_{y}}}\biggr)
=\kappa_{\alpha_{y}}(B_{\mfh_{y},\alpha_{y}})^{-1}.
\]
Therefore we see that \eqref{kappa} is given by the ratio of $\Delta_{\II}^{\H_{y}}[a_{j_{\H}},\chi_{j_{\H}}](\exp(Y_{j_{\H}}^{\ast}))$ to $\Delta_{\II}^{\G_{\eta}}[a_{k}^{\res},\chi_{k}^{\res}](\exp(X_{k}^{\ast}))$ as in the right-hand side of the assertion.
\end{proof}

From Lemma \ref{lem:spec-1} and the descent properties of the second transfer factors (both in the twisted and untwisted cases, Lemma \ref{lem:tran-II-descent} and \cite[Lemma 4.6.7]{Kal19}), we see that the defining formula \eqref{spectral} of $\b\Delta^{\spec}_{\phi,k}$ equals 
\begin{multline}\label{spectral-2}
\Delta_{\II}^{\t\G}[a_{k}^{\res},\chi_{k}^{\res}](\eta_{0})
\cdot\chi_{k}^{\res}(\eta_{0})
\cdot \frac{\Delta_{\II}^{\H}[a_{j_{\H}},\chi_{j_{\H}}](y\exp(Y_{j_{\H}}^{\ast}))}{\Delta_{\II}^{\H}[a_{k}^{\res},\chi_{k}^{\res}](y\exp(Y_{j_{\H}}^{\ast}))}\\
\cdot\frac{\varepsilon(\T_{\G_{\theta}})\cdot\varepsilon(\T_{\G_{\eta_{0}}^{\ast}})^{-1}}{e(\G_{\eta_{0}})\cdot(-1)^{|\dot{\Xi}_{\eta_{0},\ur}|}\cdot \epsilon_{\bfS_{k},\ram}(s_{k})\cdot\epsilon_{\vartheta_{k}}^{\star}(s_{k})}\\
\cdot\frac{\vartheta_{j_{\H}}(y)}{\vartheta_{k}(s_{k})}\cdot\Delta_{\I,\III}[a_{k}^{\res},\chi_{k}^{\res}]\bigl(y\exp(Y_{j_{\H}}^{\ast}), \eta\exp(X_{k}^{\ast})\bigr),
\end{multline}
where $\chi_{k}^{\res}(\eta_{0})$ is as in Lemma \ref{lem:tran-II-descent}:
\[
\chi_{k}^{\res}(\eta_{0})
\colonequals 
\prod_{\alpha_{\res}\in\dot{\Phi}(\G_{\eta_{0}},\bfS_{k}^{\nat})}
\chi_{k,\alpha_{\res}}^{\res}(l_{\alpha}).
\]
Since $\chi_{k}^{\res}(\eta_{0})$ is independent of $\eta_{0}$ by Lemma \ref{lem:chi-constant}, let us write $\chi_{k}^{\res}(\bfS_{k}^{\nat})$\symdef{chi-res-k-S}{$\chi_{k}^{\res}(\bfS_{k}^{\nat})$} for $\chi_{k}^{\res}(\eta_{0})$ in the following.

\begin{lem}\label{lem:spec-3}
The quantity $\Delta_{\II}^{\t\G}[a_{k}^{\res},\chi_{k}^{\res}](\eta_{0})$ is given by
\[
\lambda_{k,\ur}^{\res}\cdot(-1)^{r_{k,\ur}^{\res}+|\dot{\Xi}_{\eta_{0},\ur}|}
\cdot
\prod_{\alpha_{\res}\in\dot{\Phi}(\G_{\eta_{0}},\bfS_{k}^{\nat})_{\ur}}
f_{(\G_{\eta_{0}},\bfS_{k}^{\nat})}(\alpha_{\res})
\cdot\prod_{\begin{subarray}{c}\alpha_{\res}\in\dot{\Phi}_{\res}(\G,\bfS_{k})_{\sym}\\ N(\alpha)(\nu_{0})\neq1\end{subarray}}
\lambda_{\alpha_{\res}},
\]
where we put $r_{k,\ur}^{\res}\colonequals \sum_{\alpha_{\res}\in\dot{\Phi}_{\res}(\G,\bfS_{k})_{\ur}}e_{\alpha_{\res}}r$ and $\lambda_{k,\ur}^{\res}=\prod_{\alpha_{\res}\in\dot{\Phi}_{\res}(\G,\bfS_{k})_{\ur}}\lambda_{\alpha_{\res}}$.
\end{lem}

\begin{proof}
By definition, we have 
\[
\Delta_{\II}^{\t\G}[a_{k}^{\res},\chi_{k}^{\res}](\eta_{0})
=
\prod_{\begin{subarray}{c}\dot{\alpha}_{\res}\in\dot{\Phi}_{\res}(\G,\bfS_{k})\\ N(\alpha)(\nu_{0})\neq1\end{subarray}}
\chi_{k,\alpha_{\res}}^{\res}\biggl(\frac{N(\alpha)(\nu_{0})-1}{a_{k,\alpha_{\res}}^{\res}}\biggr).
\]
Since $\nu_{0}\theta$ is topologically semisimple, the valuation of $N(\alpha)(\nu_{0})-1$ is zero whenever $N(\alpha)(\nu_{0})\neq1$.
Hence, for any $\alpha_{\res}\in\dot{\Phi}_{\res}(\G,\bfS_{k})$ such that $N(\alpha)(\nu_{0})\neq1$, we can compute each factor as follows (cf.\ \cite[Lemma 4.7.1]{Kal19}):
\begin{description}
\item[The case where $\alpha_{\res}$ is asymmetric]
Since $\chi_{k,\alpha_{\res}}^{\res}$ is the trivial character of $F_{\alpha_{\res}}^{\times}$ in this case, we have 
\[
\chi_{k,\alpha_{\res}}^{\res}\biggl(\frac{N(\alpha)(\nu_{0})-1}{a_{k,\alpha_{\res}}^{\res}}\biggr)=1.
\]
\item[The case where $\alpha_{\res}$ is symmetric unramified]
Since $\chi_{k,\alpha_{\res}}^{\res}$ is the unique nontrivial quadratic unramified character of $F_{\alpha_{\res}}^{\times}$ and 
\[
\val_{F}(a_{k,\alpha_{\res}}^{\res})
=\val_{F}(\langle H_{\alpha_{\res}},X_{k}^{\ast}\rangle)
=-r
\in\val_{F}(F_{\alpha_{\res}}^{\times}),
\]
we have
\[
\chi_{k,\alpha_{\res}}^{\res}\biggl(\frac{N(\alpha)(\nu_{0})-1}{a_{k,\alpha_{\res}}^{\res}}\biggr)
=
(-1)^{e_{\alpha_{\res}}r}.
\]
(Recall that $e_{\alpha_{\res}}$ denotes the ramification index of $F_{\alpha_{\res}}/F$.)
\item[The case where $\alpha_{\res}$ is symmetric ramified]
Since $\nu_{0}\theta$ is topologically semisimple and $N(\alpha)(\nu_{0})$ belongs to the kernel of the norm map $\Nr\colon F_{\alpha_{\res}}^{\times}\rightarrow F_{\pm\alpha_{\res}}^{\times}$, we have $N(\alpha)(\nu_{0})\equiv-1 \pmod{\mfp_{F_{\alpha_{\res}}}}$ whenever $N(\alpha)(\nu_{0})\neq1$.
By noting that $\chi_{k,\alpha_{\res}}^{\res}$ is tamely ramified, we get
\[
\chi_{k,\alpha_{\res}}^{\res}\biggl(\frac{N(\alpha)(\nu_{0})-1}{a_{k,\alpha_{\res}}^{\res}}\biggr)
=
\chi_{k,\alpha_{\res}}^{\res}(-2a_{k,\alpha_{\res}}^{\res,-1}).
\]
As we have $\Tr_{F_{\alpha_{\res}}/F_{\pm\alpha_{\res}}}(a_{k,\alpha_{\res}}^{\res})=0$, we have $\Nr_{F_{\alpha_{\res}}/F_{\pm\alpha_{\res}}}(a_{k,\alpha_{\res}}^{\res})=-a_{k,\alpha_{\res}}^{\res,2}$.
Hence $\chi_{k,\alpha_{\res}}^{\res}(-2a_{k,\alpha_{\res}}^{\res,-1})=\chi_{k,\alpha_{\res}}^{\res}(2a_{k,\alpha_{\res}}^{\res})=\lambda_{\alpha_{\res}}$.
\end{description}

Therefore, we get
\begin{align*}
\Delta_{\II}^{\t\G}[a_{k}^{\res},\chi_{k}^{\res}](\eta_{0})
&=
\prod_{\begin{subarray}{c}\dot{\alpha}_{\res}\in\dot{\Phi}_{\res}(\G,\bfS_{k})_{\ur}\\ N(\alpha)(\nu_{0})\neq1\end{subarray}}
(-1)^{e_{\alpha_{\res}}r}
\prod_{\begin{subarray}{c}\dot{\alpha}_{\res}\in\dot{\Phi}_{\res}(\G,\bfS_{k})_{\ram}\\ N(\alpha)(\nu_{0})\neq1\end{subarray}}\lambda_{\alpha_{\res}}\\
&=
(-1)^{r_{k,\ur}^{\res}}\prod_{\begin{subarray}{c}\dot{\alpha}_{\res}\in\dot{\Phi}_{\res}(\G,\bfS_{k})_{\ur}\\ N(\alpha)(\nu_{0})=1\end{subarray}}
(-1)^{e_{\alpha_{\res}}r}
\prod_{\begin{subarray}{c}\dot{\alpha}_{\res}\in\dot{\Phi}_{\res}(\G,\bfS_{k})_{\ram}\\ N(\alpha)(\nu_{0})\neq1\end{subarray}}\lambda_{\alpha_{\res}}.
\end{align*}

We compute $(-1)^{e_{\alpha_{\res}}r}$ for $\alpha_{\res}\in\dot{\Phi}_{\res}(\G,\bfS_{k})_{\ur}$ satisfying $N(\alpha)(\nu_{0})=1$ by noting whether $\alpha_{\res}\in\Xi_{\eta_{0}}$ (i.e., $\alpha_{\res}$ appears in the Heisenberg quotient of $\G_{\eta_{0}}$ with respect to $(\x,\frac{r}{2})$) or not.
When $\alpha_{\res}\in\Xi_{\eta_{0}}$, we have
\[
\begin{cases}
\frac{r}{2}\in e_{\alpha_{\res}}^{-1}\Z & \text{if $f_{(\G_{\eta_{0}},\bfS_{k}^{\nat})}(\alpha_{\res})=+1$,}\\
\frac{r}{2}\in e_{\alpha_{\res}}^{-1}(\Z+\frac{1}{2}) & \text{if $f_{(\G_{\eta_{0}},\bfS_{k}^{\nat})}(\alpha_{\res})=-1$,}
\end{cases}
\]
by \cite[Proposition 4.5.1]{Kal19}.
This is equivalent to that
\[
\begin{cases}
e_{\alpha_{\res}}r \equiv 0 \pmod{2} & \text{if $f_{(\G_{\eta_{0}},\bfS_{k}^{\nat})}(\alpha_{\res})=+1$,}\\
e_{\alpha_{\res}}r \equiv 1 \pmod{2} & \text{if $f_{(\G_{\eta_{0}},\bfS_{k}^{\nat})}(\alpha_{\res})=-1$.}
\end{cases}
\]
By noting that these conditions are simply swapped when $\alpha_{\res}\notin\Xi_{\eta_{0}}$, we see that
\begin{align*}
(-1)^{e_{\alpha_{\res}}r}
=
\begin{cases}
f_{(\G_{\eta_{0}},\bfS_{k}^{\nat})}(\alpha_{\res})\\
-f_{(\G_{\eta_{0}},\bfS_{k}^{\nat})}(\alpha_{\res})
\end{cases}
=
\begin{cases}
-f_{(\G_{\eta_{0}},\bfS_{k}^{\nat})}(\alpha_{\res})\cdot\lambda_{\alpha_{\res}} & \text{if $\alpha_{\res}\in\Xi_{\eta_{0}}$,} \\
f_{(\G_{\eta_{0}},\bfS_{k}^{\nat})}(\alpha_{\res})\cdot\lambda_{\alpha_{\res}} & \text{if $\alpha_{\res}\notin\Xi_{\eta_{0}}$}
\end{cases}
\end{align*}
(recall that $\lambda_{\alpha_{\res}}=-1$ since $F_{\alpha_{\res}}/F_{\pm\alpha_{\res}}$ is unramified).
Thus we get
\[
\prod_{\begin{subarray}{c}\dot{\alpha}_{\res}\in\dot{\Phi}_{\res}(\G,\bfS_{k})_{\ur}\\ N(\alpha)(\nu_{0})=1\end{subarray}}
(-1)^{e_{\alpha_{\res}}r}
=
(-1)^{|\dot{\Xi}_{\eta_{0},\ur}|}\prod_{\begin{subarray}{c}\dot{\alpha}_{\res}\in\dot{\Phi}_{\res}(\G,\bfS_{k})_{\ur}\\ N(\alpha)(\nu_{0})=1\end{subarray}}
f_{(\G_{\eta_{0}},\bfS_{k}^{\nat})}(\alpha_{\res})\cdot\lambda_{\alpha_{\res}}.
\]
Again noting that $\lambda_{\alpha_{\res}}=-1$, we have
\[
\prod_{\begin{subarray}{c}\dot{\alpha}_{\res}\in\dot{\Phi}_{\res}(\G,\bfS_{k})_{\ur}\\ N(\alpha)(\nu_{0})=1\end{subarray}}
\lambda_{\alpha_{\res}}
=
\lambda_{k,\ur}^{\res}\cdot
\prod_{\begin{subarray}{c}\dot{\alpha}_{\res}\in\dot{\Phi}_{\res}(\G,\bfS_{k})_{\ur}\\ N(\alpha)(\nu_{0})\neq1\end{subarray}}
\lambda_{\alpha_{\res}}.
\]
Recalling that $\Phi(\G_{\eta_{0}},\bfS_{k}^{\nat})=\{\alpha_{\res}\in\Phi_{\res}(\G,\bfS_{k}) \mid N(\alpha)(\nu_{0})=1\}$, we get the assertion.
\end{proof}

\begin{lem}\label{lem:Kaletha--Kottwitz2}
We have
\[
\varepsilon(\bfS_{k}^{\nat})\cdot\varepsilon(\T_{\G_{\eta_{0}}^{\ast}})^{-1}
=
e(\G_{\eta_{0}})\prod_{\dot{\alpha}_{\res}\in \dot{\Phi}(\G_{\eta_{0}},\bfS_{k}^{\nat})_{\sym}}f_{(\G_{\eta_{0}},\bfS_{k}^{\nat})}(\alpha_{\res})\cdot\lambda_{\alpha_{\res}}
\]
\end{lem}

\begin{proof}
This is a variant of the formula of Kaletha--Kottwitz (Proposition \ref{prop:Kaletha--Kottwitz}), which is stated in \cite[Corollary 4.11]{Kal15}.
\end{proof}

By noting that $\{\alpha_{\res}\in\dot{\Phi}_{\res}(\G,\bfS_{k}) \mid N(\alpha)(\nu_{0})\neq1\}=\dot{\Phi}_{\res}(\G,\bfS_{k})\smallsetminus\dot{\Phi}(\G_{\eta_{0}},\bfS_{k}^{\nat})$, Lemmas \ref{lem:spec-3} and \ref{lem:Kaletha--Kottwitz2} imply that \eqref{spectral-2} equals
\begin{multline}\label{spectral-3}
\chi_{k}^{\res}(\bfS_{k}^{\nat})
\cdot (-1)^{r_{k,\ur}^{\res}} \cdot \lambda_{k,\ur}^{\res}
\cdot \!\!\!\!\!\! \prod_{\dot{\alpha}_{\res}\in\dot{\Phi}_{\res}(\G,\bfS_{k})_{\sym}} \!\!\!\!\!\! \lambda_{\alpha_{\res}}
\cdot \!\!\!\!\!\! \prod_{\dot{\alpha}_{\res}\in\dot{\Phi}(\G_{\eta_{0}},\bfS_{k}^{\nat})_{\ram}} \!\!\!\!\!\! f_{(\G_{\eta_{0}},\bfS_{k}^{\nat})}(\alpha_{\res})\\
\cdot \frac{\Delta_{\II}^{\H}[a_{j_{\H}},\chi_{j_{\H}}](y\exp(Y_{j_{\H}}^{\ast}))}{\Delta_{\II}^{\H}[a_{k}^{\res},\chi_{k}^{\res}](y\exp(Y_{j_{\H}}^{\ast}))}
\cdot \frac{\varepsilon(\T_{\G_{\theta}})\cdot\varepsilon(\bfS_{k}^{\nat})^{-1}}{\epsilon_{\bfS_{k},\ram}(s_{k})
\cdot\epsilon_{\vartheta_{k}}^{\star}(s_{k})}\\
\cdot\frac{\vartheta_{j_{\H}}(y)}{\vartheta_{k}(s_{k})}\cdot\Delta_{\I,\III}[a_{k}^{\res},\chi_{k}^{\res}]\bigl(y\exp(Y_{j_{\H}}^{\ast}), \eta\exp(X_{k}^{\ast})\bigr).
\end{multline}

We put
\[
\chi_{j}^{\res}(\bfS_{j_{\H}})
\colonequals 
\prod_{\dot{\alpha}_{\res}\in\dot{\Phi}(\H,\bfS_{j_{\H}})}\chi_{j}^{\res}(l_{\alpha}).
\]\symdef{chi-res-j-S-j-H}{$\chi_{j}^{\res}(\bfS_{j_{\H}})$}

\begin{lem}\label{lem:Delta-II-ratio}
We have
\[
\frac{\Delta_{\II}^{\H}[a_{j_{\H}},\chi_{j_{\H}}](y\exp(Y_{j_{\H}}^{\ast}))}{\Delta_{\II}^{\H}[a_{k}^{\res},\chi_{k}^{\res}](y\exp(Y_{j_{\H}}^{\ast}))}
=
\zeta_{\chi_{j_{\H}}/\chi_{j}^{\res},S_{j_{\H}}}(y\exp(Y^{\ast}_{j_{\H}}))
\cdot
\chi_{j}^{\res}(\bfS_{j_{\H}}).
\]
\end{lem}

\begin{proof}
Since $y\exp(Y_{j_{\H}}^{\ast})$ is regular semisimple in $\H$, we have
\[
\Delta_{\II}^{\H}[a_{j_{\H}},\chi_{j_{\H}}](y\exp(Y_{j_{\H}}^{\ast}))
=
\prod_{\dot{\alpha}_{\res}\in\dot{\Phi}(\H,\bfS_{j_{\H}})}\chi_{j_{\H}}\biggl(\frac{\alpha_{\res}(y\exp(Y_{j_{\H}}^{\ast}))-1}{a_{j_{\H},\alpha_{\res}}}\biggr),
\]
\[
\Delta_{\II}^{\H}[a_{k}^{\res},\chi_{k}^{\res}](y\exp(Y_{j_{\H}}^{\ast}))
=
\prod_{\dot{\alpha}_{\res}\in\dot{\Phi}(\H,\bfS_{j_{\H}})}\chi_{k}^{\res}\biggl(\frac{\alpha_{\res}(y\exp(Y_{j_{\H}}^{\ast}))-1}{a_{k,\alpha_{\res}}^{\res}}\biggr).
\]
By Lemma \ref{lem:a-res-H}, we have $a_{k,\alpha_{\res}}^{\res}=l_{\alpha}\cdot a_{j_{\H},\alpha_{\res}}$.
Thus we get
\begin{align*}
\frac{\Delta_{\II}^{\H}[a_{j_{\H}},\chi_{j_{\H}}](y\exp(Y_{j_{\H}}^{\ast}))}{\Delta_{\II}^{\H}[a_{k}^{\res},\chi_{k}^{\res}](y\exp(Y_{j_{\H}}^{\ast}))}
=
\frac{\Delta_{\II}^{\H}[a_{j_{\H}},\chi_{j_{\H}}](y\exp(Y_{j_{\H}}^{\ast}))}{\Delta_{\II}^{\H}[a_{j_{\H}},\chi_{k}^{\res}](y\exp(Y_{j_{\H}}^{\ast}))}
\cdot
\prod_{\dot{\alpha}_{\res}\in\dot{\Phi}(\H,\bfS_{j_{\H}})}\chi_{k}^{\res}(l_{\alpha}).
\end{align*}
Here, on the right-hand side, the ratio of two second transfer factors is given by $\zeta_{\chi_{j_{\H}}/\chi_{k}^{\res},S_{j_{\H}}}(y\exp(Y^{\ast}_{j_{\H}}))$ by\cite[Lemma 4.6.6]{Kal19}.
Thus, by noting that both $\chi_{k}^{\res}$ and $\chi_{j}^{\res}$ induce the same set of $\chi$-data on $\Phi(\H,\bfS_{j_{\H}})$, we get the assertion.
\end{proof}

Now recall that Proposition \ref{prop:descent-lemma} associates to $j\in\t\mcJ^{\G}_{\G_{\eta}}$ a unique element $y\in\mfH_{\eta}$ and $(D,j_{\H})\in \bbD(y,\eta)\times\mcJ^{\H}_{\H_{y}}$.
Also recall that we have fixed an element $\ul{\eta}_{j}\in \t{S}_{j}$.
We put $\ul{y}_{j}\colonequals \t\xi_{D}(\ul{\eta}_{j})\in S_{j_{\H}}$.
Then, by Lemma \ref{lem:tran-I} and Proposition \ref{prop:delta-III-ratio}, we have
\begin{multline}\label{eq:tran-III-ratio}
\frac{\Delta_{\I,\III}[a_{k}^{\res},\chi_{k}^{\res}](y\exp(Y_{j_{\H}}^{\ast}), \eta\exp(X_{k}^{\ast}))}{\Delta_{\I,\III}[a_{k}^{\res},\chi_{k}^{\res}](\ul{y}_{j}\exp(Y_{j_{\H}}^{\ast}), \ul{\eta}_{k}\exp(X_{k}^{\ast}))}\\
=
\frac{\vartheta_{k}(s_{k})}{\vartheta_{j_{\H}}(y/\ul{y}_{j})}
\cdot\zeta_{\desc}(s_{k})
\cdot \zeta_{\chi^{\res}_{j}/\chi_{j_{\H}},S_{j_{\H}}}(y/\ul{y}_{j}).
\end{multline}
(Recall that $\eta=s_{k}\ul{\eta}_{k}$.)
Therefore, by using Lemma \ref{lem:Delta-II-ratio}, we see that \eqref{spectral-3} equals
\begin{multline}\label{spectral-4}
\chi_{k}^{\res}(\bfS_{k}^{\nat})
\cdot \chi_{j}^{\res}(\bfS_{j_{\H}})
\cdot (-1)^{r_{k,\ur}^{\res}} \cdot \lambda_{k,\ram}^{\res}
\cdot \varepsilon(\T_{\G_{\theta}})\cdot\varepsilon(\bfS_{k}^{\nat})^{-1}
\\
\cdot\vartheta_{j_{\H}}(\ul{y}_{j})
\cdot\zeta_{\chi_{j_{\H}}/\chi_{j}^{\res},S_{j_{\H}}}(\ul{y}_{j}\exp(Y^{\ast}_{j_{\H}}))
\cdot\Delta_{\I,\III}[a_{k}^{\res},\chi_{k}^{\res}](\ul{y}_{j}\exp(Y_{j_{\H}}^{\ast}), \ul{\eta}_{k}\exp(X_{k}^{\ast}))\\
\cdot\epsilon_{\bfS_{k},\ram}(s_{k})
\cdot\epsilon_{\vartheta_{k}}^{\star}(s_{k})
\cdot\zeta_{\desc}(s_{k})
\cdot\prod_{\dot{\alpha}_{\res}\in\dot{\Phi}(\G_{\eta_{0}},\bfS_{k}^{\nat})_{\ram}}f_{(\G_{\eta_{0}},\bfS_{k}^{\nat})}(\alpha_{\res}),
\end{multline}
where we put $\lambda_{k,\ram}^{\res}=\prod_{\dot{\alpha}_{\res}\in\dot{\Phi}_{\res}(\G,\bfS_{k})_{\ram}}\lambda_{\alpha_{\res}}$.

Now let us examine the factors contained in \eqref{spectral-4}.
The factor $\varepsilon(\T_{\G_{\theta}})$ is obviously independent of $j$.
Let $j'\in\t\mcJ^{\G}_{\G_{\eta}}$ and $k'\in\t\mcJ^{\G_{\eta}}_{G_{\eta}}(j')$ such that $k$ and $k'$ are $G$-conjugate.
Suppose that $g\in G$ be an element such that $k'=[g]\circ k$ and $\ul{\eta}_{k'}={}^{g}\ul{\eta}_{k}$.
Since the $F$-rational isomorphism $[g]\colon\bfS_{k}\rightarrow\bfS_{k'}$ gives a $\Gamma$-equivariant isomorphism $\Phi(\G,\bfS_{k})\rightarrow\Phi(\G,\bfS_{k'})$ compatible with twists, we get $r_{k,\ur}^{\res}=r_{k',\ur}^{\res}$, $\lambda_{k,\ram}^{\res}=\lambda_{k',\ram}^{\res}$, and $\varepsilon(\bfS_{k}^{\nat})=\varepsilon(\bfS_{k'}^{\nat})$.
Lemma \ref{lem:chi-constant} implies that $\chi_{k'}^{\res}(\bfS_{k'}^{\nat})=\chi_{k'}^{\res}({}^{g}\eta_{0})=\chi_{k}^{\res}(\eta_{0})=\chi_{k}^{\res}(\bfS_{k})$.
It is a routine work to check that the factors $\chi_{j}^{\res}(\bfS_{j_{\H}})$, $\vartheta_{j_{\H}}(\ul{y}_{j})$, $\zeta_{\chi_{j_{\H}}/\chi_{j}^{\res},S_{j_{\H}}}(\ul{y}_{j}\exp(Y^{\ast}_{j_{\H}}))$, and $\Delta_{\I,\III}[a_{k}^{\res},\chi_{k}^{\res}](\ul{y}_{j}\exp(Y_{j_{\H}}^{\ast}), \ul{\eta}_{k}\exp(X_{k}^{\ast}))$ do not change even if we replace $(j,k)$ with $(j',k')$.

We summarize our discussion so far.
We obtained 
\begin{multline}\label{eq:spectral-tf-final}
\b\Delta^{\spec}_{\phi,k}
=\chi_{k}^{\res}(\bfS_{k}^{\nat})
\cdot \chi_{j}^{\res}(\bfS_{j_{\H}})
\cdot (-1)^{r_{k,\ur}^{\res}} \cdot \lambda_{k,\ram}^{\res}
\cdot \varepsilon(\T_{\G_{\theta}})\cdot\varepsilon(\bfS_{k}^{\nat})^{-1}
\cdot\vartheta_{j_{\H}}(\ul{y}_{j})\\
\cdot\zeta_{\chi_{j_{\H}}/\chi_{j}^{\res},S_{j_{\H}}}(\ul{y}_{j}\exp(Y^{\ast}_{j_{\H}}))
\cdot\Delta_{\I,\III}[a_{k}^{\res},\chi_{k}^{\res}](\ul{y}_{j}\exp(Y_{j_{\H}}^{\ast}), \ul{\eta}_{k}\exp(X_{k}^{\ast}))\\
\cdot\epsilon_{\bfS_{k},\ram}(s_{k})
\cdot\epsilon_{\vartheta_{k}}^{\star}(s_{k})
\cdot\zeta_{\desc}(s_{k})
\cdot\prod_{\alpha_{\res}\in\dot{\Phi}(\G_{\eta_{0}},\bfS_{k}^{\nat})_{\ram}}f_{(\G_{\eta_{0}},\bfS_{k}^{\nat})}(\alpha_{\res}).
\end{multline}
Moreover, we checked that all factors contained in the first and second lines of the right-hand side of \eqref{eq:spectral-tf-final} depend only on the $G$-conjugacy class of $k$ and are independent of $\eta$.
In other words, our remaining task is to check that
\begin{align}\label{ramified-spec}
\epsilon_{\bfS_{k},\ram}(s_{k})
\cdot\epsilon_{\vartheta_{k}}^{\star}(s_{k})
\cdot\zeta_{\desc}(s_{k})
\cdot\prod_{\alpha_{\res}\in\dot{\Phi}(\G_{\eta_{0}},\bfS_{k}^{\nat})_{\ram}}f_{(\G_{\eta_{0}},\bfS_{k}^{\nat})}(\alpha_{\res}).
\end{align}
depends only on the $G$-conjugacy class of $k$ and is independent of $\eta$.

We note that all the factors in \eqref{ramified-spec} are products over sets of ramified (restricted) roots.
Thus, there is nothing to prove if $\Phi(\G,\bfS_{k})$ and $\Phi_{\res}(\G,\bfS_{k})$ do not contain a ramified symmetric element.
For example, a sufficient condition for this is that $\bfS$ splits over a finite extension $E$ of $F$ whose ramification index $e(E/F)$ is odd.
Indeed, we have the following diagram:
\[
\xymatrix@R=10pt{
&F_{\alpha_{\res}}\ar@{}[r]|*{\subset}&F_{\alpha}\ar@{}[r]|*{\subset}& E\\
F\ar@{}[r]|*{\subset}&F_{\pm\alpha_{\res}}\ar@{}[r]|*{\subset}\ar@{}[u]|{\bigcup}&F_{\pm\alpha}\ar@{}[u]|{\bigcup}
}
\]
Hence, if $e(E/F)$ is odd, then the extension $F_{\alpha}/F_{\pm\alpha}$ and $F_{\alpha_{\res}}/F_{\pm\alpha_{\res}}$ cannot be quadratic ramified.
Let us record this observation here.

\begin{thm}\label{thm:spec-well-def-unram}
The spectral transfer factor $\Delta^{\spec}_{\phi,k}$ is well-defined if $\bfS$ splits over a finite extension $E$ of $F$ whose ramification index $e(E/F)$ is odd.
In particular, the twisted endoscopic character relation (Expectation \ref{expect:TECR}) is satisfied by Kaletha's toral supercuspidal $L$-packets.
\end{thm}

What we will do in the rest of paper is to show that the quantity \eqref{ramified-spec} indeed depends only on the $G$-conjugacy class of $k$ and is independent of $\eta$ in the case where $\G=\GL_{n}$.

\begin{rem}
Recall that the members of $\Pi_{\phi}^{\G}$ are parametrized by the set $\mcJ^{\G}_{G}$.
In the case of standard endoscopy, in \cite[Section 5.3]{Kal19}, Kaletha introduced the paring
\[
\langle-,-\rangle_{\mfw}\colon\mcJ^{\G}_{G}\times\pi_{0}(S_{\phi}^{+})\rightarrow\C^{\times};\quad
\langle j,s\rangle_{\mfw}
\mapsto
\langle\inv(j_{\mfw},j),s\rangle
\]
(see \cite[1155 page]{Kal19} for the details).
This is nothing but the spectral transfer factor in the sense of this paper in the untwisted case.
In other words, we have
\[
\Delta^{\spec}_{\phi,j}
=
\langle\inv(j,j_{\mfw}),s\rangle
\]
when $\theta$ is trivial.
We may understand that Kaletha's proof of the standard endoscopic character relation (\cite[Theorem 6.3.4]{Kal19}) shows that our definition of contains $\Delta^{\spec}_{\phi,j}$ coincides with $\langle\inv(j,j_{\mfw}),s\rangle$.
\end{rem}

\section{$\GL_{n}$ consideration}\label{sec:GL}

In the following, let $\G\colonequals \GL_{n}$, where $n$ is even, and $\theta\colonequals J_{n}{}^{t}(-)^{-1}J_{n}^{-1}$, where $J_{n}$ is the anti-diagonal matrix whose $(i,n+1-i)$-entry is given by $(-1)^{i-1}$.
(Recall that the assumption on the parity of $n$ is harmless for our purpose; see Remark \ref{rem:Prasad}).

\subsection{Twisted elliptic maximal tori of $\GL_{n}$}
Let us assume that $(\t\bfS,\bfS)$ is an $F$-rational twisted maximal torus of $\G$ whose $\bfS$ is elliptic (then $(\t\bfS,\bfS)$ is necessarily elliptic).
It is well-known that there exists a finite extension $E$ of $F$ of degree $n$ such that $\bfS$ is isomorphic to $\Res_{E/F}\Gm$.

\begin{lem}\label{lem:GL-involution}
There exists an element $\tau_{\theta}\in\Aut_{F}(E)$ of order $2$ such that $\theta_{\bfS}(s)=\tau_{\theta}(s)^{-1}$ for any $s\in S\cong E^{\times}$.
\end{lem}

\begin{proof}
Note that $\theta_{\bfS}$ is of the form $x {}^{t}(-)^{-1}x^{-1}$ for some $x\in\GL_{n}(F)$.
Since $\theta_{\bfS}$ preserves $S\subset \GL_{n}(F)$, the map $X\mapsto x{}^{t}Xx^{-1}$ preserves $\mfs\subset\mathfrak{gl}_{n}(F)$.
As the map $X\mapsto x{}^{t}Xx^{-1}$ is an involutive $F$-algebra homomorphism on $\mfs\cong E$, it is given by an element $\tau_{\theta}$ of $\Aut_{F}(E)$ whose order is either $1$ or $2$.
In other words, we have $\theta_{\bfS}(s)=\tau_{\theta}(s)^{-1}$ for any $s\in S\cong E^{\times}$.

Let us show that $\tau_{\theta}$ is not trivial.
For the sake of contradiction, we suppose that $\tau_{\theta}$ is trivial.
Then the automorphism $\theta_{\bfS}$ of $\bfS$ is given by $s\mapsto s^{-1}$ on $S\cong E^{\times}$.
Hence we have $\bfS^{\nat}(F)\subset\bfS^{\theta_{\bfS}}(F)=S^{\theta_{\bfS}}=\{\pm1\}$.
On the other hand, since $(\t{\bfS},\bfS)$ is a twisted maximal torus, there exists an element $g\in \G$ such that ${}^{g}\bfS=\T$ and $\theta_{\bfS}$ is mapped to $\theta|_{\T}$.
This implies that the torus $\bfS^{\nat}$ is isomorphic to $\T^{\theta,\circ}$ over $\overline{F}$.
In particular, the rank of $\bfS^{\nat}$ is given by $n/2$.
However, there is no $F$-rational torus whose rank is nonzero such that the set of $F$-valued points is of order at most $2$.
Hence we get a contradiction.
\end{proof}

In the following, we let $\tau_{\theta}\in\Aut_{F}(E)$ be the element as in Lemma \ref{lem:GL-involution}.
Let $E_{\pm}$ be the fixed field of $\tau_{\theta}$ in $E$.

\subsection{Roots of elliptic maximal tori of $\GL_{n}$}
We next recall a description of the set of roots of $\bfS$ in $\GL_{n}$ following \cite[Sections 3.1 and 3.2]{Tam16} (see also \cite[Sections 3.2 and 5.1]{OT21}).
First we fix a set $\{g_{1},\ldots,g_{n}\}$ of representatives of the quotient $\Gamma/\Gamma_{E}$ such that $g_{1}=\id$.
Then we get an isomorphism $\bfS(\overline{F})\cong\prod_{i=1}^{n} \overline{F}^{\times}$ which maps $x\in E^{\times}\cong\bfS(F)$ to $(g_{1}(x),\ldots,g_{n}(x))$.
Then the projections 
\[
\delta_{i}
\colon 
\bfS(\overline{F})\xrightarrow{\cong}\prod_{i=1}^{n} \overline{F}^{\times}\rightarrow \overline{F}^{\times}
;\quad
(x_{1},\ldots,x_{n})\mapsto x_{i}
\]
form a $\Z$-basis of $X^{\ast}(\bfS)$.
The set $\Phi(\G,\bfS)$ of roots of $\bfS$ in $\G$ is given by
\[
\biggl\{
\begin{bmatrix}g_{i}\\g_{j}\end{bmatrix}\colonequals \delta_{i}-\delta_{j}
\,\bigg\vert\,
1\leq i \neq j \leq n
\biggr\}
\]
and the set $\dot{\Phi}(\G,\bfS)$ is described as follows:
\[
(\Gamma_{E}\backslash\Gamma/\Gamma_{E})'
\xrightarrow{1:1}
\dot{\Phi}(\G,\bfS)
;\quad
\Gamma_{E}g_{i}\Gamma_{E}\mapsto \Gamma\cdot\begin{bmatrix}1\\g_{i}\end{bmatrix},
\]
where $(\Gamma_{E}\backslash\Gamma/\Gamma_{E})'$ is the set of nontrivial double-$\Gamma_{E}$-cosets in $\Gamma$.

Suppose that $E/F$ is tamely ramified in the following.
We simply write $e$ (resp.\ $f$) for the ramification index $e(E/F)$ (resp.\ residue degree $f(E/F)$).
We first recall an explicit choice of a set of representatives of 
$\Gamma/\Gamma_E$, following \cite[Section 3.2]{Tam16}.
Let $\mu_{E}$ denote the set of roots of unity in $E$.
We take uniformizers $\varpi_{E}$ and $\varpi_{F}$ of $E$ and $F$, respectively, so that $\varpi_{E}^e=\zeta _{E/F}\varpi_{F}$ for some $\zeta_{E/F} \in \mu_{E}$.
We fix a primitive $e$-th root $\zeta_{e}$ of unity 
and an $e$-th root $\zeta_{E/F, e}$ of $\zeta_{E/F}$,
and put $\zeta_{\phi}\colonequals \zeta_{E/F, e}^{q-1}$.
Then $L\colonequals E[\zeta_{e}, \zeta_{E/F, e}]$ is a tamely ramified extension of $F$
which contains the Galois closure of $E/F$ and is unramified over $E$.
The Galois group $\Gal(L/F)$ of the extension $L/F$ is given by the semi-direct product
$\langle \sigma \rangle \rtimes \langle \phi \rangle$, where
\begin{align*}
&\sigma \colon \zeta \mapsto \zeta \quad (\zeta \in \mu_{L}), \quad \varpi_{E} \mapsto \zeta_e \varpi_{E} \\
&\phi \colon \zeta \mapsto \zeta ^q \quad (\zeta \in \mu_{L}), \quad \varpi_{E} \mapsto \zeta_{\phi}\varpi_{E}
\end{align*}
and
$\phi \sigma \phi ^{-1}=\sigma ^q$.
Moreover, as explained in \cite[Proposition 3.3 (i)]{Tam16}, we can take a set of representatives of $\Gamma/\Gamma_{E}$ to be
\[
\{\Gamma_{F}/\Gamma_{E}\}
\colonequals 
\{
\sigma^{k}\phi^{i}
\mid
0\leq k \leq e-1,\, 0\leq i \leq f-1
\}.
\]
Here we implicitly regard each $\sigma^{k}\phi^{i}\in\Gamma_{L/F}$ as an element of $\Gamma_{F}$ by taking its extension to $\overline{F}$ from $L$.
We note that, as $L/E$ is unramified, there exists an integer $c$ such that 
$\Gal(L/E)=\langle \sigma^{c} \phi^{f}\rangle$.
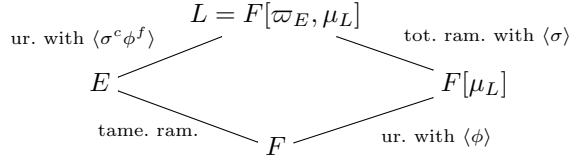
\[
\xymatrix@R=10pt{
&L=F[\varpi_{E},\mu_{L}]&\\
E\ar@{-}[ur]^-{\text{ur.\ with $\langle\sigma^{c}\phi^{f}\rangle$}\quad\quad\quad}&&F[\mu_{L}]\ar@{-}[ul]_-{\quad\quad\quad\quad\text{tot.\ ram.\ with $\langle\sigma\rangle$}}\\
&F\ar@{-}[ul]^-{\text{tame.\ ram.}\quad}\ar@{-}[ur]_-{\quad\quad\quad\text{ur.\ with $\langle\phi\rangle$}}&
}
\]

We recall a fact about symmetric ramified roots of $\bfS$ in $\G$.

\begin{prop}[{\cite[Proposition 5.3]{OT21}}]\label{prop:GL-roots-classification}
Let $\alpha\in\Phi(\G,\bfS)$ be a root of the form $\left[\begin{smallmatrix}1\\g\end{smallmatrix}\right]$ for some $g=\sigma^{k}\phi^{i}$.
The root $\alpha$ is symmetric ramified if and only if $g=\sigma^{\frac{e}{2}}$ ($e$ must be even in this case).
\end{prop}

\begin{lem}\label{lem:tau-ram}
  Suppose that $e$ is even.
  If there exists a $\theta_{\bfS}$-stable toral character of $S$, then $E/E_{\pm}$ must be ramified.
\end{lem}

\begin{proof}
Let $\vartheta$ be a $\theta_{\bfS}$-stable toral character of $S$.
If we let $r\in\R_{>0}$ be the depth of $\vartheta$, then we can take a $\theta_{\bfS}$-stable element $X^{\ast}\in\mfs^{\ast}_{-r}$ representing $\vartheta|_{S_{r}}$ (\cite[Lemma 5.4]{Oi25}; see the discussion in Section \ref{subsec:TCF}).
By the torality of $\vartheta$, $X^{\ast}$ must satisfy Yu's condition \textbf{GE2} (see \cite[Section 8]{Yu01}), which means that $\val_{F}(\langle X^{\ast},H_{\alpha}\rangle)=-r$ for any $\alpha\in\Phi(\G,\bfS)$.

We identify $\mfs^{\ast}\cong E^{\ast}=\Hom_{F}(E,F)$ with $E$ via the $F$-linear isomorphism $[Y\mapsto \Tr_{E/F}(XY)]\leftrightarrow X$.
Write $X$ for the element of $E$ corresponding to $X^{\ast}\in E^{\ast}$ under this identification.
If we write $\alpha=\left[\begin{smallmatrix}g_{i}\\g_{j}\end{smallmatrix}\right]$ as in the above notation, then we have $\langle X^{\ast},H_{\alpha}\rangle=g_{i}(X)-g_{j}(X)$.
Thus, for any $\alpha$ belonging to the $\Gamma$-orbit of $\alpha=\left[\begin{smallmatrix}1\\ g\end{smallmatrix}\right]$ with $g=\sigma^{k}\phi^{i}$, we have $\val_{F}(\langle X^{\ast},H_{\alpha}\rangle)=\val_{F}(X-g(X))$.

Now, for the sake of contradiction, let us suppose that $E/E_{\pm}$ is unramified.
Since $X^{\ast}$ is $\theta_{\bfS}$-invariant and the above identification between $E^{\ast}$ and $E$ is Galois-equivariant, $X$ must satisfy $-\tau_{\theta}(X)=X$.
We write $X=\varpi_{E}^{t}u$ with $t\in\Z$ and $u\in\mcO_{E}^{\times}$.
Then we have $\sigma^{\frac{e}{2}}(X)\equiv(-1)^{t}X \pmod{\mfp_{E}^{t+1}}$.
On the other hand, we have $\tau_{\theta}\sigma^{\frac{e}{2}}(X)\equiv(-1)^{t+1}X \pmod{\mfp_{E}^{t+1}}$.
Since $E/E_{\pm}$ is unramified, $\tau_{\theta}\sigma^{\frac{e}{2}}\neq\id$.
Thus, by considering the condition \textbf{GE2} for $g=\sigma^{\frac{e}{2}}$ and $g=\tau_{\theta}\sigma^{\frac{e}{2}}$, we get 
\[
\val_{F}(X-\sigma^{\frac{e}{2}}(X))
=r
=\val_{F}(X-\tau_{\theta}\sigma^{\frac{e}{2}}(X)).
\]
However, this is impossible because we have
\[
X-\sigma^{\frac{e}{2}}(X)\equiv X-(-1)^{t}X \pmod{\mfp_{E}^{t+1}}
\quad\text{and}
\]
\[
X-\tau_{\theta}\sigma^{\frac{e}{2}}(X)\equiv X-(-1)^{t+1}X \pmod{\mfp_{E}^{t+1}}
\]
and exactly one of these is nonzero.
\end{proof}

\subsection{Computation of spectral transfer factors}
Now we go back to the situation as in Section \ref{sec:TECR}.
Thus the explanation given in the previous subsections are applied to the $F$-rational elliptic twisted maximal torus $(\t\bfS_{k},\bfS_{k})$ of $(\t\G,\G)$.
Recall that we want to show that the quantity \eqref{ramified-spec}, which is given by
\begin{align*}
\epsilon_{\bfS_{k},\ram}(s_{k})
\cdot\epsilon_{\vartheta_{k}}^{\star}(s_{k})
\cdot\zeta_{\desc}(s_{k})^{-1}
\cdot\prod_{\alpha_{\res}\in\dot{\Phi}(\G_{\eta_{0}},\bfS_{k}^{\nat})_{\ram}}f_{(\G_{\eta_{0}},\bfS_{k}^{\nat})}(\alpha_{\res}),
\end{align*}
depends only on the $G$-conjugacy class of $k$ and is independent of $\eta$.

Suppose that $\bfS_{k}$ is isomorphic to $\Res_{E/F}\Gm$, where $E/F$ is a tamely ramified extension of degree $n$.
If the ramification index $e$ of $E/F$ is odd, then so is that of the Galois closure of $E/F$, which implies that $\bfS_{k}$ splits over a finite extension of $F$ with odd ramification index.
Since such a case is already treated in Theorem \ref{thm:spec-well-def-unram}, we assume that $e$ is even in the following.
In particular, by Lemma \ref{lem:tau-ram}, $E/E_{\pm}$ is a ramified quadratic extension with Galois group generated by $\tau_{\theta}=\sigma^{\frac{e}{2}}$.

Note that the toral invariant is always trivial when $\G=\GL_{n}$ (see \cite[Proposition 4.4]{OT21}), hence the character $\epsilon_{\bfS_{k},\ram}$ is trivial.
Moreover, we have

\begin{lem}
If $e$ is even, then the character $\epsilon_{\vartheta_{k}}^{\star}$ is trivial.
\end{lem}

\begin{proof}
Recall that, for any $s\in S_{k}$, $\epsilon_{\vartheta_{k}}^{\star}(s)$ is defined to be the product of $\epsilon_{\alpha}(s)$ over $\ddot{\alpha}\in\ddot{\Xi}(\G,\bfS_{k})$ whose restricted root $\alpha_{\res}$ is ramified.
If the ramification index $e$ of $E/F$ is even, then there exists a ramified symmetric root of $\bfS$ in $\G$ by Proposition \ref{prop:GL-roots-classification}.
Then, as discussed in \cite[Section 4.7]{Kal19} (see also \cite[Section 6.4]{OT21}), the depth $r$ of the toral character $\vartheta_{k}$ of $S_{k}$ is given by $\frac{2s+1}{e}$ for some integer $s$.
However, this implies that the set $\Xi(\G,\bfS_{k})$ of roots appearing in the Heisenberg space is empty (see \cite[Remark 5.10]{OT21}).
Thus we get the assertion.
\end{proof}

Hence we are reduced to investigate the following product:
\begin{align}\label{spectral-GL}
\prod_{\begin{subarray}{c}\ddot{\alpha}\in\ddot{\Phi}(\G,\bfS_{k})_{\asym}\\\alpha_{\res}: \ram \end{subarray}} \epsilon_{\alpha}(s_{k})
\prod_{\begin{subarray}{c}\dot{\alpha}\in\dot{\Phi}(\G,\bfS_{k})_{\ur}\\ \alpha_{\res}: \ram \end{subarray}} \epsilon_{\alpha}(s_{k})
\prod_{\dot{\alpha}_{\res}\in\dot{\Phi}(\G_{\eta_{0}},\bfS_{k}^{\nat})_{\ram}}f_{(\G_{\eta_{0}},\bfS_{k}^{\nat})}(\alpha_{\res}).
\end{align}

\begin{lem}\label{lem:res-toral-inv-GL}
The third product in \eqref{spectral-GL} equals
\[
\prod_{\ddot{\alpha}_{\res}\in\dot{\Phi}(\G_{\eta_{0}},\bfS_{k}^{\nat})_{\ram}^{(\asym)}}f_{(\G_{\eta_{0}},\bfS_{k}^{\nat})}(\alpha_{\res})
\cdot
\prod_{\dot{\alpha}_{\res}\in\dot{\Phi}(\G_{\eta_{0}},\bfS_{k}^{\nat})_{\ram}^{(\ur)}}f_{(\G_{\eta_{0}},\bfS_{k}^{\nat})}(\alpha_{\res}).
\]
\end{lem}

\begin{proof}
Let $\dot{\alpha}\in\dot{\Phi}(\G,\bfS_{k})_{\ram}$ be an element satisfying $\dot{\alpha}_{\res}\in\dot{\Phi}(\G_{\eta_{0}},\bfS_{k}^{\nat})_{\ram}^{(\ram)}$.
Then $\alpha$ is fixed by $\theta_{\bfS}$.
Indeed, we may suppose that $\alpha$ is of the form $\left[\begin{smallmatrix}1\\\sigma^{\frac{e}{2}}\end{smallmatrix}\right]$.
Since $\tau_{\theta}=\sigma^{\frac{e}{2}}$, 
\[
\theta_{\bfS}(\alpha)
=\tau_{\theta}\cdot\begin{bmatrix}1\\\sigma^{\frac{e}{2}}\end{bmatrix}^{-1}
=\sigma^{\frac{e}{2}}\begin{bmatrix}\sigma^{\frac{e}{2}}\\1\end{bmatrix}
=\begin{bmatrix}1\\\sigma^{\frac{e}{2}}\end{bmatrix}
=\alpha.
\]
Hence, we have $f_{(\G_{\eta_{0}},\bfS_{k}^{\nat})}(\alpha_{\res})=f_{(\G,\bfS_{k})}(\alpha)$ as noted in the proof of Proposition \ref{prop:toral-descent-ram-ram}.
Again by using that $f_{(\G,\bfS_{k})}(\alpha)=1$ (\cite[Proposition 4.4]{OT21}), we get $f_{(\G_{\eta_{0}},\bfS_{k}^{\nat})}(\alpha_{\res})=1$.
\end{proof}

\begin{lem}\label{lem:existence-order-2}
There exists an element of $\t{S}_{k}$ of order $2$.
\end{lem}

\begin{proof}
We utilize a realization of $\t{G}$ as the space of bilinear forms as in \cite[Section 1.2]{Wal10} (see also \cite[Section 3.6]{Li13}).

Let $V$ be an $n$-dimensional $F$-vector space equipped with basis $\{e_{i}\}_{i=1,\ldots,n}$.
We let $\t{\theta}$ be a symplectic form on $V$ such that the representation matrix of $\t{\theta}$ with respect to $\{e_{i}\}_{i=1,\ldots,n}$ is $J_{2n}$, i.e., $\t{\theta}(e_{k},e_{l})=(-1)^{k-1}\delta_{k,2n+1-l}$.
Let $\Hom_{F}^{\mathrm{nondeg}}(V\otimes_{F}V,F)$ denote the space of non-degenerate $F$-bilinear forms on $V$.
Note that $\Hom_{F}^{\mathrm{nondeg}}(V\otimes_{F}V,F)$ has a bi-$\GL_{F}(V)$-torsor structure by
\[
(g\cdot q\cdot g')(v,v')
\colonequals q(g^{-1}v,g'v')
\]
for any $q\in\Hom_{F}^{\mathrm{nondeg}}(V\otimes_{F}V,F)$ and $g,g'\in\GL_{F}(V)$.
(Thus we may regard $\t{\theta}$ as a ``base point'' of $\Hom_{F}^{\mathrm{nondeg}}(V\otimes_{F}V,F)$.)
Then $\Hom_{F}^{\mathrm{nondeg}}(V\otimes_{F}V,F)$ is identified with $\t{G}=\GL_{n}(F)\rtimes\theta$ bi-$\GL_{n}(F)$-equivariantly by the following association:
\[
\Hom_{F}^{\mathrm{nondeg}}(V\otimes_{F}V,F)
\leftrightarrow
\GL_{n}(F)\rtimes\theta
\colon\quad g\cdot\t{\theta} \mapsto g\rtimes\theta.
\]

Let us examine how the condition that $g\rtimes\theta$ is of order $2$ can be rephrased on the space $\Hom_{F}^{\mathrm{nondeg}}(V\otimes_{F}V,F)$.
The order of $g\rtimes\theta\in\t{G}$ is $2$ if and only if we have $g\cdot\t{\theta}=\t{\theta}\cdot g^{-1}$.
Let $\iota$ be the involution on the space $\Hom_{F}^{\mathrm{nondeg}}(V\otimes_{F}V,F)$ given by swapping two entries of $V\otimes_{F}V$, i.e.,
\[
\iota(q)(v,v')=q(v',v)
\]
for $q\in\Hom_{F}^{\mathrm{nondeg}}(V\otimes_{F}V,F)$ and $v,v'\in V$.
Then we have $\iota(g\cdot\t{\theta})=-\t{\theta}\cdot g^{-1}$.
Indeed, we have
\[
\iota(g\cdot\t{\theta})(v,v')
=(g\cdot\t{\theta})(v',v)
=\t{\theta}(g^{-1}v',v)
=-\t{\theta}(v,g^{-1}v')
=-(\t{\theta}\cdot g^{-1})(v,v')
\]
for $v,v'\in V$ (we used that $\t{\theta}$ is symplectic in the third equality).
Hence, $g\rtimes\theta$ is of order $2$ if and only if $\iota(g\cdot\t{\theta})=-g\cdot\t{\theta}$, in other words, $g\cdot\t{\theta}$ is symplectic.

Now we note that elements of $\t{S}_{k}$ can be realized in $\Hom_{F}^{\mathrm{nondeg}}(V\otimes_{F}V,F)$ in the following way (\cite[Section 1.3]{Wal10}).
Recall that $S_{k}\cong E^{\times}$ and we have a degree $2$ subextension $E/E_{\pm}$ with Galois group $\langle\tau_{\theta}\rangle$.
For any $x\in E^{\times}$, we define an $F$-bilinear form $\t{x}$ on $E$ by
\[
\t{x}(v,v')\colonequals \Tr_{E/F}(v\tau_{\theta}(v')x).
\]
Then, by choosing an $F$-basis of $E$, we can embed $\{\t{x}\mid x\in E^{\times}\}$ in $\Hom_{F}^{\mathrm{nondeg}}(V\otimes_{F}V,F)$.
This subset realizes $\t{S}_{k}$.

Therefore, in order to show the claim, it suffices to find an element $x\in E^{\times}$ such that $\t{x}$ is symplectic.
If we let $x\in E^{\times}$ be any element satisfying $\Tr_{E/E_{\pm}}(x)=0$, then $\t{x}$ is symplectic.
\end{proof}

\begin{prop}\label{prop:GL-rem-rest-index-sets}
If $e$ is even, then we have natural identifications
\[
\{\ddot{\alpha}\in\ddot{\Phi}(\G,\bfS_{k})_{\asym} \mid \dot{\alpha}_{\res}: \ram\}
\xrightarrow{1:1}
\dot{\Phi}(\G_{\eta_{0}},\bfS_{k}^{\nat})_{\ram}^{(\asym)}
\colon
\ddot{\alpha}\mapsto \dot{\alpha}_{\res},
\]
\[
\{\ddot{\alpha}\in\dot{\Phi}(\G,\bfS_{k})_{\ur} \mid \dot{\alpha}_{\res}: \ram\}
\xrightarrow{1:1}
\dot{\Phi}(\G_{\eta_{0}},\bfS_{k}^{\nat})_{\ram}^{(\ur)}
\colon
\ddot{\alpha}\mapsto \dot{\alpha}_{\res}.
\]
\end{prop}

\begin{proof}
We consider only the case of asymmetric roots with ramified restriction since the case of symmetric unramified roots with ramified restriction can be treated in the same manner.
To show that the association $\ddot{\alpha}\mapsto\dot{\alpha}_{\res}$ gives the asserted identification, we must check the following:
\begin{enumerate}
\item
$\alpha$ and $\theta(\alpha)$ belong to the same class in $\ddot{\Phi}(\G,\bfS_{k})_{\asym}$;
\item
any $\alpha\in\Phi(\G,\bfS_{k})_{\asym}$ whose $\alpha_{\res}$ is ramified descends to $\G_{\eta_{0}}$.
\end{enumerate}
As investigated in the proof of Lemma \ref{lem:ram-res-from-asym}, we must have $\theta(\alpha)\neq\alpha$.
Moreover, if we let $\tau_{\alpha}$ be the nontrivial element of $\Gal(F_{\alpha_{\res}}/F_{\pm\alpha_{\res}})$, then we have $\tau_{\alpha}(\alpha)=-\theta(\alpha)$.
This implies the condition (1).
For the condition (2), we note that $\alpha$ descends to $\G_{\eta_{0}}$ if and only if $\alpha$ descends to $\G_{\eta'_{0}}$ for any topologically semisimple $\eta'_{0}\in\t{S}_{k}$, which is equivalent to $\alpha(\eta^{\prime2}_{0})=1$ (see Lemma \ref{lem:ram-res-from-asym} and its proof).
Any element $\eta'_{0}$ of $\t{S}$ of order $2$, which exists by Lemma \ref{lem:existence-order-2}, satisfies the latter condition.
\end{proof}

By combining Lemma \ref{lem:res-toral-inv-GL} with Proposition \ref{prop:GL-rem-rest-index-sets}, we see that \eqref{spectral-GL} equals the following product:
\begin{align}\label{spectral-GL-2}
\prod_{\dot{\alpha}_{\res}\in\dot{\Phi}(\G_{\eta_{0}},\bfS_{k}^{\nat})_{\ram}^{(\asym)}\sqcup\dot{\Phi}(\G_{\eta_{0}},\bfS_{k}^{\nat})_{\ram}^{(\ur)}}\epsilon_{\alpha}(s_{k})\cdot f_{(\G_{\eta_{0}},\bfS_{k}^{\nat})}(\alpha_{\res}).
\end{align}
Recalling that $\eta=s_{k}\ul{\eta}_{k}$, we see that \eqref{spectral-GL-2} equals
\begin{align}\label{spectral-GL-3}
\prod_{\dot{\alpha}_{\res}\in\dot{\Phi}(\G_{\ul{\eta}_{k}},\bfS_{k}^{\nat})_{\ram}^{(\asym)}\sqcup\dot{\Phi}(\G_{\ul{\eta}_{k}},\bfS_{k}^{\nat})_{\ram}^{(\ur)}}f_{(\G_{\ul{\eta}_{k}},\bfS_{k}^{\nat})}(\alpha_{\res})
\end{align}
by Propositions \ref{prop:toral-descent-asym-ram} and \ref{prop:toral-descent-ur-ram}.

\begin{prop}
If $e$ is even, then \eqref{spectral-GL-3} depends only on the $G$-conjugacy class of $k$ and is independent of $\eta$.
\end{prop}

\begin{proof}
It is obvious that \eqref{spectral-GL-3} is independent of $\eta$.
Let $j'\in\t\mcJ^{\G}_{\G_{\eta}}$ and $k'\in\t\mcJ^{\G_{\eta}}_{G_{\eta}}(j)$ such that $k$ and $k'$ are $G$-conjugate.
Suppose that $g\in G$ be an element such that $k'=[g]\circ k$ and $\ul{\eta}_{k'}={}^{g}\ul{\eta}_{k}$.
Then the $g$-conjugation induces $F$-rational isomorphisms $[g]\colon\G_{\ul{\eta}_{k}}\rightarrow\G_{\ul{\eta}_{k'}}$ and $\bfS_{k}^{\nat}\rightarrow\bfS_{k'}^{\nat}$.
Hence we get the assertion.
\end{proof}

We summarize what we obtained.

\begin{thm}\label{thm:spec-well-def-GL}
The spectral transfer factor $\Delta^{\spec}_{\phi,k}$ is well-defined for $\GL_{n}$ with even $n$.
In particular, the twisted endoscopic character relation (Expectation \ref{expect:TECR}) is satisfied.
\end{thm}

\subsection{A consequence}\label{subsec:Arthur=Kaletha}

\begin{lem}\label{lem:Arthur=Kaletha}
Let $\H$ be either a quasi-split special orthogonal or symplectic group over $F$ which is an endoscopic group of $(\G,\theta)$.
Let $\phi_{\H}$ be a toral supercuspidal $L$-parameter of depth $r\in\R_{>0}$ in the sense of Kaletha (Definition \ref{defn:toral-L-par}).
Suppose that $\mcS_{\phi_{\H}}\colonequals \pi_{0}(\bfZ_{\h\G}(\mathrm{Im}(\phi_{\H}))/\bfZ_{\h\G})$ is trivial.
Then $\h\xi\circ\phi_{\H}$ is toral supercuspidal as an $L$-parameter of $\G$ of depth $r\in\R_{>0}$.
\end{lem}

\begin{proof}
We put $\phi\colonequals \h\xi\circ\phi_{\H}$.
Obviously $\phi$ is trivial on $\SL_2(\C)$.
By the assumption that $\mcS_{\phi_{\H}}$ is trivial, we see that $\phi$ is irreducible as an $n$-dimensional representation of $W_{F}$, which implies that $\phi$ is discrete.
By these observations, to check that $\phi$ is toral of depth $r>0$, it suffices to check that the conditions (1) and (2) of Definition \ref{defn:toral-L-par} are satisfied by $\phi$.
The condition (1) is obviously satisfied, so let us consider (2).
Since $\phi_{\H}$ is toral supercuspidal of depth $r\in\R_{>0}$, $\bfZ_{\h\H}(\phi_{\H}(I_{F}^{r}))$ is a maximal torus of $\h\H$ containing $\phi_{\H}(P_{F})$.
We note that $\bfZ_{\h\G}(\phi(I_{F}^{r}))$ is a Levi subgroup of $\h\G$ by (the proof of) \cite[Lemma 5.2.2 (1)]{Kal19}.
In other words, $\bfZ_{\h\G}(\phi(I_{F}^{r}))$ is a $\h\theta$-stable Levi subgroup of $\h\G$ whose $\bfZ_{\h\H}(\phi_{\H}(I_{F}^{r}))=\bfZ_{\h\G}(\phi(I_{F}^{r}))^{\h\theta,\circ}$ is a maximal torus of $\h\H$.
This implies that the Levi subgroup $\bfZ_{\h\G}(\phi(I_{F}^{r}))$ is necessarily a ($\h\theta$-stable) maximal torus of $\G$.
\end{proof}

Now we arrive at the following consequence.

\begin{thm}\label{thm:Arthur=Kaletha}
Let $\H$ be either a quasi-split special orthogonal or symplectic group over $F$.
Let $\Pi_{\phi_{\H}}^{\H}$ be a toral supercuspidal $L$-packet with $L$-parameter $\phi_{\H}$ in the sense of Kaletha (see Section \ref{sec:LLC}).
Let $\Pi_{\phi_{\H},\Art}^{\H}$ be the $L$-packet of $\H$ corresponding to $\phi_{\H}$ in the sense of Arthur (\cite[Theorem 2.2.1]{Art13}).
Then we have $\Pi_{\phi_{\H}}^{\H}=\Pi_{\phi_{\H},\Art}^{\H}$.
\end{thm}

\begin{proof}
Recall that both $\Pi_{\phi_{\H}}^{\H}$ and $\Pi_{\phi_{\H},\Art}^{\H}$ are bijective to the set of irreducible characters of the ``$S$-group'' $\mcS_{\phi_{\H}}$ (\cite[Section 5.3]{Kal19} and \cite[Theorem 2.2.1]{Art13}).
In particular, these sets have the same cardinality.
We first note that we may assume $|\Pi_{\phi_{\H}}^{\H}|=|\Pi_{\phi_{\H},\Art}^{\H}|=1$ by a standard argument based on the theory of standard endoscopy.
Indeed, suppose that $|\Pi_{\phi_{\H}}^{\H}|=|\Pi_{\phi_{\H},\Art}^{\H}|>1$.
Then the $S$-group contains a nontrivial element, which means that the $L$-parameter $\phi_{\H}$ factors through the $L$-group of a nontrivial standard endoscopic group $\H'$ of $\H$.
Let $\phi_{\H'}$ be an $L$-parameter of $\H'$ such that its lift to $\H$ is $\phi_{\H}$.
By \cite[Theorem 6.3.4]{Kal19} and \cite[Theorem 2.2.1]{Art13}, both $\Pi_{\phi_{\H}}^{\H}$ and $\Pi_{\phi_{\H},\Art}^{\H}$ satisfy the standard endoscopic character relation with $\Pi_{\phi_{\H'}}^{\H'}$ and $\Pi_{\phi_{\H'},\Art}^{\H'}$ at any elliptic strongly regular semisimple element of $H$, respectively.
Therefore, if we can show that $\Pi_{\phi_{\H'}}^{\H'}=\Pi_{\phi_{\H'},\Art}^{\H'}$, then we see that the signed sum of the characters of members of $\Pi_{\phi_{\H}}^{\H}$ coincides with that of $\Pi_{\phi_{\H},\Art}^{\H}$ for any elliptic strongly regular semisimple element of $H$.
Since the strongly regular semisimple locus of $\H$ is Zariski dense in the regular semisimple locus of $\H$, we see that the signed sum of $\Pi_{\phi_{\H}}^{\H}$ and that of $\Pi_{\phi_{\H},\Art}^{\H}$ coincide on the elliptic regular semisimple locus of $H$.
Hence, by the orthogonality relation of the elliptic inner product (\cite[Theorem 3]{Clo91}), we get $\Pi_{\phi_{\H}}^{\H}=\Pi_{\phi_{\H},\Art}^{\H}$.
Since the order of $\Pi_{\phi_{\H'}}^{\H'}$ (or $\Pi_{\phi_{\H'},\Art}^{\H'}$) is smaller than that of $\Pi_{\phi_{\H}}^{\H}$, by repeating this argument inductively, we are reduced to the case where $|\Pi_{\phi_{\H}}^{\H}|=|\Pi_{\phi_{\H},\Art}^{\H}|=1$.
(Note that if $\H$ is quasi-split special orthogonal or symplectic, then so is $\H'$.)

Let us put $\phi\colonequals \h\xi\circ\phi_{\H}$.
When $|\Pi_{\phi_{\H}}^{\H}|=|\Pi_{\phi_{\H},\Art}^{\H}|=1$, or equivalently, $\mcS_{\phi_{\H}}$ is trivial, $\phi$ is a toral supercuspidal $L$-parameter of $\GL_{n}$ by Lemma \ref{lem:Arthur=Kaletha}, where $n$ is such that $\H$ is a twisted endoscopic group of $(\GL_n,\theta)$.
Thus we can apply Theorem \ref{thm:spec-well-def-GL}; $\Pi_{\phi}^{\G}$ and $\Pi_{\phi_{\H}}^{\H}$ satisfy the twisted endoscopic character relation, i.e., we have
\[
\Delta^{\spec}_{\phi,\pi}\Phi_{\t{\pi}}(\delta)
=
\sum_{\gamma\in H/{\st}} \mr{\Delta}(\gamma,\delta)
\Phi_{\pi_{\H}}(\gamma)
\]
for any elliptic strongly regular semisimple element $\delta\in\t{G}$, where $\pi$ and $\pi_{\H}$ are the unique members of $\Pi_{\phi}^{\G}$ and $\Pi_{\phi_{\H}}^{\H}$, respectively.
Similarly, we also have
\[
\Phi_{\t{\pi}_{\Art}}(\delta)
=
\sum_{\gamma\in H/{\st}} \mr{\Delta}(\gamma,\delta)
\Phi_{\pi_{\H,\Art}}(\gamma)
\]
for any elliptic strongly regular semisimple element $\delta\in\t{G}$, where $\pi_{\Art}$ and $\pi_{\H,\Art}$ be the unique members of $\Pi_{\phi,\Art}^{\G}$ and $\Pi_{\phi_{\H},\Art}^{\H}$, respectively.
We note that, for any elliptic strongly regular semisimple element $\delta\in\t{G}$, there exists an elliptic strongly $\G$-regular semisimple element $\gamma\in H$ satisfying $(\gamma,\delta)\in\mcD$ at most uniquely up to stable conjugacy.
In other words, the index set of the above sums can be thought of as a singleton at most.
Moreover, for any elliptic strongly $\G$-regular semisimple element $\gamma\in H$, there exists an elliptic strongly regular semisimple element $\delta\in\t{G}$.
(These facts follow from, e.g., an explicit parametrization of semisimple conjugacy classes of these groups; see \cite[Sections 1.3 and 1.9]{Wal10}.)
Since we have $\pi=\pi_{\Art}$ by \cite{OT21}, we get $\Phi_{\pi_{\H}}(\gamma)=\Delta_{\phi,\pi}^{\spec}\Phi_{\pi_{\H,\Art}}(\gamma)$ for any elliptic strongly $\G$-regular semisimple element $\gamma\in H$ (recall that $\mr{\Delta}(\gamma,\delta)\neq0$ whenever $(\gamma,\delta)\in\mcD$).
As the strongly $\G$-regular semisimple locus of $\H$ is Zariski dense in the regular semisimple locus of $\H$, we see that the identity $\Phi_{\pi_{\H}}(\gamma)=\Delta_{\phi,\pi}^{\spec}\Phi_{\pi_{\H,\Art}}(\gamma)$ holds for any elliptic regular semisimple element $\gamma\in H$.
Therefore, again by the orthogonality relation of the elliptic inner product, we conclude that $\pi_{\H}=\pi_{\H,\Art}$ (and also $\Delta_{\phi,\pi}^{\spec}=1$).
\end{proof}

\begin{rem}
  We note that Arthur's local Langlands correspondence is established only up to the action of outer automorphisms for quasi-split even special orthogonal groups.
  Therefore, when $\H$ is quasi-split even special orthogonal group, the identity $\Pi_{\phi_{\H}}^{\H}=\Pi_{\phi_{\H},\Art}^{\H}$ of the above theorem is only up to the action of outer automorphisms.
\end{rem}

\newpage
\input{symbols.ind}

\end{document}